%% file: rw9.tex
\documentclass[12pt]{amsart}
\usepackage{amscd}
\usepackage{amssymb}
\usepackage{graphicx}
\usepackage{color}
\usepackage[centering,text={15.5cm,22cm},
		marginparwidth=20mm]{geometry}

\usepackage{mathrsfs}

\usepackage[all]{xy}
\usepackage{xypic}

\newtheorem{theorem}[equation]{Theorem}

\newtheorem{lemma}[equation]{Lemma}

\newtheorem{claim}[equation]{Claim}
\newtheorem{proposition}[equation]{Proposition}
\newtheorem{corollary}[equation]{Corollary}
\newtheorem{conjecture}[equation]{Conjecture}
\newtheorem{definition-lemma}[equation]{Definition-Lemma}

\theoremstyle{definition}
\newtheorem{definition}[equation]{Definition}
\newtheorem{question}[equation]{Question}

\newtheorem{construction}[equation]{Construction}

\newtheorem{example}[equation]{Example}

\theoremstyle{remark}
\newtheorem{remark}[equation]{Remark}
\newtheorem{remarks}[equation]{Remarks}

\numberwithin{equation}{section}
\numberwithin{figure}{section}

\renewcommand{\theta}{\fop}

\newcommand{\NN} {\mathbb{N}}
\newcommand{\ZZ} {\mathbb{Z}}
\newcommand{\QQ} {\mathbb{Q}}
\newcommand{\RR} {\mathbb{R}}
\newcommand{\bR} {\RR}

\newcommand{\Conf}{\operatorname{Conf}}
\newcommand{\CC} {\mathbb{C}}

\newcommand{\PP} {\mathbb{P}}
\newcommand{\VV} {\mathbb{V}}
\renewcommand{\AA} {\mathbb{A}}

\newcommand{\invlim}{\varprojlim}
\newcommand {\s}  {{\bf s}}

\newcommand {\mts} {\tilde{\s}}
\newcommand {\shA}  {\mathcal{A}}
\newcommand {\oshA} {\overline{\shA}}

\newcommand {\shC}  {\mathcal{C}}

\newcommand {\shH}  {\mathcal{H}}

\newcommand {\shM}  {\mathcal{M}}

\newcommand {\shO}  {\mathcal{O}}

\newcommand {\shX}  {\mathcal{X}}

\newcommand\scat{\mathrm{scat}}

\newcommand {\foD}  {\mathfrak{D}}

\newcommand {\foT}  {\mathfrak{T}}

\newcommand {\up} {\operatorname{up}}
\newcommand {\midd} {\operatorname{mid}}
\newcommand {\low} {\operatorname{ord}}

\newcommand {\fod}  {\mathfrak{d}}

\newcommand {\fog}  {\mathfrak{g}}

\newcommand {\foj}  {\mathfrak{j}}

\newcommand {\fom}  {\mathfrak{m}}

\newcommand {\fop}  {\mathfrak{p}}

\newcommand {\fov}  {\mathfrak{v}}

\newcommand {\Aut}  {\operatorname{Aut}}

\newcommand {\can}  {\mathrm{can}}

\newcommand {\Ext}  {\operatorname{Ext}}

\newcommand {\GL}  {\operatorname{GL}}

\newcommand {\Gr}  {\operatorname{Gr}}
\newcommand {\CG} {\operatorname{CGr}}
\newcommand {\g} {\mathbf{g}}

\newcommand {\Hom}  {\operatorname{Hom}}

\newcommand {\id}  {\operatorname{id}}

\newcommand {\inc}  {\mathrm{in}}

\newcommand {\Int}  {\operatorname{Int}}

\newcommand {\kk} {\Bbbk}

\newcommand {\liminv} {\varprojlim}

\newcommand {\loc} {\mathrm{loc}}

\newcommand {\M} {\mathcal{M}}

\newcommand {\Mono} {\operatorname{Mono}}

\newcommand {\mch}{\vee} 
\newcommand {\mth} {\operatorname{th}}

\newcommand {\ord}  {\operatorname{ord}}

\newcommand {\PGL}  {\operatorname{PGL}}

\newcommand {\uf} {\mathrm{uf}}
\newcommand {\semf} {\mathrm{sf}}
\newcommand {\groups}{\mathrm{groups}}

\newcommand {\Proj} {\operatorname{Proj}}

\newcommand {\rank} {\operatorname{rank}}

\newcommand {\Scatter} {\operatorname{Scatter}}

\newcommand {\Sing} {\operatorname{Sing}}
\newcommand {\SL}  {\operatorname{SL}}
\newcommand {\Spec} {\operatorname{Spec}}
\newcommand {\Span} {\operatorname{Span}}

\newcommand {\TV}{\operatorname{TV}}

\newcommand {\Supp} {\operatorname{Supp}}

\newcommand  {\todo}[1]{{\marginpar{\tiny #1}}}

\newcommand {\trop} {\mathrm{trop}}

\newcommand {\cA} {\mathcal A}
\newcommand {\cXt} {\cA^{\mch}_{\prin}(\bZ^T)}
\newcommand {\prin} {\operatorname{prin}}
\newcommand {\cAp} {\cA_{\prin}}

\newcommand {\cAg} {\ocA_{\prin}^{\s}}
\newcommand {\cAgs} {\cA_{\prin}}

\newcommand {\cXrt} {\cA^{\mch}_{\prin}(\bR^T)}

\newcommand {\cApdr} {\cA^{\mch}_{\prin}(\bR^T)}

\newcommand {\cAdr} {\cA^{\mch}(\bR^T)}

\newcommand {\LIM}{\overline{\up(\cA_{\prin})}}

\DeclareMathOperator{\cHom}{\mathcal{H}\mathnormal{om}}
\DeclareMathOperator{\lcm}{lcm}

\def\mapright#1{\smash{
  \mathop{\longrightarrow}\limits^{#1}}}

\def\bP{\Bbb P}

\def\SL{\operatorname{SL}}

\def\Aut{\operatorname{Aut}}

\def\Supp{\operatorname{Supp}}

\def\Gr{\operatorname{Gr}}

\def\GL{\operatorname{GL}}
\def\Conv{\operatorname{Conv}}

\def\cO{\Cal O}

\def\rank{\operatorname{rank}}

\def\bZ{\Bbb Z}
\def\bC{\Bbb C}
\def\bQ{\Bbb Q}
\def\bG{\Bbb G}
\def\bR{\Bbb R}
\def\bA{\Bbb A}

\def\ofoD{\overline{\foD}}

\def\oP{\bar{P}}

\def\oD{\bar{D}}

\def\oSigma{\bar{\Sigma}}

\def\tS{\tilde S}

\def\oS{\overline{S}}

\def\tS{\tilde{S}}

\def\oN{\overline{N}}

\def\oA{\overline{A}}

\def\cX{\Cal X}

\def\oV{\overline{V}}

\def\cH{\Cal H}
\def\cM{\Cal M}

\def\Spec{\operatorname{Spec}}
\def\Proj{\operatorname{Proj}}
\def\PGL{\operatorname{PGL}}
\def\tM{\widetilde{M}}

\def\ord{\operatorname{ord}}

\def\cO{\Cal O}

\def\Sing{\operatorname{Sing}}

\def\Spec{\operatorname{Spec}}
\def\Im{\operatorname{Im}}

\def\Hom{\operatorname{Hom}}
\def\Cal{\mathcal}

\def\Efi#1#2#3#4#5{\displaystyle
#1\!\!-\!\!#2
\!\!-\!\!#3
\!\!-\!\!#4
\hskip-24.2pt\lower4.5pt\hbox{${\scriptstyle|}
\hskip-3.35pt\lower6pt\hbox{$#5$}$}}

\def\Evia#1#2#3#4#5{\displaystyle
#1\!\!-\!\!#2
\!\!-\!\!#3
\hskip-24.2pt\lower4.5pt\hbox{${\scriptstyle|}
\hskip-3.35pt\lower6pt\hbox{$#4\!\!-\!\!\!-\!\!\!-\!\!$}$\hskip2.3pt${\scriptstyle|}
\hskip-3.35pt\lower6pt\hbox{$#5$}$}}

\def\Ezia#1#2#3#4{\displaystyle
#1\!\!-\!\!#2
\hskip-14.8pt\lower4.5pt\hbox{${\scriptstyle|}
\hskip-3.35pt\lower6pt\hbox{$#3\!\!-\!\!$}$\hskip2.3pt${\scriptstyle|}
\hskip-3.35pt\lower6pt\hbox{$#4$}$}}

\def\Efia#1#2#3#4#5#6{\displaystyle
#1\!\!-\!\!#2
\!\!-\!\!#3
\!\!-\!\!#4
\hskip-24.2pt\lower4.5pt\hbox{${\scriptstyle|}
\hskip-3.35pt\lower6pt\hbox{$#5$}$\hskip5.7pt${\scriptstyle|}
\hskip-3.35pt\lower6pt\hbox{$#6$}$}}

\def\Esi#1#2#3#4#5#6{\displaystyle
#1\!\!-\!\!#2
\!\!-\!\!#3
\!\!-\!\!#4\!\!-\!\!#5
\hskip-24.2pt\lower4.5pt\hbox{${\scriptstyle|}
\hskip-3.35pt\lower6pt\hbox{$#6$
\lower3pt\hbox{\ }}$}}

\def\Esia#1#2#3#4#5#6#7{\displaystyle

#1\!\!-\!\!#2
\!\!-\!\!#3
\!\!-\!\!#4\!\!-\!\!#5
\hskip-24.2pt\lower4.5pt\hbox{${\scriptstyle|}
\hskip-3.35pt\lower6pt\hbox{$#6$\hskip-3.8pt\lower4.5pt\hbox{${\scriptstyle|}
\hskip-3.35pt\lower6pt\hbox{$#7$}$}}
\lower3pt\hbox{\ }$}}

\def\Ese#1#2#3#4#5#6#7{\displaystyle
#1\!\!-\!\!#2
\!\!-\!\!#3
\!\!-\!\!#4\!\!-\!\!#5\!\!-\!\!#6
\hskip-33.6pt\lower4.5pt\hbox{${\scriptstyle|}
\hskip-3.35pt\lower6pt\hbox{$#7$
\lower3pt\hbox{\ }
}$}}

\def\Esea#1#2#3#4#5#6#7#8{\displaystyle
#1\!\!-\!\!#2
\!\!-\!\!#3
\!\!-\!\!#4\!\!-\!\!#5\!\!-\!\!#6\!\!-\!\!#7
\hskip-33.6pt\lower4.5pt\hbox{${\scriptstyle|}
\hskip-3.35pt\lower6pt\hbox{$#8$
\lower3pt\hbox{\ }
}$}}

\def\Eei#1#2#3#4#5#6#7#8{\displaystyle
#1\!\!-\!\!#2
\!\!-\!\!#3
\!\!-\!\!#4\!\!-\!\!#5\!\!-\!\!#6\!\!-\!\!#7
\hskip-43.2pt\lower4.5pt\hbox{${\scriptstyle|}
\hskip-3.35pt\lower6pt\hbox{$#8$
\lower3pt\hbox{\ }
}$}}

\def\Eeia#1#2#3#4#5#6#7#8#9{{\displaystyle
#1\!\!-\!\!#2
\!\!-\!\!#3
\!\!-\!\!#4\!\!-\!\!#5\!\!-\!\!#6\!\!-\!\!#7\!\!-\!\!#8
\hskip-52.2pt\lower4.5pt\hbox{${\scriptstyle|}
\hskip-3.35pt\lower6pt\hbox{$#9$
\lower3pt\hbox{\ }
}$}}}

\def\trop{\operatorname{trop}}
\def\prin{\operatorname{prin}}

\def\os{\overline{S}}

\def\oXi{\overline{\Xi}}

\def\ocA{\overline{\Cal A}}

\def\tN{\widetilde{N}}
\def\tn{\tilde{n}}
\def\tK{\widetilde{K}}
\def\tS{\tilde{S}}
\def\ts{\tS}
\def\ts7{\tilde{S}_7}

\def\mth{\operatorname{th}}

\def\tv{\tilde{v}}

\def\bP{\Bbb P}

\def\cC{\Cal C}

\def\Aut{\operatorname{Aut}}

\def\Supp{\operatorname{Supp}}

\def\Gr{\operatorname{Gr}}

\def\Hom{\operatorname{Hom}}

\def\cO{\Cal O}

\def\rank{\operatorname{rank}}

\def\bZ{\Bbb Z}
\def\bC{\Bbb C}
\def\bQ{\Bbb Q}
\def\bG{\Bbb G}
\def\bR{\Bbb R}
\def\bA{\Bbb A}

\def\tS{\tilde S}

\def\oS{\overline{S}}

\def\tS{\tilde{S}}

\def\oN{\overline{N}}

\def\cX{\Cal X}
\def\cH{\Cal H}
\def\cM{\Cal M}

\def\Spec{\operatorname{Spec}}
\def\Proj{\operatorname{Proj}}
\def\PGL{\operatorname{PGL}}

\def\Im{\operatorname{Im}}

\def\ord{\operatorname{ord}}

\def\Sing{\operatorname{Sing}}
\def\Cal{\mathcal}

\def\oN{\overline{N}}
\def\os7p{\oS_7'}
\def\os7{\oS_7}
\def\on6{\oN_6}
\def\n6{\oN_6}

\def\Mono{\operatorname{Mono}}

\def\brg0{(\bR_{\geq 0})}

\def\mydate{\ifcase\month \or January\or February\or March\or
April\or May\or June\or July\or August\or September\or October\or 
November\or December\fi \space\number\day,\space\number\year}

\begin{document}

\title[Canonical bases for cluster algebras]
{Canonical bases for cluster algebras}
\author{Mark Gross} 
\address{DPMMS, Centre for Mathematical Sciences,
Wilberforce Road, Cambridge, CB3 0WB, United Kingdom}
\email{mgross@dpmms.cam.ac.uk}
\author{Paul Hacking}
\address{Department of Mathematics and Statistics, Lederle Graduate
Research Tower, University of Massachusetts, Amherst, MA 01003-9305}
\email{hacking@math.umass.edu}

\author{Sean Keel}
\address{Department of Mathematics, 1 University Station C1200, Austin,
TX 78712-0257}
\email{keel@math.utexas.edu}
\author{Maxim Kontsevich}
\address{IH\'ES, Le Bois-Marie 35, route de Chartres, 91440 Bures-sur-Yvette,
France} 
\email{maxim@ihes.fr}
\date{October 3, 2016}

\begin{abstract} In \cite{GHK11}, Conjecture 0.6, the first three authors conjectured 
that the ring of regular 
functions on a natural class of affine log Calabi-Yau varieties 
(those with maximal
boundary) has a canonical vector space basis parameterized by the integral tropical
points of the mirror. Further, the structure constants for the multiplication rule in this basis should be
given by counting broken lines (certain combinatorial objects, morally the tropicalisations of
holomorphic discs). 

Here we prove the conjecture 
in the case of cluster varieties, where the statement is
a more precise form of the Fock-Goncharov dual basis conjecture, \cite{FG06},
Conjecture 4.3. In particular, 
under suitable hypotheses, for each $Y$ the partial compactification
of an affine cluster variety $U$ given by allowing some frozen 
variables to vanish, 
we obtain canonical bases for $H^0(Y,\cO_Y)$ extending to a basis
of $H^0(U,\cO_U)$.
Each choice of seed canonically identifies the parameterizing sets of these
bases 
with integral points in a polyhedral cone. These results specialize
to basis results
of combinatorial representation theory. For example, 
by considering the open double
Bruhat cell $U$ in the basic affine space $Y$ we obtain a canonical basis of
each irreducible representation of $\SL_r$, parameterized by a set which each
choice of seed identifies with integral points of a lattice polytope.
These bases and polytopes are
all constructed essentially without representation theoretic considerations.

Along the way, our methods prove a number of conjectures in cluster theory,
including positivity of the Laurent phenomenon for cluster algebras of
geometric type.
\end{abstract}

\maketitle
\tableofcontents
\bigskip


\section*{Introduction} 
\subsection{Statement of the main results}

Fock and Goncharov conjectured that the algebra of
functions on a cluster variety has a canonical vector space basis parameterized by the tropical
points of the mirror cluster variety. Unfortunately, as shown in 
\cite{P1} by the first three authors of
this paper, this conjecture is usually false: in general 
the cluster variety may have far too few global functions. One can only expect a power
series version of the conjecture, holding in the ``large complex structure 
limit,'' and
honest global functions parameterized by a subset of the mirror tropical points. For the
conjecture to hold as stated, one needs further affineness assumptions. 
Here we apply methods developed in the study of mirror symmetry,
in particular {\it scattering diagrams}, introduced by Kontsevich and Soibelman
in \cite{KS06} for 
two dimensions and by Gross and Siebert in \cite{GSAnnals} for all dimensions,
{\it broken lines}, introduced by Gross in \cite{G09} and developed further 
by Carl, Pumperla and Siebert in
\cite{CPS}, and {\it theta 
functions}, introduced by Gross, Hacking,  Keel and Siebert, see 
\cite{GHK11}, \cite{CPS}, \cite{GS12}, and \cite{GHKS}, to prove the
conjecture in this corrected form. We give in addition a formula for the
structure constants in this basis, non-negative integers given by counts of 
broken lines. Definitions of all these objects, essentially combinatorial
in nature, in the context of cluster algebras will be given in later sections. 
Here are more precise statements of our results. 

For basic cluster variety notions we follow
the notation of \cite{P1}, \S 2, for convenience, as we have collected
there a number of definitions across the literature; nothing there is
original. We recall some of this notation in Appendices 
\ref{LDsec} and \ref{rdsec}. 
The various flavors of cluster varieties are all varieties of the form
$V=\bigcup_{\s} T_{L,\s}$, where $T_{L,\s}$ is a copy of the algebraic torus
\[
T_L := L \otimes_{\bZ} \bG_m = \Hom(L^*,\bG_m)=\Spec \kk[L^*]
\]
over a field $\kk$ of characteristic zero, and $L = \bZ^n$ is a lattice,
indexed by $\s$ running over a set of \emph{seeds} (a seed being
roughly an ordered basis for $L$). The birational 
transformations induced by the inclusions of two different copies of
the torus are compositions of \emph{mutations}.
Fock and Goncharov introduced a simple way
to dualize the mutations, and using this define the 
\emph{Fock-Goncharov dual}\footnote{Roughly one can view the Fock-Goncharov
dual as the mirror variety, but this is not always precisely the case.
With some additional effort, one can make this precise ``at the boundary,''
but we shall not do so here.},
$V^{\vee}=\bigcup_{\s} T_{L^*,\s}$. We write $\bZ^T$ for the
tropical semi-field of
integers under $\max,+$. There is a notion of the set of
$\bZ^T$-valued points of $V$, written as 
$V(\bZ^T)$. This can also be viewed as being canonically in bijection with 
$V^{\trop}(\ZZ)$,
the set of divisorial discrete valuations on the field of rational functions
of $V$ where the canonical volume form has a pole, see \S \ref{tropsec}. 
Each choice of seed $\s$ determines an identification $V(\bZ^T) = L$. 

Our main object of study
is the \emph{$\cA$ cluster variety with principal coefficients}, 
$\cA_{\prin}=\bigcup_{\s} T_{\tN^{\circ},\s}$. (See Appendices \ref{LDsec}
and \ref{rdsec} for notation.)
This comes with a canonical fibration over a torus $\pi: \cA_{\prin} \to T_M$,
and a canonical free action by a torus $T_{N^\circ}$. We let $\cA_t := 
\pi^{-1}(t)$. The fibre $\cA_e \subset \cA_{\prin}$
($e \in T_M$ the identity) 
is the Fock-Goncharov $\cA$ variety (whose algebra of regular functions is the Fomin-Zelevinsky
upper cluster algebra). 
The quotient $\cA_{\prin}/T_{N^\circ}$ is the Fock-Goncharov $\cX$ variety. 

\begin{definition}
\label{globalmonomialdef}
A  \emph{global monomial} on a cluster variety $V = \bigcup_{\s \in S}
T_{L,\s}$  
is a regular function on $V$ which restricts to a character 
on some torus $T_{L,\s}$ in the atlas. For $V$ an 
$\cA$-type cluster variety a global monomial
is the same as a cluster monomial. One defines the \emph{upper
cluster algebra} $\up(V)$ associated to $V$ by
$\up(V) := \Gamma(V,\cO_V)$, and the \emph{ordinary cluster algebra}
$\ord(V)$
to be the subalgebra of $\up(V)$ generated by global monomials. 
\end{definition}

For example,
$\ord(\cA)$ is the original cluster algebra defined by Fomin and Zelevinsky
in \cite{FZ02a}, and $\up(\cA)$ is the corresponding upper cluster algebra
as defined in \cite{BFZ05}.

Given a global monomial $f$ on $V$, there is a seed $\s$
such that $f|_{T_{L,\s}}$ is a character $z^m$, $m\in L^*$. Because 
the seed $\s$ gives an identification of $V^{\vee}(\ZZ^T)$
with $L^*$, we obtain an element $\g(m)\in V^{\vee}(\ZZ^T)$, which we show 
is well-defined (independent of the open set $T_{L,\s}$), see 
Lemma \ref{fgccthg2}.  This is the 
\emph{$g$-vector} of the global
monomial $f$. We show this notion of $g$-vector coincides with
the notion of $g$-vector from \cite{FZ07} in the $\cA$ case, see Corollary
\ref{gvecthm}.
Let $\Delta^+(\ZZ) \subset V^{\mch}(\bZ^T)$ be the set of $g$-vectors of
all global monomials on $V$.
Finally, we write $\can(V)$ for the $\kk$-vector space with basis 
$V^{\mch}(\bZ^T)$, i.e., 
\begin{equation}
\label{canVdef}
\can(V) := \bigoplus_{q \in V^{\mch}(\bZ^T)} \kk \cdot \vartheta_q
\end{equation}
(where $\vartheta_q$ for the moment indicates the abstract
basis element corresponding 
to $q \in V^{\mch}(\bZ^T)$).

Fock and Goncharov's dual basis conjecture says that $\can(V)$ 
is canonically identified with the vector space 
$\up(V)$, and so in particular $\can(V)$ should have a canonical 
$\kk$-algebra structure. 
Note that such an algebra structure is determined by its structure constants,
a function
\[
\alpha: V^{\mch}(\bZ^T) \times V^{\mch}(\bZ^T) \times V^{\mch}(\bZ^T) \to \kk
\]
such that for fixed $p,q$, $\alpha(p,q,r) = 0$ for all but finitely many $r$ 
and
\[
\vartheta_p \cdot \vartheta_q = \sum_{r} \alpha(p,q,r) \vartheta_r.
\]

With this in mind, we have:

\begin{theorem} \label{mainth} Let $V$ be one of $\cA,\cX,\cA_{\prin}$. 
The following hold:
\begin{enumerate}
\item There are canonically defined
non-negative {\it structure constants} 
\[
\alpha: V^{\mch}(\bZ^T) \times V^{\mch}(\bZ^T) \times V^{\mch}(\bZ^T) 
\to \bZ_{\geq 0}\cup \{\infty\}.
\]
These are given by counts of \emph{broken lines}, certain combinatorial 
objects which we will define. The value $\infty$ is not taken in the
$\cX$ or $\cA_{\prin}$ case.
\item There is a canonically defined subset 
$\Theta \subset V^{\mch}(\bZ^T)$ with $\alpha(\Theta\times\Theta\times\Theta)
\subseteq \ZZ_{\ge 0}$
such that the restriction of $\alpha$ 
gives the vector subspace $\midd(V) \subset \can(V)$ with basis indexed by 
$\Theta$ the structure of an associative commutative $\kk$-algebra.
\item $\Delta^+(\ZZ) \subset \Theta$, i.e., 
$\Theta$ contains the $g$-vector of each global monomial.
\item For the lattice structure on $V^{\mch}(\bZ^T)$ determined by any choice of seed, 
$\Theta \subset V^{\mch}(\bZ^T)$ is closed under addition. 
Furthermore, $\Theta \subset V^{\mch}(\bZ^T)$
is saturated: for $k >0$ and $x \in V^{\mch}(\bZ^T)$,  $k \cdot x \in \Theta$ if
and only if $x \in \Theta$. 
\item There is a canonical $\kk$-algebra map $\nu: \midd(V) \to \up(V)$ 
which sends $\vartheta_q$ for 
$q \in \Delta^+(\ZZ)$ to the corresponding global monomial. 
\item The image $\nu(\vartheta_q) \in \up(V)$ is a universal positive Laurent 
polynomial (i.e., a Laurent
polynomial with non-negative integral coefficients in the cluster variables for each seed).
\item $\nu$ is injective for $V= \cA_{\prin}$ or $V=\cX$. Furthermore,
$\nu$ is injective for $V=\cA$ under the 
additional assumption that there 
is a seed $\s = (e_1,\dots,e_n)$ for which all the covectors $\{e_i,\cdot\}$, $i \in I_{\uf}$, 
lie in a strictly convex cone. 
When $\nu$ is injective we have canonical inclusions
\[
\ord(V) \subset \midd(V) \subset \up(V). 
\]
\end{enumerate}
\end{theorem}

There is an analog
to Theorem \ref{mainth} for $\cA_t$ (the main difference is that the 
\emph{theta functions}, i.e., the
canonical basis for $\midd(\cA_t)$, are only defined up to scaling each individual element, and the 
structure constants will not in general be integers). Injectivity in (7) holds for very general $\cA_t$.
See Theorem \ref{mainthax}.

Note that (5-6) immediately imply: 

\begin{corollary}[Positivity of the Laurent Phenomenon]
\label{posLP}
Each cluster variable of an $\cA$-cluster
algebra is a Laurent polynomial with non-negative integer coefficients in the cluster variables of any given seed.
\end{corollary}

This was conjectured by Fomin and Zelevinsky in their original paper
\cite{FZ02a}.
Positivity was obtained independently in the skew-symmetric case
by [LS13], by an entirely different argument. In our proof the 
positivity in (1) and (6) both come from positivity in the scattering diagram, a powerful
tool fundamental to the entire paper.
See Theorem \ref{scatdiagpositive}.

We conjecture that injectivity in (7) holds for all $\cA_t$ (without the convexity assumption). Note (7)
includes the linear independence of cluster monomials, which has already been established (without 
convexity assumptions) for skew-symmetric cluster algebras in \cite{CKLP}, 
by a very
different argument. The linear independence of cluster monomials in the
principal case
also follows easily from our scattering diagram technology, as pointed out to
us by Greg Muller. See Theorem \ref{properlaurentproperty}.

When there are frozen variables, one obtains
a partial compactification $V \subset \oV$ (where the frozen variables are allowed to
take the value $0$) for $V = \cA$, $\cA_{\prin}$ or $\cA_t$. 
The notions of $\ord$, $\up$, $\can$, and
$\midd$ extend naturally to $\oV$.  See Construction \ref{fvss}.

Of course if $\ord(V) = \up(V)$, and we have injectivity in (7), $\ord(V) = \midd(V) = \up(V)$
has a canonical basis $\Theta$ with the given properties. Also,
$\ord(V) = \up(V)$ implies, under certain hypotheses, $\ord(\oV) = \up(\oV)$,
see Lemma \ref{midocalem}. Such partial compactifications are
essential for representation-theoretic applications:

\begin{example}
\label{SLrexample}
Let $G=\SL_r$. Choose a Borel subgroup $B$ of $G$, $H\subset B$ a maximal
torus, and let $N=[B,B]$ be the unipotent radical of $B$. These choices determine a cluster
variety structure (with frozen variables) on $\ocA = G/N$, with
$\up(\ocA) = \ord(\ocA) = \cO(G/N)$, the ring of regular functions on $G/N$, 
see \cite{GLS}, \S 10.4.2. 

Theorem \ref{mainth} implies that 
these choices canonically determine a vector space basis  
$\Theta \subset \cO(G/N)$. Each
basis element is an $H$-eigenfunction for the natural (right) action of $H$ 
on $G/N$.
For each character $\lambda \in \chi^*(H)$, $\Theta \cap \cO(G/N)^{\lambda}$ is 
a basis of the weight space $\cO(G/N)^{\lambda} =: V_{\lambda}$. The 
$V_{\lambda}$ are the collection
of irreducible representations of $G$, each of which thus inherits a basis, 
canonically determined by the choice of $H \subset B \subset G$.

We give, combining our results
with results of T.\ Magee, much more precise results, 
see Corollary \ref{slrcor} below. 

Canonical bases for $\cO(G/N) $ have been constructed by Lusztig.
Here we will obtain bases by a procedure very
different from Lusztig's, as a special case of the more general
\cite{GHK11}, Conjecture 0.6, which applies in theory to
any variety with the right sort of volume form. See Remark
\ref{clusterindependenceremark} for further commentary on this.
\qed
\end{example}

The tools necessary for the proof of Theorem \ref{mainth} are
developed in the first six sections of the paper, with the proof given
in \S\ref{midclustersection}. This material is summarized in more detail
in \S\ref{intro1.2}.

The second part of the paper turns to criteria for the full Fock-Goncharov
conjecture to hold. Precisely:

\begin{definition}
\label{fgconjdef}
We say the \emph{full Fock-Goncharov conjecture holds} for a cluster variety
$V$ if the map $\nu:\midd(V)\rightarrow\up(V)$ of Theorem \ref{mainth}
is injective, 
\[
\hbox{$\up(V)=\can(V)$, and $\Theta=V^{\vee}(\ZZ^T)$}.
\]
Note this implies $\midd(V)=\up(V) = \can(V)$.
\end{definition}

We prove a number of criteria which guarantee the full Fock-Goncharov conjecture
holds. One such condition, which seems to be very natural in our setup and
is implied, say, by the existence of a maximal green sequence, is:

\begin{proposition}
[Proposition \ref{egmscprop}] 
If the set $\Delta^+(\ZZ)$ of all $g$-vectors
of global monomials of $\shA$ in $\shA^{\vee}(\ZZ^T)$ is not
contained in a half-space under the identification of $\shA^{\vee}(\ZZ^T)$
with $M^{\circ}$ induced by some choice of seed, then the full 
Fock-Goncharov conjecture 
holds for $\cA_{\prin}$, $\cX$, very general $\cA_t$ and, if the convexity condition (7) of Theorem \ref{mainth} holds, for $\cA$. 
\end{proposition}

Many of the results in the second part of the paper are proved using 
a generalized notion of convex function or convex polytope, see 
\S\S\ref{intro1.3} and \ref{intro1.4} for more details.

In \S\ref{candsec}, we turn to results on partial compactifications. 
We first explain how convex polytopes in our sense give rise, under
suitable hypotheses, to compactifications of $\cA$-type cluster varieties
and toric degenerations of such. This connects our constructions to
the mirror symmetry picture described in \cite{GHK11}, and in particular
describes a partial compactification of $\cA_{\prin}$ as giving a degeneration
of a family of log Calabi-Yau varieties to a toric variety. Partial
compactifications via frozen variables are also important in representation
theoretic applications, as already indicated in Example \ref{SLrexample}.
We prove results for such partial compactifications which,
combined with recent results of T.\ Magee \cite{Magee}, \cite{TimThesis},
yield strong representation-theoretic results, see \S\ref{intro1.4} for
more details.

We now turn to a more detailed summary of the contents of the paper.

\subsection{Towards the main theorem.}
\label{intro1.2}

\S\ref{scatdiagsection} is devoted to the construction of the fundamental
tool of the paper, \emph{scattering diagrams}. While \cite{GSAnnals}
defined these in much greater generality, here they are 
collections of walls living in a vector space
with attached functions constructed canonically from a choice of seed data.
A precise definition can be found in \S\ref{defconstsection}.
Here we simply highlight the main new result 
Theorem \ref{scatdiagpositive}, whose proof, being fairly technical,
is deferred to Appendix \ref{scatappendix}. This says that the functions
attached to walls of a scattering diagram associated to seed data have
positive coefficients. All positivity results in this paper flow from this
fundamental observation, and indeed many of our arguments use this in an essential way. 
For the reader's convenience, we give in 
\S\ref{consistentscatdiagramsection} an elementary construction of
the relevant scattering diagrams, drawing on the method given in 
\cite{KS13}. Since a scattering diagram depends on a choice of
seed, \S\ref{mutinvsec} shows how scattering diagrams associated to
mutation equivalent seeds are related. This shows that a scattering
diagram has a chamber structure indexed by seeds mutation equivalent
to the initial choice of seed. 

In \S\ref{tropsec}, we review some notions of tropicalizations of cluster
varieties, showing that scattering diagrams naturally live in such
tropicalizations. Indeed, the scattering diagram which is associated to
a cluster variety $V$ lives naturally in the tropical space of the
Fock-Goncharov dual $V^{\vee}(\RR^T)$.
These tropicalizations, crucially, can only be viewed
as piecewise linear, rather than linear, spaces, with a choice of seed
giving an identification of the tropicalization with a linear space. Already
the mutation combinatorics becomes apparent:

\begin{theorem}[Lemma \ref{fgcclem} and Theorem \ref{fgccth}] 
\label{fgcclemth}
For each seed $\s = (e_1,\dots,e_n)$ of a $\cA$-cluster 
variety, the \emph{(Fock-Goncharov) cluster chamber} associated to $\s$ is
\[
\shC^+_{\s} := \{ x \in \cA^{\vee}(\bR^T)\,|\, (z^{e_i})^T(x) \le 0 \text{ for all } 
1 \leq i \leq n\},
\]
where $(z^{e_i})^T$ denotes the tropicalization of the monomial $z^{e_i}$,
see \S\ref{tropsec}.
The collection $\Delta^+$ of such subsets of $\cA^{\vee}(\RR^T)$ 
over all mutation equivalent seeds form the maximal cones of a 
simplicial fan, the \emph{(Fock-Goncharov) cluster complex}. 
The Fomin-Zelevinsky exchange graph is the dual graph of this fan.
\end{theorem}

The collection of cones $\Delta^+$ was introduced by Fock and Goncharov, 
who conjectured
they formed a fan. It is not at all obvious from the definition that the interiors of 
the cones cannot overlap. Our description of the chamber structure induced
by a scattering diagram in fact shows that part of the chamber
structure coincides with the collection of cones $\Delta^+$. This shows
the fact they form a fan directly. In addition,
the set $\Delta^+(\ZZ)$ of Theorem \ref{mainth} consists of the integral points
of the union of cones in $\Delta^+$.

\S\ref{blsec} gives the definition of broken line, the second principal
combinatorial tool of the paper. These were originally introduced in
\cite{G09} and developed further in \cite{CPS} as tropical replacements for
Maslov index two disks. In \cite{GHK11}, they were used to define 
\emph{theta functions}, which are, in principle, formal sums over
all broken lines with fixed boundary conditions. 
The relevance of theta functions for us comes in \S\ref{basection}. Here
we show the direct relationship between scattering diagrams and the
$\cA$ cluster algebra. We show that if we associate a suitable torus
$T_L$ to each
chamber of the scattering diagram associated to a mutation of the initial
seed, then the walls separating the chambers can be interpreted as giving
birational
maps between these tori. Gluing together these copies of $T_L$
gives the $\cA$ cluster variety, see Theorem \ref{rpth}.
Further, a theta function $\vartheta_p$ depends on a point
$p\in\cA^{\vee}(\ZZ^T)$. If for a given choice of $p$, $\vartheta_p$
is in fact a finite sum, then $\vartheta_p$ is a global function on
$\cA$. We show that this holds in particular when $p$ lies in the cluster
complex $\Delta^+$, and in this case $\vartheta_p$ agrees with the cluster
monomial with $g$-vector given by $p$.
Because of the positivity result Theorem \ref{scatdiagpositive},
$\vartheta_p$ is in any event always a power series with positive coefficients.
Thus we get positivity of the Laurent phenomenon, Theorem \ref{pilpth}, as
an easy consequence of our formalism.

In \S\ref{cgvectorsec} we begin with what is another essential observation
for our approach.
A choice of initial seed $\s$ provides a partial compactification
$\ocA_{\prin}^{\s}$ of $\cA_{\prin}$ by allowing the variables $X_1,\ldots,
X_n$ (the principal coefficients) to be zero. These variables induce a flat map
$\pi: \ocA_{\prin}^{\s}\rightarrow \AA^n_{X_1,\ldots,X_n}$, with $\cA$ being
the fibre over $(1,\ldots,1)$. Our methods easily show:

\begin{theorem} \label{cftth} (Corollary \ref{mutcor}, (1)) 
The central fibre $\pi^{-1}(0) \subset \ocA_{\prin}^{\s}$ is the algebraic
torus $T_{N^{\circ},\s}$. 
\end{theorem}

Though immediate from our scattering diagram methods, the result is not obvious from the
original definitions: indeed, it is equivalent to the sign-coherence of $c$-vectors, see
Corollary \ref{sccveccor}.  

The last major ingredient in the proof of Theorem \ref{mainth} is a formal
version of the Fock-Goncharov conjecture. As mentioned above, 
this conjecture does not hold in general, 
but in \S\ref{formalFGcsec}, we show that the Fock-Goncharov conjecture holds 
in a formal neighbourhood of the torus fibre of $\ocA^{\s}_{\prin}\rightarrow
\AA^n$. 
We show the structure constants given in Theorem \ref{mainth}, (1),
have a tropical interpretation and
determine an associative product on $\can(\cA_{\prin})$, except that 
$\vartheta_p \cdot \vartheta_q$ 
will in general be an infinite sum of theta functions. Further,
canonically associated to each universal Laurent polynomial $g \in 
\up(\cA_{\prin})$ is a formal power series  
$\sum_{q \in \cA_{\prin}^{\vee}(\bZ^T)} \alpha_q \vartheta_q$ which converges
to $g$ in a formal neighbourhood of the central fibre. 
For the precise statement 
see Theorem \ref{ffgth}, which we interpret as saying that the 
Fock-Goncharov dual basis conjecture  always 
holds in the {\it large complex structure limit}.
This is all one should expect from log Calabi-Yau mirror symmetry, 
in the absence of further affineness assumptions.
A crucial
point, shown in the proof of Theorem \ref{ffgth}, is that the expansion
of $g\in \up(\cA_{\prin})$ is independent of the choice of seed $\s$
determining the compactification $\ocA_{\prin}^{\s}$, i.e., is independent
of which degeneration is used to perform the expansion.

In \S\ref{midclustersection}, we introduce the middle cluster
algebra $\midd(\cA_{\prin})$. The idea is that while we don't know that every 
regular function on $\cA_{\prin}$ can be written as a linear
combination of theta functions, there is a set $\Theta\subset
\cA_{\prin}^{\vee}(\ZZ^T)$ indexing those $p$ for which $\vartheta_p$ is a 
regular function on $\cA_{\prin}$. These in fact yield a vector space basis
for a subalgebra of $\up(\cA_{\prin})$ which necessarily includes all
cluster monomials, hence includes the ordinary cluster algebra. With this
in hand, Theorem \ref{mainth} becomes a summary of the results proved up to
this point. We then deduce the result for $\cX$ and $\cA$-type cluster
varieties from the $\cA_{\prin}$ case.

\subsection{Convexity conditions}
\label{intro1.3}

We now turn to the use of convexity conditions to prove the Fock-Goncharov
conjecture in a number of different situations, as covered in
\S\ref{ppsec}. To motivate the concepts,
let us define a \emph{partial minimal model} of a log Calabi-Yau variety
$V$.
This is an inclusion $V\subset Y$ as an open subset 
such that the canonical volume form on $V$ has a simple pole along each
irreducible divisor of the
boundary $Y \setminus V$. For example, 
a partial minimal model for an algebraic torus 
is the same as a toric compactification. We wish to extend
elementary constructions
of toric geometry to the cluster case. For example,
the partial compactification $\cA \subset \ocA$ determined
by frozen variables is a partial minimal model. 

The generalisation of the cocharacter lattice $N \subset N_{\bR}$ of the 
algebraic torus $T_N := N \otimes \bG_m$
is the tropical set $V(\bZ^T) \subset V(\bR^T)$ of $V$. 
The main difference between the torus 
and the general case is that $V(\bR^T)$ is not in general a vector space. Indeed,
the identification of $V(\ZZ^T)$ with the cocharacter lattices
of various charts of $V$ induce piecewise linear (but not linear) identifications between
the cocharacter lattices. As a result, a piecewise straight path in $V(\bR^T)$ 
which is straight under 
one identification $V(\bR^T) = N_{\bR}$ will be bent under another. 
Thus the usual notions of straight lines, convex functions or convex sets do  
not make sense on $V(\RR^T)$. 

The idea for generalizing the notion of convexity is to instead make
use of broken lines, which are piecewise linear paths in $V(\RR^T)$.
Using broken lines in place of straight lines we can say which piecewise linear
functions, and thus which polytopes, are convex, see Definition 
\ref{imcdef}. 
Each regular function $W: V \to \bA^1$ has a canonical piecewise linear 
tropicalisation
$w:=W^T: V(\bR^T) \to \bR$, which we conjecture is convex in the sense
of Definition \ref{imcdef}, see Conjecture \ref{dcon}.  The conjecture is easy 
for $W \in \ord(V) \subset \up(V)$, see Proposition \ref{convcor}. 
Each {\it convex} piecewise linear $w$ gives a \emph{convex}
polytope $\Xi_w = \{x\,|\, w(x) \geq -1 \}$ and a \emph{convex}
cone $\{x \in V^{\mch}(\bR^T)\,|\, w(x) \geq 0 \}$, where italics 
indicates convexity in our broken line sense. 
We believe the existence of a 
bounded polytope is equivalent to the full Fock-Goncharov conjecture:

\begin{conjecture} \label{ffgconj} The full Fock-Goncharov conjecture holds for $\cA_{\prin}$ 
if and only if
the tropical space $\cA^{\vee}_{\prin}(\bR^T)$ contains a full dimensional bounded polytope, 
convex in our sense.
\end{conjecture}

The examples of \cite{P1}, \S 7, show
that for the full Fock-Goncharov conjecture to hold, we need to assume $V$ has enough global functions. 
In that case tropicalizing a general function gives (conjecturally) a bounded {\it convex} polytope.
As we are unable to prove Conjecture \ref{dcon} except in the monomial case we use a restricted
version (which happily still has wide application): 

\begin{definition} \label{gmdef} A cluster variety $V$ 
has \emph{Enough Global Monomials} (EGM)
if for each valuation $0 \neq v \in V^{\trop}(\bZ)$ there is a global monomial 
$f$ with $v(f) < 0$. 
\end{definition}

The condition that $V$ has EGM is equivalent to the existence of 
$W \in \ord(V)$ whose associated convex polytope $\Xi_{W^T}$ is bounded,
see Lemma \ref{egmlem}.

The following theorem demonstrates the value of the EGM condition:

\begin{theorem} \label{egmth} Let $V$ be a cluster variety. 
Then:
\begin{enumerate}
\item (Proposition \ref{cafin}, Corollary \ref{fgcor})
If $V^{\vee}$ satisfies the EGM condition, then
the multiplication rule on $\can(V)$ is polynomial, i.e., for given
$p,q \in V^{\mch}(\bZ^T)$, $\alpha(p,q,r) =0$ for all but
finitely many $r \in V^{\mch}(\bZ^T)$. This gives 
$\can(V)$ the structure of a finitely generated commutative 
associative $\kk$-algebra.
\item (Proposition \ref{markscor})
If $V=\cA_{\prin}$ and $V$ satisfies the EGM condition, then there 
are canonical inclusions
\[
\ord(V) \subset \midd(V) \subset \up(V) \subset \can(V).
\]
\end{enumerate}
\end{theorem}

\begin{remark} \label{markovremark}
We believe, based on calculations in \cite{M13}, 
\S7.1, that the conditions of the theorem ($\cA_{\prin}$ has EGM, and 
$\Theta = \cA^{\mch}_{\prin}(\bZ^T)$)
hold for the cluster variety associated with the once punctured torus,
see some details in Examples \ref{standardexample} and \ref{examples2}.
However, the equality 
$\up(\cA) = \can(\cA)$ is expected to
fail, and in particular in this case we expect the full Fock-Goncharov 
holds for $\cA_{\prin},\cX,$ and very general $\cA_t,$ but not for $\cA$. 
\end{remark} 

We note that $\cA_{\prin}$ has Enough Global Monomials in many cases:

\begin{proposition} \label{egmprop} Consider the following conditions on a cluster algebra $\cA$:
\begin{enumerate}
\item The exchange matrix has full rank, $\up(\cA)$ is generated by finitely
many cluster variables, and $\Spec(\up(\cA))$ is a smooth affine variety. 
\item $\cA$ has an acyclic seed.
\item $\cA$ has a seed with a maximal green sequence.
\item For some seed, the cluster complex $\Delta^+(\ZZ) \subset 
\cA^{\mch}(\bR^T)$ is not
contained in a half-space.
\item $\cA_{\prin}$ has Enough Global Monomials.
\end{enumerate}
Then (1) implies (5) (Proposition \ref{egmprop1}). Furthermore,
(2) implies (3) implies (4) implies (5) (Propositions \ref{mgsrem}, 
\ref{gsprop}, and \ref{ldegmprop}). 
Finally (4) implies the full Fock-Goncharov conjecture, 
for $V= \cA_{\prin}$, $\cX$, or very general $\cA_t$, or, under the convexity assumption
(7) of Theorem \ref{mainth}, for $\cA$ (Proposition \ref{egmscprop}). 
\end{proposition}

\begin{example} \label{nex} 
A recent paper \cite{GY13} of Goodearl and Yakimov announces the equality 
$\up = \ord$ for all double Bruhat cells in semi-simple groups. In this
case, Yakimov has furthermore announced the existence 
of a maximal green sequence.
Many cluster varieties $\cA$ associated to a marked bordered surface with
at least two punctures also have a maximal green sequence, see \cite{CLS},
\S 1.3 for a summary of known results on this. The recent \cite{GS16},
Theorem 1.12,1.17,
shows that (4) holds for the Fock-Goncharov cluster varieties of $\PGL_m$ local systems on
{\it most} decorated surfaces. Together with Proposition \ref{egmprop} these results 
imply the full Fock-Goncharov theorem in any of these cases. 

We note that for the cluster algebra associated to a marked bordered surface,
a canonical basis of $\up(\cX)$ parameterized by $\cA(\bZ^T)$ has been previously obtained by Fock-Goncharov, 
\cite{FG06}, Theorem 12.3. They show that
the $\cA$ and $\cX$ varieties have natural modular meaning as moduli spaces of local systems. They identify
$\cA(\bZ^T)$ with a space of integer laminations (isotopy classes of disjoint loops with
integer weights) and their associated basis element is a natural 
function given by trace of monodromy around a loop. 
We checked, together with 
A.\ Neitzke, that our basis agrees with the Fock-Goncharov basis of trace
functions in the case of a sphere with four punctures, for primitive elements of the tropical set. 
Our theta function basis 
comes canonically from the
cluster structure (it does not depend on any modular interpretation). 
\end{example}

\begin{remark}
\label{clusterindependenceremark}
In general, we conjecture the bases we construct for rings of global functions
on cluster varieties $V$ or partial compactifications
$\oV$ are intrinsic to the 
underlying log Calabi-Yau variety $V$, and do not depend on the particular
cluster structure on $V$. This is a non-trivial statement: there exist
varieties with multiple cluster structures (in particular different atlases of
tori for the same variety). Yan Zhou will show in her Ph.D.\ thesis that the 
(principal coefficient version of) 
the cluster variety associated
to the once-punctured torus is an example. 

This conjecture is suggested by \cite{GHK11}, Conjecture 0.6, and
the results of \cite{GHK11}, \cite{GHK12}
and \cite{GHKII} prove this in the case of the $\cX$ cluster varieties 
where the skew-symmetric form has rank two,
which includes the case of the sphere with four punctures. 
Thus we have the (at least to us) 
remarkable conclusion that in many cases where bases occur because of
some extrinsic interpretation of the spaces, in fact this extrinsic
interpretation is irrelevant. For example, the theta functions
given by trace functions above, which would appear to depend on the 
realization of the cluster variety as a moduli space of local systems, 
are actually intrinsic to the underlying variety. 
In the case of Example \ref{SLrexample}, where bases may arise
from representation
theory, our basis does not use the group-theoretic aspects of the
spaces. The suggestion that the canonical basis is independent of the
cluster structure may surprise some, as understanding the canonical basis
was the initial motivation for the Fomin-Zelevinsky definition of
cluster algebras.
\end{remark} 

Returning to the role of convexity notions, we note that
our formula for the structure constants $\alpha$ of Theorem \ref{mainth},
(1) is given by counting broken lines. As a result, our notion of convexity
interacts nicely with the multiplication rule. 
This allows us to generalize basic polyhedral constructions from toric geometry 
in a straight-forward way. 

A polytope $\Xi \subset V^{\mch}(\bR^T)$ convex in our sense determines 
(by familiar Rees-type constructions for graded rings) 
a compactification of $V$. Furthermore, for any choice of seed, 
$V^{\vee}(\RR^T)$ is identified with a linear space $\RR^n$ and 
$\Xi$ with an ordinary convex polytope. Our construction also gives
a flat degeneration of this compactification of $V$ to the ordinary
polarized toric variety for $\Xi \subset \bR^n$. See \S \ref{candsec}. We expect
this specializes to  a uniform construction of many degenerations of representation theoretic
objects to toric varieties,  see e.g., \cite{C02}, \cite{AB}, and \cite{KM05}. 
Applied to the Fock-Goncharov moduli spaces of $G$-local systems, this will
give for the first time compactifications of character varieties with nice 
(e.g., toroidal anti-canonical)
boundary. See Remark \ref{crems}.  The polytope can
be chosen so that the boundary of the compactification is very simple, a union
of toric varieties. For example, let $\Gr^o(k,n) 
\subset \Gr(k,n)$ be the open subset where
the frozen variables for the standard cluster structure are non-vanishing. 
Then the boundary $\Gr(k,n)\setminus \Gr^o(k,n)$ consists of a union of
certain Schubert cells. We obtain 
using a polytope an 
alternative compactification
where the Schubert cells (which are highly non-toric) are replaced by 
toric varieties. See Theorem \ref{toricbth}.  

The Fock-Goncharov conjecture is the cluster special case of 
\cite{GHK11}, Conjecture 0.6, which says (roughly) that affine log CYs with maximal
boundary come in canonical dual pairs with the tropical set of one parameterizing
a canonical basis of functions on the other. 
We can view the conjecture as having two parts: First, that the vector space, $\can$, 
with this basis $V^{\trop}(\bZ)$ is naturally an algebra in a such a way that 
$V^{\vee} := \Spec(\can)$ is an affine log CY. And then furthermore, that this log CY is
the mirror -- in the cluster case the Fock-Goncharov dual (it is natural to further
ask if this is the mirror in the sense of HMS but we do not consider this question here). 
Our deepest mirror theoretic result is the following weakening of the first 
part: 

\begin{theorem} Assume $\cA_{\prin}^{\vee}$ has EGM. Then for $V = \cX,\cA_{\prin}$ or $\cA_t$ for very general $t$,
or, under the convexity assumption (7) of Theorem \ref{mainth}, $\cA$, $\can(V)$ (with
structure constants as in \ref{mainth}) is a finitely generated algebra
and $\Spec(\can(V))$ is a log canonical Gorenstein $K$-trivial affine variety of
dimension $\dim(V)$.
\end{theorem}
For the proof see Theorem \ref{cfcor}.

\subsection{Representation-theoretic applications}
\label{intro1.4}
We turn to \S\ref{pcrtsec}. Here we study features of
partial compactifications coming from frozen variables. As explained
in Example \ref{SLrexample}, these partial compactifications are often
the relevant ones in representation-theoretic examples. In particular,
for a partial minimal model $\cA \subset \ocA$, often 
the vector subspace 
$\up(\ocA) \subset \up(\cA)$ is more important than 
$\up(\cA)$ itself. For example there
is a cluster structure with frozen variables for the open double Bruhat
cell $U$ in a semi-simple group $G$. Then $\up(\cA)$ is the ring of functions
on the open double Bruhat cell and $\up(\ocA)=H^0(G,\shO_G)$.
Of course $H^0(G,\cO_G)$ is the most important
representation of $G$. However, one cannot expect a 
canonical basis of $\up(\ocA)$, i.e., one determined by the 
intrinsic geometry of $\ocA$. For example, 
$G$ has
no non-constant global functions which are eigenfunctions for the action of $G$ on itself.
But we expect, and in the myriad cases 
above can prove, that the affine log Calabi-Yau open subset 
$\cA \subset \ocA$ has a canonical basis, $\Theta$, and we believe 
that $\Theta \cap \up(\ocA)$, the set of theta functions on
$\cA$ 
that extends regularly to all of $\ocA$, is a basis for $\up(\ocA)$, 
canonically associated to the choice of
log Calabi-Yau open subset $\cA \subset \ocA$, see \cite{P1}, Remark 1.10.
This is not a basis of $G$-eigenfunctions, but
they are eigenfunctions
for the associated maximal torus, which is the subgroup of $G$ that
preserves $U$. This is
exactly what one should expect: the basis is not intrinsic to $G$,
instead it is (we conjecture) intrinsic to the pair $U \subset G$, see
Remark \ref{clusterindependenceremark}.

We shall now describe in more detail
what can be proved for partial compactifications 
of cluster varieties
coming from frozen variables. A key point is a technical but
combinatorial hypothesis, that 
\emph{each variable has an \it optimized seed}, see Definition \ref{osdef}
and Lemmas \ref{quiveroptimized} and \ref{opslem}. The main need for
this hypothesis is Proposition \ref{boundaryprop}, which states that if
a linear combination of theta functions extends across a boundary divisor
then each theta function in the sum extends across the divisor. Thus
the middle cluster algebra, in this case, behaves well with respect to
boundary divisors.
Happily, this condition holds for the cluster structures on the Grassmannian, 
and, for $\SL_r$, for the cluster structure
on a maximal unipotent subgroup $N \subset \SL_r$, the basic affine space $\cA = G/N$, and the Fock-Goncharov
cluster structure on $(\cA \times \cA \times \cA)/G$, see Remark \ref{osrem}. 

Let us now work with the principal cluster variety $\cA_{\prin}$.
Consider the partial compactification $\cA_{\prin}
\subset \ocA_{\prin}$ by allowing the frozen variables to be zero.
Each boundary divisor $E \subset \ocA_{\prin}$ 
gives a point $E \in \cA_{\prin}(\bZ^T)$ and thus (in general conjecturally), 
a canonical
theta function $\vartheta_E$ on $\cA_{\prin}^{\mch}$. We then define the 
\emph{potential} 
$W = \sum_{E \subset \partial \ocA_{\prin}} \vartheta_E \in 
\up(\cA_{\prin}^{\vee})$ as the sum of these theta functions. We have its 
piecewise linear tropicalisation $W^T: \cA^{\vee}_{\prin}(\bR^T) \to \bR$. 
This defines
a cone 
\begin{equation} \label{spcone}
\Xi := \{x \in \cA^{\mch}_{\prin}(\bR^T)\,|\, W^T(x) \geq 0\} 
\subset \cA^{\vee}_{\prin}(\bR^T).
\end{equation}

\begin{theorem}[Corollaries \ref{spcor} and \ref{ffgcor}]
\label{spth} 
Assume that each frozen index $i$ has an optimized seed. Then:
\begin{enumerate}
\item
$W^T$ and $\Xi$ are convex in our sense. 
\item The set 
$\Xi \cap \Theta$ 
parameterizes a canonical basis of an algebra $\midd(\ocA_{\prin})$, and
\[
\midd(\ocA_{\prin}) = \up(\ocA_{\prin}) \cap \midd(\cA_{\prin}) \subset 
\up(\cA_{\prin}).
\]
\item
Now assume further that we have Enough Global Monomials on 
$\cA_{\prin}^{\vee}$. If for some seed $\s$, $\Xi$
is contained in the convex hull of $\Theta$ 
(which itself contains the convex hull of $\Delta^+(\ZZ)$) then 
$\Theta = \cA_{\prin}^{\mch}(\bZ^T)$, $\midd(\ocA_{\prin}) = \up(\ocA_{\prin})$ is finitely generated, and the
integer points $\Xi\cap\shA_{\prin}^{\vee}(\ZZ^T) \subset \cA_{\prin}^{\mch}(\bZ^T)$ parameterize a canonical basis. 
\end{enumerate}
\end{theorem}

Each choice of seed identifies $\cA_{\prin}^{\mch}(\bZ^T)$ with a lattice, and 
the cone
$\Xi \subset \cA_{\prin}^{\vee}(\bR^T)$ with a rational polyhedral
cone, described by canonical linear inequalities given by the tropicalisation of
the potential. Note that $\Xi$ is convex in our generalized sense.

We show, making use of recent results of Magee and Goncharov-Shen, \cite{Magee}, \cite{TimThesis},
\cite{GS16}, that 
in the representation theoretic examples
which were the original motivation for the definition of cluster algebras our 
polyhedral
cones $\Xi$ specialize to the piecewise linear 
parameterizations of canonical bases of Berenstein and
Zelevinsky \cite{BZ01}, Knutson and Tao \cite{KT99}, and Goncharov and
Shen \cite{GS13}:

\begin{corollary} \label{slrcor} 
Let $G = \SL_{r+1}$ and let 
$\cA \subset \ocA$ be the Fomin-Zelevinsky cluster variety for the basic
affine space $G/N$. 
\begin{enumerate}
\item
All the hypotheses, and thus the conclusions, of Theorem \ref{spth} hold.
In particular
$\Xi\cap\shA^{\vee}(\ZZ^T) \subset \cA^{\mch}(\bZ^T)$ parameterizes a canonical theta function
basis of $\cO(G/N)$. 
\item 
Our potential $W$ agrees with the (representation theoretically defined) 
potential function of Berenstein-Kazhdan, \cite{BK07}. 
\item
The maximal torus $H$ acts canonically on $\ocA$, preserving the open set $\cA \subset \ocA$. 
\item
Each theta function is an $H$-eigenfunction, and there is a canonical map
$$
w: \cA^{\mch}(\bZ^T) \to \chi^*(H)
$$
(the target is the character lattice of $H$), linear for the linear structure given by any seed,
which sends an integer point to the $H$-weight of the corresponding theta function. The slice 
$$
\Xi(\ZZ^T) \cap w^{-1}(\lambda)
$$
parameterizes a canonical theta function basis of the eigenspace $\cO(G/N)^{\lambda} = V_{\lambda}$, the
corresponding irreducible representation of $G$. 
\item
For a natural choice of seed the {\it cone} $\Xi$ is canonically identified with the Gelfand-Tsetlin cone.
\end{enumerate}
\end{corollary}

\begin{corollary} \label{conf3cor} 
Let $\cA \subset \ocA$ be the Fock-Goncharov cluster variety for  
\[
\Conf_3(G/N) := ((G/N)^{\times 3})/G.
\]
\begin{enumerate}
\item
All the hypotheses, and thus the conclusions, of Theorem \ref{spth} hold.
In particular the cone 
$\Xi\cap\shA^{\vee}(\ZZ^T) \subset \cA^{\mch}(\bZ^T)$ parameterizes a canonical theta function
basis of $\cO(\Conf_3(G/N))$. 
\item 
Our potential function $W$ agrees with the (representation theoretically defined) 
potential function of Goncharov-Shen, \cite{GS13}. 
\item
$H^{\times 3}$ acts canonically on $\ocA$, preserving the open subset 
$\cA \subset \ocA$. 
\item
Each theta function is an $H$-eigenfunction, and there is a canonical map
\[
w: \cA^{\mch}(\bZ^T) \to \chi(H^{\times 3})
\]
linear for the linear structure given by any seed,
which sends an integer point to the $H^{\times 3}$-weight of the corresponding theta function. The slice
\[
\Xi(\ZZ^T) \cap w^{-1}(\alpha,\beta,\gamma)
\]
parameterizes a canonical theta function basis of the eigenspace 
\[
\cO(G/N)^{(\alpha,\beta,\gamma)} = (V_{\alpha} \otimes V_{\beta} \otimes V_{\gamma})^G.
\]
In particular the number of integral points in 
$\Xi(\ZZ^T) \cap w^{-1}(\alpha,\beta,\gamma)$ is the corresponding 
Littlewood-Richardson coefficient.
\item
For a natural choice of seed, the cone $\Xi$ is canonically identified with the Knutson-Tao Hive.
\end{enumerate} 
\end{corollary} 

These corollaries are proven at the end of \S \ref{fullFGsection}. 

We stress here that the above representation theoretic results 
come \emph{for free} from general properties of our 
mirror symmetry construction: any
partial minimal model $V \subset Y$ of an affine log Calabi-Yau variety with maximal boundary determines 
(in general conjecturally) 
a cone $\Xi \subset V^{\vee}(\bR^T)$ with the analogous meaning. We
are getting these basic representation theoretic results without representation theory! 

We recover the remarkable Gelfand-Tsetlin and Hive polytopes for a particular
choice of seed. Different (among the infinitely many possible) choices of seed give in general combinatorially
different cones, whose integer points parameterize the same theta function basis. The canonical
object is the {\it convex cone} 
$\Xi \subset \cA^{\vee}(\bR^T)$ cut out by $W^T$, different (by piecewise linear mutation) 
identifications
of $\cA^{\vee}(\bR^T)$ with a vector space give different incarnations of $\Xi$ as convex cones in the usual sense. 
\medskip

Potentials were considered in the work of Goncharov and Shen, 
\cite{GS13}, which
in turn built on work of Berenstein and Zelevinsky, \cite{BZ01} and
Berenstein and Kazhdan, \cite{BK00}, \cite{BK07}. The potential constructed
by Goncharov and Shen 
has a beautiful representation theoretic
definition, and was found in many situations to coincide with known constructions
of Landau-Ginzburg potentials. 
On the other hand, the construction of the potential in terms of theta functions
coincides precisely with the construction of the mirror Landau-Ginzburg 
potential as carried out in \cite{G09},\cite{CPS}. The latter work can
be viewed as a tropicalization of the descriptions of the potential in terms
of holomorphic disks in \cite{CO06},\cite{A07}.
Thus our construction explains the emergence of the 
Landau-Ginzburg
potentials in \cite{GS13}. Our potentials are determined by the 
cluster structure (and
conjecturally,
just the underlying log Calabi-Yau variety), and in particular are independent of
any modular or
representation theoretic interpretation of the cluster variety. This gives, as in 
Remark \ref{clusterindependenceremark}, 
the remarkable suggestion that e.g., the representation
theoretically defined Goncharov-Shen potential, 
which would seem to depend heavily on the modular interpretation
of $\ocA = \Conf_3(G/N)$, is actually intrinsic to the partial minimal model $\cA \subset \ocA$. 

\emph{Acknowledgments}: 
We received considerable
inspiration from conversations with A. Berenstein, V.\ Fock, S.\ Fomin, A.\ Goncharov, 
B.\ Keller, B.\ Leclerc, J.~Koll\'ar, G.\ Muller, G.\ Musiker, A.\ Neitzke, D.\ Rupel,
M.\ Shapiro, B.\ Siebert, Y.\ Soibelman, and D.\ Speyer. 

The first author
was partially supported by NSF grant DMS-1262531 and a Royal Society
Wolfson Research Merit Award, the second by
NSF grants DMS-1201439 and DMS-1601065, and the third by NSF grant DMS-0854747. 
Some of the research was conducted when the first and third authors visited the
fourth at I.H.E.S, during the summers of 2012 and 2013. 

\section{Scattering diagrams and chamber structures}
\label{scatdiagsection}

\subsection{Definition and constructions}
\label{defconstsection}

Here we recall the basic properties of scattering diagrams, the main technical
tool in this paper. Scattering diagrams appeared first in \cite{KS06} in two
dimensions, and then in all dimensions in \cite{GSAnnals}, with another 
approach in a more specific case in \cite{KS13}.
Here we give a self-contained treatment restricted to the specific case
needed in this paper.

We start with a choice of \emph{fixed data} $\Gamma$ as defined in \cite{P1},
which for the reader's convenience is described at the beginning of Appendix 
\ref{LDsec}. In brief, this entails a 
lattice $N$ with dual lattice $M=\Hom(N,\ZZ)$, a skew-symmetric form
\[
\{\cdot,\cdot\}:N\times N\rightarrow \QQ,
\]
sublattices $N_{\uf}, N^{\circ}\subseteq N$ with $N_{\uf}$ a saturated
sublattice and $N^{\circ}$ a sublattice of finite index with dual
lattice $M^{\circ}=\Hom(N^{\circ},\ZZ)$, an index set
$I=\{1,\ldots,n\}$ with $|I|=\rank N$ and a subset $I_{\uf}\subseteq I$ with
$|I_{\uf}|=\rank N_{\uf}$, as well as positive integers $d_i$, $i\in I$.
Finally we also choose an initial seed $\s$, i.e., a basis $e_1,\ldots,
e_n$ of $N$. See Appendix \ref{LDsec} for the precise properties that all this
data must satisfy. 

For the construction of the scattering diagram associated to this data, 
we will require

\medskip

{\bf The Injectivity Assumption}. The map $p_1^*:N_{\uf}\rightarrow M^{\circ}$ 
given by $n\mapsto \{n,\cdot\}$ is injective.

\medskip

While this does not hold for a general choice of fixed data, it does hold
in the principal coefficient case (see Appendix \ref{rdsec}), and results
in this paper
about arbitrary cluster varieties and algebras will be proved via the
principal case.

Set 
\[
N^+ :=N^+_{\s} := \left\{\sum_{i\in I_{\uf}} a_i e_i \,\bigg|\, a_i \geq 0, 
\sum a_i > 0 \right\}.
\]
Choose a linear function $d: N  \to \bZ$ such 
$d(n) > 0$ for $n \in N^+$.

Under the Injectivity Assumption, one can choose a strictly
convex top-dimensional cone
$\sigma\subseteq M_{\RR}$, with associated monoid $P:=\sigma \cap M^{\circ}$,
such that $p_1^*(e_i)\in J:=P\setminus P^{\times}$ for all $i\in I_{\uf}$.
Here $P^{\times}=\{0\}$ is the group of units of the monoid $P$.
This gives the monomial ideal $J\subseteq\kk[P]$ in the monoid
ring $\kk[P]$ over a field $\kk$ of characteristic zero, and we write 
$\widehat{\kk[P]}$ for the completion with respect to $J$.

We define the \emph{module of log derivations} of $\kk[P]$ as
\[
\Theta(\kk[P]):=\kk[P]\otimes_{\ZZ} N^{\circ},
\]
with the action of $f\otimes n$ on $\kk[P]$ being given by
\[
(f\otimes n)(z^m)=f\langle n,m\rangle z^m,
\]
so we write $f\otimes n$ as $f\partial_n$. Let $\widehat{\Theta(\kk[P])}$
denote the completion of $\Theta(\kk[P])$ with respect to the ideal $J$.

Using this action, if $\xi\in J\widehat{\Theta(\kk[P])}$, then
\[
\exp(\xi)\in \Aut(\widehat{\kk[P]})
\]
makes sense using the Taylor series for the exponential. We have the Lie bracket
\[
[z^m\partial_n,z^{m'}\partial_{n'}]=z^{m+m'}
\partial_{\langle n,m'\rangle n'-\langle n',m\rangle n}.
\]
Then $\exp(J\Theta(\kk[P]))$ can be viewed as a subgroup of the
group of continuous automorphisms of $\widehat{\kk[P]}$ which are the
identity modulo $J$, with
the group law of composition coinciding with the group law coming from
the Baker-Campbell-Hausdorff formula.

Define the sub-Lie algebra of $\Theta(\kk[P])=\kk[P]\otimes_{\ZZ} N^{\circ}$
\[
\fog:= \bigoplus_{n\in N^+}\fog_n
\]
where $\fog_n$ is the one-dimensional subspace of 
$\Theta(\kk[P])$ spanned by $z^{p_1^*(n)} \partial_{n}$.
We calculate that $\fog$ is in fact closed under Lie bracket:
\begin{align}
\label{Liebracketform}
\begin{split}
[z^{p_1^*(n)}\partial_n,z^{p_1^*(n')}\partial_{n'}]
= {} & z^{p_1^*(n+n')}\left( \langle p_1^*(n'),n\rangle\partial_{n'}
-\langle p_1^*(n),n'\rangle\partial_{n}\right)\\
= {} & z^{p_1^*(n+n')}\left( \{n', n\}\partial_{n'}-\{n, n'\}\partial_n\right)\\
= {} &\{n',n\} z^{p_1^*(n+n')}\partial_{n+n'}.
\end{split}
\end{align}

We have
\[
\fog^{> k} := \bigoplus_{d(n) > k} \fog_n \subset \fog
\]
a Lie subalgebra,
and $\fog^{\leq k} := \fog/\fog^{>k}$ a nilpotent Lie algebra. We let
$G^{\leq k} := \exp(\fog^{\leq k})$ be the corresponding nilpotent group.
This group, as a set, is just $\fog^{\leq k}$, but multiplication is
given by the Baker-Campbell-Hausdorff formula. We set 
\[
G := \exp(\fog):= \lim_{\longleftarrow} G^{\leq k}
\]
the corresponding pro-nilpotent group. We have the canonical set bijections
\[
\exp: \fog^{\leq k} \to G^{\le k} \quad \hbox{and} \quad \exp: \lim_{\longleftarrow} 
\fog^{\le k}  \to G.
\]

For $n_0 \in N^+$ we define
\begin{align*}
\fog_{n_0}^{\|} &=\bigoplus_{k>0} \fog_{k \cdot n_0} \subset \fog \quad \text{(note this is a
Lie subalgebra}) \\
G_{n_0}^{\|}    &= \exp(\fog_{n_0}^{\|}) \subset G.
\end{align*}
Note that by the commutator formula \eqref{Liebracketform},
$\fog^{\|}_{n_0}$, hence $G^{\|}_{n_0}$, is abelian.

In what follows, noting that $G$ is a subgroup of 
$\Aut_{\kk}(\widehat{\kk[P]})$, we will often describe elements of
$G_{n_0}^{\|}$ as follows. 

\begin{definition} \label{WCdef} Let $n_0 \in N^+$, $m_0 := p_1^*(n_0)$ and 
$f = 1+\sum_{k=1}^{\infty} c_k z^{km_0} \in \widehat{\kk[P]}$.
Define $\theta_{f}$ to be the automorphism
of $\widehat{\kk[P]}$ given by
\[
\theta_{f}(z^m)=f^{\langle n_0',m\rangle} z^m,
\]
where $n_0'$ is the generator of the monoid $\RR_{\ge 0} n_0 \cap N^{\circ}$.
\end{definition}

\begin{lemma}
\label{thetaflemma}
For $n_0\in N^+$,
$G_{n_0}^{\|} \subset \Aut(\widehat{\kk[P]})$ is the subgroup
of automorphisms of the form $\theta_f$
for $f$ as in Definition \ref{WCdef} with the given $n_0 \in N^+$. 
More specifically, $\exp(\sum_{k>0} c_k z^{kp_1^*(n_0)}\partial_{kn_0})
\in G^{\|}_{n_0}$
acts as the automorphism $\theta_f$ with $f=\exp(\sum_{k>0} d^{-1}k
c_kz^{kp_1^*(n_0)})$, where $d\in\QQ$ is the smallest positive rational number
such that $dn_0\in N^{\circ}$.
\end{lemma}

\begin{proof} 
Let $H \subset \Aut(\widehat{\kk[P]})$ be the set of 
$\theta_f$ of the given form. Then $H$ is a subgroup as $\theta_{f_1}\circ
\theta_{f_2}=\theta_{f_1f_2}$. Note that
$\sum_{k>0} c_k z^{kp_1^*(n_0)}\partial_{kn_0}=\left(\sum_{k>0}d^{-1}kc_k z^{kp_1^*
(n_0)}\right) \partial_{dn_0}$, where $d\in \QQ$ is as described in the
statement. The exponential of this vector field is
easily seen to act as $\theta_f$ with 
$f=\exp(\sum_{k>0}d^{-1}kc_kz^{kp_1^*(n_0)})$.
Hence
$G_{n_0}^{\|} \subset H$. From this, we see also that if $\log(f)=
\sum_{k>0} c_kz^{kp_1^*(n_0)}$, then $\theta_f=\exp(\sum_{k>0} dk^{-1}c_k
z^{kp_1^*(n_0)}\partial_{kn_0})$, and the latter lies in $G^{\|}_{n_0}$.
\end{proof}

\begin{definition} \label{walldef} A \emph{wall} in $M_{\RR}$ (for $N^+$ and $\fog$)
is a pair $(\fod,g_{\fod})$ such that
\begin{enumerate}
\item $g_{\fod}\in G_{n_0}^{\|}$ for some primitive $n_0 \in N^+$.
\item $\fod \subset n_0^{\perp} \subset M_{\RR}$ is a $(\rank N-1)$-dimensional 
convex (but not necessarily strictly convex) rational polyhedral cone. 
\end{enumerate}

The set $\fod \subset M_{\bR}$ is called the \emph{support} of the wall 
$(\fod,g_{\fod})$.
\end{definition}

\begin{remark}
Using Lemma \ref{thetaflemma}, we often write a wall as $(\fod,f_{\fod})$
for $f_{\fod}\in \widehat{\kk[P]}$, necessarily of the form
$f_{\fod}=1+\sum_{k\ge 1} c_kz^{km_0}$. We shall use this notation 
interchangeably without comment.
\end{remark}

\begin{definition}
\label{KSscatdiagdef}
A \emph{scattering diagram $\foD$ for $N^+$ and $\fog$} 
is a set
of walls such that for every degree $k>0$, there are only a finite
number of $(\fod,g_{\fod})\in\foD$ with the image of $g_{\fod}$ in $G^{\le k}$
not the identity.

If $\foD$ is a scattering diagram, we write
\[
\Supp(\foD)=\bigcup_{\fod\in\foD} \fod, \quad\quad
\Sing(\foD)=\bigcup_{\fod\in\foD} \partial\fod\quad\cup
\bigcup_{\fod_1,\fod_2\in\foD\atop
\dim \fod_1\cap\fod_2=n-2} \fod_1\cap\fod_2
\]
for the support and singular locus of the scattering diagram.
If $\foD$ is a finite scattering diagram, then its support is a finite
polyhedral cone complex.
A \emph{joint} is an $(n-2)$-dimensional cell of this complex, so that
$\Sing(\foD)$ is the union of all joints of $\foD$.
\end{definition}

\begin{remark}
We will often (especially in Appendix \ref{scatappendix}) want to use
a slightly more general notion of scattering diagram, where the elements
attached to walls lie in some other choice of group $G'$ arising from
an $N^+$-graded Lie algebra $\fog'$. In this case we talk about a scattering
diagram for $\fog'$. For example, any scattering diagram for $\fog$ induces
a finite scattering diagram for $\fog^{\le k}$ by taking the image of
the attached group elements under the projection $G\rightarrow G^{\le k}$.
\end{remark}

Given a scattering diagram $\foD$, we obtain the 
\emph{path-ordered product}. Assume given
a smooth immersion
\[
\gamma:[0,1]\rightarrow M_{\RR}\setminus\Sing(\foD)
\]
with endpoints not contained in the support of $\foD$. Assume
$\gamma$ is transversal to each wall of $\foD$ that it crosses.
For each degree $k>0$, we can find numbers
\[
0<t_1\le t_2\le\cdots\le t_s<1
\]
and elements $\fod_i\in\foD$ with the image of $g_{\fod_i}$
in $G^{\le k}$ non-trivial such that
\[
\gamma(t_i)\in\fod_i,
\]
$\fod_i\not=\fod_j$ if $t_i=t_j$, and $s$ taken as large as
possible. (The $t_i$ are the 
times at which the path $\gamma$ hits a wall. We allow $t_i = t_{i+1}$ because
we may have two different walls $\fod_i,\fod_{i+1}$ which span the same hyperplane.)

For each $i$, define 
\[
\epsilon_{i} = \begin{cases} +1& \langle n_0,\gamma'(t_i)\rangle<0,\\
-1& \langle n_0,\gamma'(t_i)\rangle>0,
\end{cases}
\]
where $n_0\in N^+$ with $\fod\subseteq n_0^{\perp}$.
We then define
\[
\theta^k_{\gamma,\foD}=g^{\epsilon_s}_{\fod_s}\cdots
g^{\epsilon_1}_{\fod_1}.
\]
If $t_i=t_{i+1}$, then $\fod_i, \fod_{i+1}$ span the same hyperplane 
$n_0^{\perp}$, hence $g_{\fod_i},g_{\fod_{i+1}}\in G^{\|}_{n_0}$. Thus,
since this latter group is abelian, $g_{\fod_i}$
and $g_{\fod_{i+1}}$ commute, so this product is well-defined. 
We then take
\[
\theta_{\gamma,\foD}=\lim_{k\rightarrow\infty} \theta^k_{\gamma,\foD} \in G.
\]

We note that $\theta_{\gamma,\foD}$ depends only on its homotopy class (with
fixed endpoints) in $M_{\bR} \setminus \Sing(\foD)$. We also note that the
definition can easily be extended to piecewise smooth paths
$\gamma$, provided that the path always crosses a wall if it intersects it.

\begin{definition} Two scattering diagrams $\foD$, $\foD'$ are 
\emph{equivalent} if $\theta_{\gamma,\foD}=\theta_{\gamma,\foD'}$
for all paths $\gamma$ for which both are defined.
\end{definition}

Call $x \in M_{\bR}$ {\it general} if
there is at most one rational hyperplane $n_0^{\perp}$ with $x \in 
n_0^{\perp}$. 
For $x$ general and $\foD$ a scattering diagram, 
let $g_x(\foD) := \prod_{\fod \ni x} g_{\fod} \in G_{n_0}^{\|}$.
One checks easily:

\begin{lemma}
\label{easyequivlemma}
Two scattering diagrams $\foD,\foD'$ are equivalent if and only if
$g_x(\foD) = g_x(\foD')$ for all general $x$.
\end{lemma}

\begin{definition} A scattering diagram $\foD$ is \emph{consistent}
if $\theta_{\gamma,\foD}$ only depends on the endpoints of $\gamma$
for any path $\gamma$ for which $\theta_{\gamma,\foD}$ is defined.
\end{definition}

\begin{definition} 
We say a wall $\fod \subset n_0^{\perp}$ is \emph{incoming} if
\[
p_1^*(n_0)\in\fod.
\]
Otherwise, we say the wall is \emph{outgoing} (note in any case 
$p_1^*(n_0)$ lies in the span of the wall $n_0^{\perp}$).

We call $-p_1^*(n_0)$ the {\it direction} of the wall. 
(This terminology comes from the case $N = \bZ^2$, where an outgoing wall is
then a ray containing its direction vector, thus one that points outward.)
\end{definition}

We need one particular scattering diagram, determined by the fixed data and
seed data. Setting $v_i=p_1^*(e_i)$, $i\in I_{\uf}$, we start with the 
scattering diagram
\[
\foD_{\inc,\s}:=\{(e_i^{\perp},1+z^{v_i})\,|\, i\in I_{\uf}\}.
\]

The main result on scattering diagrams, which follows easily from 
Theorem \ref{KSlemma}, is the following. A more general version of this
was proved in two dimensions in \cite{KS06}, and in a much more general
context in all dimensions in \cite{GSAnnals}. A simpler argument which applies
to the case at hand was given in \cite{KS13}, which shall be reviewed  
in \S\ref{consistentscatdiagramsection}.

\begin{theorem} \label{KSlemmaGS} 
There is a scattering diagram $\foD_{\s}$ satisfying:
\begin{enumerate}
\item $\foD_{\s}$ is consistent,
\item $\foD_{\s} \supset \foD_{\inc,\s}$,
\item $\foD_{\s} \setminus \foD_{\inc,\s}$ consists only of outgoing walls. 
\end{enumerate}

Moreover, $\foD_{\s}$ satisfying these three properties is unique up to 
equivalence. 
\end{theorem} 

The crucial positivity result satisfied by $\foD_{\s}$ is now easily stated:

\begin{theorem}
\label{scatdiagpositive}
The scattering diagram $\foD_{\s}$ is equivalent to a scattering diagram
all of whose walls $(\fod,f_{\fod})$ satisfy $f_{\fod}=(1+z^m)^c$ for
some $m = p^*(n), n \in N^+$ and $c$ a positive integer. In particular,
all nonzero coefficients of $f_{\fod}$ are positive integers.
\end{theorem}

The proof is given in Appendix \ref{scatappendix}. The basic idea
is that the construction of the scattering diagram $\foD_{\s}$ can
be reduced to repeated applications of the following example:

\begin{example}
\label{basicscatteringexample}
Take $N=N^{\circ}=N_{\uf}=\ZZ^2$,
$d_1,d_2=1$ and the skew-symmetric form $\{\cdot,\cdot\}:N\times N
\rightarrow \QQ$ given by the matrix $\epsilon=\begin{pmatrix} 0&1\\ -1&0
\end{pmatrix}$, where $\epsilon_{ij}=\{e_i,e_j\}$. Let $f_1,f_2$ be the dual
basis of $e_1,e_2$, and 
write $A_1=z^{f_1}$, $A_2=z^{f_2}$. We get
\[
\foD_{\inc,\s}=\{(e_1^{\perp},1+A_2),(e_2^{\perp},1+A_1^{-1})\}.
\]
Then one checks easily that
\[
\foD_{\s}=\foD_{\inc,\s}\cup\{(\RR_{\ge 0}(1,-1), 1+A_1^{-1}A_2)\}.
\]
See Figure \ref{scat11}.
(See for example \cite{GPS}, Example 1.6.)
\end{example}

\begin{figure}
\input{scat11.pstex_t}
\caption{Scattering diagram for Example \ref{basicscatteringexample}.}
\label{scat11}
\end{figure}
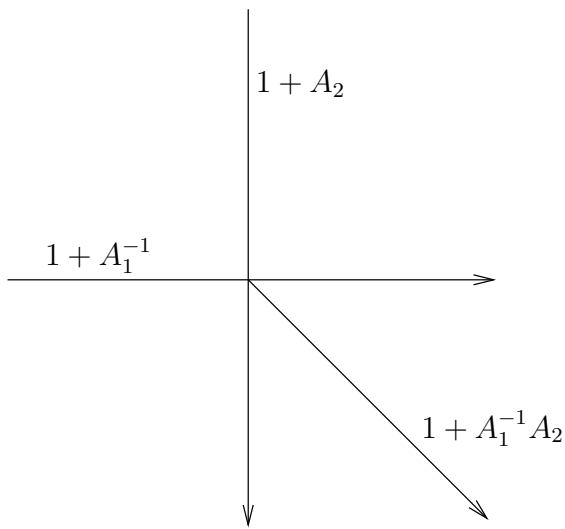

\begin{example}
\label{bcexample}
Take $N=N_{\uf}=\ZZ^2$, with basis $e_1,e_2$, and take
$N^{\circ}$ to be the sublattice generated by $b e_1, c e_2$.
Further take $d_1=b$, $d_2=c$, where $b,c$ are two positive integers, and
take the skew-symmetric form to be the same as in the previous example.
Then $f_1=e_1^*/b$, $f_2=e_2^*/c$. Taking as before $A_1=z^{f_1}$, $A_2=z^{f_2}$,
we get
\[
\foD_{\inc,\s}=\{(e_1^{\perp},1+A_2^c),(e_2^{\perp},1+A_1^{-b})\}.
\]
For most choices of $b$ and $c$, this is a very complicated scattering
diagram. A very similar scattering diagram, with functions $(1+A_2)^b$ and
$(1+A_1)^c$, has been analyzed in \cite{GP10}, but it is easy to translate
this latter diagram to the one considered here by replacing $A_1$ by $A_1^{-1}$
and using the change of lattice trick, which is given in Step IV of the proof
of Proposition \ref{basicpositivity}. All rays of $\foD_{\s}\setminus
\foD_{\inc,\s}$ are contained strictly in the fourth quadrant (i.e.,
in particular are not contained in an axis). Without giving the
details, we summarize the results. 
There are two
linear operators $S_1, S_2$ given by the matrices in the basis $f_1, f_2$
as
\[
S_1 = \begin{pmatrix} -1 & -b\\ 0 & 1\end{pmatrix}, \quad
S_2 = \begin{pmatrix} 1 & 0\\ -c & -1\end{pmatrix}.
\]
Then $\foD_{\s}\setminus \foD_{\inc,\s}$ is invariant under $S_1$ and
$S_2$, in the sense that if $(\fod, f_{\fod}(z^m))
\in \foD_{\s}\setminus \foD_{\inc,\s}$, we have 
$(S_i(\fod), f_{\fod}(z^{S_i(m)}))\in \foD_{\s}\setminus \foD_{\inc,\s}$
provided $S_i(\fod)$ is contained strictly in the fourth quadrant.
It is also the case that
applying $S_2$ to $(\RR_{\ge 0}(1,0), 1+A_1^{-b})$ or
$S_1$ to $(\RR_{\ge 0}(0,-1), 1+A_2^c)$ gives an element of 
$\foD_{\s}\setminus \foD_{\inc,\s}$. Further, $\foD_{\s}$ contains a
discrete series of rays consisting of those rays in the fourth quadrant
obtained by applying $S_1$ and $S_2$ alternately
to the above rays supported on $\RR_{\ge 0}(1,0)$ and $\RR_{\ge 0}(0,-1)$.
These rays necessarily have functions of the form $1+A_1^{-b\alpha}
A_2^{-b\beta}$
or $1+A_1^{c\alpha}A_2^{c\beta}$ for various choices of $\alpha$ and $\beta$.
If $bc < 4$, we obtain a finite diagram. 
(Moreover, the corresponding $\cA$ cluster variety is the cluster variety of finite type  \cite{FZ03a} associated to the root system $A_2$, $B_2$, or $G_2$ for $bc=1$, $2$, or $3$ respectively.)
If $bc \ge 4$, these rays converge to the rays contained in the two eigenspaces of
$S_1\circ S_2$ and $S_2\circ S_1$. These are rays of slope
$-(bc\pm\sqrt{bc(bc-4)})/2b$. This gives a complete description of the rays
outside of the cone spanned by these two rays. The expectation is that
every ray of rational slope appears in the interior of this cone, and
the attached functions are in general unknown. However, in the $b=c$
case, it is known \cite{R12} that the function attached 
to the ray of slope $-1$ is
\[
\left(\sum_{k=0}^{\infty} {1\over (b^2-2b)k+1} 
\begin{pmatrix}
(b-1)^2k\\ k
\end{pmatrix}
A_1^{-bk} A_2^{bk}\right)^b.
\]
The chamber structure one sees outside the quadratic irrational cone is very
well-behaved and familiar in cluster algebra theory. In particular,
the interiors of the first, second and third quadrants are all connected
components of $M^{\circ}_{\RR}\setminus\Supp(\foD)$, and there are for $bc\ge 4$
an infinite number of connected components in the fourth quadrant.
We will see in \S\ref{tropsec} that this chamber structure is precisely
the Fock-Goncharov cluster complex.

On the other hand, it is precisely the rich
structure inside the quadratic irrational
cone which scattering diagram technology brings
into the cluster algebra picture.
\end{example}

\begin{figure}
\input{scat13.pstex_t}
\caption{Scattering diagram for Example \ref{bcexample}, $b=1, c=3$. The
unlabelled rays intersecting the interior of the 
fourth quadrant have attached functions
$1+A_1^{-3}A_2^3$, $1+A_1^{-2}A_2^3$, $1+A_1^{-3}A_2^6$, and $1+A_1^{-1}A_2^3$
in clockwise order.}
\label{scat13}
\end{figure}
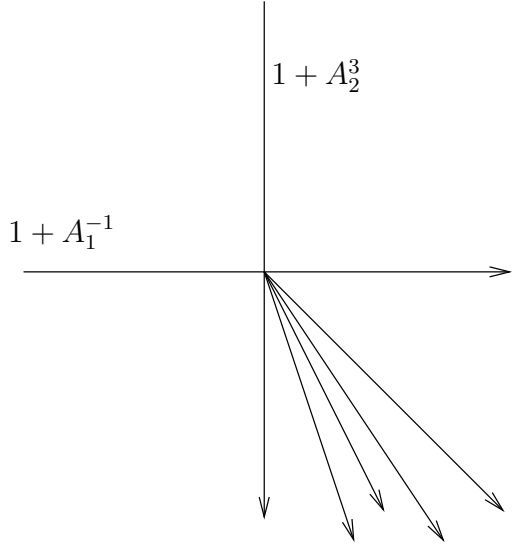

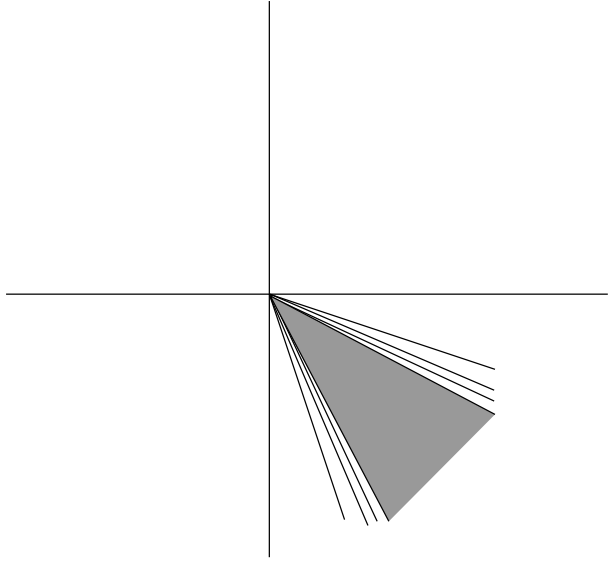
\begin{figure}
\input{diagram33.pstex_t}
\caption{The general appearance of the scattering diagram of
Example \ref{bcexample} for $bc>4$.}
\label{diagram33}
\end{figure}

\subsection{Construction of consistent scattering diagrams}
\label{consistentscatdiagramsection}

In this subsection we give more details about the construction of scattering
diagrams, and in particular give results leading to the proof of Theorem
\ref{KSlemmaGS}. This material can be skipped on first reading, but is
recommended before reading the more difficult material on
scattering diagrams in Appendix \ref{scatappendix}.

Let $\foD$ be a scattering diagram. If we set
\begin{align}
\label{chamberdef}
\begin{split}
\cC^+:= {} & \{m\in M_{\RR} \,|\, m|_{N^+} \ge 0 \}, \\
\cC^-:= {} & \{m\in M_{\RR}\,|\, m|_{N^+} \le 0 \}
\end{split}
\end{align}
then since any wall spans a hyperplane $n_0^{\perp}$, for some $n_0 \in N^+$,
\[
\Supp(\foD)\cap \Int(\cC^{\pm}) =\emptyset.
\]

In particular, if $\foD$ is a consistent scattering diagram, then
$\theta_{\gamma,\foD}$ for $\gamma$ a path with initial point in $\cC^+$
and final point in $\cC^-$ is independent of the particular choice
of path (or endpoints in $\cC^{\pm}$). 
Thus we obtain a well-defined element $\theta_{+,-}\in G$ which
only depends on the scattering diagram $\foD$. 

\begin{theorem}[Kontsevich-Soibelman]
\label{maximscorrespondence} The assignment of $\theta_{+,-}$ to $\foD$ 
gives a one-to-one correspondence between equivalence classes of consistent
scattering diagrams and elements $\theta_{+,-}\in G$.
\end{theorem}

This is a special case of \cite{KS13}, 2.1.6. For the reader's convenience we 
include the short proof: 

\begin{proof}
We need to show how to construct $\foD$ given $\theta_{+,-}\in G$.
To do so, choose any $n_0\in N^+$ primitive and a point 
$x\in n_0^{\perp}$ general. 
Then we can determine $g_x(\foD)$ as follows, 
noting by Lemma \ref{easyequivlemma} that this information for all such
$n_0$ and general $x$ determines $\foD$ up to equivalence.
We can write
\begin{equation}
\label{xsplitting}
\fog=\fog^x_{+}\oplus\fog^x_0\oplus\fog^x_-
\end{equation}
with
\[
\fog^x_+=\bigoplus_{n\in N^{+} \atop \langle n,x\rangle >0} \fog_n ,\quad
\fog^x_-=\bigoplus_{n\in N^+\atop \langle n,x\rangle <0} \fog_n,\quad
\fog^x_0=\bigoplus_{n\in N^+\atop \langle n,x\rangle =0}
\fog_n.
\]
Each of these subspaces of $\fog$ are closed under Lie bracket, thus
defining subgroups $G_{\pm}^x, G_0^x$ of $G$. Note by the generality
assumption on $x$, we in fact have $\fog_0^x=\fog_{n_0}^{\|}$.
This splitting induces a unique
factorization $g=g^x_+ \cdot g^x_0\cdot  g^x_-$ for any element $g\in G$.
Applying this to $\theta_{+,-}$ gives a
well-defined element $g_0^x\in G_0^x$. We need to show that the set of
data $g_0^x$ determines a scattering diagram $\foD$ such that
$g_x(\foD)=g_0^x$ for all general $x \in M_{\bR}$.
To do this, one needs to know that to any finite order $k$, the hyperplane 
$n_0^{\perp}$ is subdivided into a finite number of
polyhedral cones $\fod_1,\ldots,\fod_p$ such that the image of $g_0^x$
in $G^{\le k}$ is constant for $x\in \fod_i$. This is clear 
because
the number of $n\in N^+$ with $d(n)\le k$ is finite, as then the decomposition
\eqref{xsplitting} varies discretely with $x$ to order $k$. 


We need to show that $\foD$ satisfies the condition that
$\theta_{\gamma,\foD}=\theta_{+,-}$ for any path $\gamma$ 
from the positive to the negative chamber and that
$\theta_{\gamma,\foD}$ only depends on endpoints of $\gamma$. 
To do so, we work modulo $\fog^{>k}$ for any $k$, so we can
assume $\foD$ has a finite number of walls. Choose a general point
$x_0\in \shC^+$. Take a general two-dimensional subspace of $M_{\RR}$
containing $x_0$, and after choosing a metric, let $\gamma$ be a semi-circle
in the two-dimensional subspace with endpoints $x_0$ and $-x_0$ and center
$0$. Then
$\theta_{\gamma,\foD}=g_0^{x_n}\cdots g_0^{x_1}$ for points $x_1,
\ldots,x_n$ contained in walls crossed by $\gamma$ and $g_0^{x_i}$ the
element of $G_0^{x_i}$ determined by the factorization of $\theta_{+,-}$
above. Note that if $x_i$ lies in the hyperplane $n_i^{\perp}$, 
all the wall-crossing automorphisms of walls
traversed by $\gamma$ \emph{before} crossing $n_i^{\perp}$ lie in $G^{x_i}_-$
and all those from walls traversed by $\gamma$ \emph{after} crossing
$n_i^{\perp}$ lie in $G^{x_i}_+$.
It then follows inductively that the factorization of
$\theta_{+,-}$ given by $x_i$ takes the form $(g_+)g_0^{x_i}(g^{x_{i-1}}_0
\cdots g^{x_1}_0)$ for some $g_+\in G^{x_i}_+$. Indeed, for $i=1$,
this just follows from the definition of $g_0^{x_1}$, while if true for
$i-1$, then we have $\theta_{+,-}=g'\cdot  (g_0^{x_{i-1}}\cdots g_0^{x_1})$ is
a decomposition of $\theta_{+,-}$ induced by the splitting
$\fog=(\fog_+^{x_i}\oplus \fog_0^{x_i}) \oplus\fog_-^{x_i}$, and the
claim then follows by the definition of $g_0^{x_i}$.
In particular, for $i=n+1$,
taking $x_{n+1}=-x_0$ and noting that $G^{x_{n+1}}_{-}= G$, one sees that
$\theta_{+,-}=g^{x_n}_0\cdots g_0^{x_1}=\theta_{\gamma,\foD}$.

Next we show the independence of path for the $\foD$ we have constructed, 
again modulo $\fog^{>k}$. It
is sufficient to check $\theta_{\gamma_{\foj},\foD}=\id$ as an element
of $G^{\le k}$ 
for any small loop $\gamma_{\foj}$ around any joint $\foj$ of $\foD$.
Take $x'$ a general point in $\foj$, $n\in N^+$ such that 
$n^{\perp}\supseteq \foj$, and choose $x, x''$ to be points in $n^{\perp}$
near $x'$ on either side of the joint $\foj$. Let $\gamma,\gamma''$
be two semi-circular paths with endpoints $x_0$ and $-x_0$ and passing
through $x, x''$ respectively. Then
up to orientation $\gamma(\gamma'')^{-1}$ is freely homotopic to 
$\gamma_{\foj}$ in $M_{\RR}\setminus\Sing(\foD)$. 
Thus $\theta_{\gamma_{\foj},\foD}=\theta_{\gamma'',\foD}^{-1}
\theta_{\gamma,\foD}=\theta_{+,-}^{-1}\theta_{+,-}=\id$.

Thus we have established the one-to-one correspondence between
consistent scattering diagrams $\foD$
and elements of $G$.
\end{proof}

Following \cite{KS13}, we give an alternative 
parameterization of $G$, as follows.
For any $n_0\in N^+$ primitive, we get the splitting
\begin{equation}
\label{n0splitting}
\fog=\fog^{n_0}_+\oplus\fog^{n_0}_0\oplus \fog^{n_0}_-
\end{equation}
where 
\[
\fog^{n_0}_+ := \bigoplus_{\{n_0,n\}>0} \fog_n,
\quad
\fog^{n_0}_- := \bigoplus_{\{n_0,n\}<0} \fog_n,
\quad
\fog^{n_0}_0 := \bigoplus_{\{n_0,n\}=0} \fog_n.
\]
These give rise to subgroups
$G^{n_0}_{\pm}, G^{n_0}_0$ of $G$. We drop the $n_0$ when it is clear from 
context.
Again, this allows us to factor any $g\in G$ as
$g=g_+\circ g_0\circ g_-$ with $g_{\pm}\in G_{\pm}$, $g_0\in G_0$.
We can further decompose $\fog_0=\fog_0^{\|}\oplus\fog_0^{\perp}$,
where $\fog_0^{\|} := \fog_{n_0}^{\|}$,  while
$\fog_0^{\perp}$ involves those summands of $\fog_0$ coming from $n$ not
proportional to $n_0$. Note that $[\fog_0,\fog_0^{\perp}]\subseteq
\fog_0^{\perp}$. Indeed, if $n_1+n_2=kn_0$ with $\{n_i,n_0\}=0$ for $i=1,2$,
we then have $\{n_1,n_2\}=0$ so that $[\fog_{n_1},\fog_{n_2}] = 0$ by (\ref{Liebracketform}).
Thus we have a projection homomorphism
$G_0\rightarrow G_0^{\|}$ with kernel $G_0^{\perp}$. In particular,
the factorization $g=g_+\circ g_0\circ g_-$ yields an element
$g_{0}^{\|}\in G_{0}^{\|}$ via this projection. We then have a map (of sets) 
\[
\Psi:G\rightarrow \prod_{\scriptstyle\hbox{$n_0\in N^+$
primitive}} G_{n_0}^{\|}.
\]
\begin{proposition}
\label{Gparam}
$\Psi$ is a set bijection
\end{proposition}
\begin{proof}
$\Psi$ is induced by an analogous map to order $k$, 
\[
\Psi_k:G^{\leq k} \to \prod \exp(\fog^{\|}_{n_0}/\fog^{\|}_{n_0}\cap\fog^{>k}).
\] 
One checks easily that this is a bijection order by order.
\end{proof}

\begin{theorem}
\label{KSlemma}  Let $\foD$ be a consistent scattering diagram
corresponding to $\theta_{+,-} \in G$. 
The following hold:
\begin{enumerate}
\item For each $n_0\in N^+$, to any fixed finite order, 
there is an open neighbourhood 
$U \subset n_0^{\perp}$ of $p^*(n_0)$
such that $g_x(\foD) = \Psi(\theta_{+,-})_{n_0} \in G_0^x = G_{n_0}^{||}$ for all
general $x \in U$.
Here $\Psi(g)_{n_0}$ denotes the component of $\Psi(g)$ indexed
by $n_0$.
\item
$\foD$ is equivalent to a diagram with only
one wall in $n_0^{\perp}$ containing $p^*(n_0)$ for each $n_0\in N^+$, and the group element
attached to this wall is 
$\Psi(\theta_{+,-})_{n_0}$.
\item Set
\[
\foD_{\inc}:=\{(n_0^{\perp}, \Psi(\theta_{+,-})_{n_0})\,|\,
\hbox{$n_0\in N^+$ primitive}\}.
\]
Then $\foD$ is equivalent to a consistent scattering diagram $\foD'$
such that $\foD' \supseteq \foD_{\inc}$
and $\foD' \setminus\foD_{\inc}$ consists only of outgoing walls.
Furthermore, up to equivalence $\foD'$ is the unique consistent
scattering diagram with this property.
\item The equivalence class of a consistent scattering diagram is determined by
its set of incoming walls.
\end{enumerate}
\end{theorem}

We note first that (3) of Theorem \ref{KSlemma} implies Theorem \ref{KSlemmaGS}.
Indeed, let $g_i\in G$ for $i\in I_{\uf}$ be the group element corresponding
to $1+z^{v_i}$, so that the initial scattering can be written as
$\foD_{\inc,\s}=\{(e_i^{\perp}, g_i)\,|\, i\in I_{\uf} \}$.
By Proposition \ref{Gparam}
there is a unique element $g \in G$ with 
\[
\Psi(g)_{n} = \begin{cases} g_i & n=e_i, i\in I_{\uf},\\
1 & \hbox{otherwise}.
\end{cases}
\]
Now apply Theorem \ref{KSlemma} with $\theta_{+,-}=g$.

\begin{proof}[Proof of Theorem \ref{KSlemma}] First note that statement
(1) implies (2). Further, (1), along with Theorem
\ref{maximscorrespondence} and Proposition \ref{Gparam}, implies (4), which
in turn gives the uniqueness in (3).
Note (1) implies that, to the given finite order, $\foD$ is equivalent
to a diagram having only one incoming wall contained
in $n_0^{\perp}$, and the attached group element is $\Psi(\theta_{+,-})_{n_0}$.
Now we can replace this single wall by an
equivalent collection of walls consisting of $(n_0^{\perp},\Psi(\theta_{+,-})_{n_0})$
and a number
of outgoing walls contained in $n_0^{\perp}$ with attached group element
$\Psi(\theta_{+,-})^{-1}_{n_0}$. This gives the existence in (3).

Thus it suffices to prove (1). We work modulo
$\fog^{>k}$, so we may assume $\foD$ is finite, and compare the splittings
\eqref{xsplitting} coming from a choice of $x\in n_0^{\perp}$ near
$p^*(n_0)$ and \eqref{n0splitting}, after replacing $\fog$ with
$\fog/\fog^{>k}$. For each $n\in N^+$ there
exists an open neighbourhood
$U_n \subset n_0^{\perp}$ of $p^*(n_0)$ such that $\langle p^*(n_0),n \rangle >0$ (resp.\ $< 0$)
implies $\langle x,n\rangle >0$ (resp.\ $< 0$) for all $x \in U_n$. 
Since $\fog^{n_0}_{\pm}$ is now a finite sum of $\fog_n$'s,
we can find a single $U$ so that
$\fog^{n_0}_{\pm}\subseteq \fog_{\pm}^x$ for all $x \in U$. If $x$ is general 
inside $n_0^{\perp}$ we also have $\fog_0^x=\fog^{\|}_{n_0}$. 

Now write
\[
\theta_{+,-} = g^{n_0}_+ \cdot g^{n_0}_0 \cdot g^{n_0}_-
\]
as in \eqref{n0splitting}. Then we can further
factor
\[
g^{n_0}_0 = h^{x}_+ \cdot h^{x}_0 \cdot h^{x}_-
\]
as in \eqref{xsplitting}. Note $h^x_{\pm} \in G_{0}^{\perp}$,
$h^x_0 \in G_{n_0}^{\|} = G^x_0$. Since the projection $G_0\rightarrow
G_0^{\|}$ is a group homomorphism with kernel
$G_0^{\perp}$, 
the image of $g^{n_0}_0$ in $G_{n_0}^{\|}$ is $h^x_0$, which thus
coincides with $\Psi(\theta_{+,-})_{n_0}$ by definition of the latter.
We have
\[
\theta_{+,-} = (g^{n_0}_+ \cdot h^x_+)\cdot h^x_0 \cdot (h^x_- \cdot g^{n_0}_-)
\]
which is then the (unique) factorisation from \eqref{xsplitting}.
Thus
\[
g_x(\theta_{+,-}) = h^x_0 = \Psi(\theta_{+,-})_{n_0}
\]
for any general $x \in U$.
\end{proof}

\subsection{Mutation invariance of the scattering diagram} 
\label{mutinvsec}

We now study how the scattering diagram $\foD_{\s}$ constructed from
seed data defined in the previous subsection
changes under mutation. This is crucial for uncovering the
chamber structure of these diagrams and giving the connection with the
exchange graph and cluster complex.

Thus let $k\in I_{\uf}$ and $\s'=\mu_k(\s)$ be the mutated seed (see e.g.,
\cite{P1}, (2.3)). 
To distinguish the two Lie algebras involved, we write $\fog_{\s}$ and
$\fog_{\s'}$ for the Lie algebras arising from these two different
seeds. We recall that the Injectivity Assumption is independent of the
choice of seed.

\begin{definition}
\label{halfspacedef}
\label{tkdef}
We set
\[
\cH_{k,+}:=\{m\in M_{\RR}\,|\, \langle e_k,m\rangle \ge 0\}, 
\quad \cH_{k,-}:=\{m\in M_{\RR}\,|\, \langle e_k,m\rangle \le 0\}.
\]
For $k\in I_{\uf}$, 
define the piecewise linear transformation
$T_k:M^{\circ}\rightarrow M^{\circ}$ by, for $m\in M^{\circ}$,
\begin{equation}
\label{Tkdefinition}
T_k(m):=\begin{cases}
m+v_k\langle d_ke_k,m\rangle &m\in \cH_{k,+}\\
m& m\in \cH_{k,-}.
\end{cases}
\end{equation}
As we will explain in \S \ref{tropsec}, $T_k$ is the tropicalisation of $\mu_k$
viewed as a birational map between tori.
We will write $T_{k,-}$ and $T_{k,+}$ to be the linear transformations
used to define $T_k$ in the regions $\cH_{k,-}$ and $\cH_{k,+}$
respectively.

Define the scattering diagram $T_k(\foD_{\s})$ to be the scattering diagram
obtained by: 
\begin{enumerate}
\item for each wall $(\fod,f_{\fod})\in \foD_{\s}\setminus
\{\fod_k\}$, where $\fod_k:=(e_k^{\perp}, 1+z^{v_k})$, 
we have one or two walls in $T_k(\foD_{\s})$ given as
\[
\big(T_k(\fod\cap\cH_{k,-}), T_{k,-}(f_{\fod})\big),
\quad
\big(T_k(\fod\cap\cH_{k,+}), T_{k,+}(f_{\fod})\big),
\]
throwing out the first or second of these 
if $\dim \fod\cap\cH_{k,-}<\rank M-1$ or $\dim\fod\cap\cH_{k,+}<\rank M-1$,
respectively. Here for $T:M^{\circ}\rightarrow M^{\circ}$ linear, we write
$T(f_{\fod})$ for the formal power series obtained by applying $T$ to
each exponent in $f_{\fod}$.
\item $T_k(\foD_{\s})$ also contains the wall 
$\fod_k':=(e_k^{\perp}, 1+ z^{-v_k})$.
\end{enumerate} 
\end{definition}

The main result of this subsection is:

\begin{theorem}
\label{Adiagrammutations}
Suppose the Injectivity Assumption is satisfied.
Then $T_k(\foD_{\s})$ is a consistent 
scattering diagram for $\fog_{\mu_k(\s)}$ and $N^+_{\mu_k(\s)}$. 
Furthermore, $\foD_{\mu_k(\s)}$ and $T_k(\foD_{\s})$ are equivalent.
\end{theorem}

The main point in the proof, which is not at all obvious from the
definition, is that $T_k(\foD_{\s})$ is a scattering diagram for $\fog_{\s'}$,
$N^+_{\s'}$, where $\s'=\mu_k(\s)$. Formally,
consistency will be easy to check using consistency of 
$\foD_{\s}$. It will follow easily that by construction
$\foD_{\s'}$ and $T_k(\foD_{\s})$ have the same incoming walls, 
so the theorem will then follow from the uniqueness in Theorem \ref{KSlemmaGS}. 

The main problem to overcome is that the functions attached to walls
of $\foD_{\s}$ and $\foD_{\s'}$ live in two different completed monoid
rings, $\widehat{\kk[P]}$ and $\widehat{\kk[P']}$, for $P$ a monoid chosen
to contain $v_i, i\in I_{\uf}$, and $P'$ a monoid chosen to contain
$v_i', i\in I_{\uf}$. 
We need first a common monoid $\oP$ containing both $P$ and $P'$.

\begin{definition}
\label{Pbardef}
Let $\sigma\subseteq M^{\circ}_{\RR}$ be a top-dimensional cone
containing $v_i$, $i\in I_{\uf}$, and $-v_k$, and such
that $\sigma\cap (-\sigma)=\RR v_k$. Set
$\oP=\sigma\cap M^{\circ}$, 
and $J=\oP\setminus (\oP \cap \RR v_k)=\oP \setminus \oP^{\times}$.
\end{definition}

Given such a choice of $\oP$, we can find $P$, $P'$ contained in $\oP$.
However, we have an additional problem that $\foD_{\s}$
is not trivial modulo $J$. Indeed, $v_k \not \in J$, 
while one of the initial walls of 
$\foD_{\inc,\s}$ is $(e_k^{\perp},1 + z^{v_k})$.
In particular, the wall-crossing automorphism associated to 
\[
\fod_k:=(e_k^{\perp},1+z^{v_k})
\]
is not an automorphism of the ring
$\widehat{\kk[\oP]}$, but rather of the localized ring
$\widehat{\kk[\oP]}_{1+z^{v_k}}$. (Here the hats denote completion with respect to $J$.)
This kind of situation is dealt with in \cite{GSAnnals}, see especially
\S4.3. However the current situation is quite a bit simpler, so we will
give the complete necessary arguments here and in Appendix \ref{scatappendix}.

We will use the notation $\theta_{\fod_k}$ for the automorphism
of $\widehat{\kk[\oP]}_{1+z^{v_k}}$ associated to crossing the wall
$\fod_k$ from $\cH_{k,-}$ to $\cH_{k,+}$. Explicitly,
\begin{equation}
\label{thetafodkdef}
\theta_{\fod_k}(z^m)=z^m(1+z^{v_k})^{-\langle d_ke_k, m\rangle}.
\end{equation}

In this situation, define 
\[
N_{\s}^{+,k}:=
\left\{\sum_{i\in I_{\uf}}a_ie_i \,\bigg |\, \hbox{$a_i\in\ZZ_{\ge 0}$ for 
$i\not=k$, $a_k\in\ZZ$, and $\sum_{i \in I_{\uf} \setminus \{ k \}} a_i > 0$} \right\}.
\]
We note that by the definition of the mutated seed $\s'$, $N_{\s}^{+,k} = N_{\s'}^{+,k}$,
so we indicate it by $N^{+,k}$. 

We now extend the definition of scattering diagram. 

\begin{definition} 
\label{escdef} 
A \emph{wall} for $\oP$ and 
ideal $J$ is a pair $(\fod, f_{\fod})$ with $\fod$ as in Definition 
\ref{walldef}, but with $n_0 \in N^{+,k}$, and 
$f_{\fod}=1+\sum_{k=1}^{\infty} c_kz^{kp^*(n_0)} \in \widehat{\kk[\oP]}$ 
congruent to $1$ mod $J$. The {\it slab} for the seed $\s$ means the pair 
$\fod_k=(e_k^{\perp}, 1 + z^{v_k})$.  Note since $v_k \in \oP^{\times}$ this 
does not qualify as a wall. Now a scattering diagram $\foD$ is
a collection of walls and possibly this single slab, with the condition that
for each $k>0$,
$f_{\fod}\equiv 1 \mod J^k$ for all but finitely many walls in $\foD$.
\end{definition}

Note that crossing a wall or slab $(\fod,f_{\fod})$ now induces
an automorphism of $\widehat{\kk[\oP]}_{1+z^{v_k}}$ of the form 
$\theta^{\pm 1}_{f_{\fod}}$ 
(with the localization only needed when a slab is crossed).

The following is proved in Appendix \ref{scatappendix}:

\begin{theorem}
\label{KSlemma2} 
There exists a scattering diagram $\overline{\foD}_{\s}$ in the sense of
Definition \ref{escdef} such that 
\begin{enumerate}
\item $\overline\foD_{\s}\supseteq \foD_{\inc,\s}$,
\item $\overline\foD_{\s}\setminus \foD_{\inc,\s}$ consisting 
only of outgoing walls, and 
\item $\theta_{\gamma,\foD}$ as an automorphism
of $\widehat{\kk[\oP]}_{1+z^{v_k}}$ only depends on the endpoints of
$\gamma$.
\end{enumerate}
Furthermore, $\overline\foD_{\s}$ with these properties is unique up to equivalence. 

Finally, $\overline\foD_{\s}$ is also a scattering diagram
for the data $\fog_{\s}, N^+_{\s}$, and as such 
is equivalent to $\foD_{\s}$. 
\end{theorem}

\begin{remark}
\label{noekwalls}
Note in particular that the theorem implies 
$\foD_{\s} \setminus \foD_{\inc,\s}$ does not contain any
walls contained in $e_k^{\perp}$ besides $\fod_k$. Indeed, no wall
of $\ofoD_{\s}$ is contained in $e_k^{\perp}$: only the slab $\fod_k$
is contained in $e_k^{\perp}$.
\end{remark}

\begin{proof}[Proof of Theorem \ref{Adiagrammutations}]
We write $\s'=\mu_k(\s)$, $\s'=(e_i'\,|\,i\in I)$.

We first note that we can choose representatives for $\foD_{\s}$,
$\foD_{\s'}$ which are scattering diagrams in the sense of Definition
\ref{escdef}, by Theorem \ref{KSlemma2}. Furthermore, $T_k(\foD_{\s})$
is also a scattering diagram in the sense of Definition \ref{escdef} for the 
seed $\s'$:
this follows since if $z^m\in J^i$ for some $i$, we also have
$z^{T_{k,\pm}(m)}\in J^i$. Thus by the uniqueness statement of
Theorem \ref{KSlemma2}, $T_k(\foD_{\s})$ and $\foD_{\s'}$ are equivalent
if (1) these diagrams are equivalent to diagrams which have the same set of
slabs and 
incoming walls; (2) $T_k(\foD_{\s})$ is consistent. We carry out these two 
steps.

\emph{Step I. Up to equivalence, $T_k(\foD_{\s})$ and $\foD_{\s'}$
has the same set of slabs and incoming walls.}

If $\fod\in\foD_{\s}$ is outgoing, the wall $\fod$ contributes 
to $T_k(\foD_{\s})$ and is
also outgoing, so let us consider the incoming walls of $T_k(\foD_{\inc,\s})$.
Setting $v_i'=p^*(e_i')$, already
$\foD_{\inc,\s'}$ contains the slab for $\s'$
\[
((e_k')^{\perp},1+z^{v_k'})=
(e_k^{\perp},1+z^{-v_k})=\fod_k',
\]
which lies in $T_k(\foD_{\inc,\s})$ by construction. Next consider
the wall $(e_i^{\perp},1+z^{v_i})$, for $i\not=k$. 
We have three cases to consider, based on whether $\langle v_i,e_k\rangle$
is zero, positive or negative.

First if $\langle v_i,e_k\rangle =0$, then $T_k$ takes the plane
$e_i^{\perp}$ to itself (in a piecewise linear way), and $T_{k,+}(v_i)
=T_{k,-}(v_i)=v_i$. Thus the wall $(e_i^{\perp}, 1+z^{v_i})$
contributes two walls $(e_i^{\perp}\cap\cH_{k,\pm},1+z^{v_i})$
whose union is the wall $((e_i')^{\perp},1+z^{v_i})$, as
$e_i'=e_i$ and $v_i'=v_i$ in this case. Up to equivalence, we can
replace these two walls with the single wall $((e_i')^{\perp},
1+z^{v_i})$.

If $\langle v_i, e_k\rangle >0$, then consider the wall
\[
\fod_{i,+}:=\big(T_k(\cH_{k,+}\cap e_i^{\perp}), 1+z^{T_{k,+}(v_i)}\big)
\in T_k(\foD_\s).
\] 
This wall contains the ray $\RR_{\ge 0}T_{k,+}(v_i)$,
so this is an incoming wall. Note that if 
$m\in \cH_{k,+}\cap e_i^{\perp}$, we have, with $\epsilon$ as given
in \eqref{epsilondef}, 
\begin{align*}
\langle e_i',T_k(m)\rangle = {} & \langle e_i+[\epsilon_{ik}]_+e_k,
m+v_k\langle d_ke_k,m\rangle\rangle\\
= {} & \{e_k,e_i\}\langle d_ke_k,m\rangle 
+d_k\{e_i,e_k\}\langle e_k,m\rangle\\
= {} & 0.
\end{align*}
Thus $T_k(\cH_{k,+}\cap e_i^{\perp})$ is a half-space contained in 
$(e_i')^{\perp}$, and furthermore $1+z^{T_{k,+}(v_i)}=1+z^{v_i'}$
since
\[
T_k(v_i)=v_i+v_k\epsilon_{ik}=v_i'.
\]
Thus we see that the wall $\fod_{i,+}$ of $T_k(\foD_{\inc,\s})$ is
half of the wall $((e_i')^{\perp},1+z^{v_i'})$ of $\foD_{\inc,\s'}$.

If $\langle v_i,e_k\rangle<0$, then the wall
$\fod_{i,-}:=\big(T_k(\cH_{k,-}\cap e_i^{\perp}),1+z^{T_{k,-}(v_i)}\big)
\in T_k(\foD_\s)$
coincides with $(\cH_{k,-}\cap e_i^{\perp}, 1+z^{v_i})$, 
and $\cH_{k,-}$ also contains $\RR_{\ge 0}v_i$, so $\fod_{i,-}$ 
is an incoming wall.
But also $v_i'=v_i$, $e_i'=e_i$ in this case. Thus $\fod_{i,-}$ is
again half of the wall $((e_i')^{\perp},1+z^{v_i'})$.

In summary, we find that after splitting some of the walls of $\foD_{\inc,\s'}$
in two, $T_k(\foD_{\inc,\s})$ and $\foD_{\inc,\s'}$ have the same set of
incoming walls, and thus, making a similar change to $\foD_{\s'}$, we
see that $T_k(\foD_{\s})$ and $\foD_{\s'}$ have the same set of incoming
walls.

\smallskip

\emph{Step II.}
$\theta_{\gamma,T_k(\foD_{\s})}=\id$ for any loop
$\gamma$ for which this automorphism is defined. 

\smallskip

Indeed, the only place a problem can occur is for $\gamma$ a loop around a 
joint of $\foD_{\s}$
contained in the slab $\fod_k$, as this is where $T_k$ fails to be
linear. To test this, consider a loop $\gamma$ around a joint contained
in $\fod_k$. Assume that it has basepoint in the half-space $\cH_{k,-}$
and is split up as $\gamma=\gamma_1\gamma_2\gamma_3
\gamma_4$, where $\gamma_1$ immediately crosses $\fod_k$, $\gamma_2$
is contained entirely in $\cH_{k,+}$, crossing all
walls of $\foD_{\s}$ which contain $\foj$ and intersect the interior
of $\cH_{k,+}$, $\gamma_3$
crosses $\fod_k$ again, and $\gamma_4$ then crosses all relevant walls
in the half-space $\cH_{k,-}$. 

Let $\theta_{\fod_k}$, $\theta_{\fod_k'}$ be the wall-crossing automorphisms
for crossing $\fod_k$ or $\fod_k'$ passing from $\shH_{k,-}$ to
$\shH_{k,+}$, as in \eqref{thetafodkdef}. Then by Remark \ref{noekwalls},
$\theta_{\gamma_1,\foD_{\s}}=\theta_{\fod_k}$
and $\theta_{\gamma_3,\foD_{\s}} =\theta_{\fod_k}^{-1}$. 

Let $\alpha:\kk[M^{\circ}]\rightarrow\kk[M^{\circ}]$
be the automorphism induced by $T_{k,+}$, i.e.,
\[
\alpha(z^m)=z^{m+v_k\langle d_ke_k,m\rangle}.
\]
Then note that
\begin{align*}
\theta_{\gamma_1,T_k(\foD_{\s})}= {} & \theta_{\fod_k'}\\
\theta_{\gamma_2,T_k(\foD_{\s})} = {} & \alpha\circ \theta_{\gamma_2,\foD_{\s}}\circ
\alpha^{-1}\\
\theta_{\gamma_3,T_k(\foD_{\s})} = {} &\theta_{\fod_k'}^{-1}\\
\theta_{\gamma_4,T_k(\foD_{\s})}= {} & \theta_{\gamma_4,\foD_{\s}}.
\end{align*}
Thus to show $\theta_{\gamma,\foD_{\s}}=\theta_{\gamma,T_k(\foD_{\s})}$,
it is enough to show that
\[
\alpha^{-1}\circ \theta_{\fod_k'}=\theta_{\fod_k}.
\]
But
\begin{align*}
\alpha^{-1}(\theta_{\fod_k'}(z^m)) = {} &
\alpha^{-1}((1+z^{-v_k})^{-\langle d_ke_k,m\rangle}z^m)\\
= {} & (1+z^{-v_k})^{-\langle d_ke_k,m\rangle}z^{m-v_k\langle d_ke_k,m\rangle}\\
= {} & z^m(z^{v_k}+1)^{-\langle d_ke_k,m\rangle}\\
= {} & \theta_{\fod_k}(z^m),
\end{align*}
as desired.
\end{proof}

\begin{construction}[The chamber structure]
\label{chamberstructureremark}
Suppose given fixed data $\Gamma$ satisfying the Injectivity
Assumption and seed data $\s$. We then obtain for every seed
$\s'$ obtained from $\s$ via mutation a scattering diagram $\foD_{\s'}$.
In each case we will choose a representative for the scattering diagram
with minimal support.

Note by construction and Remark \ref{noekwalls}, irrespective
of the representative of $\foD_{\s}$ used, $\foD_{\s}$ contains
walls whose union of supports is $\bigcup_{k\in I_{\uf}}e_k^{\perp}$.
Furthermore, we have $\shC^{\pm}\subseteq M_{\RR}$ given by 
\eqref{chamberdef}, which can be written more explicitly as
\begin{align*}
\shC^+_{\s}:=\shC^+=&\{m\in M_{\RR}\,|\, \langle e_i,m\rangle \ge 0\quad \forall
i\in I_{\uf}\},\\
\shC^-_{\s}:=\shC^-=&\{m\in M_{\RR}\,|\, \langle e_i,m\rangle \le 0\quad \forall
i\in I_{\uf}\},
\end{align*}
Then $\shC^{\pm}_{\s}$ are the closures of connected components of 
$M_{\RR}\setminus\Supp(\foD_{\s})$. Similarly, we see that
taking $\shC^{\pm}_{\mu_k(\s)}$ to be the chambers
where all $e_i'$ are positive (or negative), we
have that $\shC^{\pm}_{\mu_k(\s)}$ is the closure of 
a connected component of $M_{\RR}
\setminus\Supp(\foD_{\mu_k(\s)})$, so that $T_k^{-1}(\shC^{\pm}_{\mu_k(\s)})$ 
is the closure of a connected component of $M_{\RR}\setminus\Supp(\foD_{\s})$.
Note that the closures of $T_k^{-1}(\shC^+_{\mu_k(\s)})$ and 
$\shC^+_{\s}$ have a common
codimension one face given by the intersection with $e_k^{\perp}$. This
gives rise to the following chamber structure for a subset of
$M_{\RR}\setminus \Supp(\foD_{\s})$.

Recall from Appendix \ref{LDsec} the infinite oriented tree $\foT$ (or
$\foT_{\s}$) used for parameterizing seeds obtained via mutation of $\s$. 
In particular, for any vertex $v$ of $\foT$, 
there is a simple path from the root
vertex to $v$, indicating a sequence of mutations $\mu_{k_1},\ldots,\mu_{k_p}$ 
and hence a piecewise linear transformation
\[
T_{v}=T_{k_p}\circ\cdots\circ T_{k_1}:M_{\RR}\rightarrow M_{\RR}.
\]
Note that $T_{k_i}$ is defined using the basis vector $e_{k_i}$
of the seed $\mu_{k_{i-1}}\circ\cdots \circ\mu_{k_1}(\s)$, not the
basis vector $e_{k_i}$ of the original seed $\s$. By applying Theorem
\ref{Adiagrammutations} repeatedly, we see that 
\begin{equation}
\label{Tvequation}
T_{v}(\foD_{\s})=\foD_{\s_v}
\end{equation}
(where $T_{v}$ applied
to the scattering diagram $\foD_{\s}$ is interpreted as the composition
of the actions of each $T_{k_i}$) and 
\[
\shC^{\pm}_{v}:=T_{v}^{-1}(\shC^{\pm}_{\s_v})
\]
is the closure of a connected component of $M_{\RR}\setminus\Supp(\foD_{\s})$. 

Note that the map from vertices of $\foT$ to chambers of $\Supp(\foD_{\s})$
is never one-to-one. Indeed, if $v$ is the vertex obtained by following
the edge labelled $k$ twice starting at the root vertex, one checks
that $\shC^{\pm}_v=\shC^{\pm}_{\s}$, even though
$\mu_k(\mu_k(\s))\not=\s$ (see \cite{P1}, Remark 2.5).

Thus we have a chamber structure on a subset of $M_{\RR}$; in general,
the union of the cones $\shC^{\pm}_{v}$ do not form a dense subset
of $M_{\RR}$. 

Since we will often want to compare various aspects of this geometry 
for different seeds, we will write the short-hand $v\in \s$ for
an object parameterized by a vertex $v$ where the root of the tree
is labelled with the seed $\s$. In particular:

\begin{definition}
\label{FGchamberdef}
We write
$\shC^{\pm}_{v\in\s}$ for the chamber of $\Supp(\foD_{\s})$ corresponding
to the vertex $v$.
We write $\Delta_{\s}^{\pm}$ for the set of chambers
$\shC^{\pm}_{v\in\s}$ for $v$ running over all vertices of $\foT_{\s}$.
We call elements of $\Delta_{\s}^+$ \emph{cluster chambers}.
\end{definition}
\end{construction}

\section{Basics on tropicalisation and the Fock-Goncharov cluster complex} 
\label{tropsec}
We now explain that the chamber structure of Construction 
\ref{chamberstructureremark}
coincides with the Fock-Goncharov cluster complex. To do so, we first 
recall the basics of tropicalisation.

For a lattice $N$ with $M=\Hom(N,\ZZ)$, let
$Q_{\semf}(N)$ be the subset of elements of the field of fractions of 
$\kk[M] = H^0(T_N,\cO_{T_N})$ which can be expressed as
a ratio of Laurent polynomials with non-negative integer coefficients.
Then $Q_{\semf}$ is a semi-field under ordinary multiplication
and addition. For any semi-field $P$,
restriction to the monomials $M\subset Q_{\semf}(N)$ gives a canonical bijection
\[
\Hom_{\semf}(Q_{\semf}(N),P) \to \Hom_{\groups}(M,P^{\times}) = N 
\otimes_{\bZ} P^{\times}
\]
where the first $\Hom$ is maps of semi-fields,
$P^{\times}$ means the multiplicative group of $P$, and in the last
tensor product we mean $P^{\times}$ viewed as $\bZ$-module.
Following \cite{FG09}  we define
the $P$-valued points of $T_N$ to be $T_N(P) :=\Hom_{\semf}(Q_{\semf}(N),P)$. A {\it positive} birational map
$\mu: T_N \dasharrow T_N$ means a birational map for which the pullback $\mu^*$ induces an isomorphism
on $Q_{\semf}(N)$. Obviously it gives an isomorphism on $P$-valued points. Thus it makes sense to
talk about $X(P)$ for any variety $X$ with a positive atlas of tori, for example many
of the various flavors of cluster variety.

There are two equally good semi-field structures on $\ZZ$,
the max-plus and the min-plus structures. Here addition is either maximum
or minimum, and multiplication is addition. We notate these as
$\ZZ^T$ and $\ZZ^t$ respectively, thinking of capital $T$ for the max-plus
tropicalization and little $t$ for the min-plus tropicalization.
We similarly define $\RR^T$ and $\RR^t$.
Thus taking $P=\ZZ^T$ or $\ZZ^t$, we obtain the sets of tropical points
$X(\ZZ^T)$ or $X(\ZZ^t)$. The former is the convention used by Fock and
Goncharov in \cite{FG09}, so we refer to this as the Fock-Goncharov
tropicalization. The latter choice in fact coincides with $X^{\trop}(\ZZ)$
as defined in \cite{P1}, Def.\ 1.7, defined as a subset of the set of discrete
valuations. We refer to this as the geometric tropicalization. It will
turn out both are useful. There is the obvious isomorphism of semi-fields
$x \mapsto -x$ from $\bZ^T \to \bZ^t$.
This induces a canonical sign-change identification
$i:X(\bZ^T) \to X(\bZ^t)$.

Given a positive birational map $\mu:T_N\dasharrow T_N$, we use
${\mu}^T:N\rightarrow N$ and ${\mu}^t:N\rightarrow N$
to indicate the induced maps $T_N(\ZZ^T)\rightarrow T_N(\ZZ^T)$ and
$T_N(\ZZ^t)\rightarrow T_N(\ZZ^t)$ respectively. For the geometric 
tropicalization, this coincides with the map on discrete valuations
induced by pullback of functions, see \cite{P1}, \S 1. 
For cluster varieties the two types
of tropicalisation are obviously equivalent. The geometric tropicalisation has 
the advantage
that it makes sense for any log Calabi-Yau variety, while the Fock-Goncharov tropicalisation is restricted
to (Fock-Goncharov) {\it positive spaces}, i.e., spaces obtained by gluing
together algebraic tori via positive birational maps.
We will use both notions, $X(\RR^t)$ because in many cases it is more
natural to think in terms
of valuations/boundary divisors, and $X(\RR^T)$
because, as we indicate below, the scattering
diagram for building $\cA_{\prin}$ lives naturally in $\cA^{\mch}_{\prin}(\bR^T)$ (because of
already established cluster sign conventions).

One computes easily that for the basic mutation
\begin{equation}
\label{mutFG}
{\mu}_{(n,m)}: T_N \dasharrow T_N, \quad \mu_{(n,m)}^*(z^{m'}) = z^{m'}
(1 + z^m)^{\langle m',n\rangle}
\end{equation}
the Fock-Goncharov tropicalisation is
\begin{equation}
{\mu}_{(n,m)}^T: N= T_N(\ZZ^T) \to T_N(\ZZ^T)=N, \quad
x \mapsto x + [\langle m,x\rangle]_{+} n
\end{equation}
while the geometric tropicalisation (see \cite{P1}, (1.4)) is
\begin{equation}
\label{mutgeometric}
{\mu}_{(n,m)}^t: N = T_N(\ZZ^t) \to T_N(\ZZ^t)=N,\quad
 x \mapsto x + [\langle m,x\rangle]_- n.
\end{equation}

Thus:
\begin{proposition} \label{fgtprop}
$T_k: M^\circ \to M^\circ$ defined in \eqref{Tkdefinition}
is the Fock-Goncharov tropicalisation of
\[
{\mu}_{(v_k,d_k e_k)}: T_{M^\circ} \dasharrow T_{M^\circ}.
\]
\end{proposition}

A rational function $f$ on a cluster variety $V$ is called {\it positive} if its
restriction to each seed torus is positive, i.e., can be expressed as a ratio of sums
of characters with positive integer coefficients. We can then define its
Fock-Goncharov tropicalisation
$f^T: V(\bR^T) \to \bR$ by $f^T(p) = -p(f)$. Similarly, for $f$ positive,
we have its
geometric tropicalisation
$f^t: V(\bR^t) \to \bR$ which
for each $v \in V(\RR^t)$ has value $f^t(v) = v(f)$. Using the identification
of $V(\ZZ^t)$ with $V^{\trop}(\ZZ)$, $v$ is interpreted as a valuation and
$f^t(v)$ coincides with $v(f)$, the value of $v$ on $f$. In particular,
this value is defined regardless of whether $f$ is positive.
We have a commutative diagram
\begin{equation} \label{ftcd}
\begin{CD} 
V(\bR^T) @>i>> V(\bR^t) \\
@V{f^T}VV      @V{f^t}VV \\
\bR     @=    \bR
\end{CD}
\end{equation}
where $i$ is the canonical isomorphism determined by the sign change 
isomorphism.
The definition of $f^t$ in
terms of valuations extends the definition of $f^t$, and hence $f^T$ via
this diagram, to any non-zero rational function.  We note that
\begin{equation} \label{pairings}
\begin{aligned} 
(z^m)^T(a) = {} & \langle m,-r(a)\rangle,\quad m \in M, a \in T_N(\bZ^T) \\
(z^m)^t(a) = {} & \langle m,r(a)\rangle, \quad m \in M, a \in T_N(\bZ^t) \\
(z^m)^T(a) = {} & (z^m)^t(i(a)) 
\end{aligned}
\end{equation}
where
\begin{equation}
\label{rrestriction}
r: T_N(P) = \Hom_{\semf}(Q_{\semf}(N),P) \to \Hom_{\groups}(M,P^{\times}) = N \otimes P^{\times}
\end{equation}
is the canonical restriction isomorphism.
We will almost always leave $r,i$ out from the notation.

\begin{lemma} \label{lplemma} 
\begin{enumerate}
\item For a positive Laurent polynomial $g: = \sum_{m \in M} c_m z^m \in 
Q_{\semf}(N)$ (i.e., $c_m \in \bZ_{\geq 0}$),
and $x \in T_N(\bR^T)$
\[
g^T(x) = \min_{m, c_m \neq 0} 
\langle m,-r(x)\rangle
\]
where $r$ is the canonical isomorphism \eqref{rrestriction}.
\item
If $v \in T_N^{\trop}(\bZ)$ is a divisorial discrete valuation, and $g = \sum c_m z^m$ is
any Laurent polynomial (so now $c_m \in \kk$), then
\[
v(g) =: g^t(v) = \min_{m, c_m \neq 0} v(z^m) = 
\min_{m,c_m \neq 0} \langle m,r(v)\rangle.
\]
\end{enumerate}
\end{lemma}

\begin{proof} By definition
\[
g^T(x) =
-\max_{m,c_m\neq 0} \langle m,r(x)\rangle  = \min_{m, c_m \neq 0} 
\langle m,-r(x)\rangle.
\]
This gives the first statement. For the second, we can assume 
$v\in T_N^{\trop}(\ZZ)=N$ is primitive, so part of a basis. Then
the statement reduces to an obvious statement about the $X_1$ degree 
of a linear combination of monomials in $\kk[X_1^{\pm 1},\dots,X_n^{\pm 1}]$.
\end{proof} 

Note the mutations ${\mu}_{(v_k,d_k e_k)}$
are precisely the mutations between the tori in the atlas for
$\cA^{\mch}$ (see Appendix \ref{LDsec} for the definition of the Fock-Goncharov dual $\cA^{\mch}$, and \cite{P1}, (2.5) for the mutations between $\cX$ tori
in our notation).
Thus by Theorem \ref{Adiagrammutations} and Proposition \ref{fgtprop}, 
the support of $\foD_{\s}$ viewed 
as a subset of $\cA^{\mch}(\bR^T)$ under the identification
$M^\circ_{\bR,\s} = \cA^{\mch}(\bR^T)$ (induced canonically
from the open set $T_{M^\circ,\s} \subset \cA^{\vee}$)
is independent of seed. In particular it makes
sense to talk about $\cA^{\mch}(\bR^T) \setminus \Supp(\foD_{\s})$ as being completely
canonically defined without choosing any seed.
For any seed the chambers $\shC^{\pm}_{\s} \subset M^{\circ}_{\bR,\s} = \cA^{\mch}(\bR^T)$ are
connected components of $\cA^{\mch}(\bR^T)\setminus \Supp(\foD_{\s})$.

We recall from \cite{FG11}:

\begin{definition} 
Suppose given fixed data $\Gamma$ and suppose given an initial seed. For
a seed $\s=(e_1,\ldots,e_n)$ obtained by mutation from the initial seed,
the \emph{Fock-Goncharov cluster chamber associated to $\s$} is the
subset 
\[
\{x \in \cA^{\mch}(\bR^T)\,|\,
(z^{e_i})^T(x) \leq 0 \text{ for all } i \in I_{\uf} \},
\]
identified with
\[
\{x \in \cA^{\mch}(\bR^t)\,|\,
(z^{e_i})^t(x) \leq 0 \text{ for all } i \in I_{\uf} \}
\]
via $i$.
The \emph{(Fock-Goncharov) cluster complex} $\Delta^+$ is the set of all such 
chambers.
\end{definition}

\begin{lemma} \label{fgcclem} Suppose given fixed data $\Gamma$ satisfying
the Injectivity Assumption and suppose given an initial seed. For a seed
$\s = (e_1,\dots,e_n)$ obtained by mutation from the initial seed, the chamber
$\shC^{+}_{\s} \subset M^\circ_{\bR,\s} = \cA^{\mch}(\bR^T)$ (also
identified with $\cA^{\mch}(\RR^t)$ via $i$) 
is the Fock-Goncharov cluster chamber
associated to $\s$. 
Hence the Fock-Goncharov cluster chambers are the maximal cones of a 
simplicial fan (of not necessarily strictly convex cones).
In particular
$\Delta^+$ is identified with $\Delta^+_{\s}$ for any choice of seed $\s$ giving
an identification of $\shA^{\vee}(\RR^T)$ with $M^{\circ}_{\RR,\s}$.
\end{lemma}

\begin{proof} The identification of the chamber is immediate from the 
definition. 
The result then follows from the chamber structure of
Remark \ref{chamberstructureremark} and the fact that the $T_k$ are
the Fock-Goncharov tropicalizations of the mutations $\mu_k$ for $\cA^{\vee}$.
It's obvious each maximal cone is simplicial, and each adjacent pair
of maximal cones meets along a codimension one face of each. Hence we obtain
a simplicial fan. 
\end{proof} 

\begin{construction} \label{sdprin} See Appendix \ref{rdsec} for a review of the
cluster variety with principal coefficients, $\cA_{\prin}$.
Any seed $\s$ gives rise to a scattering diagram $\foD_{\s}^{\shA_{\prin}}$
living in
\[
\widetilde M^{\circ}_{\bR,\s}=(M^{\circ}\oplus N)_{\RR,\s}
= (\tN^{\mch})^*_{\bR,\s}, 
\]
the second equality by Proposition \ref{ldpprop}, (3). 
Indeed in this situation, the Injectivity Assumption is satisfied, 
since the form
$\{\cdot,\cdot\}$ on $\widetilde N=N \oplus M^{\circ}$ 
is non-degenerate (which is the reason we use
$\cA_{\prin}$ instead of $\cA$ or $\cX$). Indeed, the vectors
$\tv_i := \{(e_i,0),\cdot\} = (v_i,e_i)$ are linearly independent.
Note by Theorem \ref{KSlemma},
$\foD^{\shA_{\prin}}_{\s}$ contains the scattering diagram
\begin{equation} \label{ageomsd}
\foD^{\shA_{\prin}}_{\inc,\s}:= {} 
\{\big((e_i,0)^{\perp},1+z^{(v_i,e_i)}\big)\,|\, i\in I_{\uf}\}.
\end{equation}

Recall from Proposition \ref{ldpprop} that we have a canonical map
$\rho:\shA_{\prin}^{\vee}\rightarrow \shA^{\vee}$ which is defined
on cocharacter lattices by the canonical projection $M^{\circ}\oplus N
\rightarrow M^{\circ}$, see \eqref{basicmapsdual}. Thus the tropicalization
\[
\rho^T: \cA_{\prin}^{\mch}(\bR^T) \to \cA^{\mch}(\bR^T)
\]
coincides with this projection, which can be viewed as the quotient
of an action of translation by $N$.
By Definition \ref{walldef}, walls of $\foD^{\cA_{\prin}}_{\s}$ are of the
form $(n,0)^{\perp}$
for $n \in N^+$. Thus all walls are invariant under translation by $N$, and
thus are inverse images of walls under $\rho^T$.
So even though
$\cA$ may not satisfy the Injectivity Assumption 
necessary to build a scattering diagram,
we see that $\Supp(\foD^{\cA_{\prin}}_{\s})$ is the inverse image of
a subset of $\shA^{\vee}(\RR^T)$ canonically defined independently of
the seed. In particular, note that the Fock-Goncharov cluster
chamber in $\shA^{\vee}(\RR^T)$ associated to the seed $\s$ (where
$(z^{e_i})^T\le 0$ for all $i\in I_{\uf}$) pulls back to the corresponding
Fock-Goncharov cluster chamber in $\shA_{\prin}^\vee(\RR^T)$. 
\end{construction}

The following was conjectured by Fock and Goncharov, \cite{FG11}, \S 1.5:
 
\begin{theorem} \label{fgccth} For any initial data the Fock-Goncharov cluster chambers in $\cA^{\mch}(\bR^T)$ are
the maximal cones of a simplicial fan.
\end{theorem}

\begin{proof} When the Injectivity Assumption holds,
this follows from Lemma \ref{fgcclem}. In
particular it holds for $\cA_{\prin}$. Now the general case follows by the above invariance
of $\foD_{\s}^{\cA_{\prin}}$ under the translation by $N$.
\end{proof} 

\begin{example}
\label{standardexample}
Consider the rank three skew-symmetric cluster algebra given by the matrix
\[
\epsilon=\begin{pmatrix} 0 & 2 & -2\\ -2 & 0 & 2\\ 2 & -2 & 0\end{pmatrix}.
\]
Then projecting the walls of $\foD_{\s}^{\cA_{\prin}}$ to $M^{\circ}_{\RR}$
via $\rho^T$, one obtains a collection of walls in a three-dimensional
vector space. One can visualize this by intersecting the walls with
the affine hyperplane $\langle e_1+e_2+e_3, \cdot\rangle = 1$. The collection
of resulting rays and lines appears on the first page of \cite{FG11}.
While Fock and Goncharov were not aware of scattering diagrams in this
context, in fact there the picture represents the same slice of the
cluster complex, and hence coincides with the scattering diagram. 

The cluster complex in fact fills up the half-space $\langle
e_1+e_2+e_3,\cdot\rangle \ge 0$. There is no path through chambers connecting
$\shC^-_{\s}$ and $\shC^+_{\s}$.

This example is particularly well-known in cluster theory, and gives the
cluster algebra associated with triangulations of the once-punctured torus.
\end{example}

\section{Broken lines} \label{blsec}

We will explain how a scattering diagram
determines a class of piecewise straight paths which will allow for the
construction of theta functions. The notion of broken line was introduced
in \cite{G09}, and developed from the point of view of defining canonical
functions in \cite{CPS} and \cite{GHK11}.

We choose fixed data $\Gamma$ and a seed $\s$ as described in Appendix
\ref{LDsec}, and assume it satisfies the Injectivity Assumption. 
This gives rise to the group $G$ described in \S\ref{defconstsection}
which acts by automorphisms of $\widehat{\kk[P]}$
for a choice of monoid $P$ containing $p_1^*(N^+)$ and with $P^{\times}=
\{0\}$. The group $G$ also acts on the rank one free
$\widehat{\kk[P]}$-module $z^{m_0}\widehat{\kk[P]}$ for any $m_0\in M^{\circ}$,
with a log derivation $f\partial_n$ acting on $z^{m_0}$ as usual to give
$f\langle n, m_0\rangle z^{m_0}$.

We then have:

\begin{definition} \label{blde}
Let $\foD$ be a scattering diagram in the sense of Definition 
\ref{KSscatdiagdef}, 
$m_0\in M^{\circ}\setminus \{0\}$ and
$Q\in M_{\RR}\setminus \Supp(\foD)$. A \emph{broken line} for $m_0$
with endpoint $Q$ is a piecewise linear continuous proper path
$\gamma:(-\infty,0]\rightarrow M_{\RR}\setminus\Sing(\foD)$ 
with a finite number
of domains of linearity. This path comes along with the data of, for each domain
of linearity $L\subseteq (-\infty,0]$ of $\gamma$, 
a monomial $c_Lz^{m_L}\in \kk[M^{\circ}]$. This data satisfies the following
properties:
\begin{enumerate}
\item $\gamma(0)=Q$.
\item If $L$ is the first (and therefore unbounded) domain of linearity
of $\gamma$, then $c_Lz^{m_L}=z^{m_0}$.
\item For $t$ in a domain of linearity $L$, $\gamma'(t)=-m_L$.
\item Let $t\in (-\infty,0)$ be a point at which $\gamma$ is not linear,
passing from domain of linearity $L$ to $L'$. Let 
\[
\foD_t=\{(\fod,f_{\fod})\in \foD\,|\, \gamma(t)\in\fod\}.
\]
Then $c_{L'}z^{m_{L'}}$ is a term in the formal power series
$\theta_{\gamma|_{(t-\epsilon,t+\epsilon)},\foD_t}(c_Lz^{m_L})$.
\end{enumerate}
\end{definition}

\begin{remark} Note that since a broken line does not pass through
a singular point of $\foD$, we can write
\[
\theta_{\gamma|_{(t-\epsilon,t+\epsilon)},\foD_t}(c_Lz^{m_L})
=c_Lz^{m_L}\prod_{(\fod,f_{\fod})\in \foD_t} f_{\fod}^{\langle n_0, m_L\rangle},
\]
where $n_0\in N^{\circ}$ is primitive, vanishes on each $\fod\in \foD_t$, 
and $\langle n_0, m_L\rangle$ is positive
by item (3) of the definition of broken line. 
It is an important feature of broken lines that we never need to invert.
\end{remark}

\begin{definition}
\label{IFdef}
Let $\foD$ be a scattering diagram, $m_0\in M^{\circ}\setminus \{0\}$, 
$Q\in M_{\RR}\setminus \Supp(\foD)$. For a broken line $\gamma$ for $m_0$ with
endpoint $Q$, define 
\[
I(\gamma) = m_0,
\]
($I$ is for initial),
\[ 
b(\gamma) = Q,
\]
and
\[
\Mono(\gamma) = c(\gamma) z^{F(\gamma)}
\]
to be the monomial attached to the
final ($F$ is for final) domain of linearity of $\gamma$. Define
\[
\vartheta_{Q,m_0}=\sum_{\gamma}\Mono(\gamma),
\]
where the sum is over all broken lines for $m_0$ with endpoint $Q$.

For $m_0=0$, we define for any endpoint $Q$
\[
\vartheta_{Q,0}=1.
\]
\end{definition}

In general, $\vartheta_{Q,m_0}$ is an infinite sum, but makes sense formally:

\begin{proposition}
\label{formalthfun}
$\vartheta_{Q,m_0}\in z^{m_0}\widehat{\kk[P]}$.
\end{proposition}

\begin{proof} 
It is clear by construction that for any broken line $\gamma$ with $I(\gamma)=m_0$,
we have $\Mono(\gamma)\in z^{m_0}\kk[P]$. So it is enough to show that for
any $k>0$, there are only a finite number of broken lines
$\gamma$ such that $I(\gamma)=m_0$, $b(\gamma)=Q$, and $\Mono(\gamma)\not
\in z^{m_0}J^k$. 

First note by the assumption that $J=P\setminus \{0\}$, there
are only a finite number of choices for $F(\gamma)$ such that $\Mono(\gamma)
\not\in z^{m_0} J^k$.
Fix a choice $m$ for $F(\gamma)$. Second, to test that there are finitely many 
broken lines with $I(\gamma)=m_0$, $b(\gamma)=Q$ and $F(\gamma)=m$, 
we can throw out any wall
$\fod\in \foD$ with $f_{\fod}\equiv 1\mod J^k$, so we can assume $\foD$ is
finite. Third, no broken line $\gamma$ with $\Mono(\gamma)\not\in z^{m_0}
J^k$ can bend more than $k$ times. Thus there are only a finite number
of possible ordered sequences of walls $\fod_1,\ldots,\fod_s$ at which 
$\gamma$ can bend. Fix one such sequence. One then sees there are at most
a finite number of broken lines with $b(\gamma)=Q$, $F(\gamma)=m$ bending
at $\fod_1,\ldots,\fod_s$. Indeed, one can start at $Q$ and trace a broken
line backwards, using that the final direction is $-m$. Crossing a wall 
$\fod_i$ and passing from domain of linearity $L$ (for smaller $t$)
to domain of linearity $L'$ (for larger $t$), one sees that
knowing the monomial attached to $L'$ restricts the choices of monomial on $L$
to a finite number of possibilities. This shows the desired finiteness.
\end{proof}

The most important general feature of broken lines is the following:

\begin{theorem} \label{blinvth}
Let $\foD$ be a consistent scattering diagram,
$m_0\in M^{\circ}\setminus\{0\}$, $Q,Q'\in M_{\RR}\setminus \Supp(\foD)$ 
two points with all coordinates irrational. Then for any path
$\gamma$ with endpoints $Q$ and $Q'$ for which $\theta_{\gamma,\foD}$
is defined, we have
\[
\vartheta_{Q',m_0}=\theta_{\gamma,\foD}(\vartheta_{Q,m_0}).
\]
\end{theorem}

\begin{proof} 
This is a special case of results of \S 4 of \cite{CPS}. 
The condition that $Q$ and $Q'$ both have irrational coordinates guarantees
that we don't have to worry about broken lines which pass through joints
(which we aren't allowing).
\end{proof}

Let us next consider how broken lines change under mutation. So let $\s$ be
a seed, $\bar P$ as in Definition \ref{Pbardef}.

\begin{proposition}
\label{brlineinvmut}
$T_k$ defines a one-to-one correspondence between broken lines for $m_0$ with
endpoint $Q$ for $\foD_{\s}$ and broken lines for $T_k(m_0)$ with endpoint
$T_k(Q)$ for $\foD_{\mu_k(\s)}$. 
In particular, depending on whether $Q\in \shH_{k,+}$ or $\shH_{k,-}$, we
have 
\[
\vartheta^{\mu_k(\s)}_{T_k(Q),T_k(m_0)}= T_{k,\pm}(\vartheta^{\s}_{Q,m_0}),
\]
where the superscript indicates which scattering diagram is used to define
the theta function, and $T_{k,\pm}$ acts linearly on the exponents
in $\vartheta^{\s}_{Q,m_0}$.
\end{proposition}
\begin{remark} \label{blrem} By the proposition and Proposition \ref{fgtprop}, 
when the Injectivity Assumption holds,
broken lines make sense in $\cA^{\mch}(\bR^T)$ independent of a choice of seed.
\end{remark}

\begin{proof}
Given a broken line $\gamma$ for $\foD_{\s}$, we define $T_k(\gamma)$ to have
underlying map $T_k\circ\gamma:(-\infty,0]\rightarrow M_{\RR}$. Subdivide
domains of linearity of $\gamma$ so that we can assume any domain of
linearity $L$ satisfies $\gamma(L)\subseteq \shH_{k,+}$ or $\shH_{k,-}$.
In the two cases, the attached monomial $c_Lz^{m_L}$ becomes
$c_Lz^{T_{k,+}(m_L)}$ or $c_Lz^{T_{k,-}(m_L)}$ respectively. We show that
$T_k(\gamma)$ is a broken line for $T_k(m_0)$ with endpoint
$T_k(Q)$, with respect to the scattering diagram $\foD_{\mu_k(\s)}$, which
is equal to $T_k(\foD_{\s})$, by Theorem \ref{Adiagrammutations}. 
Indeed, the only thing
to do is to analyze what happens when $\gamma$ crosses $e_k^{\perp}$. 
So suppose in passing from a domain of linearity $L_1$ to a domain of
linearity $L_2$, $\gamma$ crosses $e_k^{\perp}$, so that 
$c_{L_2}z^{m_{L_2}}$ is a term in 
\[
c_{L_1}z^{m_{L_1}}(1+z^{v_k})^{|\langle d_ke_k,m_{L_1}\rangle|}.
\]
Assume first that $\gamma$ passes from $\shH_{k,-}$ to $\shH_{k,+}$. Then
$c_{L_2}z^{T_{k,+}(m_{L_2})}$ is a term in
\begin{align*}
c_{L_1}z^{T_{k,+}(m_{L_1})}(1+z^{v_k})^{-\langle d_ke_k,m_{L_1}\rangle}
= {} & c_{L_1}z^{m_{L_1}+\langle d_ke_k,m_{L_1}\rangle v_k}
(1+z^{v_k})^{-\langle d_ke_k,m_{L_1}\rangle}\\
= {} & c_{L_1}z^{m_{L_1}}(1+z^{-v_k})^{-\langle d_ke_k,m_{L_1}\rangle},
\end{align*}
showing that $T_k(\gamma)$ satisfies the correct rules for bending as it
crosses the slab $\fod_k'=(e_k^{\perp},1+z^{-v_k})$ of $T_k(\foD_{\s})$.

If instead $\gamma$ crosses from $\shH_{k,+}$ to $\shH_{k,-}$, then 
$c_{L_2}z^{T_{k,-}(m_{L_2})}=c_{L_2}z^{m_{L_2}}$ is a term in
\begin{align*}
c_{L_1}z^{T_{k,-}(m_{L_1})}(1+z^{v_k})^{\langle d_ke_k,m_{L_1}\rangle}
= {} & c_{L_1}z^{m_{L_1}+\langle d_ke_k,m_{L_1}\rangle v_k}
(1+z^{-v_k})^{\langle d_ke_k,m_{L_1}\rangle}\\
= {} & c_{L_1}z^{T_{k,+}(m_{L_1})}(1+z^{-v_k})^{\langle d_ke_k,m_{L_1}\rangle},
\end{align*}
so again $T_k(\gamma)$ satisfies the bending rule at the slab $\fod_k'$.

The map $T_k$ on broken lines is then shown to be a bijection by observing
$T_k^{-1}$, similarly defined, is the inverse to $T_k$ on the set of
broken lines.
\end{proof}

The following, which shows that cluster variables are theta functions,
is the key observation for proving positivity of the
Laurent phenomenon. 

\begin{proposition} \label{oneblprop}
Let $Q\in \Int(\cC_{\s}^+)$ be a base-point and let $m\in \cC_{\s}^+\cap
M^{\circ}$. Then $\vartheta_{Q,m}=z^m$.
\end{proposition}

\begin{proof}
This says the only broken line with asymptotic direction $m$ and basepoint
$Q$ is $Q+\RR_{\ge 0}m$, with attached monomial $z^m$. To see this, suppose
we are given a broken line $\gamma:(-\infty,0]\rightarrow M_{\RR}$
with asymptotic direction $m$
which bends successively at walls $\fod_1,\ldots,\fod_q$.
For each $i$, there is an $n_i\in N^+$ such that 
$\fod_i\subseteq n_i^{\perp}$.
Multiplying $n_i$ by a positive integer
if necessary, we can assume that the monomial attached to $\gamma$ upon
crossing the wall $\fod_i$ changes by a factor $c_iz^{p^*(n_i)}$.
Now if $L_i\subseteq M_{\RR}$ is the image of the
$i^{th}$ linear segment of $\gamma$, we show inductively that
\[
L_{i+1}\subseteq H_i=\left\{m \,\bigg |\, \left\langle \sum_{j=1}^i n_j,m
\right\rangle \le 0 \right\}.
\]
Indeed, $L_1=q+\RR_{\ge 0}m$ for some $q$, so initially $L_1$ is
contained on the positive side of $n_1^{\perp}$, i.e., $n_1$ is positive
on $L_1$, and hence after bending at $n_1^{\perp}$, we see $L_2\subseteq
H_1$. Next, assume true for $i=k-1$. Then $L_k\subseteq H_{k-1}$, and
if $t_k$ is the time when $\gamma$ bends at the wall $\fod_k$, we have
$\langle n_k,\gamma(t_k)\rangle=0$ and $\langle \sum_{j=1}^{k-1} n_j,
\gamma(t_k)\rangle\le 0$ by the induction hypothesis. Thus
$\langle \sum_{j=1}^k n_j,\gamma(t_k)\rangle\le 0$. In addition,
the derivative $\gamma'$ of $\gamma$ along $L_{k+1}$ is
$-m-\sum_{j=1}^k p^*(n_j)$, and
\begin{align*}
\left\langle \sum_{j=1}^k n_j, -m-\sum_{j=1}^k p^*(n_j)
\right\rangle = {} & -\left\langle \sum_{j=1}^k n_j,m\right\rangle
-\left\{\sum_{j=1}^k n_j,\sum_{j=1}^k n_j\right\}
\\
= {} & -\left\langle \sum_{j=1}^k n_j, m\right\rangle \le 0
\end{align*}
by skew-symmetry of $\{\cdot,\cdot\}$ and $m\in\cC^+_{\s}$.
Thus $L_{k+1}\subseteq\gamma(t_k)-\RR_{\ge 0}(m+\sum_{j=1}^k p^*(n_j))
\subseteq H_k$. 

Since $\Int(\cC^+_{\s})\cap H_i=\emptyset$ for all $i$, any broken line with
asymptotic direction $m$ which bends cannot terminate at the basepoint 
$Q\in\Int(\cC^+_{\s})$. This shows that there is only one 
broken line for $m$ terminating at $Q\in \Int(\shC_{\s}^+)$.
\end{proof}

\begin{corollary}
\label{monocor}
Let $\sigma\in\Delta^+_{\s}$ be a cluster chamber, and let $Q\in\Int(\sigma)$,
$m\in\sigma\cap M^{\circ}$. Then $\vartheta_{Q,m}=z^m$.
\end{corollary}

\begin{proof}
Note $\sigma=\shC^+_{v\in\s}$ for some vertex $v$ of $\foT_{\s}$, with
associated seed $\s_v$. There is then a piecewise linear map $T_v:M_{\RR}
\rightarrow M_{\RR}$ with $T_v(\foD_{\s})=\foD_{\s_v}$, see
\eqref{Tvequation}.
Then the result follows by applying Proposition \ref{oneblprop} to
$T_v(m)$, $T_v(Q)$ and Proposition \ref{brlineinvmut}.
\end{proof}

In the next section, we will identify theta functions which are polynomials
with universal Laurent polynomials, i.e., elements of the cluster
algebra associated to the fixed and seed data. It will follow from
the above Corollary that cluster monomials are in fact theta functions.

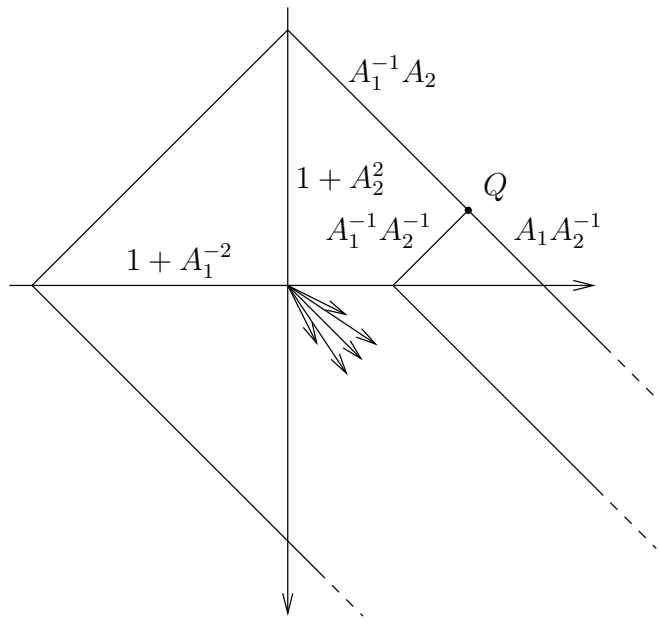
\begin{figure}
\input{brokenline11.pstex_t}
\caption{Broken lines defining $\vartheta_{Q,(1,-1)}$.}
\label{brokenline11}
\end{figure}

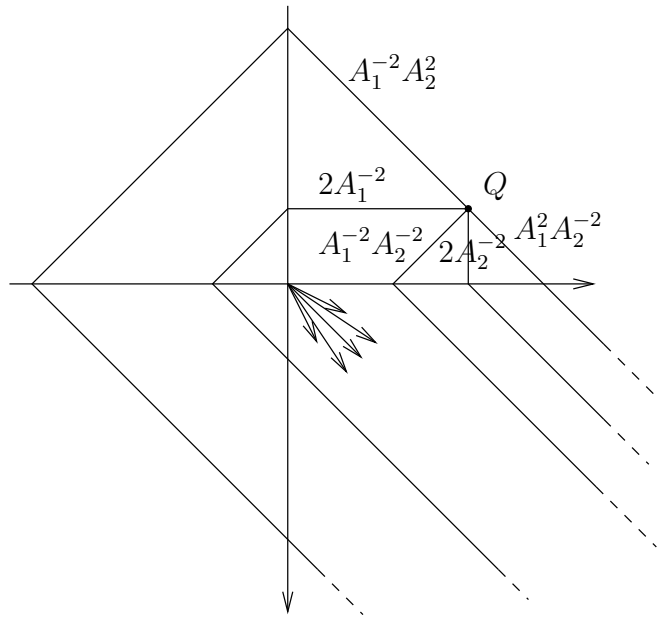
\begin{figure}
\input{brokenline22.pstex_t}
\caption{Broken lines defining $\vartheta_{Q,(2,-2)}$.}
\label{brokenline22}
\end{figure}

\begin{example}
\label{brokenlineexamples}
Figures \ref{brokenline11} and \ref{brokenline22} show some examples
of broken lines in the case of Example \ref{bcexample} with $b=c=2$.
In the first figure, we take $m=(1,-1)$, and in the second, 
$m=(2,-2)$. Neither of these lie in the cluster complex: the union
of all cones in the cluster complex is $M_{\RR}\setminus \RR_{\ge 0}(1,-1)$.
In this case the only bends occur on the original lines of $\foD_{\inc}$,
as any bending along the additional rays of the scattering diagram will
result in the broken line shooting back out, unable to reach the first
quadrant containing the basepoint $Q$. In the figures, the final
line segment is labelled with its attached monomial, so that the
theta function is a sum of these labels. One finds
\begin{align*}
\vartheta_{Q,(1,-1)}= {} &
A_1A_2^{-1}+A_1^{-1}A_2^{-1}+A_1^{-1}A_2,\\
\vartheta_{Q,(2,-2)}= {} & 
A_1^2A_2^{-2}+2A_2^{-2}+A_1^{-2}A_2^{-2}+2A_1^{-2}+A_1^{-2}A_2^2.
\end{align*}

In \cite{snowbird}, it was shown that for any $b,c$, with $Q$ lying
in the first quadrant, the $\vartheta_{Q,m}$ with $m$ ranging over all
elements of $m$ coincides with the greedy basis \cite{LLZ13}.

\end{example}

\section{Building $\cA$ from the scattering diagram and positivity of the
Laurent phenomenon} 
\label{basection}

Throughout this section we work with initial data $\Gamma$ satisfying
the Injectivity Assumption, so we obtain the cluster chamber structure
$\Delta^+_{\s}$ from $\foD_{\s}$ described
in Construction \ref{chamberstructureremark}. 
In particular, this condition holds for
initial data $\Gamma_{\prin}$, see Appendix \ref{rdsec}. 

In what follows, we will often want to deal with multiple copies of
$N, M$ etc.\ indexed either by vertices $v$ of $\foT_{\s}$ or chambers
$\sigma\in\Delta_{\s}^+$.
To distinguish these (identical) copies, we will use subscripts
$v$ or $\sigma$,  
 e.g., 
the scattering diagram $\foD_{\s_v}$ lives in $M^\circ_{\bR,v}$, and
chambers in $\foD_{\s_v}$ give, under the identification 
$M^\circ_{\bR,\s_v} = \cA^{\mch}(\bR^T)$,
the Fock-Goncharov cluster complex $\Delta^+$, 
by Lemma \ref{fgcclem}. In particular the
cluster 
chambers of $\foD_{\s_v}$ and $\foD_{\s_{v'}}$ are in canonical bijection. 

\begin{construction} \label{capcon} Fix a seed $\s$. 
We use the cluster chambers to build a positive space. 
We attach a copy of the torus
$T_{N^\circ,\sigma} := T_{N^\circ}$ to each cluster 
chamber $\sigma\in\Delta^+_{\s}$.

Given any two cluster chambers $\sigma', \sigma$ of $\Delta^+_{\s}$,
we can choose a path $\gamma$ from $\sigma'$ to $\sigma$.
We then get an automorphism $\theta_{\gamma,\foD}:\widehat{\kk[P]}
\rightarrow\widehat{\kk[P]}$ which is independent of choice of path.
If we choose the path to lie in the support
of the cluster complex, then by Remark \ref{noekwalls} 
(which shows in particular that the scattering functions on walls of the 
cluster complex are polynomials, as opposed to formal power series),
the wall crossings give birational maps of the torus and hence we can
view $\theta_{\gamma,\foD}$ as giving a well-defined map of fields of
fractions
\[
\theta_{\gamma,\foD}:\kk(M^{\circ})\rightarrow\kk(M^{\circ}).
\]
This induces a birational map
\[
\theta_{\sigma,\sigma'}:T_{N^\circ,\sigma}\dasharrow T_{N^{\circ},\sigma'}
\]
which is in fact positive.

We can then construct a space $\shA_{\scat,\s}$ by gluing together all
the tori $T_{N^{\circ},\sigma}$, $\sigma\in\Delta^+_{\s}$  
via these birational maps, see Proposition 2.4 of \cite{P1}.
We call this space (with its atlas of tori) $\cA_{\scat,\s}$. 

We write
$T_{N^\circ,\sigma \in \s} := T_{N^\circ,\sigma}$ 
if we need to make clear which seed $\s$ is being used. 
\end{construction}

We check first that mutation equivalent seeds give canonically 
isomorphic spaces.

We recall first something of the construction of $\cA$. Fix a seed $\s$. Then
we have positive spaces 
\[
\cA_{\s} = \bigcup_{v} T_{N^\circ,v},\quad\cA^{\vee}_{\s} 
=\bigcup_{v} T_{M^\circ,v},
\]
where each atlas is parameterized by vertices $v$ of the infinite tree
$\foT_{\s}$.
We write e.g., $T_{M^\circ,v\in\s} \subset \cA^{\vee}_{\s}$
for the open subset parameterized by $v$. If we obtain a seed $\s'=\s_v$
by mutation from $\s$, then we can think of the tree $\foT_{\s'}$ as
a subtree of $\foT_{\s}$ rooted at $v$, and thus we obtain natural
open immersions
\begin{equation}
\label{changeofseedisomorphisms}
\cA_{\s'} \hookrightarrow \cA_{\s},\quad \cA^{\vee}_{\s'} 
\hookrightarrow \cA^{\vee}_{\s}.
\end{equation}
These are easily seen to be isomorphisms.
Under this immersion, the open cover of $\shA_{\s'}$ is identified
canonically with the subcover of $\shA_{\s}$ indexed by vertices
of $\foT_{\s'}$ (but in either atlas there are many tori identified
with the same open set of the union). 
Because of this we view $\cA$ as independent of the choice of seed in a 
given mutation equivalence class. 

Given vertices $v,v'$ of $\foT_{\s}$, we have birational maps
\[
\mu_{v,v'}:T_{N^{\circ},v}\dasharrow T_{N^{\circ},v'},\quad
\mu_{v,v'}:T_{M^{\circ},v}\dasharrow T_{M^{\circ},v'}
\]
induced by the inclusions 
$T_{N^{\circ},v},T_{N^{\circ},v'}\subseteq
\shA_\s$ and
$T_{M^{\circ},v},T_{M^{\circ},v'}\subseteq
\shA^{\vee}_\s$ respectively.

In what follows, we use 
the same notation for the restriction of a piecewise
linear map to a maximal cone on which it is linear 
and the unique linear extension
of this restriction to the ambient vector space. 

\begin{proposition} \label{rpprop} Let $\s$ be a seed. Let $v$ be the root
of $\foT_{\s}$, $v'$ any other vertex. 
Consider the Fock-Goncharov tropicalisation $\mu_{v',v}^T: 
M^\circ_{v'} \to M^\circ_{v}$ of $\mu_{v',v}:T_{M^{\circ},{v'}}
\dasharrow T_{M^{\circ},v}$. Its 
restriction $\mu^T_{v',v}|_{\sigma'}$ 
to each cluster chamber $\sigma'\in\Delta^+_{\s_{v'}}$ is a 
linear isomorphism onto the corresponding
chamber $\sigma:=\mu^T_{v',v}(\sigma') \in \Delta^+_{\s}$.
The linear map
\[
\mu^T_{v',v}|_{\sigma'}:M^{\circ}_{\sigma'\in\s_{v'}}
\rightarrow M^{\circ}_{\sigma\in\s}
\]
induces an isomorphism
\[
T_{v',\sigma}:T_{N^\circ,\sigma \in \s }  \to T_{N^\circ,\sigma' \in 
\s_{v'}}.
\]
These glue to give an isomorphism of positive spaces 
$\shA_{\scat,\s} \to \shA_{\scat,\s_{v'}}$. 
\end{proposition}

In view of the proposition, we can view $\cA_{\scat} = \cA_{\scat,\s}$ 
as independent of the seed in a given
mutation class.

\begin{proof} It is enough to treat the case where 
$v'$ is adjacent to $v$ via an edge labelled with $k\in I_{\uf}$,
so that $\s':=\s_{v'}=\mu_k(\s)$, as in general $\mu_{v',v}$
is the inverse of a composition of mutations 
$\mu_{k_p}\circ\cdots\circ\mu_{k_1}$.
Note in this special case $\mu_{v',v}^T=T_k^{-1}$ by Proposition \ref{fgtprop},
the definition of $\shA^{\vee}$ in Appendix \ref{LDsec}, and the formula
for the $\shX$-cluster mutation $\mu_k$ (see e.g., \cite{P1}, (2.5)).
So
\[
T_{v',\sigma}: T_{N^{\circ}, \sigma \in \s} \to 
T_{N^{\circ},T_k(\sigma) \in \s'}
\]
is the isomorphism determined
by the linear map 
$T_k^{-1}|_{T_k(\sigma)}$.
The proposition amounts to showing commutativity of the diagram, for
$\sigma,\tilde\sigma\in \Delta^+_{\s}$, $\sigma'=T_k(\sigma)$, $\tilde\sigma'
=T_k(\tilde\sigma)$,
\[
\xymatrix@C=30pt
{T_{N^{\circ}, \sigma \in \s}\ar[r]^{T_{v',\sigma}}\ar@![d]_{\theta_{\sigma,
\tilde\sigma}}
& T_{N^{\circ}, \sigma'\in \s'}
\ar@![d]^{\theta_{\sigma',\tilde\sigma'}}\\
T_{N^{\circ},\tilde \sigma \in \s}\ar[r]_{T_{v',\tilde\sigma}}&
T_{N^{\circ},\tilde\sigma'\in 
\s'}}
\]
where in the left column $\theta$ indicates wall crossings in $\foD_{\s}$ while
in the right column the wall crossings are in $\foD_{\s'}$. 

If $\sigma$ and $\tilde\sigma$ are on the same side of the wall $e_k^{\perp}$, then
commutativity follows immediately from Theorem \ref{Adiagrammutations}.
So we can assume that $\sigma$ and $\tilde\sigma$ are adjacent cluster 
chambers separated
by the wall $e_k^{\perp}$, and further without loss of generality that
$e_k$ is non-negative on $\sigma$. Now by Remark \ref{noekwalls} there is
only one wall of $\foD_{\s}$ ($\foD_{\s'}$) contained in $e_k^{\perp}$, 
with support $e_k^{\perp}$ itself and attached function $1 + z^{v_k}$ 
(resp.\ $1 + z^{-v_k}$). Now it is a simple calculation: 
\begin{align*}
T_{v',\sigma}^*(\theta_{\sigma',\tilde\sigma'}^*(z^m)) = {} & 
T_{v',\sigma}^*(z^m(1+z^{-v_k})^{-\langle d_ke_k,m\rangle})\\
= {} & z^{m-v_k\langle d_ke_k,m\rangle}
(1+z^{-v_k})^{-\langle d_ke_k,m\rangle}\\
= {} & z^m(1+z^{v_k})^{-\langle d_ke_k,m\rangle}\\
= {} & \theta_{\sigma,\tilde\sigma}^*(T_{v',\tilde\sigma}^*(z^m)),
\end{align*}
This gives the desired commutativity.
\end{proof}

Next we explain how to identify $\cA_{\scat}$ with $\cA$. 

Recall for each vertex $v$ of $\foT_{\s}$ there is an associated
cluster chamber $\shC^+_v\in\Delta^+_{\s}$ in the cluster complex.
While the atlas for $\cA_{\scat,\s}$ is parameterized
by chambers of $\Delta^+_{\s}$, we can use a more redundant atlas
indexed by vertices of $\foT_{\s}$, equating $T_{N^{\circ},v}$
with $T_{N^{\circ}, \shC^+_{v}}$.
The open sets and the gluing maps in
this redundant atlas are the same as in the original, 
but in the redundant atlas a given open set might be repeated many times.

\begin{theorem} \label{rpth} Fix a seed $\s$. Let $v$ be the root of
$\foT_{\s}$, $v'$ any other vertex. Let 
$\psi^*_{v,v'}: M^\circ_{v'} \to M^\circ_{v'}$ be the linear map 
$\mu^T_{v,v'}|_{\shC^+_{v'\in\s}}$. Let 
$\psi_{v,v'}: T_{N^\circ,v'} \to T_{N^\circ,v'}$ 
be the associated map of tori.
These glue to give an isomorphism of positive spaces 
\[
\cA_{\s} := \bigcup_{v'} T_{N^\circ,v'} \to \cA_{\scat,\s} := \bigcup_{v'} 
T_{N^\circ,v'}.
\]
Furthermore, the diagram 
\begin{equation}
\label{commdiagramAAprime}
\begin{CD}
\cA_{\s} @>>> \cA_{\scat,\s} \\
@VVV        @VVV \\
\cA_{\s_{v'}} @>>> \cA_{\scat,\s_{v'}}
\end{CD}
\end{equation}
is commutative, where the right-hand vertical map is the isomorphism
of Proposition \ref{rpprop},
the left-hand vertical map the isomorphism given in 
\eqref{changeofseedisomorphisms}, 
and the horizontal maps are the isomorphisms
just described.
\end{theorem}

\begin{proof} 
Let $v',v'' \in \foT_{\s}$. The desired
isomorphism is equivalent to commutativity of the diagram 
\begin{equation}
\label{fcd}
\xymatrix@C=30pt
{
T_{N^{\circ},v'}\subset\shA_{\s}\ar[r]^{\psi_{v,v'}}\ar@![d]_{\mu_{v',v''}}
&T_{N^{\circ},v'}\subset\shA_{\scat,\s}\ar@![d]^{\theta_{\shC^+_{v'},\shC^+_{v''}}}\\
T_{N^{\circ},v''}\subset\shA_{\s}\ar[r]_{\psi_{v,v''}}&
T_{N^{\circ},v''}\subset\shA_{\scat,\s}
}
\end{equation}
%
where the right-hand vertical arrow 
is given by wall crossings in $\foD_{\s}$ between the
cluster chambers for $v',v''$. 
For this we may assume there is an oriented path from $v'$ to $v''$ in 
$\foT_{\s}$,
and thus that $v'' \in \foT_{\s_{v'}} \subset \foT_{\s}$.

The commutativity of \eqref{commdiagramAAprime} is
equivalent to the commutativity of
\begin{equation} \label{scd}
\begin{CD}
T_{N^\circ,v''} \subset \cA_{\s} @>{\psi_{v,v''}}>> T_{N^\circ,v''} \subset 
\cA_{\scat,\s} \\
@|                                        @VVV \\
T_{N^\circ,v''} \subset \cA_{\s_{v'}} @>{\psi_{v',v''}}>> T_{N^\circ,v''} \subset \cA_{\scat,\s_{v'}}
\end{CD}
\end{equation}
where the right-hand vertical map is the restriction of the isomorphism $\cA_{\scat,\s} \to \cA_{\scat,\s_{v'}}$
of Proposition \ref{rpprop}. 
We argue the commutativity of \eqref{scd} first, and then show that this
implies the commutativity of \eqref{fcd}.

Each map in \eqref{scd} is an isomorphism, induced by the restrictions
of tropicalizations of various $\mu_{w,w'}$ to various cluster 
chambers. Explicitly,
on character lattices, we have the corresponding diagram
$$
\begin{CD} 
M^\circ @<{\mu^T_{v,v''}|_{C^+_{v''} \in \s}}<<  M^\circ \\
@|                         @AA\mu^T_{v',v}|_{\shC^+_{v''\in\s_{v'}}}A \\
M^\circ    @<<{\mu^T_{v',v''}|_{C^+_{v''} \in \s_{v'}}}< M^\circ
\end{CD}
$$
which is obviously commutative as tropicalization is functorial and
$\mu_{v',v''}=\mu_{v,v''}\circ \mu_{v',v}$.

Now for the commutativity of \eqref{fcd}. 
It is enough to check the case when there is an oriented edge from $v'$ to $v''$ in $\foT_{\s}$
labelled by $k \in I_{\uf}$. We claim we may also assume $\s = \s_{v'}$. Indeed,
assume we have proven 
commutativity in this case. We draw a cube, whose
back vertical face is the diagram \eqref{fcd}, and whose front vertical face is the analogous
diagram for $\s_{v'}$, which is commutative by assumption. The top and bottom horizontal
faces are instances of \eqref{scd}, and the right-hand vertical face is the commutative diagram
of atlas tori giving the isomorphism $\cA_{\scat,\s} \to \cA_{\scat,\s_{v'}}$ of
Proposition \ref{rpprop}. Finally the left-hand vertical face consists of
equality of charts or birational maps coming from inclusions of these tori
in $\shA_{\s}$ or $\shA_{\s_v}$, thus commutative.
Now commutativity of the back vertical face, \eqref{fcd} follows. 

Finally, to show \eqref{fcd} when $\s=\s_{v'}$, i.e., $v=v'$, we note
$\psi_{v,v'}$ is automatically the identity, and
$\psi_{v,v''}$ is also the identity, by Definition
\ref{tkdef}, and the identification of $T_k$ as Fock-Goncharov tropicalisation
of the birational map of tori 
$\mu_{v',v''}=\mu_k:T_{M^\circ,v'} \dasharrow T_{M^\circ,v''}$. 
Thus the commutativity amounts to showing that the wall-crossing automorphism
$\kk(M^\circ) \to \kk(M^\circ)$ 
of fraction fields, given by 
crossing the wall $e_k^{\perp}$ from the negative to the positive side, 
is the pullback on rational functions of the birational mutation 
$\mu_{k}: T_{N^\circ} \dasharrow T_{N^\circ}$.
Note the only scattering function on the wall is $1 + z^{v_k}$, so this follows
from the coordinate free formula for the birational mutation, see e.g.,
\cite{P1}, (2.6).
\end{proof} 

We can now make precise the relationship between theta functions
and cluster monomials mentioned at the end of \S\ref{blsec}. First recall:

\begin{definition}
\label{clustermonodef}
Given fixed and initial data $\Gamma,\s$, if a seed $\s_w=(e_1',\ldots,e_n')$ is given, with
$(e_i')^*$ the dual basis and $f_i'=d_i^{-1}(e_i')^*$, 
a \emph{cluster monomial} in this seed is a monomial on $T_{N^{\circ},w}
\subset\shA$ of the form $z^m$ with $m=\sum_{i=1}^na_if_i'$ and the
$a_i$ non-negative for $i\in I_{\uf}$. By the Laurent phenomenon \cite{FZ02b}, 
such a monomial always
extends to a regular function on $\shA$. A \emph{cluster monomial on $\shA$}
is a regular function which is a cluster monomial in some seed.
\end{definition}

\begin{theorem}
\label{thetaequalsmonomial}
Let $\Gamma$ be fixed data satisfying the Injectivity 
Assumption and $\s$ an initial seed. Let $Q\in \shC^+_{\s}$ and
$m\in \sigma\cap M^{\circ}$ for some $\sigma \in \Delta^+_{\s}$. Then
$\vartheta_{Q,m}$ is a positive 
Laurent polynomial which expresses a cluster monomial
of $\cA$ in the initial seed $\s$. Further, all cluster monomials can be
expressed in this way.
\end{theorem}

\begin{proof}
By Theorem \ref{rpth} and Proposition \ref{rpprop} we have a canonical 
isomorphism of positive spaces 
$\varphi: \shA \to \cA_{\scat}=\shA_{\scat,\s}$. Let $v$ be the root of 
$\foT_\s$ and
$v'$ any vertex of $\foT_{\s}$. Then
we have $T_{N^{\circ},v'}\subset \cA$, and the cluster monomials for the
seed $\s_{v'}$ are just the monomials $z^m$ on $T_{N^{\circ},v'}$ with
$m \in \cC_{\s_{v'}}^+\cap M^{\circ}_{v'}$. 
By Theorem \ref{rpth}, this is identified
with the monomial $(\psi_{v,v'}^{-1})^*(z^m)=z^{\mu^T_{v',v}(m)}$
on $T_{N^{\circ},v'}\subset \shA_{\scat,\s}$, as
$\mu^T_{v',v}$ takes $\shC^+_{\s_{v'}}\in \Delta^+_{\s_{v'}}$ to 
$\shC^+_{v'\in\s}\in \Delta^+_{\s}$ by Proposition \ref{rpprop}. So
the cluster monomials for the chart indexed by $v'$ in $\shA_{\scat}$ are of
the form $z^m$ with $m\in \shC^+_{v'\in\s}$. Furthermore, if for
each vertex $w$ of $\foT_{\s}$, $Q_w\in \shC^+_{w\in\s}$ is a general
basepoint, we have $\vartheta_{Q_{v'},m}=z^m$ for $m\in \shC^+_{v'\in\s}$
by Corollary \ref{monocor}.
By the definition of $\cA_{\scat,\s}$ in Construction \ref{capcon},
the corresponding rational function on the open set
$T_{N^\circ,v} \subset \cA_{\scat,\s}$ 
is $\theta_{\gamma,\foD}(\vartheta_{Q_{v'},m})$,
where $\gamma$ is a path from
$Q_{v'}$ to $Q\in \shC^+_{\s}$ lying in the support of $\Delta^+_{\s}$. 
But $\vartheta_{Q,m}=
\theta_{\gamma,\foD}(\vartheta_{Q_{v'},m})$ by Theorem 
\ref{blinvth}. 
Finally $\vartheta_{Q,m}$ is a positive Laurent series by Theorem 
\ref{scatdiagpositive} and the definition of broken lines. By the Laurent
phenomenon, it is also a polynomial.
\end{proof}

We can now remove the Injectivity Assumption to prove:

\begin{theorem}[Positivity of the Laurent Phenomenon] \label{pilpth} Each cluster variable of an $\cA$-cluster
algebra is a Laurent polynomial with non-negative integer coefficients in the cluster variables
of any given seed.
\end{theorem}

\begin{proof} Since, as explained in Proposition \ref{cvextrem}, each cluster variable lifts canonically from $\cA$ to
$\cA_{\prin}$, we can replace the initial data $\Gamma$ with $\Gamma_{\prin}$, 
for which the Injectivity
Assumption holds. The result then immediately follows from 
Theorem \ref{thetaequalsmonomial}.
\end{proof} 

\begin{remark} \label{capconpc} 
When fixed and initial data $\Gamma$, $\s$ has frozen variables 
there is a partial compactification of cluster varieties $\cA \subset \ocA$, 
see Construction \ref{fvss}. 
We have an analogous partial compactification $\cA_{\scat,\s} \subset 
\ocA_{\scat,\s}$, 
given by an atlas of toric varieties 
$T_{N^\circ,{v \in \s}} \subset \TV(\Sigma_{v \in \s})$. 
The choice of fans is forced by the identifications of Proposition \ref{rpprop}:
for $v$ the root of $\foT_{\s}$,
$\Sigma_{v \in \s} := \Sigma^{\s}$ ($\Sigma^{\s}$ as in Construction \ref{fvss}) and then 
$\Sigma_{v' \in \s}:= \mu_{v,v'}^t(\Sigma_{v \in \s})$.
Now Proposition  \ref{rpprop} and 
Theorem \ref{rpth} (and their proofs) extend to the partial
compactifications without change. One checks easily that all mutations in the positive 
spaces $\cA,\cA_{\scat}$, and all the linear isomorphisms between corresponding tori in the atlases 
for $\cA,\cA_{\scat,\s},\cA_{\scat,\s_v}$ preserve the monomials 
$A_i=z^{f_i}$, $i \not \in I_{\uf}$ (these are
the frozen cluster variables), so that
all the spaces come with canonical projection to 
$\bA^{\#(I\setminus I_{\uf})}$, preserved by the isomorphisms 
between these positive spaces. We shall see in the next section that in the special
case of the partial compactification of $\cA_{\prin}$, the relevant fans are particularly
well-behaved. 
\end{remark}

\section{Sign coherence of $c$- and $g$-vectors} 
\label{cgvectorsec}
We begin with some philosophy concerning log Calabi-Yau varieties
following on from the discussion of \cite{P1}, \S 1.
Suppose $V\subset U$ are both log Calabi-Yau and $V$ is a Zariski open subset
of $U$, both having maximal boundary (\cite{P1}, Definition 1.5).
The tropical sets (which are expected to parameterize
the theta function basis of functions on the mirror) of $U$ and $V$ 
are canonically
equal and we expect the mirror $U^{\vee}$ to
degenerate to the mirror $V^{\vee}$. In particular when $V=T$ is an algebraic torus, we
expect a canonical degeneration of $U^{\vee}$ to the dual torus $T^{\vee}$, under which
the theta functions degenerate to monomials (i.e., characters). 
When $U = \cA$ is an
$\cA$-cluster variety, and $T=T_{N^{\circ},\s} \subset \cA$ is a cluster torus, it
turns out 
this degeneration has a purely cluster construction: the choice of seed $\s$
determines
a canonical partial compactification $\pi: \ocA_{\prin}^{\s} \to \bA^n_{X_1,\dots,X_n}$ of
$\pi: \cA_{\prin} \to T_M$, see Proposition \ref{ldpprop} and
Remark \ref{Aprinremark}. The main point of this 
section is to show that $\pi^{-1}(0) = T_{N^{\circ}}$, Corollary \ref{mutcor}. This
degeneration is central to what follows in this paper. 
For example, we prove linear independence of theta 
functions by showing they restrict to different characters on $T_{N^{\circ}}$, and
the Fock-Goncharov conjecture, false in general, is true in a formal neighborhood
of this fibre. There are analogous degenerations (identified with this one when 
the Fock-Goncharov conjecture holds) for e.g., $\can(\cA)$, and here they are even
more central, being the main tool we have for proving properties of this algebra
(e.g., that its spectrum gives a Gorenstein log Calabi-Yau of the right 
dimension), see 
Theorem \ref{cfcor}. The equality $\pi^{-1}(0) = T_{N^{\circ}}$, while not at all obvious 
from the cluster atlas, is immediate using the alternative description $\cA_{\scat,\s}$
of the 
previous section, as we now explain. Further, there are some immediate
benefits, such as sign coherence of $c$-vectors.  

For the remainder of the paper, the only scattering diagram we will ever 
consider 
is $\foD^{\cA_{\prin}}_{\s}$, see Construction \ref{sdprin}.
So we will often omit the superscript from the notation. 

\begin{construction} \label{atlascon}
Fix a seed $\s$ for fixed data $\Gamma$.
By Construction \ref{capcon}, 
the scattering diagram $\foD_{\s} = \foD^{\cA_{\prin}}_{\s}$ gives an atlas for 
the space $\cA_{\scat,\s}$. (Technically, we should write 
$\cA_{\prin,\scat,\s}$ to
indicate we are constructing something isomorphic to $\cA_{\prin}$, however
this will make the notation even less readable.) 
This was constructed by attaching a copy $T_{\tN^\circ,\sigma}$ 
of the torus
$T_{\tN^\circ}$ to each cluster chamber $\sigma\in\Delta^+_{\s}$,
and (compositions of) wall crossing automorphisms give the birational maps between them. By
Theorem \ref{rpth} this space is canonically identified with $\cA_{\prin}$: $\cA_{\prin}$
has an atlas of tori $T_{\tN^\circ,w}$ parameterized by vertices $w$ of
$\foT_{\s}$, and we have canonical isomorphisms
$\psi_{v,w}:T_{\tN^{\circ},w}\rightarrow T_{\tN^{\circ},\shC^+_w}$
for each vertex $w$ which induce the isomorphism $\cA_{\prin} \to \cA_{\scat,\s}$. 

In what follows, if $w$ is a vertex of $\foT_{\s}$, we write $\tilde\s_w$
for the seed obtained by mutating $\tilde\s$ (see \eqref{tildeseed})
via the sequence of mutations
dictated by the path from the root $v$ of $\foT_{\s}$ to $w$.
As described in Remark \ref{Aprinremark}, the initial seed $\s$
determines the partial compactification 
$\cA_{\prin} \subset \ocA_{\prin}^{\s}$, given by the atlas of toric varieties
\[
T_{\tN^\circ,w} \subset \TV(\Sigma^{\s}_{w})
\]
where $\Sigma^{\s}_{w}$ is the cone generated by the subset of 
basis vectors of $\tilde\s_{w}$ corresponding to the second copy of $I$.

By Remark \ref{capconpc}, the seed $\s$ 
also determines a partial compactification 
$\cA_{\scat,\s} \subset \ocA_{\scat,\s}^{\s}$ (the superscript, 
thus the seed {\it close} to the overline in the
notation, is responsible for the partial compactification), 
given by an atlas of toric varieties. Explicitly, if $w$ is a vertex
of $\foT_{\s}$, the fan $\Sigma^{\s}_{w}$ yields the partial
compactification of $T_{\tN^{\circ},w}$ in $\ocA_{\prin}^{\s}$, and
this is identified with $T_{\tN^{\circ},\shC^+_{w}\in\Delta^+_{\s}}$
via $\psi_{v,w}$ under the isomorphism $\shA_{\prin}\cong \shA_{\scat,\s}$
of Theorem \ref{rpth}. Thus the fan giving the partial compactifaction
of $T_{\tN^{\circ},\shC^+_{w}\in\Delta^+_{\s}}$ is
\[
\Sigma^{\s}_{\scat,w}:=\psi_{v,w}^t(\Sigma^{\s}_{w}).
\]
In fact, this fan is easily calculated:
\end{construction}

\begin{lemma} \label{clem} The cones $\Sigma^{\s}_{\scat,w}$, 
and thus the toric varieties
in the atlas for the partial compactification
$\cA_{\scat,\s} \subset \ocA_{\scat,\s}^{\s}$, 
are the same for all $w$.
Each is equal to the cone spanned by the vectors 
$(0,e_1^*),\dots,(0,e_n^*) \in \tN^\circ$, where 
$\s = (e_1,\dots,e_n)$ and $e_1^*,\ldots,e_n^*$ denotes the dual basis.
\end{lemma}

\begin{proof} $\Sigma^{\s}_{\scat,v}$ is the given cone, by definition of the
seed $\tilde\s$. By Construction \ref{fvss}
the other fans are given by applying the geometric
tropicalisation of the birational gluing of the tori 
in the atlas for $\cA_{\scat,\s}$. 
These birational maps are given by wall crossings in $\foD_{\s}$. 
But for each wall between cluster chambers the
wall crossing is a standard mutation ${\mu}_{(\tilde n,\tilde m)}$, 
notation as in \S\ref{tropsec}, for some $\tilde n\in \tN^{\circ},
\tilde m\in \tM^{\circ}$.
The attached scattering function is $1 + z^{p^*(n,0)}$ for some 
$n$ in the convex hull of $\{e_i\,|\,i\in I\}$, and $\tilde m=p^*(n,0)$. 
But then $\langle \tilde m,(0,e_i^*)\rangle=\{(n,0),(0,e_i^*)\} \geq 0$.
Thus the geometric tropicalisation ${\mu}^t_{(\tilde n,\tilde m)}$ 
fixes all the $e_i^*$ by \eqref{mutgeometric}, and so the fan is constant. 
\end{proof}

\begin{corollary} \label{mutcor} Fix a seed $\s$, and let $v$ be the root
of $\foT_{\s}$. The following hold:
\begin{enumerate}
\item The fibre of $\pi: \ocA^{\s}_{\prin} \to \bA^n_{X_1,\dots,X_n}$ over $0$ is $T_{N^\circ}$.
(See Proposition \ref{ldpprop} for the definition of $\pi$).
\item The mutation maps  
\[
\TV(\Sigma^{\s}_w) \dasharrow \TV(\Sigma^{\s}_{w'})
\]
for the atlas of toric varieties defining $\ocA^{\s}_{\prin}$ 
are isomorphisms in a neighborhood of the fibre over $0 \in \bA^n_{X_1,\dots,X_n}$. 
\item For the
partial compactification $\shA_{\scat,\s} \subset \oshA_{\scat,\s}^{\s}$ 
with atlas corresponding to
cluster chambers of $\foD_{\s}$, the corresponding mutation map between
two charts (which by Lemma \ref{clem}
has the same domain and range) is an isomorphism in a neighborhood of the 
fibre $0 \in \bA^n_{X_1,\dots,X_n}$ and restricts to the identity on this fibre.
\end{enumerate}
\end{corollary}

\begin{proof} It is clear that (3) implies (2) implies (1). 

For (3), the scattering diagram 
$\foD_{\s}$ is trivial  modulo the $X_i$ (which pulls back to $z^{(0,e_i)}$),
because this holds for the initial walls, with attached functions
$1+z^{(v_i,e_i)}$. Now for any adjacent vertices $w,w'\in \foT_{\s}$,
the birational gluing map $\TV(\Sigma^{\s}_{\scat,w})\dasharrow \TV(\Sigma^{\s}_{\scat,w'})$
is given on the level of monomials by $z^{\tilde m}\mapsto z^{\tilde m} 
f^{\langle\tilde n,\tilde m\rangle}$ for a regular function $f$ on $\TV(\Sigma^{\s}_{\scat,w'})$
and some $\tilde n\in \tN^{\circ}$ and any $\tilde m\in \tM^{\circ}$, and by the above
$f$ is identically $1$ when restricted to the torus where the $X_i$ are zero. 
On the other hand, this birational map gives an isomorphism between the open subsets
of $\TV(\Sigma^{\s}_{\scat,w})$ and $\TV(\Sigma^{\s}_{\scat,w'})$ where $f$ is non-zero.
In particular, the gluing maps are isomorphisms in the neighbourhood of the fibre where
all $X_i$ vanish and are the identity on that fibre.
\end{proof}

The proof of the corollary shows the utility of constructing $\cA_{\prin} \subset 
\ocA_{\prin}^{\s}$ as the 
positive space $\cA_{\scat,\s} \subset \ocA_{\scat,\s}^{\s}$ associated to the cluster chambers 
in the scattering diagram $\foD_{\s}$. Next we show sign coherence of $c$-vectors
follows easily from the corollary.

In what follows, given a seed $\tilde\s_w=(\tilde e_1,\ldots,\tilde e_{2n})$ 
obtained via mutation from
$\tilde\s$, we write $\tilde\epsilon^w$ for the $n\times 2n$ exchange matrix 
for this seed, with
\begin{equation}
\label{cvectordef}
\tilde\epsilon^w_{ij}=\begin{cases}
\{\tilde e_i,\tilde e_j\}d_j & 1\le j\le n\\
\{\tilde e_i,\tilde e_j\}d_{j-n} & n+1\le j\le 2n
\end{cases}
\end{equation}
The $c$-vectors of this seed are the rows of the right-hand $n\times n$
submatrix.

\begin{corollary}[Sign coherence of $c$-vectors] \label{sccveccor} 
For any vertex $w$ of $\foT_{\s}$ and 
fixed $k$ satisfying $1 \leq k \leq n$, either
the entries $\tilde\epsilon^{w}_{k,j}$, $n+1 \leq j \leq 2n$ are all 
non-positive, or these entries are all non-negative.
\end{corollary}

\begin{proof} The result follows directly from Corollary \ref{mutcor} by writing down the
mutation in cluster coordinates. 
Following the notation given in Appendix \ref{rdsec},
we have the fixed seed $\s =(e_1,\dots,e_n)$ which determines 
$\cA_{\prin} \subset \ocA_{\prin}^{\s}$
and the family $\pi: \ocA_{\prin}^{\s} \to \bA^n_{X_1,\dots,X_n}$. 
The corresponding initial seed
for $\ocA_{\prin}^{\s}$ is 
\[
\tilde\s=((e_1,0),\dots,(e_n,0),(0,f_1),\dots,(0,f_n)),
\]
and the coordinate $X_i$ on $\AA^n$ pulls back to $z^{(0,e_i)}$ on 
$\ocA_{\prin}^{\s}$. 
These are the frozen cluster variables for $\ocA_{\prin}^{\s}$. 
Note $X_i = z^{g_{n+i}}$ where $g_i$ is the dual basis to the basis
$(d_1 e_1,0),\ldots,(d_ne_n,0),(0,e_1^*),\dots,(0,e_n^*)$ of 
$\widetilde N^{\circ}$.

A vertex $w'$ corresponds to a seed
$\s_{w'}=(e_1',\dots,e'_n)$ for $N$ with corresponding seed 
$\tilde\s_{w'} =((e_1',0),\ldots,(e_n',0),h_1',\dots,h_n')$
for $\tN$, with $\tilde\s_{w'}$ obtained from $\tilde\s$
by a sequence of mutations.  The $h_i$ are no longer necessarily given 
by the $f_i'$. Write $\tilde f_i'$, $1\le i\le 2n$ for the corresponding basis
of $\tM^{\circ}$.
The cluster variables on the corresponding torus $T_{\tN^\circ,w'}$ 
are $A_i' := z^{\tilde f_i'}$.
Say $w''$ is a vertex of $\foT_{\s}$ 
adjacent to $w'$ along an edge labelled by $k$.
Then
\[
\tilde\s_{w''} = ((e_1'',0),\dots,(e_n'',0),h_1'',\dots,h_n''),
\]
and the cluster coordinates are 
$A_i'' = z^{\tilde f_i''}$. Since the last $n$ cluster variables are
frozen, $A_{n+i}' = A_{n+i}'' = X_i$, $ 1 \leq i \leq n$.

The fan $\Sigma^{\s}_{w'}$ determining a toric variety  
in the atlas for $\ocA_{\prin}^{\s}$ consists of a single
cone spanned by $h'_1,\dots,h'_n$, and
\[
\TV(\Sigma^{\s}_{w'}) = (\bG_m^n)_{A_1',\dots,A_n'} \times \bA^n_{X_1,\dots,X_n}.
\]
Similarly
\[
\TV(\Sigma^{\s}_{w''}) = (\bG_m^n)_{A_1'',\dots,A_n''} \times 
\bA^n_{X_1,\dots,X_n}.
\]
The mutation $\mu_k: \TV(\Sigma^{\s}_{w'}) \dasharrow \TV(\Sigma^{\s}_{w''})$ 
is given by the exchange relation \cite{FZ07} (2.15) (see \cite{P1}, (2.8)
in our notation) which is, with $\tilde\epsilon=\tilde\epsilon^{w'}$,
\begin{align*}
\mu_{k}^*(A_i'') = {} & A_i' \text{ for } i \neq k \\
\mu_{k}^*(A_k'') = {} & (A_k')^{-1}( \prod_{i =1}^{2n} 
(A_i')^{[\tilde\epsilon_{ki}]_+} + \prod_{i=1}^{2n} (A_i')^{-[\tilde
\epsilon_{ki}]_-}) \\
= {} & (A_k')^{-1}( p_k^+ \prod_{i=1}^{n} (A_i')^{[\tilde\epsilon_{ki}]_+} +
p_k^{-}  \prod_{i=1}^{n} (A_i')^{-[\tilde\epsilon_{ki}]_-}) \\
\mu_{k}^*(X_i)= {} & X_i
\end{align*}
where 
\[
p_k^+  := \prod_{\substack{1 \leq i \leq n\\ \tilde\epsilon_{k,n+i} \geq 0}} X_i^{\tilde\epsilon_{k,n+i}}, \quad
p_k^- := \prod_{\substack{ 1 \leq i \leq n \\ -\tilde\epsilon_{k,n+i} \geq 0}} X_i^{-\tilde\epsilon_{k,n+i}}.
\]
Now
$\mu_k$ fails to be an isomorphism exactly along the vanishing locus of
\[
p_k^+ \prod_{i=1}^{n} (A_i')^{[\tilde\epsilon_{ki}]_+} +
p_k^{-}  \prod_{i=1}^{n} (A_i')^{-[\tilde\epsilon_{ki}]_-}.
\]
This locus is disjoint from the central fibre $0 \in \bA^n_{X_1,\dots,X_n}$ by Corollary \ref{mutcor}. 
On the other hand
it is disjoint from the central fibre if and only if exactly one of 
$p_k^+,p_{k}^-$ is the empty product, i.e., the constant monomial $1$. 
Sign coherence is the 
statement that at least one of $p_k^+$, $p_k^-$ is the empty product. 
\end{proof} 

Recall from Definition \ref{clustermonodef} the notion of cluster monomial, and also
recall from Proposition \ref{ldpprop}, (2) the 
$T_{N^{\circ}}$-action on $\shA_{\prin}$.

\begin{definition} 
\label{gvecdef}
By Proposition \ref{cvextrem}, the choice of seed $\s$ provides 
a canonical extension of each cluster monomial on $\shA$ to a cluster
monomial on $\shA_{\prin}$.
Each cluster monomial on $\cA_{\prin}$
is a $T_{N^\circ}$-eigenfunction under the above $T_{N^{\circ}}$
action. The
\emph{$g$-vector} with respect to a seed $\s$
(see \cite{FZ07}, (6.4)) associated to a cluster
monomial of $\cA$ is the $T_{N^\circ}$-weight of its lift determined by $\s$.
\end{definition}
  
We now give an alternative description of $g$-vectors, which will lead
to a more intrinsic definition of $g$-vector (Definition \ref{intrinsicgA}).
This in turn generalises to all the different flavors of cluster
varieties (Definition \ref{fgccthg0}). 

\begin{proposition} \label{gprop} Fix a seed $\s$, giving the partial compactification
$\cA_{\prin} \subset \ocA_{\prin}^{\s}$ and $T_{N^\circ}$-equivariant
$\pi: \ocA_{\prin}^{\s} \to \bA^n_{X_1,\dots,X_n}$. The central fibre 
$\pi^{-1}(0)$ is a $T_{N^\circ}$-torsor. Let $\oA$
be a cluster monomial on $\cA = \pi^{-1}(1,1,\dots,1)$ and 
$A$ the corresponding lifted cluster monomial on $\ocA_{\prin}^{\s}$.
This restricts to a regular non-vanishing $T_{N^\circ}$-eigenfunction 
along $\pi^{-1}(0)$, and
so canonically determines an element of $M^\circ$ (its weight). 
This is the $g$-vector associated to $\oA$.
\end{proposition} 

\begin{proof} Let $w\in\foT_{\s}$ determine the seed in which $\oA$ is
defined as a monomial.
By Corollary \ref{mutcor}, all mutations are isomorphisms near the central 
fibre of $\pi$, so it's enough to check that $\oA$
is regular on the toric variety 
$\TV(\Sigma^{\s}_{w})$, and restricts to a character on its central fibre. 
But this is true by construction: if the seed $\tilde s_w$ is 
$(\tilde e_1,\ldots,\tilde e_{2n})$, then the cluster
variables for the seed $\tilde\s_w$ on the torus 
$T_{\tN^\circ,w}$ are $z^{\tilde f_k}$ 
and $\Sigma^{\s}_{w}$ is the fan with rays spanned by the $\tilde e_{n+1},
\ldots,\tilde e_{2n}$. Thus the lift $A$ of $\oA$ is regular on
$TV(\Sigma^{\s}_w)$, and hence is regular in a neighbourhood of 
$\pi^{-1}(0)\subset \shA_{\prin,\s}$. Furthermore, it is non-zero on
$\pi^{-1}(0)$ since the canonical lift only involves monomials $z^{\tilde
f_1},\ldots,z^{\tilde f_n}$, which are non-vanishing on the strata
of $TV(\Sigma^{\s}_w)$.
The final statement follows since the restriction of the 
variable to the central fibre will have the same $T_{N^\circ}$-weight, 
as the map $\pi$ is $T_{N^\circ}$-equivariant, and $T_{N^\circ}$ fixes 
$0 \in \bA^n$. 
\end{proof} 

\begin{definition}
\label{intrinsicgA}
Writing $\shA=\bigcup_{\s} T_{N^{\circ},\s}$, let
$\oA$ be a cluster monomial of the form $z^m$
on a chart $T_{N^{\circ},\s'}$,
$\s'=(e_1',\ldots,e_n')$. Note that $(z^{e_i'})^T(m)\le 0$
for all $i$, so after identifying $\shA^{\vee}(\RR^T)$ with 
$M^{\circ}_{\RR,\s'}$,
$m$ yields a point in the Fock-Goncharov cluster chamber
$\shC_{\s'}^+\subseteq\shA^{\vee}(\RR^T)$, as defined in
Lemma \ref{fgcclem}. We define $\g(\oA)$ to be this point of
$\shC^+_{\s'}\subseteq \shA^{\vee}(\RR^T)$. 
\end{definition}

\begin{corollary} 
\label{gvecthm}
Let $\oA$ be a cluster monomial on $\cA$, and fix a seed $\s$ giving an
identification $\shA^{\vee}(\RR^T)=M^{\circ}_{\RR,\s}$. Then under this
identification, $\g(\oA)$ is the $g$-vector of the cluster monomial $\oA$ 
with respect to $\s$.\end{corollary}

\begin{proof} 
We first note that if $\oA$ is a monomial $z^m$ on the chart $T_{N^{\circ},\s'}$
with $\s'=\s_w$, $\s=\s_v$, then the image of $\g(\oA)$ under the identification
$\shA^{\vee}(\RR^T)=M^{\circ}_{\RR,v}$ is $\mu^T_{w,v}(m)$, where as usual
$\mu_{w,v}:T_{M^{\circ},w}
\dasharrow T_{M^{\circ},v}$ is the rational map induced by
the inclusions $T_{M^{\circ},w},T_{M^{\circ},v}\subset\shA^{\vee}$.

The choice of the seed $\s$ gives the lift of $\oA$ to a cluster monomial
$A$ on $\shA_{\prin}$. Using the identification of $\shA_{\prin}$ with
$\shA_{\scat,\s}$, $A$ is identified with a monomial of the form $z^{(m',n')}$
on the chart
$T_{\tN^{\circ},w}$ (or $T_{\tN^{\circ},\shC^+_{w\in\s}}$, depending on
how one chooses to parametrize charts of $\shA_{\scat,\s}$). 
Let $v$ be the root of $\foT_{\s}$.
By Lemma \ref{clem}, the
corresponding chart of $\ocA_{\scat,\s}^{\s}$ is the toric variety defined
by the fan $\Sigma^{\s}_{\scat,w}$. 
By Proposition \ref{gprop}, $A$ is a regular function on
$\TV(\Sigma^{\s}_{\scat,w})$ which is non-vanishing along $\pi^{-1}(0)$. 
The $T_{N^{\circ}}$ weight is the $g$-vector. Since $\Sigma^{\s}_{\scat,w}$
is the cone spanned by $(0,e_1^*),\dots,(0,e_n^*)$ in $\tN^{\circ}_{\RR}$,
where $\s=(e_1,\ldots,e_n)$, one sees that $(m',n')=(g,0)$.

Thus to show the corollary, it is enough to show that 
$m=\mu^T_{v,w}(g)\in M^{\circ}_w$. Note however a similar statement
is already true at the level of $\shA_{\prin}$. Indeed,
in the chart $T_{\tN^{\circ},w}$ of $\shA_{\prin}$, the
monomial $A$ takes the form $z^{(m,n'')}$ for some $n''\in N$, and 
$(m,n'')$ lies in the positive chamber of $\foD^{\shA_{\prin}}_{\s_w}$.
But $\shC^+_{w\in\s}$
is the image of this positive chamber under the map
$\mu^T_{w,v}$, where now $\mu_{w,v}:T_{\tM^{\circ},w}\dasharrow
T_{\tM^{\circ},v}$ is the map induced by the inclusions
$T_{\tM^{\circ},v},T_{\tM^{\circ},w}\subset \shA^{\vee}_{\prin}$. 
Now $(g,0)=(\psi_{v,w}^*)^{-1}(m,n'')$ by Theorem
\ref{rpth}, and $(\psi_{v,w}^*)^{-1}=(\mu^T_{v,w}|_{\shC^+_{w\in\s}})^{-1}$, so
we see that $(m,n'')=\mu^T_{v,w}(g,0)$.

Now because there is a well-defined map $\rho:\shA^{\vee}_{\prin}
\rightarrow\shA^{\vee}$ by Proposition \ref{ldpprop}, (4), with
$\rho^T$ given by projection onto $M^{\circ}$, this projection
$\rho^T$ is compatible with the tropicalizations
$\mu^T_{v,w}:\tM^{\circ}\rightarrow\tM^{\circ}$ and $\mu^T_{v,w}:M^{\circ}
\rightarrow M^{\circ}$, i.e., $\mu^T_{v,w}\circ\rho^T=\rho^T\circ\mu^T_{v,w}$.
Thus $\mu_{v,w}^T(m)=g$, as desired.
\end{proof}

This corollary shows us how to generalize the notion of $g$-vector
to any cluster variety:

\begin{definition}
\label{fgccthg0}
Let $V=\bigcup_{\s} T_{L,\s}$ be a cluster variety, suppose that
$f$ is a global monomial (see Definition \ref{globalmonomialdef}) 
on $V$, and let $\s$ be a seed such that
$f|_{T_{L,\s}\subset V}$ is the character $z^m$, $m\in\Hom(L,\ZZ)=L^*$.
Define the \emph{$g$-vector} of $f$ to
be the image of $m$ under the identifications of \S\ref{tropsec}:
\[
V^{\mch}(\bZ^T) = T_{L^*,\s}(\bZ^T) = L^*.
\]
We write the $g$-vector of $f$ as ${\bf g}(f)$.
\end{definition}

Note that the definition as given is not clearly independent of the choice
of seed $\s$, but for a cluster variety of $\cA$ type, the previous corollary 
shows this. This independence will be shown in general in Lemma
\ref{fgccthg2}.

By \cite{NZ}, the sign coherence for $c$-vectors (proved in
Corollary \ref{sccveccor} here),
implies a sign coherence for $g$-vectors. Here we give a much shorter proof
using the above description of $g$-vectors.

\begin{theorem}[Sign coherence of $g$-vectors]
\label{scgvec} 
Fix initial seed $\s=(e_1,
\ldots,e_n)$, with $f_i=d_i^{-1}e_i^*$ as usual.
Given any mutation equivalent seed $\s'$, 
the
$i^{\text{th}}$-coordinates of the $g$-vectors for the cluster variables of this seed, expressed
in the basis $(f_1,\dots,f_n)$, are either all
non-negative, or all non-positive. 
\end{theorem}

\begin{proof} 
By Corollary \ref{gvecthm}, the $g$-vectors in question are the generators of a 
chamber
in the cluster complex of $\s$, defined as the images of
the cluster chambers of $\foD^{\cA_{\prin}}_{\s}$ under the projection
$\rho^T$, by Theorem \ref{fgccth}.
The hyperplanes $e_i^{\perp}$ are thus walls in the cluster complex.
In particular,
$e_i$ is either non-negative everywhere on a chamber, or 
non-positive everywhere on a chamber. The theorem follows.
\end{proof} 

For future reference, we record the relationship between $c$-vectors
and the cluster chambers
in the case of no frozen variables.
Fix a seed $\s$. By Lemma \ref{fgcclem}, each mutation equivalent
seed $\s'=\s_w$ has an associated
cluster chamber $\shC^+_{w\in\s} \subset M^{\circ}_{\bR,\s}$.
This is a full dimensional
strictly simplicial cone, generated by a basis of $M^{\circ}$ consisting of
$g$-vectors of the cluster variables $A_1',\ldots,A_n'$ of
$\s_w$. The facets of $\shC^+_{w\in\s}$ are thus in natural bijection with the
elements of $\s$ (or the indices in $I=I_{\uf}$).

\begin{lemma} \label{facelem} The facet of 
$\shC^+_{w\in\s}$ corresponding to $i \in I$ is
the intersection of $\shC^+_{w\in\s}$ with the orthogonal complement of
the $c$-vector for the corresponding element of $\s^{\vee}_w$ 
(the corresponding mutation of the
Langlands dual seed $\s^{\vee}$, see Appendix \ref{LDsec}). Furthermore,
each $c$-vector for $\s^{\vee}_w$ is non-negative on $\shC^+_{w\in\s}$.
\end{lemma} 

\begin{proof} 
This is the content of \cite{NZ}, Theorem 1.2, the condition 
(1.8) of \cite{NZ} holding by our Corollary \ref{sccveccor}. 
The $g$-vectors used in \cite{NZ} are precisely the $g$-vectors of the
cluster variables $A_i'$.
\end{proof}

\section{The formal Fock-Goncharov Conjecture}
\label{formalFGcsec}
In this section we associate in a canonical way 
to every universal Laurent polynomial $g$ on $\cA_{\prin}$ 
a formal sum $\sum_{q \in \cXt} \alpha(g)(q) \vartheta_q$, 
$\alpha(g)(q) \in \kk$, which,
roughly speaking, converges to $g$ at infinity in each partial compactification 
$\cA_{\prin} \subset \ocA_{\prin}^{\s}$. To give such an expression for a single
$\s$ is quite easy, see Propositions \ref{thetabasislemma} and \ref{hatinfun}.
The crucial point is that remarkably these coefficients are independent
of $\s$, see Theorem \ref{ffgth}, our alternative to the Fock-Goncharov
conjecture (which fails in general). This establishes the connection between
$\up(\cA_{\prin})$ and $\can(\cA_{\prin})$, and is key to one of our main technical
results, see the proof of Proposition \ref{markscor}. 

Choose a seed $\s=(e_1,\dots,e_n)$. We let 
$X_i := z^{e_i}$, $I_{\s} = (X_1,\dots,X_n) \subset \kk[X_1,\ldots,X_n]$,
set
\[
\AA^n_{(X_1,\ldots,X_n),k}=\Spec \kk[X_1,\ldots,X_n]/I_{\s}^{k+1},
\quad \ocA^{\s}_{\prin,k}=\ocA_{\prin}^{\s}\times_{\AA^n_{X_1,\ldots,X_n}}
\AA^n_{(X_1,\ldots,X_n),k},
\]
and write the map induced by $\pi:\ocA_{\prin}^{\s}\rightarrow\AA^n_{X_1,\ldots,
X_n}$ also as
\[
\pi: \ocA_{\prin,k}^{\s} \to \bA^n_{(X_1,\dots,X_n),k}.
\]
We use the notation $\up(Y):=H^0(Y,\shO_Y)$ for a variety $Y$, so that 
e.g., $\up(\shA)$ is the upper cluster algebra. We define
\[
\widehat{\up(\ocA_{\prin}^\s)}=\liminv \up(\ocA_{\prin,k}^{\s}).
\]
Note that for
any $g \in \up(\cA_{\prin})$, $z^n g \in \up(\ocA_{\prin}^{\s})$ for some monomial
$z^n$ in the $X_i$. This induces a canonical inclusion
\begin{equation} \label{upginhat}
\up(\cA_{\prin}) \subset \widehat{\up(\ocA_{\prin}^\s)} \otimes_{\kk[N^+_{\s}]} 
\kk[N]
\end{equation}
where $N^+_{\s}\subset N$ is the monoid generated by 
$e_1,\ldots,e_n$. Let $\pi_N:\tM^{\circ}\rightarrow N$ be the projection,
and set
\[
\tM^{\circ,+}_{\s}=\pi_N^{-1}(N^+_{\s}).
\]

Recall a choice of seed $\s = (e_1,\dots,e_n)$
determines a scattering diagram 
$\foD_{\s} = \foD^{\cA_{\prin}}_{\s} \subset \tM^\circ_{\bR,\s}$ 
with initial walls $(e_i^{\perp},1 + z^{(v_i,e_i)})$ for $i\in I_{\uf}$.
We let $P_{\s} \subset \tM^\circ_{\s}$ be the monoid
generated by $(v_1,e_1),\ldots,(v_n,e_n)$. We have the cluster complex
$\Delta^+_{\s}$ of cones in $\tM^{\circ}_{\RR,\s}$, with cones
$\shC^+_{v}\in\Delta^+_{\s}$ for each vertex $v$ of $\foT_{\s}$.

Similarly to the above discussion, 
$\pi:\ocA_{\scat,\s}^{\s}\rightarrow \AA^n_{X_1,\ldots,X_n}$
induces maps
\[
\oshA_{\scat,\s,k}^{\s}:=\oshA_{\scat,\s}^{\s}\times_{\AA^n} 
\AA^n_{(X_1,\ldots,X_n),k}
\rightarrow \AA^n_{(X_1,\ldots,X_n),k}.
\]
The isomorphism between $\oshA_{\prin}^{\s}$ and $\oshA_{\scat,\s}^{\s}$
discussed in Construction \ref{atlascon} restricts to give an 
isomorphism between $\oshA_{\prin,k}^{\s}$ and $\oshA_{\scat,\s,k}^{\s}$. 
Furthermore, as 
$\oshA_{\scat,\s}^{\s}$ is described by gluing charts isomorphic to 
$TV(\Sigma)$ with
$\Sigma$ the cone generated by $(0,e_1^*),\ldots,(0,e_n^*)$ for every chart
by Lemma \ref{clem}, in fact $\oshA_{\scat,\s,k}^{\s}$ is 
described by gluing charts
parameterized by $\sigma\in\Delta^+_{\s}$ isomorphic to 
\[
V_{\s,\sigma,k}:=TV(\Sigma)\times_{\AA^n}\AA^n_{(X_1,\ldots,X_n),k}.
\]
Note for $\sigma,\sigma'\in \Delta^+_{\s}$, the birational map
$\theta_{\sigma,\sigma'}:TV(\Sigma)\dasharrow TV(\Sigma)$ between
the charts of $\oshA_{\scat,\s}^{\s}$ indexed by $\sigma$ and $\sigma'$
restrict to isomorphisms $V_{\s,\sigma,k}\rightarrow V_{\s,\sigma',k}$: 
this is implied by Corollary \ref{mutcor}, (3).

We choose a generic basepoint $Q_\sigma \in \sigma$ for each $\sigma\in
\Delta^+_{\s}$. Then for any $q\in\tM^{\circ}_{\s}$, by Proposition
\ref{formalthfun}, we obtain as a sum over broken lines a well-defined series
\[
\vartheta_{Q_{\sigma},q}  \in z^{q} \widehat{\kk[P_{\s}]}
\]
satisfying by Theorem \ref{blinvth}
\[
\vartheta_{Q_{\sigma},q}=\theta_{\sigma,\sigma'}^*(\vartheta_{Q_{\sigma'},q}).
\]

The following definition will yield the structure constants for the
$\vartheta$:

\begin{definition-lemma} \label{structuredef} Let $p_1,p_2,q \in 
\tM^{\circ}_{\s}$. Let $z\in\tM^{\circ}_{\RR,\s}$ be chosen generally.
There are at most
finitely many pairs of broken lines $\gamma_1,\gamma_2$ with 
$I(\gamma_i) = p_i$, $b(\gamma_i) = z$ and $F(\gamma_1) + F(\gamma_2) = q$
(see Definition \ref{IFdef} for this notation). We can then define 
\[
\alpha_z(p_1,p_2,q)= \sum_{\substack{(\gamma_1,\gamma_2) \\
                 I(\gamma_i) = p_i, b(\gamma_i) =z \\
                 F(\gamma_1) + F(\gamma_2) = q}} 
c(\gamma_1)c(\gamma_2).
\]
The integers $\alpha_z(p_1,p_2,q)$ are non-negative.
\end{definition-lemma} 

\begin{proof} 
By definition of scattering diagram for $\foD_{\s}$, all walls
$(\fod,f_{\fod})\in \foD_{\s}$ have $f_{\fod}\in\widehat{\kk[P_{\s}]}$.
Note also that because $P_{\s}$ comes from a strictly convex cone, any
element of $P_{\s}$ can only be written as a finite sum of elements in
$P_{\s}$ in a finite number of ways. In particular, as
$F(\gamma_i) \in I(\gamma_i) + P_{\s}$, we can write 
$F(\gamma_i)=I(\gamma_i)+m_i$ for $m_i\in P_{\s}$. 
Thus we have
\[
I(\gamma_1) + I(\gamma_2) + m_1 + m_2 =q,
\]
and there are only a finite number of possible $m_1$, $m_2$.
So with $p_1,p_2,q$ fixed, there are only finitely 
many possible monomial decorations that can occur on either $\gamma_i$. 
 From this finiteness is clear, c.f.\ the proof of Proposition 
\ref{formalthfun}.
The non-negativity statement follows from Theorem \ref{scatdiagpositive}, 
which implies $c(\gamma) \in \bZ_{\geq 0}$ for any broken line $\gamma$. 
\end{proof}

\begin{definition} \label{ckdef} 
For a monoid $C \subset L$ a lattice, we write $C_k \subset C$ for the subset 
of elements which can be written as a sum of $k$ non-invertible elements of $C$.
\end{definition}

\begin{proposition} \label{thetabasislemma} 
Notation as above. The following hold:
\begin{enumerate}
\item For $q \in \tM^{\circ,+}_{\s}$, 
$\vartheta_{Q_{\sigma},q}$ is a regular function on
$V_{\s,\sigma,k}$, and the $\vartheta_{Q_{\sigma},q}$ as $\sigma$ varies
glue to give a canonically defined function
$\vartheta_{q,k} \in \up(\ocA_{\prin,k}^{\s})$. 
\item For each $q\in\cA^{\vee}_{\prin}(\ZZ^T)$ and $k'\ge k$, we have
$\vartheta_{q,k'}|_{\cA_{\prin,k}^{\s}} = \vartheta_{q,k}$,
and thus the $\vartheta_{q,k}$ for $k\ge 0$ canonically define
\[
\vartheta_q \in  \widehat{\up(\ocA_{\prin}^{\s})} \otimes_{\kk[N^+_{\s}]}\kk[N].
\]
The $\vartheta_q$ are
linearly independent, i.e.,
we have a canonical inclusion of $\kk$-vector spaces
\[
\can(\cA_{\prin}) := \bigoplus_{q \in \cXt} \kk \cdot \vartheta_q 
\subset \widehat{\up(\ocA_{\prin}^{\s})} \otimes_{\kk[N^+_{\s}]} \kk[N]. 
\]
\item 
\[
\vartheta_{p_1} \cdot \vartheta_{p_2} = \sum_{q \in \tM^{\circ}_{\s}} 
\alpha_{z(q)}(p_1,p_2,q) \vartheta_q \in 
\widehat{\up(\ocA_{\prin}^{\s})} \otimes_{\kk[N^+_{\s}]} \kk[N]
\]
for $z(q)$ chosen sufficiently close to $q$. In particular, $\alpha_z(p_1,p_2,q)$
is independent of the choice of $z$ sufficiently near $q$, and we define
\[
\alpha(p_1,p_2,q):=\alpha_z(p_1,p_2,q)
\]
for $z$ chosen sufficiently close to $q$.
\item 
\[
\hbox{$\{\vartheta_q\,|\,
q\in \tM_{\s}^{\circ,+}\setminus\tM_{\s,k+1}^{\circ,+}\}$ and
$\{\vartheta_q\,|\,
q\in \pi_N^{-1}(0)\}$ }
\]
restrict to bases of $\up(\ocA_{\prin,k}^{\s})$ as $\kk$-vector space and
$\kk[N^+_{\s}]/I_{\s}^{k+1}$-module respectively.
\end{enumerate}
\end{proposition} 

\begin{proof}
Using the isomorphism $\oshA_{\prin}^{\s}$ with $\ocA_{\scat,\s}^{\s}$, 
the basic compatibility Theorem \ref{blinvth} gives the gluing statement (1).  
To prove (4), using the $N$-linearity, it is enough to prove the given
$\vartheta_q$ restrict to  basis
as $\kk[N^+_\s]/(X_1,\dots,X_n)^{k+1}$-module.
By Corollary \ref{mutcor},
the central fibre of $\ocA_{\prin}^{\s} \to \bA^n_{X_1,\dots,X_n}$ 
is the torus $T_{N^\circ}$. If $q\in \pi_N^{-1}(0)$, the
only broken lines with $q=I(\gamma)$ and 
$\Mono(\gamma) \not \in (X_1,\dots,X_n)$ are straight lines.
Thus these $\vartheta_{q}$ restrict to the basis
of characters on the central fibre. Now the result follows from the 
nilpotent Nakayama lemma (see \cite{Ma89}, pg. 58, Theorem 8.4).

(2) follows immediately from (1) and (4). 

For (3), it is enough to prove the equality in $\ocA_{\prin,k}^{\s}$ for
each $k$. The argument
is the same as the proof of the multiplication rule in \cite{GHK11},
Theorem 2.38, which, as it is very short, we 
recall for the reader's convenience. We work with the scattering diagram
$\foD_{\s}$ modulo $I_{\s}^{k+1}$, which has only finitely
many walls with non-trivial attached function. Expressing $\vartheta_{p_1}
\cdot\vartheta_{p_2}$ in the basis $\{\vartheta_q\}$ of (4), we examine
the coefficient of $\vartheta_q$. 
We choose a general 
point $z \in \tM^{\circ}_{\RR}$ very close to $q$, so that 
$z,q$ lie in the closure
of the same connected component of $\tM^{\circ}_{\RR} \setminus 
\Supp(\foD_{\s,k})$
(where $\foD_{\s,k}$ denotes the finitely many walls
non-trivial modulo $I_{\s}^{k+1}$). By definition of $\alpha_z$,
\[
\vartheta_{z,p_1} \cdot \vartheta_{z,p_2} = 
\sum_{r} \alpha_z(p_1,p_2,r)z^r.
\]
Now observe first that there is only one broken line $\gamma$
with endpoint $z$ and $F(\gamma)=q$:
this is the broken line whose image is $z+\RR_{\ge 0} q$.
Indeed, the final segment of such a $\gamma$ is on this ray, and this ray
meets no walls, other than walls containing $q$, 
so the broken line cannot bend.  Thus the coefficient of $\vartheta_{z,q}$ 
can be read off by looking at the coefficient
of the monomial $z^q$ on the right-hand side of the above equation.
This gives the desired formula to order $k$. The finiteness argument of
Definition-Lemma \ref{structuredef} then shows that for any given $q$, $z$
chosen sufficiently close to $q$ will work for all $k$.
\end{proof}

By the proposition, each $g \in \widehat{\up(\ocA_{\prin}^{\s})}$ 
has a unique expression
as a convergent formal sum $\sum_{q \in \tM^{\circ,+}_{\s}} 
\alpha_{\s}(g)(q) \vartheta_{q}$, with
coefficients $\alpha_{\s}(g)(q) \in \kk$. This immediately implies:

\begin{proposition} \label{hatinfun} 
Notation as in Proposition \ref{thetabasislemma}.
There is a unique inclusion
\[
\alpha_{\s}: \widehat{\up(\ocA_{\prin}^{\s})} \otimes_{\kk[N^+_\s]} \kk[N] 
\hookrightarrow  \Hom_{\operatorname{sets}}(\cXt= \tM^\circ_{\s},\kk)
\]
given by 
\[
g\mapsto (q\mapsto \alpha_{\s}(g)(q)).
\]
We have $\alpha_{\s}(z^n \cdot g)(q +n) = \alpha_{\s}(g)(q)$ for all $n \in N$.
\end{proposition}

\begin{definition}
\label{defS}
For $g\in\up(\cA_{\prin})$, write $g=\sum_{q\in \tM^{\circ}_{\s}} \beta_{s}(g)(q)
z^q$ on the torus chart $T_{\tN^{\circ},\s}$ of $\cA_{\prin}$
corresponding to a seed $\s$. We also have a formal expansion 
$g=\sum_{q\in \tM^\circ_{\s}} \alpha_{\s}(g)(q)\vartheta_q$
as $z^m g\in \widehat{\up(\ocA_{\prin}^{\s})}$ for some $m\in \tM^{\circ}_{\s}$.
Set 
\begin{align*}
\overline{S}_{g,\s}= {} & \{q\in \tM^{\circ}_{\s}\,|\, \beta_{\s}(g)(q)
\not=0\},\\
S_{g,\s}= {} & \{q \in \tM^{\circ}_{\s}\,|\, \alpha_{\s}(g)(q)\not=0\}.
\end{align*}
If $P_{\s}$ is the monoid generated by $\{(v_i,e_i)\,|\,i\in I_{\uf}\}$,
it is easy to check from the construction of theta functions that
\[
S_{g,\s}\subseteq \overline{S}_{g,\s}+P_{\s}.
\]
\end{definition}

\begin{remark}
Note that $\ocA_{\scat,\s,k}^{\s}$ is constructed by gluing
together various $V_{\s,\sigma,k}\cong T_{M^{\circ}}\times \Spec
\kk[X_1,\ldots,X_n]/(X_1,\ldots,X_n)^{k+1}$ via isomorphisms, and
hence $\ocA_{\scat,\s,k}^{\s}\cong T_{M^{\circ}}\times \Spec
\kk[X_1,\ldots,X_n]/(X_1,\ldots,X_n)^{k+1}$. Thus by
Proposition \ref{thetabasislemma}, (4), a collection of the $\vartheta_q$
yield a basis for regular functions on 
$T_{M^{\circ}}\times \Spec \kk[X_1,\ldots,X_n]/(X_1,\ldots,X_n)^{k+1}$
with the property that this basis restricts
to a monomial basis on the underlying
reduced space. It remains a mystery about theta functions in general
whether they satisfy some other interesting characterizing properties, such
as the heat equation satisfied by theta functions on abelian varieties.
\end{remark}

The main point of the following theorem is that on $\up(\shA_{\prin})$,
$\alpha_{\s}$ is independent of the seed $\s$.

\begin{theorem} \label{ffgth} There is a unique function
\[
\alpha: \up(\cA_{\prin}) \to \Hom_{\operatorname{sets}}(\cXt,\kk),
\]
with all the following properties:
\begin{enumerate}
\item $\alpha$ is compatible with the $\kk[N]$-module structure on $\up(\cA_{\prin})$ and the
$N$-translation action on $\cXt$ in the sense that
$$
\alpha(z^n \cdot g)(x+n) = \alpha(g)(x)
$$
for all $g \in \up(\cA_{\prin})$, $n \in N$, $x \in \cXt$.
\item For each choice of seed $\s$, the formal sum $\sum_{q \in \cXt} \alpha(g)(q) \vartheta_q$
converges to $g$ in $\widehat{\up(\ocA_{\prin}^{\s})} \otimes_{\kk[N^+_\s]}\kk[N]$.
\item If $z^n\cdot g \in \up(\ocA_{\prin}^{\s})$ then $\alpha(z^n \cdot g)(q) = 0$ unless
$\pi_{N}(q) \in N^+_\s$, and 
\[
z^n \cdot g = \sum_{\pi_{N,\s}(q) \in N^+_\s \setminus (N^+_\s)_{k+1}} 
\alpha(z^n \cdot g)(q) \vartheta_q
\mod I_{\s}^{k+1}
\]
and the coefficients $\alpha(z^n \cdot g)(q)$ are the coefficients for the expansion of 
\mbox{$z^n \cdot g$} viewed as an element of $\up(\ocA_{\prin,k}^{\s})$ 
in the basis of theta functions from Proposition \ref{thetabasislemma}.
\item For any seed $\s'$ obtained via mutations from $\s$, 
$\alpha$ is the composition of the inclusions 
\[
\up(\cA_{\prin}) \subset \widehat{\up(\ocA_{\prin}^{\s'})} \otimes_{\kk[N^+_{\s'}]} 
\kk[N] \subset 
\Hom_{\operatorname{sets}}(\cXt= \tM^\circ_{\s'},\kk)
\]
given by \eqref{upginhat} and Proposition \ref{hatinfun}. This sends a cluster monomial
$A \in \up(\cA_{\prin})$ to the delta function $\delta_{\g(A)}$ 
for its $g$-vector $\g(A) \in \cXt$. 
\end{enumerate}
In the notation of Definition \ref{defS}, $\alpha(g)(m) = \alpha_{\s'}(g)(m)$
for any seed $\s'$. In particular the sets $S_{g,\s'} \subset \cXt$ 
of Definition \ref{defS} are independent of the seed, depending only on $g$. 
\end{theorem} 

\begin{proof}
It is easy to see from Proposition \ref{thetabasislemma} 
that $\alpha_{\s}$ is the unique function which satisfies conditions 
(1-3) of the theorem for the given seed $\s$. Moreover
it satisfies (4) for $\s=\s'$. Thus it is
enough to show that $\alpha_{\s}$ is independent of the choice of seed.

The basic idea is that $\alpha_{\s}$ 
expresses $g$ as a sum of theta functions. As the theta functions are linearly
independent, the expression is unique. But as the sums can be infinite, we make
the argument in the appropriate formal neighborhood. 

For a seed $\s = (e_i\,|\,i\in I)$ we write
$\Sigma^{\s}$ for the fan in $\tN^\circ = N^\circ \oplus M$ with rays
spanned by the $(0,d_if_i)$. We write
$\oSigma^{\s}$ for the fan in $M$ with rays spanned by the
 $d_if_i$.

Clearly for the invariance it is enough to consider two adjacent 
seeds, say $\s = (e_i\,|\,i\in I)$ and $\s' = (e'_i\,|\,i\in I)$ obtained, 
without loss of generality, by mutation of the first basis vector $e_1$.

We consider the union of the two tori
$T_{\tN^\circ,\s},T_{\tN^\circ,\s'}$ in the atlas for $\cA_{\prin}$, 
glued by the mutation
$\mu_1$, which we recall is given by
\[
\mu_1^*:z^{(m,n)} \mapsto z^{(m,n)} \cdot (1 + z^{(v_1,e_1)})^{-\langle
(d_1e_1,0),(m,n)\rangle}, 
\]
$(m,n) \in \tM^{\circ}= M^\circ \oplus N$, see \cite{P1}, (2.6). We
will partially compactify this union by gluing 
the toric varieties 
\[
\mu_1:\TV(\Sigma^{\s}) \dasharrow \TV(\Sigma^{\s'}),
\]
writing
\[
U:=\TV(\Sigma^{\s})\cup \TV(\Sigma^{\s'})
\]
under this gluing. Note this union of toric varieties is \emph{not}
part of the atlas for either $\ocA_{\prin}^{\s}$ or
$\ocA_{\prin}^{\s'}$
(for either of these, the fans determining the atlases
for the toric compactifications are related by geometric tropicalisation of
the birational mutation, but here $\mu^t_1(f_1) \neq f_1'$, and thus 
$\Sigma^{\s'} \neq \mu^t_1(\Sigma^{\s})$).

Note $f_i' = f_i$ for $i \not=1$, while $f_1' = -f_1 + 
\sum_j [\{e_j,e_1\}]_+ d_jf_j$,
(see e.g., \cite{P1}, (2.3)). 
Thus the two cones $\oSigma^{\s},\oSigma^{\s'}$ share a codimension one face and
form a fan, $\oSigma$.
Let $V=\TV(\oSigma)$. By construction the rational map
$\TV(\Sigma^\s) \to \TV(\oSigma^\s)$ is regular, and the same holds for the seed
$\s'$.
Observe that $\mu_1$ commutes with the second projection $\pi: 
T_{\tN^\circ} \to T_{M}$. 
 From this it follows that 
$\pi: U \to V$ is regular. Note the toric boundary $\partial V$ has a unique
complete one-dimensional stratum $\bP^1$ and two zero strata $0_\s$, $0_{\s'}$, whose complements in the $\bP^1$ we 
write as $\bA^1_{\s'},\bA^1_{\s}$ respectively. 
We write e.g.\ $\bA^1_{\s,k} \subset V$ for the
$k^{\mth}$ order neighborhood of this curve, and e.g.\ 
$U_{\bA^1_{\s},k}$ for the scheme-theoretic
inverse image $\pi^{-1}(\bA^1_{\s,k}) \subset U$.
Finally, let 
\[
U_{\bG_m,k} = U_{\bA^1_{\s},k} \cap U_{\bA^1_{\s'},k} \subset U.
\]
We will show theta functions give a basis of functions on these formal neighborhoods. 
To make the computation transparent we introduce coordinates.

We let $X_i := z^{(0,e_i)}$, $X_i':= z^{(0,e_i')}$, observing that
$\mu_1^*(z^{(0,n)}) = z^{(0,n)}$ for all $n \in N$. 
In particular $\mu_1^*(X_i)=X_i$, $\mu_1^*(X_i')=X_i'$. Since there is
a map of fans from $\oSigma$ to the fan defining $\PP^1$ by dividing out
by the subspace spanned by $\{d_if_i\,|\,i\in I\setminus\{1\}\}$, 
there is a map
$V\rightarrow\PP^1$. We can pull back $\shO_{\PP^1}(1)$ to $V$, getting
a line bundle
with monomial sections $X,X'$ pulled back from $\PP^1$
with $X'/X = X_1'$ in the above notation. The open subset 
of $U$ where $X' \neq 0$ is given explicitly up to codimension two
by the hypersurface
\[
A_1 \cdot A_1' = X_1\prod_{j:\epsilon_{1j} \geq 0} A_i^{\epsilon_{1j}} 
+ \prod_{j:\epsilon_{1j} \leq 0} A_j^{-\epsilon_{1j}} \subset 
\bA^n_{X_1,\dots,X_n} 
\times \AA^2_{A_1,A_{1}'}\times
(\bG_m)^{n-1}_{A_2,\dots,A_{n}} 
\]
where $A_i = z^{(f_i,0)}$ and $A_1' = z^{(f_1',0)}$. 

Note the points 
\[
(f_i,0),(0,e_i) \in (M^\circ\oplus N)_{\s} = \cXt,\quad (f_1',0) \in 
(M^\circ \oplus N)_{\s'} = \cXt
\]
lie in the chambers of $\Delta^+_{\s}$ corresponding to $\s$ and $\s'$
respectively,
and thus by Proposition \ref{oneblprop} these points determine theta functions
in $\up(\cA_{\prin})$, which are of course the corresponding cluster
monomials
$A_i,X_i,A_1'$. 

We have the exactly analogous
description for the open subset $X \neq 0$.

Next observe that all but one of the functions attached to walls
in $\foD_{\s}$ is trivial modulo
the ideal  $J=(X_i\,|\,i\in I\setminus\{1\})$. 
Indeed, the unique non-trivial wall is
$((e_1,0)^{\perp}, 1 + z^{(v_1,e_1)})$. It follows from
Theorem \ref{KSlemma2} that modulo 
$J^k$ the
scattering diagram $\foD_{\s}$ has only finitely many non-trivial walls, 
and $\vartheta_{Q,m}$ is regular
on $U_{\bA^1_{\s},k}$, for $Q$ a basepoint in the distinguished chamber 
$\shC^+_{\s}$, so long as $\pi_N(m) \in \Span(e_1,\dots,e_n)$, 
$\pi_N:\tM^{\circ}\rightarrow N$ the projection. 

Let $C := \sum_{k=1}^n \NN e_i$, $C' := \sum_{k=1}^n \NN e_i'$. 
Noting $e_1' = -e_1$, we can set 
\[
\widetilde C:=\ZZ e_1 + \sum_{k=2}^{n}\NN e_k  = 
\ZZ e_1'+\sum_{k=2}^{n}\NN e_k'.
\]
Observe $U_{\bG_m,k}$ is the subscheme of $U$ defined by 
the ideal $J^k$ in the open subset $XX' \neq 0 \subset U$. 
Note that the open subset of $U$ defined by
$X X' \neq 0,\, \prod_{i\not=1} X_i\neq 0$ is the union of the two tori 
$T_{\tN^\circ,\s},T_{\tN^\circ,\s'}$.

\begin{claim} \label{basislem} The following hold:
\begin{enumerate}
\item The collection $\vartheta_{Q,m}$, $m\in \tM^{\circ}$,  
$\pi_N(m) \in C \setminus (\widetilde C_{k+1}\cap C)$ forms a $\kk$-basis of
the vector space $\up(U_{\bA^1_{\s,k}})$. 
\item The collection $\vartheta_{Q,(m,0)}$, $m \in M^\circ$, forms a basis
of $\up(U_{\bA^1_{\s,k}})$
as an $H^0({\bA^1_{\s,k}},\cO_{\bA^1_{\s,k}})$-module. 
\item The collection $\vartheta_{Q,m}$, $\pi_N(m) 
\in \widetilde{C}\setminus\widetilde{C}_k$, 
forms a $\kk$-basis of $\up(U_{\bG_m,k})$. 
\end{enumerate}
\end{claim}

\begin{proof} (2) implies (1) using the $N$-linearity of the
scattering diagram and multiplication rule with respect to the $N$-translation. Similarly, (2) implies (3) by inverting $X_1$.

For the second claim, 
by Lemma 2.30 of \cite{GHK11},
we need only prove the statement for $k=0$. To
prove linear independence it is enough to show linear independence modulo 
$(X_1^r,X_2,\ldots,X_n)$ for all $r$.
For this, again by Lemma 2.30 of \cite{GHK11}, it
is enough to check just over the fibre $X_1 = \cdots =X_n =0$. 
This is the torus 
$T_{N^{\circ}}$ and the theta functions restrict to the basis of characters. 

So it remains only to show the given theta functions generate modulo $J$. Here 
we
use the explicit description of the open subset of $U$ where $X' \neq 0$ above. 
This is an affine variety, and the ring of functions is clearly generated by the
$A_1,A_1',A_2^{\pm 1},\ldots,A_n^{\pm 1}$ as a 
$\kk[X_1,\ldots,X_n]$-algebra. On the other hand, by the explicit description
of $\foD_{\s}$ modulo the ideal $J$, for $m=\sum a_if_i=\sum a_i'f_i'
\in M^{\circ}$, 
\[
\vartheta_{Q,(m,0)}=\begin{cases} \prod_i A_i^{a_i} & a_1\ge 0\\
(A_1')^{a'_1} \prod_{i\not=1} A_i^{a'_i} & a_1 = -a_1' \le 0.
\end{cases}
\]
This shows theta functions generate $\up(U_{\AA^1_{\s,0}})$ as an
$H^0(\AA^1_{\s,0},\shO_{\AA^1_{\s,0}})$-module, hence the result.
\end{proof} 

Of course there is an analogous claim for $\s'$. 

Now we can prove $S_{g,\s} = S_{g,\s'}$ for $g\in \up(\shA_{\prin})$.

By the $N$-translation action on $\cXt$ (and the corresponding $N$-linearity of
the scattering diagrams), to prove the equality, we are free to multiply $g$ by
a monomial from the base of $\cA_{\prin} \to T_M$. Multiplying by a monomial
in the $X_i$, $i\not=1$, we can then assume 
$g$ is a regular function on the
open subset of $U$ where $XX' \neq 0$. 
Now in the notation of Definition \ref{defS},
$\pi_N(m) \in \widetilde{C}$ for $m \in P_{\s} + \overline{S}_{g,\s}$ 
or $m\in P_{\s'} + \overline{S}_{g,\s'}$.
It follows now from the fact that $\foD_{\s}$ is finite modulo $J^k$
for any $k$ that each $\vartheta_{Q,m},\vartheta_{Q',m}$
for $m \in S_{g,\s},S_{g,\s'}$ is a finite Laurent polynomial modulo $J^k$.
Here $Q, Q'$ are basepoints in the chambers indexed by $\s$ and $\s'$
respectively.

\begin{claim}
\label{claimtwo}
 Modulo $J^k$, the  sums 
$\sum_{m \in S_{g,\s}} \alpha_m \vartheta_{Q,m}$, 
$\sum_{m \in S_{g,\s'}} \alpha_m' \vartheta_{Q',m}$
are finite, and coincide with $g$ in the charts indexed by $\s$ and $\s'$
respectively.
\end{claim}

\begin{proof}
By symmetry it's enough to treat $\s$. We can multiply both
sides by a power of $X_1$, and so may assume $g$ is regular on $X' \neq 0$, and
$\pi_N(m) \in C$ for each $\alpha_m \neq 0$. Note $\pi_N(P_{\s} \setminus \{0\}) = C\setminus \{0\}$,
thus by construction modulo $(X_1,\ldots,X_n)^r$ for any $r$
the sum $\sum_{m\in S_{g,\s}}\alpha_m \vartheta_{Q,m}$ is finite and
equal to $g$. By Claim \ref{basislem},
(1), we have a (finite) expression modulo $J^k$
\[
g = \sum_{\pi_N(m) \in C \setminus (C\cap \widetilde C_{k+1})} 
\beta(m) \vartheta_{Q,m}. 
\]
Thus, for fixed $k$ and arbitrary $r \geq 1$ we have 
modulo $J^k+(X_1^r)$,
\begin{align*}
g = {}  & \sum_{\pi_N(m) \in C \setminus (C\cap \widetilde C_{k+1})} 
\beta(m) \vartheta_{Q,m} \\
  = {}  & \sum_{m \in S_{g,\s}} \alpha_m \vartheta_{Q,m}.
\end{align*}
By the linear independence these expressions are the same, for all $r$,
thus the expressions are the same modulo $J^k$. 
\end{proof}

Note that by Theorem \ref{blinvth}, for 
$m \in \pi_N^{-1}(\widetilde{C})$, $\vartheta_{Q,m}$ and $\vartheta_{Q',m}$ 
induce the same regular function $\vartheta_m$ on $U_{\bG_m,k}$.
Thus we have by Claim \ref{claimtwo} that
\[
g =\sum_{m \in S_{g,\s}} \alpha_m \vartheta_m = 
\sum_{m \in S_{g,\s'}} \alpha_m' \vartheta_m \mod J^k.
\]
Now by (3) of Claim \ref{basislem} (varying $k$) the coefficients in the sums are the same. 
\end{proof}

The theorem implies that the theta functions are a topological basis for a natural topological $\kk$-algebra
completion of $\up(\cA_{\prin})$:

\begin{corollary}
For $n\in N$, let $n^*:\Hom_{\operatorname{sets}}(\cXt,\kk)\rightarrow
\Hom_{\operatorname{sets}}(\cXt,\kk)$ denote precomposition by the
action of translation by $n$ on $\cXt$.
Let 
\[
\LIM \subset \Hom_{\operatorname{sets}}(\cXt,\kk)
\]
be the vector subspace of functions, $f$, such that for each seed $\s$, 
there exists $n\in N$
for which the restriction of $n^*(f)$ to $\cXt \setminus \pi_{N,\s}^{-1}((N^+_{\s})_k)$ has finite
support for all $k > 0$. Then we have
\[
\up(\cA_{\prin}) \subset  \LIM = \bigcap_{\s} \widehat{\up(\ocA_{\prin}^{\s})} \otimes_{\kk[N^+_{\s}]}
\kk[N] \subset \Hom_{\operatorname{sets}}(\cXt,\kk).
\]
$\LIM$ is a complete topological vector space under the weakest topology so that each inclusion
$\LIM \subset \widehat{\up(\ocA_{\prin}^{\s})} 
\otimes_{\kk[N^+_{\s}]} \kk[N]$ is continuous. 
Let $\vartheta_q = \delta_q \in \LIM$ be the delta function associated
to $q \in \cXt$. The $\vartheta_q$ are a topological basis for $\LIM$.
There is a unique topological $\kk$-algebra structure on $\LIM$
such that $\vartheta_p \cdot \vartheta_q = \sum_r \alpha(p,q,r) \vartheta_r$ with structure
constants given by Definition \ref{structuredef}.
\end{corollary}

\section{The middle cluster algebra}
\label{midclustersection}

In this section, we prove one of the main theorems of the paper, Theorem
\ref{mainth}. This is done in two steps. First, it follows from the
results of the previous section and properties of theta functions
in the $\cA_{\prin}$ case. This is easiest since the scattering diagram
technology works best for $\cA_{\prin}$. Second, we descend to the $\cA$ or
$\cA_t$
case and the $\cX$ case, with the $\cA$-type varieties appearing as
fibres of $\cA_{\prin}\rightarrow T_M$ and the $\cX$ variety as a quotient
of $\cA_{\prin}$ by $T_{N^{\circ}}$, using the $\cA_{\prin}$ case to deduce
the result for these other cases.

\subsection{The middle algebra for $\shA_{\prin}$}

Recall 
from Definition \ref{FGchamberdef} that $\Delta^+_{\s}$ is the collection of
chambers forming the cluster complex. Abstractly, by Lemma \ref{fgcclem},
this can be viewed as
giving a collection of chambers $\Delta^+$ in $\cXrt$. 

\begin{proposition} \label{posprop} Choose $m_0 \in \cXt$. If for some generic 
basepoint $Q \in \sigma \in \Delta^+$
there are only finitely many broken lines $\gamma$ with
$I(\gamma) = m_0$ and $b(\gamma) = Q$, then the same is true for any generic 
$Q' \in \sigma'\in \Delta^+$.
In particular, $\vartheta_{Q,m_0} \in \kk[\tM^\circ]$ is 
a positive universal Laurent polynomial. 
\end{proposition}

\begin{proof} 
By positivity of the scattering diagram, Theorem \ref{scatdiagpositive}, 
for any basepoint $Q$, $\vartheta_{Q,m_0}$ has only non-negative coefficients 
(though
it may have infinitely many terms). Also, we know that for basepoints in different chambers,
the $\vartheta_{Q,m_0}$ are related by wall-crossings
by Theorem \ref{blinvth}, 
which in turn are identified with
the mutations of tori in the atlas for $\cA_{\prin}$. So
the $\vartheta_{Q,m_0}$ determine a universal positive Laurent
polynomial if and only if we have finiteness of broken lines ending at any $Q$
in any chamber of $\Delta^+_{\s}$. If we vary $Q$ in the
chamber, $\vartheta_{Q,m_0}$ does not change. 
So it's enough to check that if $\vartheta_{Q,m_0}$
is a polynomial, the same is true of $\vartheta_{Q',m_0}$ 
for $Q'$ in an adjacent 
chamber $\sigma'$ to $\sigma$ close to the wall $\sigma\cap\sigma'$. We can
work in some seed. Let 
the wall be contained in $n_0^{\perp}$, $n_0 \in \tN^{\circ}$,
with $\langle n_0,Q\rangle > 0$, and denote the wall-crossing automorphism
from $Q$ to $Q'$ as $\theta$. 
Recall that 
$\theta(z^m) = z^m f^{\langle n_0,m\rangle}$,
where for walls between cluster chambers, $f$ is
some positive Laurent polynomial (in fact it has the form $1 + z^{q}$ for
some $q \in n_0^{\perp}\subset \tM^{\circ}$). 

Monomials $m \in \tM^{\circ}$ are then divided
into three groups, according to the sign of $\langle n_0,m\rangle$. 
This sign is preserved by $\theta$, as
$n_0$ takes the same value on
each exponent of a monomial term in $\theta(z^m)$ 
as $n_0$ takes on $m$.

Monomials with $\langle n_0,m\rangle =0$ 
are then invariant under $\theta$, so 
these terms in $\vartheta_{Q',m_0}$ coincide with those in $\vartheta_{Q,m_0}$.
Hence there are only a finite number of such terms in $\vartheta_{Q',m_0}$.

The sum of  terms of the form $cz^m$
in $\vartheta_{Q,m_0}$ with $\langle n_0,m\rangle > 0$, 
which we know form a Laurent polynomial, is, by the explicit formula
for $\theta$, sent to a polynomial.
So it only remains to show that there are only finitely many terms $cz^m$ 
in $\vartheta_{Q',m_0}$ with $\langle n_0,m\rangle < 0$. Suppose the contrary
is true. 
The direction vector of each broken line contributing to such terms at 
$Q'$ is towards the wall $\sigma\cap\sigma'$, and so we can
extend the final segment of any such broken line to obtain a 
broken line terminating at some point $Q''$ (depending on $m$) in the same 
chamber as $Q$. As there are no cancellations because of
the positivity of all coefficients and $\vartheta_{Q,m_0}$ does not
depend on the location of $Q$ inside the chamber by Theorem \ref{blinvth},
we see that $\vartheta_{Q,m_0}$ has an infinite number of terms, a 
contradiction.
\end{proof} 

\begin{definition} \label{thetadef} Let $\Theta \subset \cXt$ be the collection of $m_0$ such that for some (or equivalently,
by Proposition \ref{posprop}, any) generic $Q \in \sigma\in\Delta^+$ 
there are only finitely many broken
lines $\gamma$ with $I(\gamma) = m_0$, $b(\gamma) = Q$. 
\end{definition}

\begin{definition} \label{icdef} We call a subset 
$S \subset \cXt$ \emph{intrinsically closed 
under addition} if 
$p,q \in S$ and $\alpha(p,q,r) \neq 0$ implies $r \in S$.
\end{definition}

\begin{lemma}
\label{iclemma}
Let $S\subset \cXt$ be intrinsically
closed under addition. The image of $S$ in $\tM^{\circ}_{\s}$ 
(under the identification $\cXt = \tM^{\circ}_{\s}$
induced by the seed $\s$) is closed under addition for any seed $\s$. 
If for some seed $S \subset \tM^{\circ}_{\s}$ is a toric monoid (i.e.,
the integral points of a convex rational polyhedral cone), then this holds
for any seed.
\end{lemma} 

\begin{proof} 
Choose a seed $\s$. Then straight lines in Definition-Lemma \ref{structuredef}
show $\alpha(p,q,p+q) \neq 0$. 
This gives closure under addition.
Now suppose $S \subset \tM^{\circ}_{\s}$ is a toric monoid, generating
the convex rational polyhedral cone $W \subset \tM^{\circ}_{\s,\bR}$. Then
$\mu_{\s,\s'}(W) \subset \tM^{\circ}_{\s',\bR}$ is a rational polyhedral cone 
with
integral points $S \subset \tM^{\circ}_{\s'}$. As this set of integral points 
is closed under addition,
$\mu_{\s,\s'}(W)$ is convex, and so its integral points are a toric
monoid.
\end{proof}

Recall from the introduction the definition of global monomial (Definition
\ref{globalmonomialdef}).

\begin{theorem} \label{goodalgebra} Let  
\[
\Delta^+(\bZ)=\bigcup_{\sigma\in\Delta^+} \sigma \cap \cXt
\]
be the set of integral points in chambers of the cluster complex. Then
\begin{enumerate}
\item 
$\Delta^+(\ZZ)\subset\Theta$.
\item
For $p_1,p_2 \in \Theta$ 
\[
\vartheta_{p_1} \cdot \vartheta_{p_2} = \sum_{r} \alpha(p_1,p_2,r) \vartheta_r
\]
is a finite sum (i.e., $\alpha(p_1,p_2,r) =0$ for all but finitely many $r$) 
with non-negative integer coefficients. If $\alpha(p_1,p_2,r) \neq 0$,
then $r \in \Theta$. 
\item
The set $\Theta$ is intrinsically closed under addition. For any seed $\s$, the 
image of $\Theta \subset \tM^{\circ}_{\s}$ is a saturated monoid. 
\item
The structure constants $\alpha(p,q,r)$ of Definition \ref{structuredef} make the 
$\kk$-vector space with basis indexed by $\Theta$, 
\[
\midd(\cA_{\prin}) := \bigoplus_{q \in \Theta} \kk \cdot \vartheta_q
\]
into an associative commutative $\kk[N]$-algebra. 
There are canonical inclusions of $\kk[N]$-algebras
\begin{align*}
\ord(\cA_{\prin}) \subset \midd(\cA_{\prin}) & \subset \up(\cA_{\prin}) \\
                                             & \subset 
\widehat{\up(\cA_{\prin,\s})} \otimes_{\kk[N_{\s}^+]} \kk[N].
\end{align*}
Under the first inclusion a cluster monomial $Z$ is identified with 
$\vartheta_{\g(Z)}$ for $\g(Z) \in \Delta^+(\bZ)$
its $g$-vector. Under the second inclusion each $\vartheta_q$ is
identified with a universal positive Laurent polynomial.
\end{enumerate}
\end{theorem}

\begin{proof} (1) is immediate from Corollary \ref{monocor}.
For (2), first note that
the coefficients $\alpha(p_1,p_2,r)$ are non-negative by Definition-Lemma \ref{structuredef}.
Suppose $p_1,p_2 \in \Theta$. Take a generic basepoint $Q$ in some
cluster chamber. Then
$\vartheta_{Q,p_1} \cdot \vartheta_{Q,p_2}$ is the product of two 
Laurent polynomials, thus a Laurent polynomial.
It is equal to $\sum_{r} \alpha(p_1,p_2,r) \vartheta_{Q,r}$ 
by (3) of Proposition \ref{thetabasislemma}, and hence this sum must be finite,
as it involves a positive linear combination of series with positive
coefficients. Further, each $\vartheta_{Q,r}$ appearing with $\alpha(p_1,p_2,r)
\not=0$ must be a Laurent polynomial for the same reason.
Thus $r \in \Theta$ by Definition \ref{thetadef}.  
(2) then immediately implies $\Theta$ is intrinsically closed under addition.

For (4), note each $\vartheta_{Q,p}$, $p \in \Theta$ is a
universal positive Laurent polynomial by Proposition \ref{posprop}. 
For $p \in \Delta^+(\bZ)$, $\vartheta_p \in \up(\cA_{\prin})$ 
is the corresponding cluster monomial by (4) of Theorem
\ref{ffgth}. The inclusions of algebras, and the associativity of the multiplication
on $\midd$ follow from Proposition \ref{thetabasislemma}. 

Finally we complete the proof of (3) by checking that $\Theta$ is saturated. 
Assume $k \cdot q \in \Theta$ for some integer $k \geq 1$. 
Take $Q$ to be a generic basepoint in some cluster chamber. 
We show that the set of final monomials $S(q) := \{F(\gamma)\}$ 
for broken lines $\gamma$ with 
$I(\gamma) = q$, $b(\gamma) = Q$ is finite. By assumption (and the positivity of the scattering diagram), 
this holds with $q$ replaced by $kq$. So it is enough to show $m \in S(q)$ 
implies $km \in S(kq)$. 
Indeed, it is easy to see that for every broken line $\gamma$ for $q$ ending
at $Q$, there is a broken line $\gamma'$ for $kq$
with the same underlying path, such that for every domain of
linearity $L$ of $\gamma$, the exponents $m_L$ and $m'_L$
of the monomial decorations of $L$ for $\gamma$ and $\gamma'$ respectively
satisfy
$m'_L=km_L$. This completes the proof of (3), hence the theorem.
\end{proof}

The above theorem immediately implies:

\begin{corollary}
Theorem \ref{mainth} is true for $V=\shA_{\prin}$.
\end{corollary}

The following shows our theta functions are well-behaved with respect
to the canonical torus action on $\cA_{\prin}$.

\begin{proposition}
\label{thetaeigenfunctions}
Let $q\in \Theta\subset \cA^{\vee}_{\prin}(\ZZ)$. Then $\vartheta_q\in 
\up(\cA_{\prin})$ is an eigenfunction
for the natural $T_{\tK^{\circ}}$ action on $\cA_{\prin}$ (see 
Proposition \ref{ldpprop}, (2)), with weight $w(q)$  given by the canonical
map $w:\tM^{\circ} = (\tN^\circ)^* \to (\tK^{\circ})^*$ (the map being
dual to the inclusion 
$\tK^{\circ} \subset \tN^{\circ}$). In particular $\vartheta_q$ is an eigenfunction for the 
subtorus $T_{N^{\circ}} \subset T_{\tK^{\circ}}$ with weight $w(q)$ where 
$w: \tM^{\circ} \rightarrow M^{\circ}$ is given by $(m,n)\mapsto m-p^*(n)$.
\end{proposition}

\begin{proof}
Pick a seed $\s$, giving an identification $\shA^{\vee}_{\prin}(\ZZ)$ with
$\tM^{\circ}$. Pick also a general basepoint $Q\in \shC^+_{\s}$.
We need to show that for any broken line $\gamma$ in $\tM_{\RR}^{\circ}$
for $q$ with endpoint $Q$, $\Mono(\gamma)$ is a semi-invariant for
the $T_{\tK^{\circ}}$ action with weight $w(q)$. The $T_{\tK^{\circ}}$ action
on the seed torus $T_{\tN^{\circ},\s}\subset \shA_{\prin}$ is given on cocharacters by the
natural inclusion $\tK^{\circ} \subset \tN^{\circ}$. By definition of $\tK^{\circ}$
we have $w(v_i,e_i)=0$, $i \in I_{\uf}$, 
so every monomial appearing in
any function $f_{\fod}$ for $(\fod,f_{\fod})\in\foD^{\cA_{\prin}}_{\s}$
is in the kernel of $w$. The result for $T_{\tK^{\circ}}$ follows. The statement for $T_{N^{\circ}}$ now
follows from the definitions. 
\end{proof}

With more work, we will define the middle cluster algebra for $V=\cA_t$ or 
$\cX$. 

\subsection{From $\shA_{\prin}$ to $\cA_t$ and $\cX$.} 
\label{blgcss}
We now show how the various structures we have used to understand 
$\shA_{\prin}$ induce similar structures for $\cA_t$ and $\cX$.

By \cite{P1}, \S 3, each seed $\s$ (in the $\cX$, $\cA$ and $\cA_{\prin}$ cases)
gives a \emph{toric model} for $V$. The seed specifies the data of a fan
$\Sigma_{\s,V}$, consisting only of rays 
(so the boundary $\oD \subset \TV(\Sigma_{\s,V})$ of the
associated toric variety is a disjoint union of tori). The seed
also then specifies a blowup $Y \to \TV(\Sigma_{\s,V})$
with codimension two center, the disjoint union of divisors $Z_i \subset \oD_i$ 
in each of the disjoint irreducible components $\oD_i \subset \oD$. If
$D$ is the proper transform of $\oD$, then there is 
a birational map $Y \setminus D \dasharrow V$. This map
is an isomorphism outside of codimension two between 
$Y \setminus D$ and the upper bound
(see \cite{P1}, Remark 3.13, \cite{BFZ05}, Def.\ 1.1) $V_{\s} \subset V$, 
which we recall is the union of $T_{L,\s}$ with $T_{L,\s'}$ for the
{\it adjacent seeds}, $\s'=\mu_k(\s), k \in I_{\uf}$. 
In the case $V = \cA_{\prin},\cX$ or $\cA_t$ for
very general $t$, the inclusion $V_{\s} \subset V$ is an isomorphism
outside codimension two. We have
\begin{align*}
\Sigma_{\s,\cA}= {} & \{ \RR_{\ge 0} e_i \,|\, i\in I_{\uf}\},\\
\Sigma_{\s,\cX}= {} & \{ -\RR_{\ge 0} v_i \,|\, i\in I_{\uf}\}.
\end{align*}

From these toric models it is easy to determine the global monomials: 

\begin{lemma}[Global Monomials] \label{gmlemma} Notation as immediately above. 
For $m\in \Hom(L_{\s},\ZZ)$,
 the character $z^m$ on the torus $T_{L,\s} \subset V$ 
is a global monomial
if and only if $z^m$ is regular on the toric variety 
$\TV(\Sigma_{\s,V})$, which holds if and only if  $\langle m,n \rangle\geq 0$
for the primitive generator $n$ of each ray in the fan $\Sigma_{\s,V}$. For 
$\cA$-type cluster varieties
a global monomial is the same as a cluster monomial, i.e., a monomial in the 
variables of
a single cluster, where the non-frozen variables have non-negative exponent.
\end{lemma}

\begin{proof} We have a surjection $Y\rightarrow \TV(\Sigma_{\s,V})$ by
construction of $Y$, and thus a monomial $z^m$ is regular on 
$\TV(\Sigma_{\s,V})$
if and only if its pull-back to $Y$ is regular. Certainly such a function
is also regular on $Y\setminus D$. Conversely, suppose $z^m$ is not regular
on $\TV(\Sigma_{\s,V})$. Then it has a pole on some toric boundary divisor
$\oD_i$. Now the support of $Z_i\subset \oD_i$ is given by 
$1+z^{v_i}=0$ (resp.\  $1+z^{e_i}=0$) in the $\cA$ (resp.\ $\cX$) case,
as explained in \cite{P1}, \S3.2. In particular for $i\in
I_{\uf}$, $Z_i$ is non-empty, in the $\cA$ case because of the assumption
that $v_i\not=0$ for $i\in I_{\uf}$ stated in Appendix \ref{LDsec}. 
As $z^m$ has no zeros on the big torus, the divisor of zeros of
$z^m$ will not contain the center $Z_i\subset \oD_i$. 
It follows that $z^m$ has a pole
along the exceptional divisor $E_i$ over $Z_i$. Since
$E_i \cap (Y \setminus D) \neq \emptyset$,  $z^m$ is not regular on 
$Y \setminus D$. Thus we conclude that $z^m$ is regular on $Y\setminus
D$ if and only if $z^m$ is regular on $\TV(\Sigma_{\s,V})$.
Of course, $z^m$ is regular on $\TV(\Sigma_{\s,V})$ if and only
if $\langle m,n\rangle \geq 0$ for all primitive generators $n$ of rays
of $\Sigma_{\s,V}$. 

Now the rational map
$Y \setminus D \dasharrow V_{\s}$ to the upper bound
is an isomorphism outside codimension two,
so the two varieties have the same global functions. In the $\cX$ 
(or $\cA_{\prin}$) case, the inclusion
$V_{\s} \subset V$ is an isomorphism outside codimension two as well. 
This gives
the theorem for $\cX$ or $\cA_{\prin}$, and the forward direction for $\cA_t$. 
The reverse direction
for $\cA_t$ follows from the Laurent phenomenon. Indeed, 
the final statement of the lemma simply describes
the monomials regular on $\TV(\Sigma_{\s,\cA})$, and  
a monomial of the given form is the same as a cluster
monomial and these are global regular functions by the Laurent phenomenon. 
\end{proof}

Recall from Proposition \ref{ldpprop}, (4), the canonical maps
$\rho:\shA^{\vee}_{\prin}\rightarrow \shA^{\vee}$ and $\xi:\shX^{\vee}
\rightarrow \shA^{\vee}_{\prin}$ whose tropicalizations are
\[
\rho^T:(m,n)\mapsto m\quad\quad\quad \xi^T: n\mapsto (-p^*(n),-n).
\]
Note $\rho^T$ identifies $\shA^{\vee}(\ZZ^T)$ with
the quotient of $\shA_{\prin}^{\vee}(\ZZ^T)$ by the natural $N$-action.
Since $\xi$ identifies $\shX^{\vee}$ with the fibre over $e$ of
$w:\shA_{\prin}^{\vee}\rightarrow T_{M^{\circ}}$, $\xi^T$
identifies $\shX^{\vee}(\ZZ^T)$ with $w^{-1}(0)$,
where $w:\shA^{\vee}_{\prin}(\ZZ^T)\rightarrow M^{\circ}$ is the weight
map given by $w(m,n)=m-p^*(n)$.

\begin{definition} 
\label{fgccthg1}
Let $V=\bigcup_{\s} T_{L,\s}$ be a cluster variety.
Define $\shC_{\s}^+(\ZZ) \subset V^{\mch}(\bZ^T)$ to be the set of 
$g$-vectors (see Definition \ref{fgccthg0}) for global monomials 
which are characters on the seed torus $T_{L,\s} \subset V$, and 
$\Delta^+_V(\ZZ) \subset V^{\mch}(\bZ^T)$ to be the union of all 
$\shC_{\s}^+(\ZZ)$. 
\end{definition}

\begin{lemma}
\label{fgccthg2}
\begin{enumerate}
\item
For $V$ of $\cA$-type 
$\shC_{\s}^+(\ZZ)$ is the set of integral points of the
cone $\shC_{\s}^+$ in the Fock-Goncharov cluster complex corresponding 
to the seed $\s$.
\item
In any case $\shC^+_{\s}(\ZZ)$ is the set of integral points of a 
rational convex cone $\shC^+_{\s}$, and the relative interiors of
$\shC^+_{\s}$ 
as $\s$ varies are disjoint. The $g$-vector $\g(f) \in V^{\mch}(\bZ^T)$ 
depends only on the function $f$ (i.e., 
if $f$ restricts to a character on two different seed tori, the $g$-vectors
they determine are the same). 
\item
For $m \in w^{-1}(0) \cap \Delta^+_{\cA_{\prin}}(\bZ)$, 
the global monomial $\vartheta_m$ on
$\cA_{\prin}$ is invariant under the $T_{N^{\circ}}$ action
and thus gives a global function
on $\cX = \cA_{\prin}/T_{N^{\circ}}$. 
This is a global monomial and all global monomials on $\cX$ occur
this way, and $m = {\bf g}(\vartheta_m)$.
\end{enumerate}
\end{lemma}

\begin{proof} (1) In the $\cA$ case,
$\shC^+_{\s}$ is the Fock-Goncharov cone by Lemma \ref{fgcclem} and
Lemma \ref{gmlemma}.
These cones form a fan by Theorem \ref{fgccth}, and the fan statement implies
that $\g(f)$ depends only on $f$. 

The $\cA$ case of (2) immediately follows also from the discussion in 
\S\ref{cgvectorsec}.
The $\cX$ case follows from the $\cA$-case (applied to $\cA_{\prin}$)
by recalling that there is a map $\tilde p:\cA_{\prin}\rightarrow
\cX$ making $\cA_{\prin}$ into $T_{N^{\circ}}$-torsor over $\cX$, see
Proposition \ref{ldpprop}, (2). 
This map is defined on monomials by $\tilde p^*(z^n) =z^{(p^*(n),n)}$.
Pulling back a  monomial for $\cX$ under $\tilde p$ gives a 
$T_{N^{\circ}}$-invariant global monomial for $\cA_{\prin}$.
Thus there is an inclusion $\Delta^+_{\cX}(\ZZ)\subseteq w^{-1}(0)
\cap \Delta^+_{\cA_{\prin}}(\ZZ)$ by Proposition \ref{thetaeigenfunctions}. 
Conversely, if $m\in 
w^{-1}(0)$ and $m=\g(f)$ for a global monomial $f$ on $\shA_{\prin}$,
then there is some seed $\s=(e_1,\ldots,e_n)$ 
where $f$ is represented by a monomial
$z^m$ on $T_{\tN^{\circ},\s}$. Because $m\in \tM^{\circ}_{\s}$ 
lies in $w^{-1}(0)$ it is of the form $m=(p^*(n),n)$ for some $n\in N$.
By Lemma \ref{gmlemma}, $m$ is non-negative on the rays 
$\RR_{\ge 0}(e_i,0)$ of $\Sigma_{\s,\shA_{\prin}}$, 
hence $n$ is non-negative on the rays $-\RR_{\ge 0}v_i$ of $\Sigma_{\s,\shX}$.
Hence $z^n$ defines a global monomial on $\shX$. Thus
$\Delta^+_{\cX}(\ZZ)=w^{-1}(0) \cap \Delta^+_{\cA_{\prin}}(\ZZ)$.
Furthermore, one then sees that the Fock-Goncharov cones for
$\cA_{\prin}$ yield the cones for $\cX$ by intersecting with $w^{-1}(0)$.
This gives the remaining statements of (2) in the $\cX$ case, as well as
(3).
\end{proof} 

\begin{construction}[Broken lines for $\cX$ and $\cA$]
\label{blXA}
\emph{The $\cX$ case.}
Note that every function $f_{\fod}$ attached to a wall in 
$\foD^{\shA_{\prin}}_{\s}$ is a power series in 
$z^{(p^*(n),n)}$ for some $n$, thus $w$ is zero on
all exponents appearing in these functions.
Thus broken lines with both $I(\gamma)$ and initial 
infinite segment lying in $w^{-1}(0)$ remain in
$w^{-1}(0)$. In particular $b(\gamma) \in w^{-1}(0)$, and all their monomial 
decorations, e.g., $F(\gamma)$, are in $w^{-1}(0)$. 
We define these to be the broken lines in $\cX^{\mch}(\bR^T)$.\footnote{
In fact each $\foD^{\cA_{\prin}}_{\s}$ induces a collection of walls with attached functions, 
$\foD^{\cX}_{\s}$, living in $N_{\bR,\s}$, just by intersecting each wall 
with $w^{-1}(0)$ and taking the 
same scattering function. This is a consistent scattering diagram, 
and we are
getting exactly the broken lines for this diagram. We will not use this diagram,
as we can get whatever we need from $\foD^{\cA_{\prin}}$.} 

\emph{The $\cA$ case.} We define broken lines in $\cA^{\mch}(\bR^T)$ to be 
images 
of broken lines in $\cXrt$ under $\rho^T$
(applying the derivative $D\rho^T$ to the decorating monomials). 
\end{construction}

\begin{definition}
We define
\begin{enumerate}
\item
$\Theta(\cX) := \Theta(\cA_{\prin}) 
\cap w^{-1}(0) \subset \shX^{\vee}(\ZZ^T)=\cA^{\mch}_{\prin}(\bZ^T)
\cap w^{-1}(0)$. 
\item
\[
\midd(\cX) := \midd(\cA_{\prin})^{T_{N^{\circ}}} = 
\bigoplus_{q \in \Theta(\cX)} \kk\vartheta_q,
\]
where the superscript denotes the invariant part under the group action.
\end{enumerate}
\end{definition}

\begin{corollary}
\label{Xcorollary}
Theorem \ref{mainth} holds for $V=\shX$.
\end{corollary}

\begin{proof} This follows immediately from the $\shA_{\prin}$ case
by taking $T_{N^{\circ}}$-invariants.
\end{proof}

Moving on to the $\cA$ case, the following is easily checked:

\begin{definition-lemma} \label{thacasedef} 
\begin{enumerate}
\item
Define 
\[
\Theta(\cA_t) := \rho^T(\Theta(\cA_{\prin})) \subset \cA^{\mch}(\bZ^T).
\] 
Noting that
$\Theta(\cA_{\prin}) \subset \cA_{\prin}^{\mch}(\bZ^T)$ is invariant under
$N$-translation, we have
$\Theta(\cA_{\prin}) = (\rho^T)^{-1}(\Theta(\cA_t))$. Furthermore, any choice
of section $\Sigma: \cA^{\mch}(\bZ^T) \to \cXt$ of $\rho^T$
induces a bijection $\Theta(\cA_{\prin}) \rightarrow \Theta(\cA_t) \times N$. 
\item
Define $\midd(\cA_t) = \midd(\cA_{\prin}) \otimes_{\kk[N]} \kk$, where
$\kk[N] \twoheadrightarrow \kk$ is given by $t \in T_M$. Given a choice
of $\Sigma$, the collection
$\vartheta_m, m \in \Sigma(M^{\circ})$ gives a $\kk[N]$-module
basis for $\midd(\cA_{\prin})$ and thus a $\kk$-vector space basis for
$\midd(\cA_t)$. For $\midd(\cA)$ the basis is independent of the choice
of $\Sigma$, while
for $\midd(\cA_t)$ it is independent up to scaling each basis vector
(i.e., the decomposition of the vector space $\midd(\cA)$ into one 
dimensional subspaces is canonical). 
\end{enumerate}
\qed
\end{definition-lemma} 

The variety $\cA_t$ is a space $\cA_t := \bigcup_{\s} T_{N^{\circ},\s}$ with the
tori glued by birational maps which vary with $t$. It is then
not so clear how to dualize these birational maps to obtain $\cA_t^{\vee}$ as
it is not obvious how to deal with these parameters. However, 
the tropicalisations
of these birational maps are all the same (independent of $t$) and thus the 
tropical sets
$\cA_t^{\vee}(\bZ^T)$ should all be canonically identified with 
$\cA^{\vee}(\bZ^T)$. So we just take:

\begin{definition} 
$\cA_t^{\vee}(\bZ^T) := \cA^{\vee}(\bZ^T)$. 
\end{definition}

\begin{theorem} \label{mainthax} For $V = \cA_t$ the following modified 
statements of Theorem \ref{mainth} hold. 
\begin{enumerate}
\item There is a map
\[
\alpha_{\shA_t}: V^{\mch}(\bZ^T) \times V^{\mch}(\bZ^T) \times V^{\mch}(\bZ^T)
\to \kk\cup\{\infty\},
\]
depending on a choice of a section $\Sigma: \cA^{\mch}(\bZ^T) \to 
\cA_{\prin}^{\mch}(\bZ^T)$. This function is given by the formula
\[
\alpha_{\shA_t}(p,q,r) = \sum_{n \in N} 
\alpha_{\shA_{\prin}}(\Sigma(p),\Sigma(q),\Sigma(r)+n) z^n(t)
\]
if this sum is finite; otherwise we take $\alpha_{\shA_t}(p,q,r)=\infty$.
This sum is finite whenever $p,q,r\in \Theta(\shA_t)$.
\item 
There is a canonically defined subset
$\Theta \subset V^{\mch}(\bZ^T)$ given by $\Theta=\Theta(\cA_t)$
such that the restriction of the structure constants
give the vector subspace $\midd(V) \subset \can(V)$ with basis indexed by
$\Theta$ the structure of an associative commutative $\kk$-algebra.
\item $\Delta^+_V(\ZZ) \subset \Theta$, i.e., 
$\Theta$ contains the $g$-vector of each global monomial.
\item For the lattice structure on $V^{\mch}(\bZ^T)$ determined by any choice of seed, 
$\Theta \subset V^{\mch}(\bZ^T)$ is closed under addition. Furthermore
$\Theta$ is saturated.
\item There is a $\kk$-algebra map 
$\nu: \midd(V) \to \up(V)$ which sends $\vartheta_p$ for 
$p \in \Delta^+_V(\ZZ)$ to a multiple of the corresponding global monomial. 
\item There is no analogue of (6) of Theorem \ref{mainth} because the
coefficients of the $\vartheta_{Q,p}$ will generally not be integers.
\item $\nu$ is injective for very general $t$, and for all $t$ if
the vectors $v_i$, $i \in I_{\uf}$, lie in a strictly convex cone.
When $\nu$ is injective we have canonical inclusions
\[
\ord(V) \subset \midd(V) \subset \up(V). 
\]
\end{enumerate}
Taking $t=e$ gives Theorem \ref{mainth} for the $V=\cA$ case.
\end{theorem}

\begin{proof} 
For (1), note that for  $p,q\in \Theta(\cA_t)$, we have
$\Sigma(p), \Sigma(q)\in \Theta(\cA_{\prin})$ and on $\cA_{\prin}$
\begin{align*}
\vartheta_{\Sigma(p)}\cdot\vartheta_{\Sigma(q)}= {} &
\sum_{r\in \Theta(\cA_{\prin})} \alpha_{\cA_{\prin}}(\Sigma(p),\Sigma(q),
r) \vartheta_r\\
={}&\sum_{r\in \Theta(\cA_t)}\sum_{n\in N} \alpha_{\cA_{\prin}}(\Sigma(p),
\Sigma(q),\Sigma(r)+n) \vartheta_{\Sigma(r)+n}\\
={}&\sum_{r\in \Theta(\cA_t)}\vartheta_{\Sigma(r)}\cdot\left(\sum_{n\in N} 
\alpha_{\cA_{\prin}}(\Sigma(p),\Sigma(q),\Sigma(r)+n)z^n\right)
\end{align*}
using $\vartheta_{\Sigma(r)+n}=\vartheta_{\Sigma(r)}z^n$.
Note that the sums are finite because 
$\Sigma(p),\Sigma(q)\in\Theta_{\cA_{\prin}}$.
Restricting to $\cA_t$ gives the formula of (1).

The remaining statements follow easily from the definitions except for
the injectivity of (7). To see this,
fix a seed $\s$, which gives the
second projection $\pi_{N,\s}: \cA^{\mch}(\bZ^T) = 
(M^\circ \oplus N)_{\s} \to N$. Choose the section $\Sigma$ of $\rho^T$
to be $\Sigma(m)=(m,0)$. 
Note the collection of
$\vartheta_p$, $p \in B:= \Sigma(M^{\circ}) \cap \Theta(\cA_{\prin})$
are a $\kk[N]$-basis for $\midd(\cA_{\prin})$. By the choice of
$\Sigma$, the $\vartheta_p$
restrict to the basis of monomials on the central fibre $T_{N^\circ}$ of 
$\pi: \cA_{\prin,\s} \to \bA^n_{X_1,\dots,X_n}$. It follows that for any 
finite subset $S \subset B$
there is a Zariski open set $0 \in U_S \subset \bA^n$ such that $\vartheta_p$, $p \in S$ restrict
to linearly independent elements of $\up(\cA_t)$, $t \in U_S$. This gives the injectivity of $\nu$ for
very general $t$. 

Now suppose the $v_i := \{e_i,\cdot\}$, $i \in I_{\uf}$ span a strictly convex cone. We can then
pick an $n\in N^{\circ}\setminus\{0\}$ such that $\{n,e_i\}=
-\langle v_i,n\rangle>0$ for all $i$. Now pick $m \in N_{\uf}^{\perp}$ such that
$\langle m,e_f\rangle + \langle p^*(n),e_f\rangle > 0$ for $f \in I \setminus I_{\uf}$, and set
$\tn := (n,p^*(n)) + (0,m) \in \tK^{\circ}$, notation as in (2) of Proposition \ref{ldpprop}.
By construction the second projection $\pi_M(\tn) = m + p^*(n) \in M$ lies in the
interior of the orthant generated by the dual basis $e_i^*$. 
Take the 
one-parameter subgroup $T = \tn\otimes \bG_m \subset T_{\tK^\circ}$. Now, by Proposition \ref{ldpprop},
the map $\shA_{\prin,\s}\rightarrow \AA^n$ is $T_{\tK^{\circ}}$-equivariant,
where the action on $\AA^n$ is given by the map of cocharacters $\pi_M$. 
Thus $T$ has a one-dimensional orbit whose closure contains
$0\in\AA^n$. So $0$ is in the closure of 
the orbit $T \cdot x \subset \bA^n$ for all $x \in T_M \subset \bA^n$. In particular for all
$x$ and all $S$ there is some $t_{S,x}$ with $t_{S,x}\cdot x \in U_S$. Now from the $T_{\tK^\circ}$-equivariance
of the construction, Proposition \ref{thetaeigenfunctions}, the linear independence holds for all $t$. 

Changing $\Sigma$ will change the $\kk[N]$-basis for $\midd(\cA_{\prin})$, multiplying
each $\vartheta_p$ by some character $z^n$, $n \in N$. 
The restrictions to $\midd(\cA_t)$ are
then multiplied by the values $z^n(t)$.

Theorem \ref{mainth} for $V=\cA$ now follows from setting $t=e$, where
$z^n(t)=1$ for all $n$.
\end{proof}

It is natural to wonder:

\begin{question} Does the equality $\midd(\cA_{\prin}) = \up(\cA_{\prin})$ always hold?
\end{question} 

Our guess is no, but we do not know a counterexample.

Certainly $\Theta \neq \cXt$ in general, for this implies $\Theta(\cX)$,
which is defined to be $\Theta \cap w^{-1}(0)$, coincides with 
$\cX^{\mch}(\bZ^T)$, while we know that in general $\cX$ has many 
fewer global functions, see \cite{P1}, \S 7. So we
look for conditions that guarantee $\Theta =\cXt$, and $\midd = \up$.  
We turn to this in the next section.

\begin{example}
\label{examples2}
In the cases of Example \ref{bcexample}, 
the convex hull
of the union of the cones of $\Delta^+$ in $\tM^{\circ}_{\RR}$ is all of 
$\tM^{\circ}_{\RR}$. Indeed, the first three quadrants already are part of
the cluster complex.
It then follows from the fact that $\Theta$ is closed
under addition and is saturated that $\Theta=\tM^{\circ}$. 

In the case of Example \ref{standardexample}, we know that 
\[
\Delta^+(\ZZ)
=\{(m,n) \in \tM^{\circ}\,|\, \langle e_1+e_2+e_3,m\rangle \ge 0\}.
\]
It then follows again from the fact that $\Theta$ is closed under addition
that either $\Theta=\Delta^+(\ZZ)$ or $\Theta=\tM^{\circ}$. We believe,
partly based on calculations in \cite{M13}, \S7.1, that in fact the latter
holds.
\qed
\end{example}

We show the analogue of Proposition \ref{thetaeigenfunctions} for the
$\cA$ variety:

\begin{proposition}
\label{thetaeigenfunctions2}
If $q\in \Theta(\cA)\subset \cA^{\vee}(\ZZ)$, then 
$\vartheta_q\in\up(\cA)$ is an eigenfunction for the natural $T_{K^{\circ}}$
action on $\cA$.
\end{proposition}

\begin{proof}
This is essentially the same as the proof of Proposition 
\ref{thetaeigenfunctions},
noting that the monomials $z^{v_i}=z^{(v_i,e_i)}|_{\cA}$ 
are invariant under the $T_{K^{\circ}}$ action, as $v_i|_{K^{\circ}}=0$
by definition of $K^{\circ}=\ker p_2^*$.
\end{proof}

We end this section by showing that linear independence of cluster monomials
follows easily from our techniques. This was pointed out to us by Gregory
Muller. In the skew-symmetric case, this was proved in \cite{CKLP}.

\begin{theorem}
\label{properlaurentproperty} For any $\cA$ cluster variety,
there are no linear relations between cluster
monomials and theta functions in $\nu(\midd(\cA)) \subset \up(\cA)$. 
More precisely, if there is a linear relation
\[
\sum_{q \in \Theta(\cA)} \alpha_q \vartheta_q = 0
\]
in $\up(\cA)$, then $\alpha_q=0$ for all $q\in \Delta^+(\ZZ)$.
In particular
the cluster monomials in $\ord(\cA_{\prin})$ are linearly independent.
\end{theorem}

\begin{proof}
Suppose given such a relation.
We choose a seed $\s$ and a generic base point $Q \in \shC^{+}_{\s}\in\Delta^+$.
The seed gives an identification $\cA^{\vee}(\ZZ^T)=M^{\circ}$. We first
show that if $q\in \Delta^+(\ZZ)$ with $q\not\in \shC^+_{\s}$, then
$\vartheta_{Q,q}$ satisfies the \emph{proper Laurent property}, i.e., 
every monomial $z^m=z^{\sum a_if_i}$ appearing in $\vartheta_{Q,q}$
has $a_i<0$ for some $i$.

Indeed, fix a section $\Sigma: \cA^{\mch}(\bZ^T) \to \cA_{\prin}^{\vee}(\ZZ^T)$
as in Definition-Lemma \ref{thacasedef}.
As restriction to $\cA \subset \cA_{\prin}$ gives a bijection between the
cluster variables for $\cA_{\prin}$ and
the cluster variables for $\cA$, between
the theta functions $\vartheta_{q}$, $q \in \Im(\Sigma)$ and the theta functions
for $\cA$, and between the corresponding local expressions $\vartheta_{Q,q}$,
it is enough to prove the claim in the $\cA_{\prin}$ case.
This follows immediately from the
definition of broken line. Indeed, if $\gamma$ is a broken line ending at $Q$
and $F(\gamma)=\sum a_i f_i$ with $a_i \ge 0$ for all $i$, then $\gamma$
must be wholly contained in $\shC^+_{\s}$. But the unbounded direction
of $\gamma$ is parallel to $\RR_{\ge 0} m$, so it follows that $q = I(\gamma) \in \shC^{+}_{\s}$.

We then have the relation
\[
\sum_{q \in \Theta(\cA)} \alpha_q \vartheta_{Q,q} = 0 \in \kk[M^{\circ}]
\]
which we rearrange as
\[
\sum_{q \in \shC^{+}_{\s}} \alpha_q \vartheta_{Q,q} =
-\sum_{q \not \in \shC^+_{\s}} \alpha_q \vartheta_{Q,q}.
\]
The collection of $\vartheta_{Q,q}$ for $q \in \shC^+_{\s}$ are exactly
the distinct cluster monomials for the seed $\s$. In particular all of their exponents are
non-negative. Thus both sides of the equation are zero. Since the cluster 
monomials for $\s$ are linearly independent, we conclude $\alpha_q =0$ 
for all $q \in \shC^{+}_{\s}$. Varying $\s$ the result follows.
\end{proof}

\section{Convexity in the tropical space}  \label{ppsec} 

As explained in the introduction, the 
Fock-Goncharov conjecure is in general false, as a cluster variety $V$ has 
in general too few functions. The conjectured
theta functions only exist formally, near infinity, in the sense of 
\S \ref{formalFGcsec}.
The failure of convergence in general manifests itself in the existence of infinitely many
broken lines with a given incoming direction, and fixed basepoint, and non-finiteness of
the multiplication rule (for fixed $p,q$ infinitely many $r$ with $\alpha(p,q,r) \neq 0$).
The remainder of the paper is devoted to the question of finding conditions
on cluster varieties which guarantee the conjecture holds as stated. One can
only expect a theta function basis for $\up(V)$ in cases when $V$ has enough functions; 
more precisely, when $\up(V)$ is finitely generated, and the natural map
$V \to \Spec(\up(V))$ is an open immersion -- note the second condition is automatic by
the Laurent phenomenon when $V$ is $\cA$-type. 
Our main (and simple) idea is that we can replace the assumption of
enough functions by the existence of a bounded {\it convex} polytope in $V^{\vee}(\bR^T)$, cut
out by the tropicalisation of a regular function. However, our notion
of convexity is delicate, as $V(\RR^T)$ a priori only has a piecewise
linear structure. The correct notion is explained in 
\S\ref{convexityconditionssection}.

Polytopes will play
several roles. The existence of bounded convex polytopes implies
various results on convergence of theta functions.
We get finiteness of the multiplication rule, and so an algebra
structure on $\can(V)$, see Proposition \ref{fmprop}. For technical reasons, see Remark 
\ref{mandelremark}, we often have to replace enough global functions by enough 
global monomials (EGM), and can make optimal use only of convex polytopes cut out by the
tropicalisations of global monomials. EGM implies 
$\up(V) \subset \can(V)$, see Proposition \ref{cafin}, and
Proposition \ref{markscor}. Convex polytopes  give partial 
(full in the bounded polytope case) compactifications of $\Spec(\can(V))$, by copying the familiar
toric construction, see \S \ref{candsec}. 
These compactifications then degenerate to toric compactifications
under (the analog) of the canonical degeneration $\cA_{\prin,\s} \to \bA^n$ of $\cA$ to $T_{N^{\circ},\s}$.
We use these degenerations to prove $\Spec(\can(V))$ is log CY, see Theorem
\ref{cfcor}. Convex cones 
in $V^{\vee}(\bR^T)$  are
intimately related with partial minimal models $V \subset \oV$, and potential functions,
see \S \ref{fullFGsection}, and Corollary \ref{spcor}. 
Finally our methods give several sufficient conditions to guarantee the Full
Fock-Goncharov conjecture, see Proposition \ref{egmscprop}
and Corollary \ref{ffgcor}. These  statements are weaker, and more technical, than one
would hope -- the reason is our inability to prove in full generality that EGM (or better,
the existence of a bounded convex polytope) implies
$\Theta(V) = V^{\vee}(\bZ^T)$. We have only optimal control over the subset 
$\Delta^+ \subset \Theta$ (the cluster complex), and the technical statements are various
ways of saying $\Delta^+$ is sufficiently big. 

The first issue is to make sense of the notion of convexity in $V(\bR^T)$.

\subsection{Convexity conditions}
\label{convexityconditionssection}
The following is elementary:

\begin{definition-lemma} \label{mcdef} By a piecewise linear function on a real vector space
$W$ we mean a continuous function $f: W \to \bR$ piecewise linear 
with respect to a finite fan of (not necessarily strictly) convex cones.
For a piecewise linear function $f:W \to \bR$ we say $f$ is \emph{min-convex}
if it satisfies one of the following three equivalent conditions:
\begin{enumerate}
\item There are finitely many linear 
functions $\ell_1,\dots,\ell_r \in W^*$ such that 
$f(x) = \min \{\ell_i(x)\}$ for all $x \in W$.
\item $f(\lambda_1 v_1 + \lambda_2 v_2) \geq  \lambda_1 f(v_1) + 
\lambda_2 f(v_2)$
for all $\lambda_i \in \bR_{\geq 0}$ and $v_i \in W$. 
\item The differential $df$ is decreasing on straight lines. In
other words, for a directed straight line $L$ with tangent
vector $v$, and $x \in L$ general, then 
\[
(df)_{x+rv}(v) \le (df)_{x}(v),
\]
where $r\in\RR_{\ge 0}$ is general
and the subscript denotes the point at which the differential is calculated.
\end{enumerate}
In the case that $W$ is defined over $\QQ$, then in condition (3) we
can restrict to lines of rational slope.
\end{definition-lemma} 

We now define convexity for functions on $V(\RR^T)$ for $V$ a cluster
variety by generalizing
the third condition above, using broken lines instead of straight lines: 

\begin{definition} \label{imcdef} 
\begin{enumerate}
\item A \emph{piecewise linear function} $f:V(\RR^T)\rightarrow\RR$
is a function which is piecewise linear after fixing a seed $\s$ to get an
identification $V(\RR^T)=L_{\RR,\s}$. If the function is piecewise
linear for one seed it is clearly piecewise linear for all seeds.
\item
Let $f: V(\bR^T) \to \bR$ be piecewise linear, and fix a seed $\s$,
to view $f:L_{\RR,\s}\rightarrow\RR$. 
We say $f$ is \emph{min-convex for $V$} (or just \emph{min-convex} if $V$
is clear from context) if for any broken line for $V$ 
in $L_{\RR,\s}$, $df$ is
increasing on exponents of the decoration monomials 
(and thus decreasing on their negatives, which are the velocity
vectors of the underlying directed path). 
We note that this notion is independent of mutation, by the invariance of
broken lines, Proposition \ref{brlineinvmut}, and thus an intrinsic
property of a piecewise linear function on $V(\bR^T)$. 
\end{enumerate}
\end{definition}

We have a closely related notion, instead defined using
the structure constants for multiplication of theta functions. 

\begin{definition} \label{decdef} 
We say that a piecewise linear 
$f:V(\bR^T) \to \bR$
is \emph{decreasing} if for $p_1,p_2,r \in V(\bR^T)$, 
with $\alpha(p_1,p_2,r) \neq 0$,
$f(r)\geq  f(p_1) + f(p_2)$. Here $\alpha(p,q,r)$ are the structure constants
of Theorem \ref{mainth}.
\end{definition}

These two notions are not quite equivalent:

\begin{lemma}
\label{diplem1}
\begin{enumerate}
\item
If $f:V(\RR^T)\rightarrow\RR$ is min-convex, then $f$ is decreasing. 
\item
If $f\colon V(\bR^T) \to \bR$  is
decreasing, then for any seed
$\s$, we have $f\colon L_{\bR,\s} = V(\bR^T) \to \bR$ min-convex in 
the sense of Definition-Lemma \ref{mcdef}.
\end{enumerate}
\end{lemma}

\begin{proof}
(1) Let $\gamma_1$, $\gamma_2$ be broken lines. Assume $f$ is min-convex
and that $z$ very close to $r$ is the endpoint of each broken line,
with $F(\gamma_1) + F(\gamma_2) = r$. Then 
\begin{align*}
f(r) = (df)_z(r) = {} & (df)_z(F(\gamma_1)) + (df)_z(F(\gamma_2)) \\
                 \geq {} & (df)_{\gamma_1(t)}(I(\gamma_1)) + 
                 (df)_{\gamma_2(t)}(I(\gamma_2)) \\
                 = {} & f(I(\gamma_1)) + f(I(\gamma_2)),
\end{align*}
where $t\ll 0$.
Thus $f$ is decreasing.

(2) Suppose $f$ is decreasing. For any $a,b \in \bZ_{> 0}$, and the linear structure on
$V(\bR^T) = L_{\bR,\s}$ determined by any choice of seed $\s$, the
contribution of straight lines in Definition-Lemma \ref{structuredef}
(and item (1) of Theorem \ref{mainthax} in the $\cA$ case) shows 
$\alpha(a\cdot p, b \cdot q, a \cdot p + b \cdot q) \neq 0$ for all $p,q \in V(\bZ^T)$. 
Thus $f(a \cdot p + b \cdot q) \geq a f(p) + b f(q)$ for all positive
integers $a$ and $b$. By rescaling, the same is true for all positive
rational numbers $a$ and $b$ and $p,q\in V(\QQ^T)$. Min-convexity
in the sense of Definition \ref{mcdef} then follows by continuity of $f$.
\end{proof} 

We have a closely related concept, capturing the generalization of the
notion of a convex polytope. For
$\Xi\subseteq V(\RR^T)$ a closed subset, define
the \emph{cone} of $\Xi$
\[
{\bf C}(\Xi)=\overline{\{ (p,r) \,|\, p\in r\Xi, r \in\RR_{\ge 0}\}}
\subseteq V(\RR^T)\times\RR_{\ge 0}.
\]
Note the closure is only necessary if $\Xi$ is not compact, in which case
${\bf C}(\Xi)\cap V\times\{0\})$ is an asymptotic
form of $\Xi$. Denote
\[
d\Xi(\ZZ)= {\bf C}(\Xi)\cap (V(\ZZ^T)\times\{d\}),
\]
which we view as a subset of $V(\ZZ^T)$. Note for $d\not=0$,
$d\Xi(\ZZ)$ agrees with the obvious notion $d\Xi \cap V(\ZZ^T)$.

\begin{definition} We will call a closed subset $\Xi \subset V(\RR^T)$
\emph{positive} if for any non-negative integers $d_1, d_2$, any
$p_1\in d_1\Xi(\ZZ)$, $p_2\in d_2\Xi(\ZZ)$,  
and any $r \in V(\ZZ^T)$ with $\alpha(p_1,p_2,r) 
\neq 0$, we have $r \in (d_1 + d_2) \Xi(\ZZ)$.
\end{definition}

Note that if $\Xi$ is a cone, i.e., invariant under rescaling, then
this definition agrees with Definition \ref{icdef}.

For a piecewise linear function $f: V(\RR^T) \to \bR$, let
\begin{equation}
\label{Xifdef}
\Xi_f := \{x \in \cXrt\,|\, f(x) \geq -1 \}.
\end{equation}

By Lemma \ref{diplem1}, if $f$ is min-convex in the sense of Definition 
\ref{imcdef}  (or more generally, decreasing in the
sense of Definition \ref{decdef}), then under any identification
$V(\RR^T) = L_{\bR,\s}$ given by any seed, $\Xi_f \subset L_{\bR,\s}$
is a convex polytope. 

\begin{lemma} \label{sminlem} 
If a piecewise linear function $f:V(\RR^T)\rightarrow\RR$ is decreasing, then
$\Xi_f$ is positive. Furthemore, $\Xi_f$ is compact if and only if
$f:V(\RR^T)\rightarrow\RR$ is strictly negative away from $0$.
\end{lemma}

\begin{proof} 
Note that $d\Xi_f(\ZZ)=\{ p \in V(\ZZ^T)\,|\, f(p)\ge -d\}$. Thus
if $f$ is decreasing and $p_i\in d_i\Xi_f\cap V(\ZZ^T)$ with $\alpha(p_1,
p_2,r)\not=0$, then $f(p_i)\ge -d_i$ and thus $f(r)\ge f(p_1)+f(p_2)
\ge -d_1-d_2$, so $r\in (d_1+d_2)\Xi_f$. The second statement is obvious.
\end{proof}

We collect here a couple of results comparing these convexity conditions
on $\cA^{\vee}_{\prin}$ and $\cA^{\vee}$.
Recall from Proposition \ref{ldpprop}, (4), the natural map 
$\rho:\shA^{\vee}_{\prin}\rightarrow\shA^{\vee}$ with $\rho^T:\cApdr\to\cAdr$
being the canonical projection $\tM^{\circ} \to M^\circ$, 
the quotient by the $N$ translation action. 

\begin{lemma} \label{mcxcase} A piecewise linear function $f$ on 
$\cA^{\vee}(\bR^T)$ is min-convex
(resp.\ decreasing) if and only if $f\circ \rho^T$ 
on $\cA_{\prin}^{\vee}(\bR^T)$ 
is min-convex (resp.\ decreasing).
\end{lemma} 

\begin{proof} Broken lines in $\shA^{\vee}(\RR^T)$ 
are by definition images of broken lines on 
$\cA_{\prin}^{\vee}(\bR^T)$ under $\rho^T$,
which gives the min-convex statement. The decreasing statement follows
from the formula for the structure constants for $\shA$
of Theorem \ref{mainthax}, (1). 
\end{proof} 

To understand the relationship between positive polytopes 
in $\cA^{\vee}_{\prin}(\RR^T)$ and $\cA^{\vee}(\RR^T)$, we need to understand
how broken lines behave under the canonical $N$-translation on
$\cA^{\vee}_{\prin}(\RR^T)$:

\begin{lemma}\label{eqslem} For $Q \in \cApdr$ general and $n \in N$ 
there are natural bijections between the following sets of broken lines
\begin{align*}
&\{\gamma\,|\, b(\gamma) = Q, I(\gamma) = q, F(\gamma) = s\}, \\
&\{\gamma\,|\, b(\gamma) = Q, I(\gamma) = q + n, F(\gamma) = s +n\}, \\
&\{\gamma\,|\, b(\gamma) = Q+n, I(\gamma) = q+n, F(\gamma) = s + n\}.
\end{align*}
If $\alpha(p,q,r) \neq 0$, then $\alpha(p+n,q,r+n) \neq 0$ and 
$\alpha(p,q+n,r+n) \neq 0$.
\end{lemma}
\begin{proof} The implications for $\alpha$ follow from the equality of the 
sets using Definition-Lemma \ref{structuredef}. To get the bijections
between the sets, 
we first recall that every wall of $\foD^{\cA_{\prin}}$ is invariant under 
the canonical $N$-translation and is contained in a hyperplane $(n,0)^{\perp}$
for some $(n,0)\in \tN^{\circ}$. 
Thus $N$ acts on broken lines, by
translation on the underlying path, keeping the monomial decorations the same.
This gives the bijection between the second and third sets.

For bijection between the first and second sets, we need to translate
the decorations on each straight segment of $\gamma$ by $n$. This will
change the slopes of each line segment. To do this 
precisely, take $\gamma$ in the first set, say with straight decorated segments
$L_1,\dots,L_k$ taken in reverse order, with $L_k$ the infinite segment. 
Suppose the monomial attached to $L_i$ is $c_iz^{m_i}$ with $m_i\in\tM^{\circ}$.
Say the bends are at points
$x_i \in L_{i-1} \cap L_{i}$ along a wall contained
the hyperplane $(n_i,0)^{\perp}$
so that $L_i$ is parameterized (in the reverse direction to that of Definition
\ref{blde})
by $x_i + t m_i$, $0\le t \le t_i$. 
Then we define
\[
x_i' = Q + t_1(m_1 + (0,n)) + t_2(m_2 + (0,n)) + \cdots + t_{i-1}(m_{i-1} 
+ (0,n)).
\]
Observe that $x_i' \in (n_i,0)^{\perp}$. 
Let $L_i'$ be the segment $x_i' + t(m_i + (0,n))$,
$0 \le t \le s_i$, with attached monomial $c_iz^{m_i +(0,n)}$. Then
$L_1',\dots,L_k'$ form the straight pieces of a broken line $\gamma'$ 
in the second set. This gives the desired bijection. 
\end{proof} 

\begin{lemma}
\label{rescalelem}
Let $p,q,r\in \shA^{\vee}_{\prin}(\ZZ^T)$, $k$ a positive integer. 
If $\alpha(p,q,r)\not=0$, then $\alpha(kp,kq,kr)\not=0$.
\end{lemma}

\begin{proof}
This follows immediately from Definition-Lemma \ref{structuredef} and the
argument given in the proof of saturatedness in Theorem \ref{goodalgebra}.
This latter argument shows that if there is a broken line
$\gamma$ with $I(\gamma)=p$, $F(\gamma)=r$, then there is a broken line
$\gamma'$ with $I(\gamma')=kp$, $F(\gamma')=kr$.
\end{proof}

\begin{proposition} \label{tqprop} 
Suppose $\Xi\subseteq \cA_{\prin}^{\vee}(\RR^T)$ is a positive polytope defined over $\QQ$ (i.e., all faces
span rationally defined affine spaces). Then $\Xi+N_{\RR}$ is positive.
\end{proposition}

\begin{proof} 
Suppose $p_i \in d_i(\Xi + N_{\RR})\cap\cA^{\vee}_{\prin}(\ZZ^T)$, 
and $\alpha(p_1,p_2,r) \neq 0$. 
We can always write $p_i = p_i' + n_i$ with
$p'_i \in (d_i \Xi)\cap \shA^{\vee}_{\prin}(\QQ^T)$ and $n_i \in N_{\QQ}$ by
the rationality assumption. Let $k$ be a positive integer
such that $kp_i'$ and $kn_i$ are
all integral for $i=1,2$. Then because $\alpha(p_1,p_2,r)\not=0$,
$\alpha(kp_1,kp_2,kr)\not=0$ by Lemma \ref{rescalelem}, and thus
$\alpha(kp_1',kp_2',k(r-n_1-n_2)) \neq 0$ by Lemma \ref{eqslem}.
As $kp_i'\in kd_i\Xi$, positivity of $\Xi$ implies
$k(r - n_1 - n_2) \in k(d_1 + d_2) \Xi$ and thus $r \in 
(d_1 + d_2)(\Xi + N_{\RR})$. 
\end{proof}

The chief difficulty now lies in constructing min-convex functions or
positive polytopes. We turn to this next.

\subsection{Convexity criteria}

The following would be a powerful tool for construction min-convex
functions on cluster varieties:

\begin{conjecture} \label{dcon} If $0 \neq f$ is a regular function on a log
Calabi-Yau manifold $V$ with maximal
boundary, then $f^{\trop}: V^{\trop}(\bR) \to \bR$ is min-convex. Here $f^{\trop}(v) = v(f)$ for
the valuation $f$. 
\end{conjecture} 

\begin{remark}
\label{mandelremark} 
To make sense of the conjecture one needs a good theory of broken lines, currently constructed
in \cite{GHK11} in dimension two, 
and here for cluster varieties of all dimension. 
In dimension two, the conjecture has been proven by Travis Mandel, 
\cite{MandelThesis}. Also,
it is easy to see that in any case, for each seed $\s$ and 
regular function $f$, that 
$f^{T}: L_{\bR,\s} = V(\bR^T) \to \bR$, see \eqref{ftcd}, 
is min-convex in the sense of Definition \ref{mcdef}. 
Indeed this is the standard (min) tropicalisation of a Laurent polynomial. We hope to eventually
give a direct geometric description of broken lines (without reference to a scattering
diagram), for any log CY, as tropicalisations of some algebraic analog of holomorphic disks. We
expect the conjecture to follow easily from such a description. 
\end{remark} 

In fact, we can prove Conjecture \ref{dcon} for global monomials, which
gives our main tool for constructing min-convex functions 
(our inability to prove
the conjecture in general is the main reason we use the condition
EGM rather than the more natural condition of enough global functions): 

\begin{proposition} \label{convcor} For a global monomial $f$ on $V^{\vee}$, 
the tropicalisation
$f^T$ is min-convex, and in particular, by Lemma \ref{diplem1}, decreasing.
\end{proposition}

\begin{lemma} \label{diplem} 
Let $f:V^{\vee}(\RR^T)\rightarrow \RR$ be a piecewise linear function.
\begin{enumerate}
\item If $V=\shA_{\prin}$ or $\cX$ and for some choice 
of seed $f$ is the minimum of a collection of linear
functionals $\ell_i$, each of which is non-negative on all the 
initial scattering
monomials of $\foD_{\s}^{\cA_{\prin}}$
(resp., for $V = \cA$, the pullbacks $\ell_i\circ \rho^T$ 
are non-negative on the initial
scattering monomials of $\foD_{\s}^{\cA_{\prin}}$) 
then $f$ is min-convex. 
\item For $V=\cA_{\prin}$, if $f$ is linear in a neighborhood of every wall
of $\foD_{\s}^{\cA_{\prin}}$, then
$f$ is min-convex if and only if  each $\ell_i$ is 
non-negative on each of the initial scattering
monomials.
\end{enumerate}
\end{lemma}

\begin{remark} \label{evalrem} Recall from 
\S \ref{blgcss} that for any choice of seed, the scattering monomials
in $\foD^{\cA_{\prin}}_{\s}$ lie in $w^{-1}(0) = N^\circ = \shX^\vee(\bZ^T)$. 
So it makes sense to 
evaluate functions defined only on $\shX^{\vee}(\bR^T)$ on scattering monomials for $\foD^{\cA_{\prin}}$.
\end{remark} 

\begin{proof}[Proof of Lemma \ref{diplem}]
Choose a seed and suppose $f$ is the minimum of the $\ell_i$. 
We consider a broken line $\gamma$, with two consecutive monomial decorations 
$cz^m,c'z^{m'}$. Possibly
refining the linear segments, we can assume $f$ is given by $\ell\in\{\ell_i\}$ 
along the 
first segment, and $\ell'\in \{\ell_i\}$ along the second.
Let $t,t'$ be points in the domain of $\gamma$ in the two segments.
Then
\begin{align*}
(df)_{\gamma(t')}(m') - (df)_{\gamma(t)}(m) = {} & \ell'(m') - \ell(m) \\
               = {} & \ell'(m' -m) - (\ell(m) - \ell'(m)) \\
               & {} \geq \ell'(m' -m).
\end{align*}
The last inequality comes from the fact that $f = \min\{\ell_i\}$ and
$m$ lies on the side of the wall crossed by $\gamma$ where $f=\ell$. Now 
$m' - m$ is some positive multiple of the scattering monomial. This gives (1).
If $f$ is linear near any bend of the broken line, then
$\ell = \ell'$, the inequality is an equality, and the right- (and left-) 
hand side is just $\ell(m') - \ell(m)$,
which gives the equivalence of (2).
\end{proof}

\begin{proof}[Proof of Proposition \ref{convcor}] 
First consider the case $V=\cA_{\prin}$. 
Suppose $f$ is a global monomial which is a character on a chart indexed by 
$\s$. Then by Lemma \ref{gmlemma}, this character is regular on 
$\TV(\Sigma_{\s,\shA^{\vee}_{\prin}})$, i.e., it is a character
whose geometric tropicalisation \eqref{pairings}
has non-negative value on each ray in the fan
$\Sigma_{\s,\shA^{\vee}_{\prin}}$. 
These rays are spanned by
$-(v_i,e_i)$, $i \in I_{\uf}$, the negatives of the initial
scattering monomials for $\foD_{\s}^{\cA_{\prin}}$. Thus,
because of the sign change between
geometric and Fock-Goncharov tropicalisation, see \eqref{ftcd},
$f^T$ is non-negative on the initial scattering monomials.
Thus $f^T$ is min-convex by Lemma \ref{diplem}.
The same argument then applies in the $V=\cX$ case, see Remark \ref{evalrem}.
For the $V=\cA$ case, a global monomial on $\cA^{\vee}$ pulls back to
a global monomial on $\cA_{\prin}^{\vee}$ via the map $\rho:\cA_{\prin}^{\vee}
\rightarrow \cA^{\vee}$, 
and then the result follows from the $V=\cA_{\prin}$ case by
Lemma \ref{mcxcase}.
\end{proof} 

We need a slight refinement of Proposition \ref{convcor} for the
proof of finite generation of the canonical algebra in the next subsection.

\begin{lemma} \label{blblem}
Let $V$ be a cluster variety and
let $p$ be an integral point in the interior of a maximal
dimensional cone $\shC^+_{V^{\vee},\s} \subset V(\bR^T)$ (see Definition
\ref{fgccthg1}).
Then $\vartheta_p^T$ evaluated
on monomial decorations strictly increases at any non-trivial bend of a broken line
in $L^*_{\bR,\s}$.
\end{lemma}

\begin{proof}
It is enough to treat the case 
$V =\cA_{\prin}$, because global monomials on $\cX$ or $\cA$ are given
either by $T_{N^{\circ}}$-invariant global monomials on $\cA_{\prin}$ or by
restriction of global monomials on $\cA_{\prin}$ to $\cA$ respectively.
Furthermore, broken lines in $\cA^{\vee}_{\prin}(\RR^T)$ yield broken lines
in $\cX^{\vee}(\RR^T)$ and $\cA^{\vee}(\RR^T)$.

The integral points of the cluster cone
\[
\shC^+_{\cA_{\prin}^{\vee},\s} \cap \shA_{\prin}(\ZZ^T) \subset 
T_{\tN^\circ,\s}(\bZ^T) = \tN^\circ_{\s}
\]
correspond to characters of $T_{\tM^\circ,\s} \subset \cA_{\prin}^{\mch}$ which extend to
global regular functions on $\cA_{\prin}^{\mch}$.
Just as in the proof of Proposition \ref{convcor},
these are the characters $z^m$ with $m$ non-negative on the rays
spanned by $-(v_i,e_i)$, $i \in I_{\uf}$, the negatives of the
initial scattering monomials for $\foD_{\s}^{\cA_{\prin}}$.
Thus, because of the sign change between
geometric and Fock-Goncharov tropicalisation, see \eqref{ftcd},
\[
\hbox{$p \in \Int(\shC_{\shA_{\prin}^{\vee},\s}^+)\cap
\shA_{\prin}(\bZ)$ if and only if
$\vartheta_p^T((v_i,e_i)) > 0$ for all $i \in I_{\uf}$.}
\]
In this case $\vartheta_p^T$
is strictly increasing on monomial decorations as in the statement.
\end{proof}

We now introduce our key assumption
necessary for proving strong results about theta functions and the algebras
they generate.

\begin{definition}
We say that $V^{\vee}$ has \emph{Enough Global Monomials} if
for any $x \in V^{\vee}(\bZ^T)$, $x\not=0$, there is a global
monomial $\vartheta_p\in H^0(V^{\vee},\cO_{V^{\vee}})$ such that
$\vartheta_p^T(x) < 0$.
\end{definition}

\begin{lemma}\label{egmlem}
Under any of the identifications $V^{\vee}(\bR^T) = L^*_{\bR,\s}$ 
induced by a choice of seed, the set 
\[
\Xi_V:=\bigcap_{p \in \Delta_{V^{\vee}}^+(\bZ) \subset V(\bZ^T)} 
\{x \in V^{\vee}(\bR^T)\,|\, \vartheta_p^T(x) \geq -1 \} 
\]
is a closed convex subset of $V^{\vee}(\RR^T)$.
The following are equivalent:
\begin{enumerate} 
\item $V^{\vee}$ has Enough Global Monomials.
\item $\Xi_V$ is bounded, or equivalently, the intersection
of all sets $\{x\in V^{\vee}(\RR^T)\,|\, \vartheta_p^T(x) \ge 0\}$
for $p\in \Delta_{V^{\vee}}^+(\ZZ)$ equals $\{0\}$.
\item There exists a finite number of points
$p_1,\ldots,p_r \in \Delta^+_{V^{\vee}}(\ZZ)$ such that
\[
\bigcap_{i=1}^r \{x \in V^{\vee}(\RR^T)\,|\, \vartheta_{p_i}^T(x) \ge -1\}
\]
is bounded,
or equivalently, the intersection
of all sets $\{x\in V^{\vee}(\RR^T)\,|\, \vartheta_{p_i}^T(x) \ge 0\}$
for $1\le i \le r$ equals $\{0\}$.
\item There is function $g \in \ord(V^{\vee})$ whose associated polytope 
$\{x \in V^{\vee}(\RR^T)\,|\, g^T(x) \geq -1 \}$ is bounded. 
\end{enumerate}
\end{lemma}

\begin{proof}
By Remark \ref{mandelremark}, $\Xi_V$ is the intersection of 
closed rational convex polygons (with respect to any seed), 
and hence is a closed convex set. 

The equivalence of (1) and (2) is immediate from the definitions, while
(3) clearly implies (2). For the converse,
let $S$ be a sphere in
$V^{\vee}(\RR^T)=L^*_{\RR,\s}$ centered at the origin. For each $x\in S$
there is a global monomial $\vartheta_{p_x}$ such that
$\vartheta_{p_x}^T(x)<0$,
and thus there is an open subset $U_x\subset S$ on which $\vartheta_{p_x}^T$
is negative. The $\{U_x\}$ form a cover of $S$, and hence by compactness
there is a finite subcover $\{U_{x_i}\}$. Taking $p_i=p_{x_i}$ gives the
desired collection of $p_i$. 

Finally we show the equivalence of (3) and (4). 
The $\vartheta_p$, $p \in \Delta^+_{V^{\vee}}(\bZ)$
are exactly the global monomials on $V^{\vee}$, thus generators of $\ord(V^{\vee}$).  
Now for any finite collection of functions $g_i$, $(\sum g_i)^T \geq 
\min_i g_i^T$
and for the $g_i$ positive universal Laurent polynomials, for 
example for global monomials, we have equality. Thus given (3), we take
$g=\sum_i \vartheta_{p_i}$. Conversely, an element $g$ of
$\ord(V^{\vee})$ is a linear combination of some collection of 
$\vartheta_{p_i}$. Then $\bigcap_i\{x\in V^{\vee}(\RR^T)\,|\,\vartheta^T_{p_i}
\ge -1\}$ is contained in $\{x\in V^{\vee}(\RR^T)\,|\,g^T(x) \ge -1\}$, so
if the latter is bounded, so is the former.
\end{proof} 

We note that the property of EGM is preserved by Fock-Goncharov duality:

\begin{proposition} \label{ldegmprop} Let $\Gamma$ be fixed data, and $\Gamma^{\vee}$ the
Langlands dual data. We write e.g. $N^{\vee}$ for the corresponding lattice for
the data $\Gamma^{\vee}$ as in Appendix \ref{LDsec}.
For each seed $\s$, the canonical inclusion
\[
M_{\s} = M \subset M^{\circ} = M^{\vee}_{\s^{\vee}}
\]
commutes with the tropicalization of mutations, and induces an isomorphism
\[
\cX_{\Gamma}(\bR^T) = \cX_{\Gamma^{\vee}}(\bR^T).
\]
For $n \in N_{\s}$, the monomial $z^n$ on $T_{M,\s} \subset \cX_{\Gamma}$ is a global monomial
if and only if $z^{D \cdot n}$ on $T_{M^{\vee},\s^{\vee}} \subset \cX_{\Gamma^{\vee}}$ is a global
monomial. Finally,
$\cA^{\vee}_{\prin}$ has EGM if and only if $\cA_{\prin}$ has EGM. 
\end{proposition}

\begin{proof} The statement about tropical spaces is immediate from the
definitions. (Note that a similar statement does not hold at the level of
tori, so there is no isomorphism between $\cX_{\Gamma}$ and 
$\cX_{\Gamma^{\vee}}$.)
The statement about global monomials is immediate from Lemma \ref{gmlemma}. Now
the final statement follows from the definition of EGM, the isomorphism 
$\cA_{\prin} \cong \cX_{\prin}$ of Proposition \ref{ldpprop}, (1), 
and the equality $\cA^{\vee}_{\prin} = \cX_{\Gamma_{\prin}^{\vee}}$ of
Proposition \ref{ldpprop}, (3).
\end{proof} 

\subsection{The canonical algebra}

In \eqref{canVdef} we introduced $\can(V)$ as a $\kk$-vector space. 
In the presence of suitable convex objects on $V^{\vee}(\RR^T)$, we can in 
fact put an algebra structure on $\can(V)$ using the structure constants
given by $\alpha$. If further the EGM condition holds, then $\can(V)$
contains $\up(V)$, and is a finitely generated algebra. This often makes
it easier to work with $\can(V)$, as it is a more geometric object. 

Precisely:

\begin{proposition} \label{fmprop} For $V = \cA_{\prin}$ or $\cX$, 
suppose there is a compact positive polytope $\Xi\subseteq V^{\vee}(\RR^T)$.
Assume further that $\Xi$ is top-dimensional, i.e., 
$\dim\Xi=\dim V$.
Then for $p,q \in V^{\vee}(\bZ^T)$, there are at most finitely many $r$ with
$\alpha(p,q,r) \neq 0$. These give structure constants for an 
associative multiplication on
\[
\can(V) := \bigoplus_{r \in V^{\vee}(\bZ^T)} \kk \cdot \vartheta_r.
\]
If there is a compact positive polytope 
$\Xi \subseteq \cA_{\prin}^{\vee}(\bR^T)$
then the same conclusion holds for the structure constants (which
are all finite) and
multiplication rule of $\can(\cA_t)$, for all $t$. 
\end{proposition}

\begin{proof} For $\cA_{\prin}$ or $\cX$, the structure constants are defined
in terms of broken lines. The finiteness is then immediate from Lemma 
\ref{sminlem}. Indeed, given $p,q \in V(\ZZ^T)$, we have
$p\in d_1\Xi$, $q\in d_2\Xi$ for some $d_1, d_2$, by the fact that
$\Xi$ is top-dimensional, and
thus if $\alpha(p,q,r)\not=0$, then 
$r$ lies in the bounded polytope $(d_1+d_2)\Xi$.
The algebra structure is associative
by (3) of Proposition \ref{thetabasislemma}. The $\cA_t$ case follows from
the $\cA_{\prin}$ case and the definitions of the structure constants and multiplication
rule for $\can(\cA_t)$, see 
Theorem \ref{mainthax}.
\end{proof}

\begin{corollary} \label{cafin} For $V = \cA_{\prin}$ or $\cX$ assume $V^{\mch}$ has 
Enough
Global Monomials. For $V = \cA_t$ assume $\cA_{\prin}^{\vee}$ has Enough 
Global Monomials. Then $\alpha$ defines a $\kk$-algebra structure on
$\can(V)$.
\end{corollary}

\begin{proof} 
The case of $V = \cA_t$ follows from the case of $\cA_{\prin}$ so
we may assume $V$ is either $\cX$ or $\cA_{\prin}$. 
Using the Enough Global Monomials hypothesis and Lemma \ref{egmlem},
we can find a finite collection $p_1,\ldots,p_r \in \Delta^+_{V^{\vee}}(\bZ)$
such that the intersection of the finite collection
of polytopes $\Xi_{\vartheta^T_{p_i}}$ is bounded. But since
$\vartheta^T_{p_i}$ is min-convex by Proposition \ref{convcor}, each
of these polytopes is positive by Lemma \ref{sminlem}. Thus the 
result follows from Proposition \ref{fmprop}.
\end{proof}

Finite generation of $\can(V)$ is a special case of a much more general
result. 

\begin{theorem} \label{fgprop} Let $V=\cA_{\prin}$ or $\cX$,
assume $V^{\vee}$ has Enough
Global Monomials, and
let $\Xi\subseteq V^{\vee}(\RR^T)$ be a positive polytope,
which we assume is rationally defined and not necessarily compact. 
Then
\[
\tilde S_{\Xi} := \bigoplus_{d\ge 0}\bigoplus_{q \in d\Xi(\ZZ)}
\kk \vartheta_q T^d \subset \can(V)[T]
\]
is a finitely generated $\kk$-subalgebra.
\end{theorem}

\begin{proof} Note that $\tilde S:=\tilde S_{\Xi}$ 
is a subalgebra of $\can(V)[T]$ by the definition of positive polytope. 

As in the proof of Corollary \ref{cafin}, we can choose $p_1,
\ldots,p_r$ so that $\bigcap_i \Xi_{\vartheta^T_{p_i}}$ is
a compact positive polytope.
Moreover, because boundedness of the intersection is preserved by 
small perturbation
of the functions, we can assume that each $p_i$
is in the interior of some maximal dimensional cluster cone $\shC^+_{\s_i}$. 
Note the seed $\s_i$
is then uniquely determined  by $p_i$, by Lemma \ref{fgccthg2}.
It follows that $\vartheta_{p_i}^T$ is strictly 
increasing on the monomial decorations at
any non-trivial bend of any broken line in $L^*_{\bR,\s_i}$ 
by Lemma \ref{blblem}. 

For $0 \le j \le r$, we consider
\[
\tilde S_j := \bigoplus_{d\ge 0}
\bigoplus_{\scriptstyle\hbox{$\vartheta_{p_i}^T(q) \geq 0$ 
for $i \leq j$}} \kk \vartheta_qT^d \subset \tilde S.
\]
Here the second sum is over all $q \in d\Xi(\ZZ)$ satisfying the stated
condition.
This is a subalgebra of $\tilde S$ by Proposition \ref{convcor}.
Similarly, we define
\[
\overline{S}_j \subset \tilde S[U]
\]
to be the vector subspace spanned by all 
$\vartheta_q T^d U^s$, $s \geq 0$ where
$\vartheta_{p_i}^T(q) \geq 0$ for $i < j$,
and $\vartheta_{p_j}^T(q) \geq -s$. Then
$\overline{S}_j$ is a graded subalgebra of $\tilde S[U]$ by Proposition 
\ref{convcor}
(graded by $U$-degree) and $\tS_j \subset \oS_j$ is the degree zero part.

The result then follows from the following claim, noting that
$\tS=\tS_0$ and for $j\ge 1$ there is a natural surjection $\oS_j
\twoheadrightarrow \tS_{j-1}$ by sending $U\mapsto 1$, $\vartheta_q\mapsto
\vartheta_q$.

\begin{claim}
$\oS_j$ is a finitely generated $\kk$-algebra.
\end{claim}
\emph{Proof}.
We argue first that $\oS_j/U \cdot \oS_j$ is finitely generated.
We work on $\cA_{\prin}^{\mch}(\bR^T) = L^*_{\bR,\s_j}$, 
so that the multiplication rule is defined using
broken lines for $\foD^{\cA_{\prin}}_{\s_j}$, as described by Proposition
\ref{thetabasislemma}, (3) and Definition-Lemma \ref{structuredef}.
Note $\vartheta_{p_j}^T$ is linear on $L^*_{\bR,\s_j}$. If $\vartheta_q T^dU^s
\in \oS_j$, then
modulo $U$, $\vartheta_q T^d U^s =0$ unless
$s = -\vartheta_{p_j}^T(q)$, for otherwise $\vartheta_q T^dU^{s-1} \in \oS_j$.
By Lemma \ref{blblem}, $\vartheta_{p_j}^T$ is strictly increasing on
monomial decorations at any non-trivial bend
of a broken line and
thus the only broken lines that will contribute (modulo $U$)
to $\vartheta_{q_1}T^{d_1} U^{-\vartheta_{p_j}^T(q_1)} \cdot \vartheta_{q_2} 
T^{d_2} U^{-\vartheta_{p_j}^T(q_2)}$ are straight, thus modulo $U$,
\[
\vartheta_{q_1} T^{d_1} U^{-\vartheta_{p_j}^T(q_1)} \cdot
\vartheta_{q_2} T^{d_2} U^{-\vartheta_{p_j}^T(q_2)} =
\vartheta_{q_1+q_2} T^{d_1+d_2} U^{-\vartheta_{p_j}^T(q_1+q_2)}
\]
(addition here in $L^*_{\s_j}$).
Thus $\oS_j/U \cdot \oS_j$ is the monoid ring associated to the rational convex cone
\[
{\bf C}\left(\Xi\cap \bigcap_{ i < j}\{ q \in \cA_{\prin}^{\vee}(\RR^T)\,|\, 
\vartheta^T_{p_i}(q) \geq  0\}\right) \subseteq L^*_{\RR,\s_j}\oplus\RR,
\]
and is thus finitely generated.
Now by decreasing induction on the $U$-degree, to show $\oS_j$ is 
finitely generated,
it is enough to show that its degree $0$ subalgebra, $\tS_j$, is finitely generated.
This is obvious if the set 
\[
\bigcap_{ i \leq j} \{q \in \cA_{\prin}^{\vee}(\RR^T)\,|\, \vartheta_{p_i}^T(q) \geq 0\}
\]
is $\{0\}$, and by assumption this is true for sufficiently large $j$. By the
above, in any case, to prove $\tS_j$ is finitely generated, it is enough to
show $\oS_{j+1}$ is finitely generated. So we are done by induction.
\end{proof}

\begin{corollary}
\label{fgcor}
For $V=\cA_{\prin}$ or $\cX$, suppose that $V^{\vee}$ 
has Enough
Global Monomials. For $V=\cA_t$, assume $\cA_{\prin}^{\vee}$ has Enough
Global Monomials. Then $\can(V)$ is a finitely generated $\kk$-algebra.
\end{corollary}

\begin{proof} 
In the $V=\cA_{\prin}$ or $\cX$ cases, apply the theorem with 
$\Xi=\cA_{\prin}^{\vee}(\RR^T)$ or $\cX_{\prin}^{\vee}(\RR^T)$, 
which are trivially positive. Then $\can(\cA_{\prin})$ is the
degree $0$ part of the finitely generated ring $\tS$ with respect to the
$T$-grading.

Since $\can(\cA_t)$ is a quotient of $\can(\cA_{\prin})$ by
construction of $\alpha_{\cA_t}$ in Theorem \ref{mainthax}, (1),
$\can(\cA_t)$ is also finitely generated. 
\end{proof}

\begin{proposition} \label{markscor} Assume $\cA^{\mch}_{\prin}$ has EGM. 
Then for each universal
Laurent polynomial $g$ on $\cA_{\prin}$, the 
function $\alpha(g)$ of Theorem \ref{ffgth} has finite support
(i.e., $\alpha(g)(q) = 0$ for all but finitely many $q \in \cXt$), and 
$g \mapsto \sum_{q} \alpha(g)(q) \vartheta_q$ gives inclusions of $\kk$-algebras
\[
\ord(\cA_{\prin}) \subset \midd(\cA_{\prin})  \subset \up(\cA_{\prin})
\subset \can(\cA_{\prin})  
 \subset \widehat{\up(\ocA_{\prin}^{\s})} \otimes_{\kk[N^+_{\s}]} \kk[N].
\]
\end{proposition} 

\begin{proof} 
Let $g$ be a universal Laurent polynomial on $\shA_{\prin}$.
By Theorem \ref{ffgth} the sets 
$S_g:=S_{g,\s}$ of Definition \ref{defS} are independent of the seed $\s$. 
We claim that for each global monomial $\vartheta_p$ on $\shA_{\prin}^{\vee}$, 
there is a constant $c_{p}$ such that 
\[
S_g \subset \{x \,|\, \vartheta_p^T(x) \geq c_{p}\} 
\subset \cA^{\mch}_{\prin}(\bR^T).
\]
To see that this is sufficient to prove the proposition, note that by 
Lemma \ref{egmlem}, there
are a finite number
of $p_i$ such that
the intersection of the 
sets where $\vartheta_{p_i}^T(x)\ge 0$
is the origin in $\shA^{\vee}_{\prin}(\RR^T)$. Thus, if the claim is true,
$S_{g}$, the support of $\alpha(g)$, is a finite set. 
The inclusion of algebras follows by Proposition \ref{thetabasislemma}. So 
it's enough to establish the claim.

Let $\vartheta_p$ be a global monomial which is a character on the seed 
torus for $\s$. We follow the notation of Definition \ref{defS}. Thus
$S_g = S_{g,\s} \subset \overline{S}_{g,\s} + P_{\s}$, where 
$\overline{S}_{g,\s}$ itself depends on the seed $\s$ and $g$. 
The tropicalization $\vartheta_p^T$ 
of global monomials $\vartheta_p$ which restrict to a character on
the seed torus $T_{\tM^\circ,\s} \subset \cA^{\mch}_{\prin}$ 
are identified with integer points of the dual cone
$P_{\s}^{\mch}$ (i.e., elements non-negative on each of the initial scattering monomials), see the proof of Proposition \ref{convcor}. 

Note $\vartheta_p^T$ 
is linear on $\tM^\circ_{\RR,\s}$ (though for
our purposes, its min-convexity will be enough).  
Since $\overline{S}_{g,\s}$ is a finite set, for any such $p \in P_{\s}^{\mch}$, there
is constant $c_{p}$ such that 
\[
S_{g} \subset \overline{S}_{g,\s} + P_{\s} \subset \{ x\,|\, \vartheta_p^T(x) \geq c_{p}\}.
\]
This completes the proof. 
\end{proof} 

\subsection{Conditions implying $\cA_{\prin}$ has EGM and the full 
Fock-Goncharov conjecture}

We begin by showing some standard conditions in cluster theory,
namely acyclicity of the quiver or existence of a maximal green
sequence both imply a weaker condition which in turn 
implies both the EGM condition and the full Fock-Goncharov conjecture.
This suggests that this weaker condition is perhaps a more natural one
in cluster theory. This point has been explored in \cite{Muller}.

\begin{definition} 
\label{lccdef}
We say a cluster variety
$\cA$ has \emph{large cluster complex}
if for some seed $\s$,
$\Delta^+(\ZZ) \subset \cA^{\mch}(\bR^T) = M^\circ_{\bR,\s}$ 
is not contained in a half-space.
\end{definition}

\begin{proposition} 
\label{egmscprop} 
If $\cA$ has large cluster complex, then $\cA_{prin}$ has EGM,
$\Theta=\cA_{\prin}^{\vee}(\ZZ^T)$, and the full 
Fock-Goncharov conjecture (see Definition \ref{fgconjdef})
holds for $\cA_{\prin}$, $\cX$, very general $\cA_t$ and, if the convexity 
condition (7) of Theorem \ref{mainth} holds, for $\cA$. 
\end{proposition}

\begin{proof} 
Assume EGM fails for $\cA_{\prin}$. Then we have some point $0 \neq x 
\in \cA_{\prin}(\bZ^T)$ and
$\vartheta_{p}^T(x) \geq 0$ for all $p \in \Delta^+(\bZ) 
\subset \cA^{\mch}_{\prin}(\bZ^T)$.
Take any seed $\s$. We can compute $\vartheta_p^T(x)$ 
by using the corresponding positive Laurent polynomial $\vartheta_{Q,p} \in 
\kk[\tM^\circ_\s]$, for $Q$ a point in
the distinguished chamber $\shC^+_{\s}$ of $\foD^{\cA_{\prin}}_{\s}$. 
Thus using Lemma \ref{lplemma} (leaving the canonical isomorphism $r$
out of the notation),
\[
0 \leq \vartheta_{Q,p}^T(x) = \min_{\substack{ I(\gamma) = p \\
b(\gamma) = Q}} \langle F(\gamma),-x\rangle  \leq \langle p,-x\rangle.
\]
Here the minimum is over all broken lines $\gamma$ contributing
to $\vartheta_{Q,p}$ and 
the final inequality comes from the fact that one of the broken lines is the obvious straight line. 
Thus $\Delta^+(\ZZ)$ is contained in the half-space $\{\langle\cdot,-x
\rangle \geq 0\} \subset \tM^\circ_{\s,\bR}$. 
Since $\Delta^+(\ZZ) \subset \cA_{\prin}^{\mch}(\ZZ^T)$ is the inverse image of
$\Delta_{\cA}^+(\ZZ) \subset \cA^{\mch}(\ZZ^T)$ under the map
$\rho^T:\shA_{\prin}^{\vee}(\ZZ^T)\rightarrow\shA^{\vee}(\ZZ^T)$, 
the EGM statement follows. Now $\Theta =\cA^{\mch}_{\prin}(\bZ^T)$ 
since $\Delta^{+}(\ZZ) \subset \Theta$ and $\Theta$ is saturated and
intrinsically closed under addition, see Theorem \ref{goodalgebra}. 
Since $\cA_{\prin}$ satisfies EGM, so does $\cA_{\prin}^{\vee}$ by
Proposition \ref{ldegmprop}, and
the full Fock-Goncharov conjecture
for $\cA_{\prin}$ then follows from Corollary \ref{cafin} and 
Proposition \ref{markscor}. The $\cA_t, \cX$ and $\cA$ cases then follow 
as in the proofs of Corollary \ref{Xcorollary} and Theorem \ref{mainthax}. 
\end{proof}

\begin{lemma} \label{gsprop} Let $\s = (e_i)$ be a seed. 
Suppose for some vertex $w$ of $\foT_{\s}$ the cluster chamber 
$\shC^+_{w\in\s}\subset M^\circ_{\bR, \s} = \cA^{\mch}(\bR^T)$
meets the interior of $\shC^-_{\s}$. Then the following hold:
\begin{enumerate}
\item $\shC^+_{w\in\s}= \shC^-_{\s}$. 
\item $\cA$ has large cluster complex.
\end{enumerate}
\end{lemma}

\begin{proof} Obviously (1) implies (2). (1) follows from the fact that each cluster
chamber is a chamber in $\foD^{\cA_{\prin}}$, and each $e_i^{\perp} \subset 
M^\circ_{\bR}$ is a union of
walls from $\foD^{\cA_{\prin}}_{\s}$, see Theorem \ref{fgccth}. 
\end{proof}

\begin{corollary} \label{mgsrem} 
Consider the following conditions on a skew-symmetric cluster algebra $\cA$:
\begin{enumerate}
\item $\cA$ has an acyclic seed.
\item $\cA$ has a seed with a maximal green sequence. (For the definition,
see \cite{BDP}, Def.\ 1.8.)
\item $\cA$ has large cluster complex.
\end{enumerate}
Then (1) implies (2) implies (3).
\end{corollary} 

\begin{proof} (1) implies (2) is \cite{BDP}, Lemma 1.20. 
For (2) implies (3), let $\s$ be an initial seed and
$\s'$ the seed obtained by mutations in a maximal green sequence. By definition, the
$c$-vectors for $\s'$ have non-positive entries. 
By Lemma \ref{facelem}
the $c$-vectors are the equations for the walls of the cluster  chamber $\shC^+_{\s'}$, thus
the hypothesis of Proposition \ref{gsprop} holds. 
\end{proof}

We make an aside here connecting with work of Reineke \cite{R10} in the acyclic skew-symmetric case. 
In this case $\foD_{\s}=\foD_{\s}^{\cA_{\prin}}$ has a natural interpretation in terms of moduli of quiver
representations. Consider skew-symmetric fixed and initial data with no frozen
variables. Set $N^{\oplus}=\{\sum a_ie_i \,|\, a_i\ge 0\}$, and let
$\widehat{\kk[N^{\oplus}]}$ be the completion of the polynomial ring
$\kk[N^{\oplus}]$ with respect to the maximal monomial ideal. Let $P
\subseteq \widetilde M=M\oplus N$ be a monoid as in \S 1 containing
all $(v_i,e_i)$, so that 
$G$, the pro-nilpotent group of \S\ref{defconstsection} (in the $\cA_{\prin}$ 
case) 
acts by automorphisms of $\widehat{\kk[P]}$
as usual. Note there is an embedding $\widehat{\kk[N^{\oplus}]}
\hookrightarrow \widehat{\kk[P]}$ given by $z^n\mapsto z^{(p^*(n),n)}$.
The action of $G$ on $\widehat{\kk[P]}$ then induces an action on
$\widehat{\kk[N^{\oplus}]}$. Indeed, one checks immediately that an
automorphism (for $d\in N^+$)
\[
z^{(m,n)}\mapsto z^{(m,n)} f(z^{(p^*(d),d)})^{\langle (d,0), (m,n)\rangle}
\]
induces the automorphism on $\widehat{\kk[N^{\oplus}]}$ given by
\[
z^n\mapsto z^n f(z^d)^{-\{d,n\}}.
\]

\begin{proposition}
\label{reinekeprop} 
Suppose given fixed skew-symmetric data with no frozen variables
along with an acyclic seed $\s=(e_1,\ldots,e_n)$. Let $Q$
be the associated quiver.\footnote{We remind the reader that because of
the assumption made in Appendix \ref{LDsec} that $v_i\not=0$ for any
$i\in I_{\uf}$, the quiver $Q$ has no isolated vertex.}
Each $x \in M_{\bR}$ gives a {\it stability} in the sense of
\cite{R10}. Assume there is a unique primitive $d \in N^+_{\s}$ with 
$x \in d^{\perp}$. For each $i \in I$ let
\[
Q^i(z^d) = \sum_{k \geq 0} \chi(\shM^x_{kd,i}(Q)) z^{kd}
\]
where $\shM^x_{d,i}(Q)$ is the framed moduli space (framed by
the vector spaces $V_j$ with $\dim V_j=0$ unless $j=i$, in which
case $\dim V_j=1$) of
semi-stable representations of $Q$ with dimension vector $d$ and $x$-slope zero,
(see \cite{R10},\S 5.1) and $\chi$ denotes topological Euler characteristic.
Let $d = \sum d_i e_i$ for some $d_i\in \NN$. Then
\[
f(z^d):= (Q^i)^{\frac{1}{d_i}}  \text{ for } i \in I, d_i \neq 0, 
\]
depends only on $Q$ and $x$ (i.e., is independent of the vertex $i \in I$).
Furthermore, for arbitrary $y\in N_{\RR}$,
$g_{(x,y)}(\foD_{\s})$ (see Lemma \ref{easyequivlemma}) acts 
on $\widehat{\kk[N^{\oplus}]}$ 
by
\[
z^n \to z^n \cdot f^{-\{d,n\}}
\]
and on 
$\widehat{\kk[P]}$ by
$$
z^{(m,n)} \mapsto z^{(m,n)} \cdot f(z^{(p^*(d),d)})^{\langle d,m \rangle}.
$$ 
\end{proposition}

\begin{proof}
The equality of the $(Q^i)^{\frac{1}{d_i}}$ follows from the argument in the proof of \cite{R10}, Lemma 3.6. 

If $d=e_i$ for some $i$ then one checks easily that $\cM^x_{e_i,i}$ is a 
point and $\cM_{ke_i,j}^x=\emptyset$ for $i\not=j$ or $k>1$. Thus
$f(z^d)=1+z^{e_i}$ and the formula for $g_{x,y}(\foD_\s)$ holds by Remark~\ref{noekwalls}.

Let $G$ be the pro-nilpotent group of \S\ref{defconstsection} (in the $\cA_{\prin}$ case) associated to the completion of the Lie algebra
$$\fog= \bigoplus_{n \in N^+} \fog_n = \bigoplus_{n \in N^+} \kk \cdot z^{(p^*(n),n)}\partial_{(n,0)}.$$
The group $G$ acts faithfully on $\widehat{\kk[P]}$
but need not
act faithfully via restriction on $\widehat{\kk[N^{\oplus}]}$. It turns
out, however, that there is a subgroup $G'\subseteq G$ which does
act faithfully on $\widehat{\kk[N^{\oplus}]}$ and that all automorphisms
attached to walls in $\foD_{\s}$ are elements of $G'$. 
We see this as follows.

Consider the subspace
$$\fog' := \bigoplus_{\substack{n \in N^+ \\ p^*(n) \neq 0}} \fog_n \subset \fog.$$
By the commutator formula (\ref{Liebracketform}) we have $[\fog,\fog] \subset \fog'$, and in particular $\fog'$ is a Lie subalgebra of $\fog$. 
Let $G' \subset G$ denote the associated pro-nilpotent subgroup.
Because of the assumption that $p^*(e_i)\not=0$ for any unfrozen $i$, 
all automorphisms associated to initial walls of $\foD_{\s}$ 
lie in $G'$, and thus by the inductive construction of the
scattering diagram in \S\ref{subsecKSlemalg}, all automorphisms associated
to outgoing walls of $\foD_{\s}$ also lie in $G'$.

Let $G''$ denote the pro-nilpotent group acting faithfully on $\widehat{\kk[N^{\oplus}]}$ associated to the completion of the Lie algebra
$$\fog'' := \bigoplus_{\substack{n \in N^+ \\ p^*(n) \neq 0}} \kk \cdot z^{n}\partial_{p^*(n)}.$$
Then the restriction of the action of $G$ on $\widehat{\kk[P]}$ to $\widehat{\kk[N^{\oplus}]}$ is given by the group homomorphism $G \rightarrow G''$ 
associated to the Lie algebra homomorphism
$$\fog \rightarrow \fog'', \quad z^{(p^*(n),n)}\partial_{(n,0)} \mapsto -z^n\partial_{p^*(n)}.$$
This homomorphism restricts to an isomorphism $G' \rightarrow G''$, and
in particular the restriction of the $G'$ action on $\widehat{\kk[P]}$
to $\widehat{\kk[N^{\oplus}]}$ is faithful.

Assume now that the indices are ordered so that $Q$ has arrows from
the vertex with index $i$ to the vertex with index $j$ only if $i>j$.
We compute $\theta_{+,-}\in G'$, the automorphism associated to a path from
the positive to the negative chamber, in two different ways.

First, there is a sequence of chambers connecting $\shC^+_{\s}$ to
$\shC^-_{\s}$ via the mutations $\mu_n,\mu_{n-1}$, $\ldots,\mu_1$. Indeed,
it is easy to check that the $c$-vectors obtained by mutating 
$\mu_n,\mu_{n-1},\ldots,$ $\mu_i$ are precisely $e_1,\ldots,e_{i-1},-e_i,
\ldots,-e_n$, and the chamber corresponding to this sequence of mutations 
is precisely the dual of the cone generated by the $c$-vectors, see
Lemma \ref{facelem}. Thus in particular, we can find a path $\gamma$
from $\shC^+_{\s}$ to $\shC^-_{\s}$ which only crosses the walls
$e_n^{\perp},\ldots, e_1^{\perp}$ in order. Note that the element of $G'$
attached to the wall $e_j^{\perp}$ acts on $\widehat{\kk[N^{\oplus}]}$
by $z^n\mapsto z^n(1+z^{e_j})^{-\{e_j,n\}}$, which agrees with the automorphism in \cite{R10} written as
$T_{i_j}$ (noting that \cite{R10} uses the opposite sign convention for
the skew form $\{\cdot,\cdot\}$ associated to the quiver). From this
we conclude that $\theta_{+,-}=T_{i_1}\circ\cdots\circ T_{i_n}$, the
left-hand side of the equality of Theorem 2.1 of \cite{R10}. 

On the other hand, choose a stability condition $x$ and
consider the path $\gamma$ from $\shC^+_\s$ to $\shC^-_{\s}$ parameterized
by $\mu$, with $\gamma(\mu)= x-\mu\sum_i e_i^*$, with domain sufficiently
large so the initial and final
endpoints lie in the positive and negative orthants respectively. Then
a dimension vector has $\gamma(\mu)$-slope $0$ if and only if it has $x$-slope
$\mu$. Thus if the description in the statement of the theorem of $g_{x,y}(\foD_{\s})$
is correct, then $\theta_{\gamma,\foD_{\s}}$ coincides with the
right-hand side of the equality of Theorem 2.1 of \cite{R10}. By the uniqueness
of the factorization of $\theta_{+,-}$ from the proof of Theorem 
\ref{maximscorrespondence} and the faithful action of $G'$ on $\widehat{\kk
[N^{\oplus}]}$ shown above, we obtain the result.
\end{proof}

Because non-negativity of Euler characteristics for the quiver moduli
spaces appearing in the above statement is known (\cite{R14}), 
this gives an alternate proof of positivity of the scattering diagram in this 
case.

\begin{remark}
Since the initial version of this paper was released, Bridgeland
\cite{Bridge} developed a Hall algebra version of scattering diagrams
in the context of quiver representation theory, and the above
result follows conceptually from his results.
\end{remark}

\begin{example} \label{finite_type_scattering}
Let $Q$ be a quiver given by an orientation of the Dynkin diagram of a simply-laced finite dimensional simple Lie algebra.
Then the dimension vectors of the indecomposable complex representations of $Q$ are the positive roots of the associated root system $\Delta$ (Gabriel's theorem). Moreover, for each positive root $d$, there is a unique indecomposable representation $V$ with dimension vector $d$, and $\Hom(V,V)=\bC$. See e.g. \cite{BGP73}.

The $\cA$ cluster variety associated to $Q$ is the cluster variety of finite type associated to the root system $\Delta$ \cite{FZ03a}.
Using Proposition~\ref{reinekeprop} we can give an explicit description of the scattering diagram $\foD$ for $\cA_{\prin}$ as follows.

First we observe that a representation of $Q$ that contributes to $\foD$ is a direct sum of copies of an indecomposable representation.
Let $d \in N^+$ be a primitive vector and $x \in M_{\bR}$ be such that $x^{\perp} \cap N = \bZ \cdot d$.
Suppose $W$ is an $x$-semistable representation of $Q$ with dimension vector a multiple of $d$, and consider the decomposition of $W$ into indecomposable representations. 
By $x$-semistability and our assumption $x^{\perp} \cap N = \bZ \cdot d$, each factor must have dimension vector a multiple of $d$.
By Gabriel's theorem, we see that $d$ is a positive root and $W$ is a direct sum of copies of the associated indecomposable representation.
 
We see that the walls of $\foD$ are in bijection with the positive roots of $\Delta$.
Let $d \in \Delta_+$ be a positive root and $V$ the indecomposable representation with dimension vector $d$.
Let $\fod \subset d^{\perp} \subset M_{\bR}$ be the locus of $x \in M_{\bR}$ such that $V$ is  $x$-semistable of $x$-slope zero, that is,
$\langle x,d \rangle = 0$ and $\langle x,d' \rangle \le 0$ for $d'$ the dimension vector of any subrepresentation of $V$.
Then $\fod$ is a rational polyhedral cone in $M_{\bR}$, and is non-empty of real codimension $1$. Indeed, there exists $x \in d^{\perp}$ such that $V$ is $x$-stable by \cite{K94}, Remark~4.5 and \cite{S92}, Theorem~6.1, and this is an open condition on $x \in d^{\perp}$.
Now let $x \in \fod$ be a point such that $x^{\perp} \cap N = \bZ \cdot d$. Then the $x$-semistable representations of $x$-slope zero are the direct sums of copies of $V$.  

Let us now examine the moduli space $\M^x_{kd,i}$. An object in this moduli
space is a direct sum $V^{\oplus k}=\CC^k\otimes V$ of $k$ copies of the 
unique indecomposable
representation of dimension vector $d$, along with the framing, a choice of 
a vector $v=(v_1,\ldots,v_k) \in \CC^k\otimes V_i$. Such an object is stable
if and only if $v$ is not contained in a proper subrepresentation of 
$V^{\oplus k}$ of the form $W\otimes V$ for some subspace $W\subseteq \CC^k$. 
In order for this to be the case, the $v_1,\ldots,v_k$ must be linearly
independent elements of $V_i$, and hence span a $k$-dimensional subspace
of $V_i$. The automorphism group of $V^{\oplus k}$ is $\GL_k$,
which has the effect
of changing the basis of the subspace spanned by $v_1,\ldots,v_k$.
Now it follows easily from the definitions that, for each $k \in \bZ_{\ge 0}$ 
and $i \in I$ such that $d_i \neq 0$, the moduli space $\cM^x_{kd,i}$ of 
$x$-semistable representations with framing at vertex $i$ is isomorphic to the 
Grassmannian $\Gr(k,d_i)$.

So, in the notation of Proposition~\ref{reinekeprop},
$$Q^i(z^d)=\sum_{k \ge 0} \chi(\Gr(k,d_i))z^{kd} = \sum_{k \ge 0} {d_i \choose k} z^{kd} = (1+z^d)^{d_i}$$
and
$$f(z^d)=Q^i(z^d)^{1/d_i}=1+z^d.$$
Thus the wall of $\foD$ associated to $d \in \Delta_+$ is
$$(\fod \times N_{\bR}, 1+z^{(p^*(d),d)}).$$

For example, suppose $Q$ is the quiver with vertices $1,2,3$, and arrows from 1 to 2 and 2 to 3. This is an orientation of the Dynkin diagram $A_3$.
We have the following isomorphism types of indecomposable representations: 
$$1 \rightarrow 0 \rightarrow 0, \quad 0 \rightarrow 1 \rightarrow 0, \quad 0 \rightarrow 0 \rightarrow 1,$$ 
$$1 \stackrel{\sim}{\rightarrow} 1 \rightarrow 0, \quad 0 \rightarrow 1 \stackrel{\sim}{\rightarrow} 1, \quad 1 \stackrel{\sim}{\rightarrow} 1 \stackrel{\sim}{\rightarrow} 1.$$
(Here the numbers denote the dimension of the vector space at the vertex, and the symbol $\sim$ over an arrow indicates that the corresponding linear transformation is an isomorphism.) 
We write $A_i=z^{e_i^*}$ and $X_i=z^{e_i}$.
Then the walls of $\foD$ are
$$(e_1^{\perp}, 1+A_2X_1), \quad (e_2^{\perp},1+A_1^{-1}A_3X_2), \quad (e_3^{\perp},1+A_2^{-1}X_3) $$
$$(\bR e_3^*+\bR_{\ge 0}(e_1^*-e_2^*)+N_{\bR}, 1+A_1^{-1}A_2A_3X_1X_2),$$
$$(\bR e_1^*+\bR_{\ge 0}(e_2^*-e_3^*)+N_{\bR}, 1+A_1^{-1}A_2^{-1}A_3X_2X_3),$$
$$(\bR_{\ge 0}(e_1^*-e_2^*)+\bR_{\ge 0}(e_2^*-e_3^*)+N_{\bR}, 1+A_1^{-1}A_3X_1X_2X_3).$$
For example, the indecomposable representation with dimension vector $(1,1,1)$ has subrepresentations with dimension vectors $(0,1,1)$ and $(0,0,1)$. So the associated wall has support $\fod \subset (e_1+e_2+e_3)^{\perp}$ defined by the inequalities $e_2+e_3 \leq 0$ and $e_3 \leq 0$. This gives the last wall in the list.
\end{example}

\begin{example} \label{acyclic_scattering}
Kac generalized Gabriel's theorem to the case of an arbitrary quiver $Q$ (without edge loops) as follows \cite{K80},\cite{K82}:
Let $\fog$ be the Kac-Moody algebra associated to the underlying graph of $Q$.
Then the dimension vectors of indecomposable complex representations of $Q$ are the positive roots of $\fog$.

The roots $\Delta$ of $\fog$ are divided into real and imaginary roots. 
The real roots are the translates  of the simple roots $e_1,\ldots,e_n$ under the action of the Weyl group.
Let $\chi \colon N \times N \rightarrow \bZ$ be the asymmetric bilinear form defined by 
$$\chi(d,d')=\sum_{i \in I} d_id'_i - \sum_{a \colon i \rightarrow j} d_{i}d'_{j}.$$
Then for representations $V$ and $V'$ of $Q$ with dimension vectors $d$ and $d'$, 
\[
\chi(d,d')=\chi(V,V'):= \dim \Hom(V,V') - \dim \Ext^1(V,V').
\]

Let $d \in \Delta^+$ be a positive root. 
We have $\chi(d,d)=1$ if $d$ is real and $\chi(d,d)\le 0$ if $d$ is imaginary.
The indecomposable representations of dimension vector $d$ depend on $1-\chi(d,d)$ parameters. 
We say $d$ is a \emph{Schur} root if there exists a representation $V$ of $Q$ with dimension vector $d$ such that $\Hom(V,V)=\bC$. 

Assume that $Q$ is acyclic. Let $\foD$ be the scattering diagram for the associated $\cA_{\prin}$ cluster variety.
We study $\foD$ using Proposition~\ref{reinekeprop}.

We show that each wall of $\foD$ is contained in $d^{\perp}$ for $d$ a primitive Schur root.
First, as in Example~\ref{finite_type_scattering}, each wall is contained in $d^{\perp}$ for $d$ a primitive positive root (note that the set of roots is saturated by \cite{K80}, Proposition~1.2).
It remains to show that $d$ is necessarily a Schur root. Otherwise, a representation $V$ with dimension vector $d$ deforms to a decomposable representation $V'$ \cite{K82}, Proposition~1(b). Then, for $x \in M_{\bR}$ such that $x^{\perp} \cap N = \bZ \cdot d$, $V'$ is $x$-unstable and so $V$ is $x$-unstable (as $x$-semistability is an open condition). A representation with dimension vector a multiple of $d$ is $x$-unstable for the same reason, using \cite{S92}, Theorem~3.8. Now by Proposition~\ref{reinekeprop} we see that there does not exist a wall of $\foD$ contained in $d^{\perp}$.
 
For a real Schur root $d$ there is a unique wall contained in $d^{\perp}$ which can be described explicitly as in Example~\ref{finite_type_scattering}. 
We remark that $d$ is a real Schur root iff there is an indecomposable representation $V$ of $Q$ with dimension vector $d$ such that $\Hom(V,V)=\bC$ and $\Ext^1(V,V)=0$. (Moreover, $V$ is uniquely determined by $d$.) Such a representation $V$ is an exceptional object in the category of representations of $Q$ in the sense of \cite{B90}.  

For an imaginary Schur root the associated walls involve contributions from positive dimensional moduli spaces of semistable representations of $Q$.
The case of the imaginary root $d=(1,1)$ for the quiver $Q$ with vertices $1,2$ and two arrows from 1 to 2 is described in \cite{R10}, \S6.1.
(This is the case $b=c=2$ of Example~\ref{bcexample}.)
\end{example}

We next give another condition for the EGM condition to hold. While
this may appear very technical, it is in fact very important for
group-theoretic examples, see Remark \ref{EGMrem}.

\begin{proposition} \label{mcexprop} 
\begin{enumerate}
\item
Let $U = \Spec(A)$ be an affine variety over a field $\kk$, 
and $f_1,\dots,f_n$ generators of $A$ as a $\kk$-algebra.
For each divisorial discrete valuation $v: Q(U)^* \to \bZ$ (where
$Q(U)$ denotes the function field of $U$) which does not have center
on $U$ (or equivalently, for each boundary divisor $E \subset Y\setminus U$ 
in any partial compactification $U \subset Y$), $v(f_i) < 0$ for some $i$. 
\item
Suppose $V$ is a cluster variety, $U= \Spec(\up(V))$ is a smooth affine
variety, and $V \to U$ is an open immersion. Let $f_1,\dots,f_n$ generate 
$\up(V)$ as a $\kk$-algebra. Then
$f = \min(f_1^T,\dots,f_n^T)$ is strictly negative on $V(\bZ^T) \setminus \{0\}$.
\end{enumerate}
\end{proposition}

\begin{proof} 
(1) Let $U \subset V$ be an open immersion with complement an irreducible
divisor $E$. Suppose each $f_i$ is regular along $E$. Then the inclusion 
$H^0(V,\cO_V) \subset H^0(U,\cO_U)$ is an equality. 
Thus the inverse birational map 
$V \dasharrow U$ is regular, which implies $U = V$. Thus (1) follows.

(2) Since the restriction $H^0(U,\cO_U) \to H^0(V,\cO_V)$ to the open
subset $V \subset U$ is an isomorphism, it follows that $U \setminus V \subset U$ has codimension
at least two. Thus $U$ itself is log Calabi-Yau by \cite{P1}, Lemma 1.4, and the
restriction $(\omega_U)|_V$ of the holomorphic volume form is 
a scalar multiple
of $\omega_V$. In addition $V(\bZ^T) = U(\bZ^T)$. Now (2)
follows from (1).
\end{proof}

\begin{proposition} \label{egmprop1} If the canonical map 
\[
p_2^*|_{N^{\circ}}:N^{\circ} \to  N_{\uf}^*,\quad
n \mapsto  \{n, \cdot\}|_{N_{\uf}}
\]
is surjective, then 
\begin{enumerate}
\item $\pi:\cA_{\prin} \to T_M$ is 
isomorphic to $\cA \times T_{M}$.
\item
We can choose $p^*: N \to M^{\circ}$ so that the induced map 
$p^*: N \otimes_{\ZZ} \bQ \to M^{\circ} \otimes_{\ZZ} \bQ$ is 
an isomorphism.
\item The map induced by the choice of $p^*$ in (2),
$p: \cA \to \cX$, is finite. 
\item
If furthermore for each $0 \neq x \in \cA(\bZ^T)$ we can find a cluster 
variable $A$ with $A^T(x) <0$,
then $\cA$ (and $\cA_{\prin}$) has
Enough Global Monomials. This final condition holds if 
$\ord(\cA) = \up(\cA)$ is finitely generated
and $\Spec(\up(\cA))$ is a smooth affine variety. 
\end{enumerate}
\end{proposition}    

\begin{proof} (1) is Lemma \ref{taprinlem}. (3) follows from (2).
So we assume
$p_2^*|_{N^{\circ}}$ is surjective and show we can 
choose $p^*$ to have finite cokernel, or
equivalently, so $p^*$ is injective. We follow the notation of \cite{P1},
\S 2.1. By the assumed
surjectivity, $p^*$ is injective iff the induced map $p^*|_K: K \to N_{\uf}^{\perp} \subset M^\circ$ 
is injective. 
We can replace $p^*$ by $p^* + \alpha$ for any map $\alpha:N \to 
N_{\uf}^{\perp} \subset M^\circ$ which vanishes on $N_{\uf}$, i.e.,
factors through a map $\alpha: N/N_{\uf} \to N_{\uf}^{\perp}$. 
Note by the assumed surjectivity that $K$ and $N_{\uf}^{\perp}$ have the same rank,
and moreover the restriction $p^*|_{N_{\uf}}=p_1^*$ 
(which is unaffected by the addition of $\alpha$) is
injective. In particular $p^*|_{K \cap N_{\uf}}: K \cap N_{\uf} \to N_{\uf}^{\perp}$ is injective. 
Thus we can choose $\beta: K \to N_{\uf}^{\perp}$, vanishing on 
$K \cap N_{\uf}$ (i.e., factoring
through a map $\beta: K/ K \cap N_{\uf} \to N_{\uf}^{\perp}$) so that 
$p^*|_K + \beta: K \to N_{\uf}^{\perp}$ is injective. By viewing the
determinant of $ p^*|_K+m\cdot\beta$ for $m$ an integer as a polynomial
in $m$, we see that $p^*|_K+m\cdot\beta$ is injective for all but a finite
number of $m$. For sufficiently divisible $m$, 
$m \cdot \beta: K/K \cap N_{\uf} \to N_{\uf}^{\perp}$ extends 
to $\alpha: N/N_{\uf} \to N_{\uf}^{\perp}$ under
the natural inclusion $K/K \cap N_{\uf} \subset N/N_{\uf}$. Now 
$p^* + \alpha: N \to M^\circ$ is injective as required. This shows (2).

For (4), when $\cA_{\prin} \to T_M$ is a trivial bundle, it follows that
\[
\cA_{\prin}(\bZ^T) = \cA(\bZ^T) \times M.
\]
So we have Enough Global Monomials so long as we can find cluster variables on 
$\cA$ with the
given condition. The final statement of (4)
follows from Proposition \ref{mcexprop}.
\end{proof} 

\begin{remarks} \label{EGMrem} Every double Bruhat cell is an affine 
variety by \cite{BFZ05}, Prop. 2.8 and smooth by \cite{FZ99}, Theorem 1.1.
The surjectivity condition in the statement
of  Proposition \ref{egmprop1} holds for all double Bruhat cells by 
\cite{BFZ05}, Proposition 2.6
(the Proposition states that the exchange matrix has full rank, but the proof shows
the surjectivity). So by the proposition,
$\cA_{\prin}$ has Enough Global Monomials for double Bruhat cells for which the upper
and ordinary cluster algebras are the same. This holds for the open double
Bruhat cell of $G$ and the $G/N$ ($N \subset G$ maximal unipotent) for $G = \SL_n$ by
\cite{BFZ05}, Remark 2.20, and is announced in \cite{GY13} for all 
double Bruhat cells of all semi-simple $G$.
\end{remarks}

\subsection{Compactifications from positive polytopes}
\label{candsec}

In this subsection, we will use positive polytopes in $\cA_{\prin}^{\vee}
(\RR^T)$ to create partial compactifications of $\Spec(\can(\cA_{\prin}))$
which fibre over an affine space $\AA^n$. The fibre over $0$ will be a toric variety,
and the general fibre is log Calabi-Yau.

Fix seed data for a cluster variety,
$\s = (e_1,\dots,e_n)$ and let $N^{\oplus}_{\s} \subset N$
be the monoid generated by the $e_i$. Similarly, let
$N^{\oplus}_{\s,\RR}\subset N_{\RR}$ be the cone generated by the 
$e_i$.
The choice of seed gives an identification
$\cXt = \tM^\circ_{\s} = M^\circ \oplus N$ and in particular determines a second projection $\pi_N: \cXt \to N$
(which depends on the choice of seed). We have the canonical inclusion
$N \subset \cXt$ given in each seed by $N = 0 \oplus N \subset \tM^\circ_{\s} = \cXt$, and canonical
translation action of $N$ on $\cXt$ making $\can(\cA_{\prin})$ into a $\kk[N]$-module. 

Now assume given a compact, positive, rationally defined top-dimensional 
polytope 
$\Xi\subseteq \cA^{\vee}_{\prin}(\RR^T)$.
We let $S =\can(\cA_{\prin})$. By Proposition \ref{fmprop} 
and the compactness of $\Xi$, $S$ 
is a $\kk[N]$-algebra with $\kk$-algebra structure constants $\alpha(p,q,r)$. 

\begin{lemma} The set $\pi_N^{-1}(N^{\oplus}_{\s,\RR})\subset \cA^{\vee}_{\prin}
(\RR^T)$ is a positive polytope. Denote by $S_{N^{\oplus}_{\s}}$ the 
degree $0$ part of the ring $\tilde S_{\pi_N^{-1}(N^{\oplus}_{\s,\RR})}$ defined in 
Theorem \ref{fgprop}.
Then $S_{N^{\oplus}_{\s}}$ is a finitely generated $\kk[N^{\oplus}_{\s}]$-algebra.
\end{lemma}

\begin{proof} Positivity follows from the fact that 
$\pi_N(m) \in N^{\oplus}_{\s}$ for each
scattering monomial $m$ in $\foD^{\cA_{\prin}}_{\s}$. The finite generation
statement then follows from Theorem \ref{fgprop}.
\end{proof} 

Let $\tilde\Xi:=\Xi+N_{\RR}$, and
\[
\Xi^+ := \tilde\Xi \cap \pi_N^{-1}({N^{\oplus}_{\s,\RR}}).
\]
Then $\tilde\Xi$ is positive by Proposition \ref{tqprop}, and as the
intersection of two positive sets is positive, $\Xi^+$ is positive.
Hence the associated graded rings $\tilde S_{\tilde\Xi}$ and
$\tilde S:=\tilde S_{\Xi^+}$ (graded by $T$) defined via Theorem
\ref{fgprop} are finitely generated. Note that
$S_{N^{\oplus}_{\s}}$ is the set of homogeneous elements of degree $0$ in the
localization $\tilde S_T$. Thus we have an inclusion
$\Spec(S_{N^{\oplus}_{\s}}) \subset \Proj(\tS)$ an open subset, with complement the
zero locus of $T \in H^0(\Proj(\tS),\cO(1))$. The inclusion of
$\kk[N_{\s}^{\oplus}]=\vartheta_0\kk[N_{\s}^{\oplus}]$ in the degree 0 part 
of $\tS$ induces a morphism
$\Proj(\tS) \to \Spec(\kk[N^{\oplus}_{\s}]) = \bA^n_{X_1,\dots,X_n}$. This morphism
is flat, since $\tS$ is a free $\kk[N^{\oplus}_{\s}]$-module.

\begin{theorem} \label{polylem} The central fibre of
\[
(\Spec(S_{N^{\oplus}_{\s}}) \subset \Proj(\tS) ) \to \bA^n
\]
is the polarized toric variety $T_{N^\circ} \subset \bP_{\oXi}$ given
by\footnote{Although $\oXi\subseteq M^{\circ}_{\RR}$ is only a rationally
defined polyhedron rather than a lattice polyhedron,
we can still define $\PP_{\oXi}=\Proj 
\bigoplus_{d=0}^{\infty} \kk^{d\oXi\cap M^{\circ}}$.} the
polyhedron $\oXi=\rho^T(\Xi)$ where $\rho:\cA^{\vee}_{\prin}\rightarrow
\cA^{\vee}$ is the natural map of Proposition \ref{ldpprop}, (4). 
\end{theorem}

\begin{proof} 
This follows from the multiplication rule. Indeed, since all the
scattering monomials project under $\pi_N$ into the interior of 
$N^{\oplus}_{\s}$, $z^{F(\gamma)}$ vanishes modulo 
the maximal ideal of $\kk[N^{\oplus}_{\s}]$ for any broken line $\gamma$ that bends, see e.g.\ 
the proof of Corollary \ref{cafin}. Thus 
\[
\tS \otimes_{\kk[N^{\oplus}_{\s}]} (\kk[N^{\oplus}_{\s}]/(X_1,\dots,X_n)) = 
\bigoplus_{d \geq 0} \bigoplus_{q \in d \cdot \oXi} \kk \cdot \vartheta_q 
\cdot T^d
\]
with multiplication induced by $\vartheta_p \cdot \vartheta_q = \vartheta_{p + q}$
(addition in $M^\circ$). This is the coordinate ring of $\bP_{\oXi}.$
\end{proof}

\begin{example}
Consider the fixed data 
and seed data given in Example \ref{basicscatteringexample}.
The scattering diagram for $\cA_{\prin}$ in this case has three walls,
pulled back from the walls of the scattering diagram for $\cA$ as
given in Example \ref{basicscatteringexample}, with attached functions
$1+A_2X_1$, $1+A_1^{-1}X_2$ and $1+A_1^{-1}A_2X_1X_2$. Here, with 
basis $e_1,e_2$ of $N$ and dual basis $f_1,f_2$ of $M$, we have
$A_i=z^{(f_i,0)}$ and $X_i=z^{(0,e_i)}$.

Take $\oXi\subseteq M^{\circ}_{\RR}$ to be the pentagon with vertices
(with respect to the basis $f_1, f_2$) $(1,0)$, $(0,1)$, $(-1,0)$, $(0,-1)$,
and $(1,-1)$, which we write as $w_1,\ldots,w_5$. Then $\oXi$ pulls back to 
$\tM^{\circ}_{\RR}$ to give a polytope
$\Xi$. It is easy to see that $\Xi$ is a positive polytope. Further,
write $\vartheta_i:=\vartheta_{(w_i,0)}$,
$\vartheta_0=\vartheta_{(0,0)}$. Then it is not difficult to describe
the ring $\tilde S$ determined by $\Xi^+$ as the graded ring generated
in degree $1$ by $\vartheta_0,\ldots,\vartheta_5$, with relations
\begin{align*}
\vartheta_1\cdot\vartheta_3= {} & X_1 \vartheta_2\vartheta_0+ \vartheta_0^2,\\
\vartheta_2\cdot\vartheta_4= {} & X_2 \vartheta_3\vartheta_0+ \vartheta_0^2,\\
\vartheta_3\cdot\vartheta_5= {} & \vartheta_4\vartheta_0+ X_1 \vartheta_0^2,\\
\vartheta_4\cdot\vartheta_1= {} & \vartheta_5\vartheta_0+ X_1 X_2 \vartheta_0^2,\\
\vartheta_5\cdot\vartheta_2= {} & \vartheta_1\vartheta_0+ X_2 \vartheta_0^2.
\end{align*}
These equations define a family of projective varieties in $\PP^5$,
parameterized by $(X_1,X_2)\in \AA^2$. For $X_1X_2\not=0$, we obtain a smooth
del Pezzo surface of degree $5$. The boundary (where $\vartheta_0=0$) is
a cycle of five projective lines. When $X_1=X_2=0$, we obtain a toric surface
with two ordinary double points.
\end{example}

\begin{theorem} 
\label{cfcor} 
Assume that $\cA_{\prin}^{\vee}$ has Enough Global
Monomials, $\Xi$ is given as above, and that 
$\kk$ is an algebraically closed field of characteristic
zero. Let $V$ be one of $\cX,\cA,\cA_t$ or $\cA_{\prin}$. We note  $\can(V)$
has a finitely generated $\kk$-algebra structure by Proposition \ref{cafin}. 
Define $U := \Spec(\can(V))$. 

Define $Y := \Proj(\tS_{\tilde\Xi}) \to T_M$ (constructed above) in case $V = \cA_{\prin}$, 
and
for $V:=\cA_t$, take instead its fibre over $t \in T_M$ (we are not defining
$Y$ in the $V =\cX$ case), so by construction 
we have an open immersion $U \subset Y$. Define $B:= Y \setminus U$.
The following hold:
\begin{enumerate}
\item
In all cases $U$ is a Gorenstein scheme with
trivial dualizing sheaf.
\item
For $V = \cA_{\prin}$, $\cX$, or $\cA_t$ for $t$ general, 
$U$ is a $K$-trivial Gorenstein log canonical variety. 
\item
For $V = \cA_{\prin}$ or $\cA_t$ for $t$ general, or all $\cA_t$ assuming
there exists a seed $(e_1,\dots,e_n)$ and
a strictly convex cone containing all of $v_i := \{e_i,\cdot\}$ for $i \in I_{\uf}$, we have
$U \subset Y$ is a minimal model. In other words,
$Y$ is a (in the $\cA_{\prin}$ case relative to $T_M$) projective normal variety,
$B \subset Y$ is a reduced Weil divisor, $K_Y + B$ is trivial, and 
$(Y,B)$ is log canonical.
\end{enumerate}
\end{theorem}

\begin{proof} First we consider the theorem in the cases $V \neq \cX$.
Note that (3) implies (2) by restriction.

We consider the family $(\Proj(\tS),B) \to \bA^n$ constructed above,
where $B$ is the divisor given by $T=0$ with its reduced structure.
Using Lemma \ref{kollarslemma} below, the condition that on a fibre $Z$,
$U \subset Z \setminus B_Z$ is a minimal model (in the sense
of the statement) is open, and it holds for the central fibre as it is
toric by Theorem \ref{polylem}.
Thus the condition holds for fibres over some non-empty Zariski open subset
$0 \in W \subset \bA^n$. This gives (3) for $\cA_t$ with $t$ general.
The  convexity
condition (on the $v_i$) implies there is a one-parameter subgroup of $T_{N^\circ}$
which pushes a general point of $\bA^n$ to $0$ (see the proof of Theorem
\ref{mainthax}), and now (3) for $\cA_t$ for all $t$
follows by
the $T_{N^\circ}$-equivariance.

Now note given seed data $\Gamma$ the convexity assumption holds for the
seed data 
$\Gamma_{\prin}$. Thus the final paragraph applies with $\cA = \cA_{\Gamma_{\prin}}$
and so in particular $\Spec(\can(\cA_{\prin}))$ is Gorenstein with trivial dualizing
sheaf. The same then holds for the fibres of the flat map
$\Spec(\can(\cA_{\prin})) \to T_M$, which are $U = \Spec(\can(\cA_t))$ (for
arbitrary $t \in T_M$). This gives (1).

Finally we consider the case $V = \cX$. The graded ring construction above
applied with seed data $\Gamma_{\prin}$
gives a degeneration of a compactification of 
$\Spec(\can(\cA_{\prin})) \subset Y$
(which is now a fibre of the family) to a toric compactification of
$T_{\tN^{\circ}}$. The torus
$T_{N^{\circ}}$ acts on the family, trivially on the base,
and the quotient gives an isotrivial degeneration of an analogously defined
compactification of $\Spec(\can(\cX))$ to a toric compactification of $T_M$.
We leave the details of the construction (which is exactly analogous to the 
construction of $\Proj(\tS)$ above)
to the reader. Now exactly the same openness argument applies.
\end{proof}

We learned of the following result, and its proof, from J. Koll\'ar.

\begin{lemma}[Koll\'ar] \label{kollarslemma} Let $\kk$ be an algebraically 
closed field of characteristic zero.
Let $p \colon X \rightarrow S$ be a proper flat morphism of schemes of finite 
type over $\kk$, and $B \subset X$ a closed subscheme which is flat over $S$. 
Let $(X_0,B_0)$ denote the fiber of $(X,B)/S$ over a closed point $0 \in S$.
Assume that $S$ is regular and for $s=0 \in S$ the following hold:
\begin{enumerate}
\item $X_s$ is normal and Cohen--Macaulay.
\item $B_s \subset X_s$ is a reduced divisor.
\item The pair $(X_s,B_s)$ is log canonical.
\item $\omega_{X_s}(B_s) \simeq \cO_{X_s}$.
\item $H^1(X_s,\cO_{X_s})=0.$
\end{enumerate}
Then the natural morphism $\omega_{X/S}(B)|_{X_0} \rightarrow \omega_{X_0}(B_0)$ 
is an isomorphism,
and there exists a Zariski open neighbourhood $0 \in V \subset S$ such that the 
conditions (1-5) hold for all $s \in V$.
In particular, $X_s \setminus B_s$ is a $K$-trivial Gorenstein log canonical variety 
for all $s \in V$.
\end{lemma}
\begin{proof}
We are free to replace $S$ by an open neighbourhood of $0 \in S$ and will do so 
during the proof without further comment.

By assumption $\omega_{X_0}(B_0) \simeq \cO_{X_0}$ and $X_0$ is Cohen--Macaulay. 
So $$\cO_{X_0}(-B_0)=\cHom_{\cO_{X_0}}(\omega_{X_0}(B_0),\omega_{X_0})$$ is Cohen-Macaulay 
by \cite{K13}, Corollary~2.71, p.~82. It follows that $B_0$ is Cohen--Macaulay by \cite{K13}, Corollary~2.63, p.~80.

The base $S$ is regular by assumption, so $0 \in S$ is cut out by a regular sequence. Since $X_0$ and $B_0$ are Cohen-Macaulay, and $(X,B) \rightarrow S$ is proper and flat, we may assume that $X$ and $B$ are Cohen--Macaulay. Now $\cO_X(-B)$ is Cohen--Macaulay by \cite{K13}, Corollary~2.63, and $\omega_X(B)=\cHom_{\cO_X}(\cO_X(-B),\omega_X)$ is Cohen--Macaulay by \cite{K13}, Corollary~2.71.
The relative dualizing sheaf $\omega_{X/S}$ is identified with $\omega_X \otimes (p^*\omega_S)^{\vee}$, so $\omega_{X/S}(B)$ is also Cohen--Macaulay.
It follows that $\omega_{X/S}(B)|_{X_0}$ is Cohen--Macaulay, and so in particular satisfies Serre's condition $S_2$.
The natural map $\omega_{X/S}(B)|_{X_0} \rightarrow \omega_{X_0}(B_0)$ is an isomorphism in codimension $1$ (because $X_0$ is smooth in codimension $1$)
and both sheaves are $S_2$, hence the map is an isomorphism. Now $\omega_{X_0}(B_0) \simeq \cO_{X_0}$ implies that we may assume $\omega_{X/S}(B) \simeq \cO_X$ using $H^1(X_0,\cO_{X_0})=0$.

The conditions (1),(2), and (5) are open conditions on $s \in S$ because $(X,B) \rightarrow S$ is proper and flat.
So we may assume they hold for all $s \in S$.
We established above that $\omega_{X/S}(B)$ is invertible.
It follows that condition (3) is also open on $s \in S$ by \cite{K13}, Corollary~4.10, p.~159,
and that condition (4) is open on $S$ (using (5)).
\end{proof}

\begin{remarks} \label{crems} Note that directly from its definition, with
the multiplication rule counting broken lines, it is
difficult to prove anything about $\can(V)$, e.g., that it is an integral domain, or determine
its dimension. But the convexity, i.e., existence of a convex polytope 
in the intrinsic sense, gives
this very simple degeneration from which we get many properties, 
at least for very general $\cA_t$,
for free. 

There have been many constructions of degenerations of flag varieties and the like to toric varieties,
see \cite{AB} and references therein. We expect these are all instances of Theorem \ref{polylem}. 

Many authors have looked for a nice compactification of the moduli space $\shM$
of (say) rank two vector bundles
with algebraic connection on an algebraic curve $X$. We know of no satisfactory solution. For example,
in \cite{IIS} the case of $X$ the complement of $4$ points in $\bP^1$ is considered, a compactification
is constructed, but the boundary is rather nasty (it lies in $|-K|$, but this anti-canonical divisor
is not reduced). This can be
explained as follows: $\shM$ has a different algebraic structure, the $\SL_2(\bC)$ character
variety, $V$ (as complex manifolds they are the same). Note $\cM$ is covered by affine lines 
(the space of connections on a fixed bundle is an affine space), thus it is not log Calabi-Yau. Rather, it is 
the log version of uniruled, and there is no Mori theoretic reason to expect a natural compactification.
$V$ however is log Calabi-Yau, and then by Mori theory one expects (infinitely many) nice compactifications, the
minimal models, see \cite{P1}, \S 1, for an introduction to these ideas. 
When $X$ has punctures, $V$ is a cluster variety, see \cite{FST} 
and \cite{FG06}. In the case
of $S^2$ with $4$ punctures, $V$ is the universal family of affine cubic surfaces  (the complement
of a triangle of lines on a cubic surface in $\bP^3$). See \cite{GHK11}, 
Example 6.12.
Each affine cubic has an obvious normal crossing minimal
model, the cubic surface. This compactification is an instance of the above, for a natural choice
of polygon $\Xi$. The same procedure will give a minimal model compactification for any
$\SL_2$ character variety (of a punctured Riemann surface) 
by the above simple procedure that has nothing to do with Teichm\"uller theory. 
\end{remarks}

For the remainder of this section we will assume that $\cA^{\vee}_{\prin}$ 
has Enough Global Monomials. By Lemma \ref{egmlem},
there are global monomials
$\vartheta_{p_1},\ldots,\vartheta_{p_n}$ with $p_1,\ldots,p_r\in \cA_{\prin}
(\ZZ^T)$ such that $w:=\min \{\vartheta^T_{p_i}\}$ is min-convex with
\[
\Xi:=\{ x\in \cA^{\vee}_{\prin}(\RR^T)\,|\,w(x)\ge -1\}
\]
being compact. Thus we have seeds $\s_1,\dots,\s_r$ (possibly
repeated) such that $\vartheta_{p_i}$ is a character on $T_{\tM^{\circ},\s_i}$,
so that $\vartheta_{p_i}^T$ is linear after making the identification
$\cA_{\prin}^{\vee}(\RR^T)\cong \tM^{\circ}_{\s_i}$.
Furthermore, as in the proof of Theorem \ref{fgprop},
we can assume $p_i$ is in the interior of the cone $\shC^+_{\s_i}$.
We will now observe
that with these assumptions the irreducible components of the
boundary in the compactification of $\shA_{\prin}$ induced by $\Xi$ are toric.

Note for each $p_i$ there is at least one
seed where $\vartheta_{p_i}^T$ is linear. 
We assume the collection of $p_i$ is minimal
for defining $\Xi$, and thus $\{\vartheta_{p_i}^T = -1\} \cap \Xi$ is a 
union of maximal faces of $\Xi$, a non-empty closed subset of
codimension $1$.

Writing $S=\can(\shA_{\prin})$, let $\tilde S_{\Xi}$ be the graded 
algebra of Theorem \ref{fgprop}, again a finitely generated algebra.
Then $Y = \Proj(\tS_{\Xi}) \supset \Spec(S)$ is a projective variety and
$T=0$ gives a Cartier (but not necessarily reduced) boundary $D \subset Y$.

\begin{theorem} \label{toricbth} In the above situation,
the irreducible components of $D$ are projective toric 
varieties. More
precisely, for each $p_i$ we have a seed $\s_i$ such that $\vartheta_{p_i}$
is a character on $T_{\tN^{\circ},\s_i}$. Then
\[
\{\vartheta^T_{p_i}=-1\} \cap \Xi \subset \tM^{\circ}_{\RR,\s_i}
\]
is a bounded polytope. The associated projective toric variety 
is an irreducible component
of $D$, and all irreducible components of $D$ occur in this way.
\end{theorem}

\begin{proof} 
For each $i$ consider the vector subspace $I_i \subset \tS:=\tS_{\Xi}$ 
with basis $\vartheta_q \cdot T^s$
with $\vartheta^T_{p_i}(q) > -s$ and $\vartheta^T_{p_j}(q) \geq -s$ 
for $j \neq i$. 

Note that $I_i$ is an ideal of $\tS$. Indeed,
the fact that $p_i$ lies
in the interior of its cone of the cluster complex for $\shA_{\prin}^{\vee}$
implies by Lemma \ref{blblem} that $\vartheta^T_{p_i}$ is strictly increasing
at bends on monomial decorations of broken lines.
Now if $\vartheta_p T^s \in I_i$, $\vartheta_q T^w \in \tS$, 
and $\vartheta_r$ appears in $\vartheta_p \cdot \vartheta_q$, then
$\vartheta^T_{p_i}(r) > -s - w$,
and thus $\vartheta_p T^s \cdot \vartheta_q T^w \in I_i$. 

Now the definitions imply $\bigcap_{i} I_i = (T)$.
So it is enough to show that $\Proj(\tS/I_i)$ is the projective toric variety given by
the polytope $\Xi_i:=\{\vartheta^T_{p_i} =-1\} \cap \Xi \subset 
\tM^{\circ}_{\bR,\s_i}$. Now
$\tS/I_i$ has basis $\vartheta_q T^s$, $q \in s \Xi_i$. 
By the multiplication rule, and
the fact again that $\vartheta^T_{p_i}$ is strictly increasing at bends
on monomial decorations of broken lines,
the only broken line that contributes to 
$\vartheta_q T^s \cdot \vartheta_p T^w$ is the straight
broken line, and the multiplication rule on 
\[
\tS/I_i = \bigoplus_{s \geq 0} \kk \cdot (s \Xi_i \cap \tM^{\circ}_{\s_i})
\]
is given by lattice addition, i.e., 
$\Proj(\tS/I_i)$ is the projective toric variety
given by the polytope $\Xi_i$.
\end{proof}

\begin{remark} The result is (at least to us) surprising in that many cluster varieties come
with a natural compactification, where the boundary is not at all toric. For
example, order the columns
of a $k \times n$ matrix and consider the open subset
$\Gr^o(k,n) \subset \Gr(k,n)$ where the $n$ consecutive Pl\"ucker coordinates (the determinant of the
first $k$ columns, columns $2,\dots,k+1$, etc.) are non-zero. This is a cluster variety. Its
boundary in the given compactification $\Gr(k,n)$ is a union of Schubert cells (which are not toric). 
This has EGM by Proposition \ref{egmprop1}. Then generic compactifications given by bounded polytopes $\Xi$ gives an alternative
compactification in which we replace all these Schubert cells by toric varieties. We do not know, e.g.,
how to produce such a compactification by birational geometric operations beginning with $\Gr(k,n)$.
\end{remark} 

\section{Partial compactications and represention-theoretic
results} 
\label{pcrtsec}

\subsection{Partial minimal models} \label{pcsec}

As discussed in the introduction, many basic objects in representation theory, 
e.g., a semi-simple group $G$, are not log Calabi-Yau, 
and we cannot expect that they have a 
canonical basis of regular functions. 
However, in many cases the basic object is a \emph{partial minimal model} of a 
log Calabi-Yau variety, i.e., contains a Zariski open log Calabi-Yau subset whose 
volume form has a 
pole along all components of the complement. For example, the group $G$
will be a partial compactification of an open double Bruhat cell,
and this is a partial minimal model.
We have a canonical basis of functions
on the cluster variety, and from this, we conjecture one can get a
canonical basis on the partial
compactification (the thing we really care about) in the most naive possible 
way, by taking
those elements in the basis of functions for the open set which extend to regular functions
on the compactification. We are only able to prove the conjecture under rather strong
assumptions, see Corollary \ref{spcor}. Happily these conditions hold in many important
examples.

Note that a frozen variable for $\cA$ (or $\cA_{\prin}$) canonically determines
a valuation, a point of $\cA^{\trop}(\bZ)$, namely the boundary divisor where
that variable
becomes zero. See Construction \ref{fvss}.

While we have myriad (and near optimal) sufficient conditions guaranteeing
a canonical basis $\Theta$ for  $\up(\cA)$, we can only prove our conjecture
that $\Theta \cap \up(\oA) \subset \up(\cA)$ is a basis of $\up(\oA)$ under
a much stronger condition (which happily holds in the most important
representation theoretic examples):

\begin{definition} \label{osdef} We say a seed $\s = (e_i)_{i\in I}$ is 
\emph{optimized}
for $n \in \cA(\bZ^T)$
if 
\[
\hbox{$\{e_k,(r \circ i)(n)\} \geq 0$ for all $k\in I_{\uf}$,}
\]
where
\[
r\circ i: \cA(\bZ^T) \mapright{i} \cA(\bZ^t)=\cA^{\trop}(\ZZ) \mapright{r}
N^{\circ}
\]
is the composition of canonical identifications defined in 
\S\ref{tropsec}. If instead $n\in \cA(\bZ^t)=\cA^{\trop}(\ZZ)$, we say
$\s$ is \emph{optimized} for $n$ if $\{e_k, r(n)\}\ge 0$ for all $k\in I_{\uf}$.

We say $\s$ is optimized for a frozen index if it is optimized for the
corresponding point of $\cA^{\trop}(\ZZ)$.
\end{definition}

For the connection between optimal seed and our conjecture on
$\Theta \cap \up(\ocA) \subset \up(\cA)$ see Proposition
\ref{boundaryprop} and Conjecture \ref{nosq}.

\begin{lemma} 
\label{quiveroptimized}
In the skew-symmetric case, a seed is optimized for a frozen
index if and only if
in the quiver for this seed all arrows between unfrozen vertices and the given
frozen vertex point towards the given frozen vertex.
\end{lemma}

\begin{proof} Under the identification $r:\shA^{\trop}(\ZZ)\rightarrow 
N^{\circ}$ (which is just $N$ in the skew-symmetric case), 
the valuation corresponding to the divisor given by the
frozen variable indexed by $i\in I\setminus I_{\uf}$ is simply $e_i$.
Thus the seed is optimized for this frozen variable if $\{e_k,e_i\}\ge 0$
for all $k\in I_{\uf}$; this is the number of arrows from $k$ to $i$ in the
quiver, with sign telling us that they are incoming arrows.
\end{proof}

\begin{lemma} \label{opslem}
\begin{enumerate}
\item
The seed $\s$ is optimized for $n \in  \cA(\bZ^T)$
if and only if the monomial $z^{r(n)}$ on $T_{M^{\circ},\s} 
\subset \cA^{\mch}$
is a global monomial. In this case
\[
n \in \shC^+_{\s}(\ZZ) \subset \Delta_{\cA^{\vee}}^+(\ZZ)  \subset \Theta(\cA^{\mch})
\]
and the global monomial $z^{r(n)}$ is the restriction to 
$T_{M^{\circ},\s} \subset \cA^{\mch}$ of 
$\vartheta_n$. In the $\cA_{\prin}$ case, for $n\in \cA_{\prin}(\ZZ^T)$
primitive,
this holds if and only if each of the
initial scattering monomials $z^{(v_i,e_i)}$ in 
$\foD^{\cA_{\prin}}_{\s}$ is regular along the
boundary divisor of $\cA_{\prin}$ corresponding to $n$ under the identification
$i:\cA_{\prin}(\bZ^T) \to \cA_{\prin}^{\trop}(\bZ)$.
\item
$n \in \cA(\bZ^T)$ has an optimized seed if and only if $n$ 
lies in $\Delta^+_{\cA^{\vee}}(\ZZ)$.
\end{enumerate}
\end{lemma}

\begin{proof} 
For (1), the rays for the fan $\Sigma_{\s}$ giving the
toric model for $\cA^{\vee}$ are $-\bR_{\geq 0} v_k$ for $k\in I_{\uf}$. 
Note that $r(i(n))=-r(n)$, see \eqref{pairings}.
Now the statement concerning $\cA$ follows
from Lemma \ref{gmlemma}\ and Lemma \ref{fgccthg2}.
The additional statement in the $\cA_{\prin}$ case is clear from the
definitions. For (2), one notes that the forward implication is given
by (1), while for the converse, if $n\in \Delta^+_{\cA^{\vee}}(\ZZ)$, then
$n\in \shC^+_{\s}(\ZZ)$ for some seed $\s$, and then $n$ is optimized for
that seed.
\end{proof}


\begin{proposition} \label{osprop} For the standard cluster algebra structure on $\CG(k,n)$
(the affine cone over $\Gr(k,n)$ in its Pl\"ucker embedding) 
every frozen variable has an optimized seed.
\end{proposition}

\begin{proof} As was pointed out to us by Lauren Willams, for
$\Gr(k,n)$, the initial seed in \cite{GSV}, Figure 4.4, is optimized
for one frozen variable (the {\it special} upper right hand vertex for the
initial quiver).
The result follows from the cyclic symmetry of this
cluster structure.\end{proof}
\begin{remark} \label{osrem} B. LeClerc, and independently L. Shen, 
gave us an explicit sequence of mutations that shows
the proposition holds as well for the cluster structure of \cite{BFZ05}, \cite{GLS}  on
the maximal unipotent subgroup $N \subset \SL_{r+1}$, and the same argument applies to 
the Fock-Goncharov cluster structure on $(G/N \times G/N \times G/N)^G$, $G = \SL_{r+1}$. The
argument appears in \cite{TimThesis}.
\end{remark}

\begin{lemma}
\label{linindlemma}
Let $L$ be a lattice and $P \subset L$ a submonoid with
$P^{\times}=0$.
For any subset $S\subseteq L$ and collection of elements $\{Z_q \,|\, q\in S\}$
such that $Z_q\in \kk[q + (P\setminus \{0\})]$, the
subset $\{z^q + Z_q\,|\,q\in S\} \subset \kk[L]$ is linearly independent
over $\kk$.
\end{lemma}

\begin{proof} Suppose
\[
\sum_{q\in S'} \alpha_q(z^q + Z_q) = 0
\]
for $\alpha_q$ all non-zero and $S'\subseteq S$ a finite set.
Let $q'\in S'$ be minimal with respect to the partial ordering on $L$ 
given by $P$ (where $n_1 \le n_2$ means $n_2 = n_1 + p$ for some
$p \in P$). The
coefficient of $z^{q'}$ in the sum, expressed in the basis of monomials, 
must be zero. But the minimality of $q'$ implies
the monomial $z^{q'}$ does not appear in any of the $Z_q$, $q\in S'$.
Thus the coefficient of $z^{q'}$ is just $\alpha_{q'}$, a contradiction.
\end{proof} 

\begin{proposition} \label{boundaryprop}  Suppose a valuation $v \in 
\cA^{\trop}_{\prin}(\bZ)$ has an optimized seed.
If $v(\sum_{q \in \Theta} \alpha_q \vartheta_q) \geq 0$, then $v(\vartheta_q) 
\geq 0$ for all $q$ with $\alpha_q \neq 0$.
\end{proposition}

\begin{proof}
Let $\s = (e_1,\dots,e_n)$ be optimized for $v$.
Let $C$ be the strictly convex cone spanned by the exponents of the initial
scattering monomials $(v_i,e_i) \in \tM^\circ$.
Let $P=C\cap \tM^{\circ}$.
Take $Q$ a basepoint in the positive chamber of $\foD_{\s}$.
By definition
$\vartheta_{Q,q} = z^q + Z_q$ where $Z_q = \sum_{m \in q + P\setminus\{0\}} 
\beta_{m,q} z^m$ is a finite sum of monomials. By (1) of Lemma \ref{opslem} we have
$v(z^m) \geq v(z^q)$, and thus by (2) of Lemma \ref{lplemma}, $v(\vartheta_q) = v(z^q)$.

Let $r$ be the minimum of $v(\vartheta_q)$ over all $q$ with $\alpha_q \neq 0$, and
suppose $r <  0$.
Since $v(\sum \alpha_q \vartheta_q) \geq 0$, necessarily
\[
\sum_{v(z^q)=r} \alpha_q( z^q + \sum_{m: v(z^m)=r} \beta_{m,q} z^m) = 0
\in \kk[\tM^\circ].
\]
Note this is the sum of all the monomial terms in $\sum \alpha_q \vartheta_q$
which have the maximal order of pole, $|r|$, along $v$.
This contradicts Lemma \ref{linindlemma}.
\end{proof}

We believe the assumption of an optimized seed is not necessary:

\begin{conjecture} 
\label{nosq} 
The proposition holds for any $v \in \cA^{\trop}_{\prin}(\bZ)$.
\end{conjecture}

Any finite set $S \subset \cA^{\trop}_{\prin}(\bZ) \setminus \{0\}$ of primitive elements
gives a partial compactification
(defined canonically up to codimension two) $\cA_{\prin} \subset 
\ocA_{\prin}^S$,
with the boundary divisors of this partial compactification in one-to-one
correspondence with the elements of $S$ (this is true for any finite collection, $S$, of divisorial discrete
valuations on the function field of a normal variety $\cA$: there is always an open immersion
$\cA \subset \ocA^S$, with divisorial boundary $\ocA^S \setminus \cA$ corresponding to $S$, and 
$\cA \subset \ocA^S$ is unique up to changes in codimension greater than or equal to two). 


We then define
\[
\Theta(\ocA_{\prin}^S) := \{q \in \Theta(\cA_{\prin})\,|\, 
v(\vartheta_q) \geq 0 \text{ for all } v \in S \} 
\]
and $\midd(\ocA_{\prin}^S) \subset \midd(\cA_{\prin})$ the vector subspace with
basis $\Theta(\ocA_{\prin}^S)$. Similarly we define
$\ord(\ocA_{\prin}^S)$ to be the subalgebra of $\up(\ocA_{\prin}^S)$
generated by those cluster variables that are regular (generically) along all 
$v \in S$.

\begin{definition} \label{spdef} Each choice of seed $\s$ gives a pairing
\[
\langle\cdot,\cdot\rangle_{\s}: \cA_{\prin}(\bZ^T) \times \cA^{\mch}_{\prin}(\bZ^T) \to \bZ
\]
which is just the dual pairing composed with the identifications
\begin{align*}
\cA_{\prin}(\bZ^T) = {} & T_{\tN^\circ,\s}(\bZ^T) \overset r = \tN^\circ_{\s},\\
\cA^{\mch}_{\prin}(\bZ^T) = {} &
T_{\tM^\circ,\s}(\bZ^T) \overset r = \tM^\circ_{\s}.
\end{align*}
\end{definition}

\begin{lemma} \label{midocalem}
\begin{enumerate}
\item
$\midd(\ocA_{\prin}^S) \subset \midd(\cA_{\prin})$ is a subalgebra
containing $\ord(\ocA_{\prin}^S)$. If $\ord(\ocA_{\prin}^S) = \up(\ocA_{\prin}^S)$
then
\[
\ord(\ocA_{\prin}^S) =\midd(\ocA_{\prin}^S) = \up(\ocA_{\prin}^S).
\]
\item
Assume each $v\in S$ has an optimized seed. Then
\[
\midd(\ocA_{\prin}^S) = \midd(\cA_{\prin}) \cap \up(\ocA_{\prin}^S) \subset 
\up(\cA_{\prin}).
\]
If $\midd(\cA_{\prin}) = \up(\cA_{\prin})$ then
$\midd(\ocA_{\prin}^S) = \up(\ocA_{\prin}^S)$.
\item
If each $v\in S$ has an optimized seed and $\s$ is optimized for $v \in S$, 
the piecewise linear function
\[
\vartheta_{i(v)}^T = \langle\cdot,r(v)\rangle_{\s} : (\cXrt = \tM^\circ_{\bR,\s}) \to \bR
\]
is min-convex, and for all $q \in \Theta(\cA_{\prin}) \subset \cA^{\mch}_{\prin}(\bZ^T)$,
\[
\vartheta_q^T(v) = \langle r(q),r(v)\rangle_{\s} = \vartheta_{i(v)}^T(q)
\]
where $\vartheta_{i(v)}$ is the global monomial on $\cA^{\mch}_{\prin}$ corresponding to $i(v)$ (which exists by Lemma \ref{opslem}).
\end{enumerate}
\end{lemma}

\begin{remark} There are pairings 
\[
\Theta(V) \times \Theta(V^{\vee}) \to \bZ
\]
which are much more natural then Definition \ref{spdef}. Indeed,
$v \in \Theta(V)$ gives a canonical function $\vartheta_v \in \up(V)$ and,
since $\Theta(V) \subset V^{\vee}(\bZ^t)$, a valuation on $\up(V^{\vee})$. The 
analogous statements apply to $w \in \Theta(V^{\vee})$. So we 
could define a pairing by either
\[
\hbox{$\langle v,w\rangle  \mapsto w(\vartheta_v)$, or $\langle v,w \rangle \mapsto
v(\vartheta_w)$.}
\]
We conjecture these two pairings are equal. 
Lemma \ref{midocalem}, (3),  gives the result when one of $v,w$ lies in the
cluster complex. 
One can pose the same symmetry conjecture for mirror pairs  of affine log CYs (with maximal
boundary) in general, the two-dimensional case having been shown 
in \cite{MandelThesis}. 
Suppose the symmetry conjecture holds, and furthermore 
$\Theta(\cA_{\prin}) = \cA_{\prin}^{\vee}(\bZ^T)$. Then (see the  
proof of Lemma \ref{midocalem} below) 
the cone of \eqref{spcone} cut out by the tropicalisation of the potential
function is 
\[
\Xi := \{x \in \cA^{\mch}_{\prin}(\bZ^T)\,|\, W^T(x) \geq 0\} = \Theta(\cA_{\prin,S}).
\]
If furthermore Conjecture \ref{nosq} holds and $\midd(\cA_{\prin}) = \up(\cA_{\prin})$,
then $\Xi$ gives a basis of $\up(\ocA_{\prin}^S)$, canonically determined by the open
set $\cA_{\prin} \subset \ocA_{\prin}^S$ (together with its cluster structure, though we
conjecture the basis is independent of the cluster structure). See Corollary \ref{spcor}.
\end{remark}

\begin{proof}[Proof of Lemma \ref{midocalem}]
The subalgebra statement of (1) 
follows from the positivity (both of structure constants and the
Laurent polynomials $\vartheta_{Q,q}$) just as in the proof of
Theorem \ref{goodalgebra}. Every cluster variable is a theta function,
so the inclusion $\ord \subset \midd$ is clear. Now obviously if 
$\ord(\ocA_{\prin}^S) = \up(\ocA_{\prin}^S)$
then both are equal to $\midd$.

The intersection expression of (2) for the middle algebra
follows from Proposition \ref{boundaryprop}. Now obviously if $\midd(\cA_{\prin}) = \up(\cA_{\prin})$
then $\midd(\ocA_{\prin}^S) = \up(\ocA_{\prin}^S)$.

For (3), we work with the scattering diagram 
$\foD_{\s}$. Then $i(v)$ is the $g$-vector of the global
monomial $\vartheta_{i(v)}$, with $(\vartheta_{i(v)})|_{T_{\tM^\circ,\s} \subset \cA^{\mch}_{\prin}} = z^{r(i(v))}$,
by Lemma \ref{opslem}.
Using $r(v)=-r(i(v))$, one sees that 
$\vartheta_{i(v)}^T = \langle\cdot,r(v)\rangle$
is linear on $\tM^\circ$, so obviously min-convex in the sense
of Definition-Lemma \ref{mcdef}. 
Since it is the tropicalisation of a global monomial it is
also min-convex in the sense of Definition \ref{imcdef}, by Proposition
\ref{convcor}.

Now fix a base point $Q \in \shC^+_{\s} \subset \cXrt$,
and consider $\vartheta_{Q,q}$, $q \in \Theta$.  By Lemma
\ref{opslem}, (1), each scattering function is regular along
the boundary divisor corresponding to $v \in \cA_{\prin}^{\trop}(\bZ)$. 
By definition
$\vartheta_{Q,q} = z^{r(q)} + Z_{r(q)}$, where $Z_{r(q)}$ 
is a linear combination of monomials  $z^{r(q)+q'}$
with $z^{q'}$ regular along the boundary divisor corresponding to $v$. Thus
\[
\vartheta_q^T(i(v)) = v(\vartheta_{Q,q}) = \langle r(q),r(v)\rangle
\]
by Lemma \ref{lplemma}. Since $\vartheta_{i(v)}$ is the monomial 
$z^{r(i(v))}$ on
$T_{\tM^\circ,\s}$,
\[
\vartheta_{i(v)}^T(q) 
= -\langle r(q),r(i(v)) \rangle = \langle r(q), r(v) \rangle.
\]
This completes the proof of (3).
\end{proof}

\subsection{Cones cut out by the tropicalized potential}
\label{fullFGsection}

Recall a choice of seed gives a partial compactification $\cA_{\prin} \subset 
\ocA_{\prin}^{\s}$
and a map $\pi:\ocA_{\prin}^{\s} \to \bA^n_{X_1,\ldots,X_n}$. 
The boundary $\ocA_{\prin}^{\s} \setminus \cA_{\prin}$ has $n$
irreducible components, primitive elements of $\cA_{\prin}^{\trop}(\bZ)$, 
the vanishing loci of the $X_i$.

\begin{lemma} 
\label{gsoplem} 
The seed $\s$ is optimized for each of the boundary divisors of 
$\cA_{\prin} \subset \ocA_{\prin}^{\s}$. 
\end{lemma}

\begin{proof} 
If $\s=(e_1,\ldots,e_n)$,
the corresponding seed for $\cA_{\prin}$ is 
\[
\tilde\s=\big((e_1,0),\dots,(e_n,0),(0,f_1),\dots,(0,f_n)\big),
\]
and the boundary divisors correspond to the $(0,f_i)$. But
$
\{(e_i,0),(0,f_j)\} = \langle e_i,f_j\rangle = \delta_{ij}\ge 0,
$
hence the claim.
\end{proof}

We adjust slightly the notation $\ocA_{\prin}^S$ of the previous 
subsection to this case: 
\begin{definition} 
Let
\[
\Theta(\cAg) \subset \Theta \subset \cXt
\]
be the subset of points $q$ such that $\vartheta_q$ remains 
regular on the partial compactification $\ocA_{\prin}^{\s} \supset \cA_{\prin}$, 
i.e., such that 
\[
\vartheta_q \in \up(\cAg) \subset \up(\cAgs).
\]
\end{definition}

\begin{lemma} 
\label{productlem} 
Under the 
identification $\cXt = M^\circ \oplus N$, we have
$\Theta = \Theta(\cA^\mch) \times N$
and $\Theta(\cAg) = \Theta(\cA^{\mch}) \times N^+_{\s}$. 
\end{lemma}

\begin{proof} 
$\Theta$ is invariant under translation by $0 \oplus N$,
and thus $\Theta = \Theta(\cA^\mch) \times N$. 

By Lemma \ref{clem} we construct $\cAgs \subset \cAg$ from
the atlas of toric compactifications 
\[
T_{N^\circ} \times T_M \subset T_{N^\circ} \times \bA^n_{X_i}
\]
parameterized by the cluster chambers in $\Delta^+_{\s}$. 
Now take $q \in \Theta$,
and consider $\vartheta_{Q,q}$ for some basepoint in the cluster complex. This
is a positive sum of monomials, so it will be regular on the boundary of
$\cAgs \subset \cAg$ iff each summand is. 
One summand is $z^q$, so if $\vartheta_q$ is regular on $\cAg$ then 
$\pi_N(q) \in N^+_{\s}$. Thus $\Theta(\cAg)\subset \Theta(\cA^{\vee})
\times N^+_{\s}$.
But now suppose $q = (m,n)$, some $m \in \Theta(\cA^{\vee})$ and $n \in N^+_{\s}$. Then
$z^q$ is regular on the boundary. Since the initial scattering monomials
are $(v_i,e_i)$, any bend in a broken line multiplies the decorating monomial 
by a monomial regular on the boundary.
Thus $q \in \Theta(\cAg)$. This completes the proof.
\end{proof} 

We define
\[
\midd(\cAg) := \bigoplus_{q \in \Theta(\cAg)} \kk 
\vartheta_q \subset \midd(\cAgs).
\]

Recall
$\low(\cAgs) \subset \up(\cAgs)$ are the cluster and upper cluster algebras 
with principal coefficients respectively, 
with the frozen variables inverted. On the other
hand,
$\low(\cAg) \subset \up(\cAg)$ are the cluster and upper cluster algebras
with principal coefficients respectively, with the frozen variables not
inverted. 
By Lemma \ref{midocalem}, $\midd(\cAg) \subset \midd(\cA_{\prin})$ 
is a subalgebra, and
$\low(\cAg) \subset \midd(\cAg) \subset \up(\cAg)$. 

By Lemma \ref{midocalem} and Lemma \ref{gsoplem}, we have

\begin{corollary} \label{improp} If $\midd(\cA_{\prin}) = \up(\cA_{\prin})$ 
then $\midd(\ocA_{\prin}^{\s}) = \up(\ocA_{\prin}^{\s})$. 
\end{corollary}

Here is another
sufficient condition for the full Fock-Goncharov conjecture to hold, which
will prove immediately useful below:

\begin{proposition} \label{scprop} Suppose there is a min-convex function 
$w:\cXrt \to \bR$, such that $w(p) > 0$ implies 
$p \in \Theta$, and such that $w(p) > 0$ for some $p$. 
Suppose also that
there is a bounded positive polytope in $\cXt$ 
(which holds for example if $\cA^{\mch}_{\prin}$ has Enough Global Monomials). 
Then $\Theta = \cXt$. 
\end{proposition}

\begin{proof} Take any $p \in \cXt$ and $q \in \cXt$ with $w(q) > 0$. 
Then consider any $\vartheta_r$ appearing in $\vartheta_p \cdot \vartheta_{mq}$, for $m \geq 1$. By Lemma \ref{diplem1},
\[
w(r) \geq w(p) + w(mq) = w(p) + m w(q) > 0
\]
for $m$ sufficiently large. In particular $r\in \Theta$,
so for each $\vartheta_r$ that appears, 
$\vartheta_{Q,r}$ is a universal positive Laurent polynomial, for
any basepoint $Q$ in the cluster complex. 
The existence of the bounded positive polytope
implies $\vartheta_p \cdot \vartheta_{mq}$ is 
a finite sum of $\vartheta_r$. Thus the product $\vartheta_p\cdot\vartheta_{mq}$
is also a universal positive Laurent polynomial,
and thus by the positivity of the scattering diagram,
$\vartheta_{Q,p}$ must be a finite positive Laurent polynomial. Thus $p \in \Theta$.
\end{proof} 

If there are frozen variables, there is a canonical candidate for $w$ in
the proposition. The cluster algebras related to double Bruhat cells are of this
sort, and we hope that this will give a way of completing the proof of
the full Fock-Goncharov conjecture in these cases.

When we have frozen variables, this gives a partial compactification $\cA \subset \ocA$.
In this case, let us change notation slightly and write a seed $\s$ as 
\[
\s=(e_1,\dots,e_{n_u},h_1,\dots,h_{n_f}),
\]
with $n_u=\# I_{\uf}$ and $n_f=\#(I\setminus I_{\uf})$, 
and the $h_i$ are frozen. 
In this case the elements $d_i h_i \in N^\circ_{\s} = \cA^{\trop}(\bZ)$
give $n_f$ canonical boundary divisors for a partial compactification 
$\cA \subset \ocA$, and an analogous $\cA_{\prin} \subset \ocA_{\prin}$. 
An atlas for $\cA_{\prin} \subset \ocA_{\prin}$ is given by gluing the partial
compactification $T_{\tN^\circ} \subset \TV(\Sigma_{\s})$, where $\Sigma_{\s}$ 
is the fan consisting of the rays $\bR_{\geq 0} (d_ih_i,0)$. 

\begin{corollary} 
\label{spcor} 
Assume that for each $1\le j\le n_f$,
$(d_jh_j,0)\in \cA^{\trop}_{\prin}(\bZ)$ has an optimized seed,
$\s_j$. Let $W := \sum \vartheta_{i(d_jh_j,0)}$ be the \emph{(Landau-Ginzburg)
potential}, the sum of the corresponding 
global monomials on $\cA^{\mch}_{\prin}$
given by Lemma \ref{opslem}. 
Then:
\begin{enumerate}
\item
The piecewise linear function
\[
W^T : \cXrt \to \bR
\]
is min-convex and
\[
\Xi:= \{x \in \cXrt\,|\, W^T(x) \geq 0 \}
\]
is a positive polytope. 
\item
$\Xi$ has the alternative description:
\[
\Xi = \{x \in \cXrt\,|\, \langle x,(d_jh_j,0)\rangle_{\s_j} \geq 0 \text{ for all } j\}. 
\]
\item
The set 
\[
\Xi \cap \Theta(\cA_{\prin}) = \{p \in \Theta(\cA_{\prin})\,|\,
\vartheta_p \in \up(\ocA_{\prin})
\subset \up(\cA_{\prin})\}
\]
parameterizes a canonical basis of 
\[
\midd(\ocA_{\prin}) = \up(\ocA_{\prin}) \cap \midd(\cA_{\prin}) \subset 
\up(\cA_{\prin}).
\]
\end{enumerate}
\end{corollary}

\begin{proof} 
This is immediate from Lemma \ref{midocalem}. 
\end{proof} 

\begin{corollary} \label{ffgcor} Assume we have Enough Global Monomials on 
$\cA_{\prin}^{\vee}$, and every 
frozen variable has an optimized seed. Let $W$ and $\Xi$ be as in 
Corollary \ref{spcor}. If for some seed $\s$, $\Xi$ 
is contained in the convex hull $\Conv(\Theta)$
of $\Theta$ (which itself contains 
the integral points of the  cluster complex $\Delta^+(\ZZ)$) then 
$\Theta = \cA^{\mch}_{\prin}(\bZ^T)$,
$\midd(\ocA_{\prin}) = \up(\ocA_{\prin})$ is finitely generated, and the
integer points $\Xi \cap \cA^{\mch}_{\prin}(\bZ^T)$ 
parameterize a canonical basis of $\up(\ocA_{\prin})$. 
\end{corollary}

\begin{proof} 
By definition 
$\Xi := \{W^T \geq 0\}$, and $W^T$ is min-convex by Lemma \ref{opslem}. Thus 
$\Theta = \cA^{\vee}_{\prin}(\ZZ^T)$ 
by Proposition \ref{scprop}. Now the result follows from 
the inclusions 
\[
\midd(\cA_{\prin}) \subset \up(\cA_{\prin}) \subset \can(\cA_{\prin})
\]
of Corollary \ref{markscor} 
\end{proof} 

The corollary applies in important representation theoretic examples:

\begin{proof}[Proof of Corollary \ref{slrcor}] The hyptheses of Theorem \ref{spth} are proven
in \cite{Magee}, using Proposition \ref{scprop} applied to the tropicalisation of our
potential $W$. The agreement of $W$ with the Berenstein-Kazhdan potential is given in
\cite{TimThesis}. Theorem \ref{spth} is stated for $\cA_{\prin}$. But in this case it is shown
in \cite{Magee} that the exchange matrix has full rank, i.e. the equivalent conditions
of Lemma \ref{taprinlem} hold. Now the results for $\cA_{\prin}$ imply the analogous result
for $\cA$ using $T_{\tK^{\circ}}$ equivariance, as in the proof of (7) of Theorem \ref{mainthax}. 
The $H$ action is identified with the action of $T_{N_{\uf}^{\perp}}$ on $\cA^{\vee}$,  
the various statements about
$H$-weights now follow immediately from the equivariance, Proposition \ref{thetaeigenfunctions}. 
\end{proof} 

\begin{proof}[Proof of Corollary \ref{conf3cor}] We check the conditions of Theorem \ref{spth}:
The existence of an optimized seed is proven in \cite{Magee}, following suggestions of L. Shen and
B. LeClerc. The cluster variety has large cluster complex
(Definition \ref{lccdef}) by \cite{GS16},
Theorems 1.12 and 1.17. This
gives the hypotheses of Theorem \ref{spth}, (3). The equality of our 
$W$ with the Goncharov-Shen potential is given 
in \cite{TimThesis}. It is shown in \cite{TimThesis} that the exchange matrix has full rank, in
the sense of Lemma \ref{taprinlem}. Now the $\cA_{\prin}$ results imply the analogous statements for $\cA_{\prin}$
as in the proof directly above. The $H^{\times 3}$ action is identified with the $T_{N_{\uf}^{\perp}}$ action,
which gives the weight statements as above. 
\end{proof} 

\appendix
\section{Review of notation and Langlands duality} \label{LDsec} 
We first review basic cluster variety notation as adopted in \cite{P1}.
None of this is original to \cite{P1}, but we follow that source for
consistency of notation.

As in \cite{P1}, \S 2, {\it fixed data} $\Gamma$ means 

\begin{itemize} \label{seeddata}
\item  A lattice $N$ with a skew-symmetric bilinear form
\[
\{\cdot,\cdot\}:N\times N\rightarrow \QQ.
\]
\item 
An \emph{unfrozen sublattice} $N_{\uf}\subseteq N$, a saturated
sublattice of $N$. If $N_{\uf}=N$, we say the fixed data has no frozen
variables.
\item An index set $I$ with $|I|=\rank N$ and a subset $I_{\uf}\subseteq I$
with $|I_{\uf}|=\rank N_{\uf}$.
\item Positive integers $d_i$ for $i\in I$ with greatest common divisor $1$.
\item A sublattice $N^{\circ}\subseteq N$ of finite index such that
$\{N_{\uf}, N^{\circ}\}\subseteq \ZZ$, $\{N,N_{\uf}
\cap N^{\circ}\}\subseteq\ZZ$.
\item $M=\Hom(N,\ZZ)$, $M^{\circ}=\Hom(N^{\circ},\ZZ)$.
\end{itemize}

Here we modify the definition slightly, and include in the {\it fixed data}
$[\s]$ a mutation class of seed. Recall a seed $\s= (e_1,\dots,e_n)$ is a
basis of $N$ satisfying certain properties, see \cite{P1}, \S 2, for 
the precise definitions, including that of mutation. In particular,
we write $e_1^*,\ldots,e_n^*$ for the dual basis and $f_i=d_i^{-1}e_i^*$.
We write
\begin{equation}
\label{epsilondef}
\epsilon_{ij}:= \{e_i,e_j\}d_j.
\end{equation}
We shall also assume that if $i\in I_{\uf}$, then the linear functional
$\{e_i,\cdot\}$ is non-zero. (If this happens, then one can view $e_i$
as frozen.)

We have two natural maps defined by $\{\cdot,\cdot\}$:
\begin{align*}
p_1^*:N_{\uf}\rightarrow M^{\circ}\quad\quad &\quad\quad\quad\quad p_2^*:N\rightarrow
M^{\circ}/N_{\uf}^{\perp}\\
N_{\uf}\ni n\mapsto (N^{\circ}\ni n'\mapsto \{n,n'\})&\quad\quad N\ni n
\mapsto (N_{\uf}\cap N^\circ\ni n'\mapsto \{n,n'\})
\end{align*}
We also choose a map
\begin{equation}
\label{pstardef}
p^*:N\rightarrow M^{\circ}
\end{equation}
such that (a) $p^*|_{N_{\uf}}=p_1^*$ and
(b) the composed map $N\rightarrow M^{\circ}/N_{\uf}^{\perp}$ agrees with
$p_2^*$. Different choices of $p^*$
differ by a choice of map $N/N_{\uf}\rightarrow N_{\uf}^{\perp}$.
Further, if there are no frozen variables, i.e., $I_{\uf}=I$,
then $p^*=p_1^*=p_2^*$ is canonically
defined.

We also define
\[
K=\ker p_2^*,\quad K^{\circ}=K\cap N^{\circ}.
\]

Following our conventions in \cite{P1}, let $\foT$ be the infinite oriented
rooted tree with $|I_{\uf}|$ outgoing edges from each vertex, labelled by
the elements of $I_{\uf}$. Let $v$ be the root of the tree. Attach some
choice of initial seed
$\s\in [\s]$ to the vertex $v$. (We write $\foT_{\s}$ if we want to record
this choice of initial seed.)
Now each simple path starting at $v$ determines
a sequence of mutations, just mutating at the label attached to the edge.
In this way we attach a seed to each vertex of $\foT$. We write the seed
attached to a vertex $w$ as $\s_w$, and write $T_{N^{\circ},\s_w}, 
T_{M,\s_w}$
etc.\ for the corresponding tori. Mutations define birational maps between
these tori, and the associated Fock-Goncharov $\shA$, $\shX$
cluster varieties are defined by
\begin{equation} \label{aandx}
\cA_{\Gamma} = \bigcup_{w \in \foT} T_{N^{\circ},{\s_w}}, 
\quad\quad\cX_{\Gamma} =
\bigcup_{w \in \foT} T_{M,\s_w}.
\end{equation}
This parameterization of torus charts is very redundant, with 
infinitely many copies
of the same chart appearing. In particular, given a vertex $w$ of $\foT$,
one can consider the subtree $\foT_w$ rooted at $w$, with initial seed $\s_w$. 
This tree can similarly be used to define $\cA_{\Gamma}$, and the obvious
inclusion between these two versions of $\cA_{\Gamma}$ is in fact an
isomorphism, as can be easily checked.

As one expects the mirror of a variety obtained by gluing charts of the
form $T_{M^{\circ}}$ to be obtained by gluing charts of the form 
$T_{N^{\circ}}$, the mirror of $\cA$ is not $\cX$, as the latter is
obtained by gluing charts of the form $T_N$. To get the correct mirrors
of $\cA$ and $\cX$, one 
follows \cite{FG09} in defining the {\it Langlands dual} cluster varieties.
This is done by, given fixed data $\Gamma$, defining 
fixed data $\Gamma^{\mch}$ to be the fixed data:
\[
I^{\vee}:=I, \quad I_{\uf}^{\vee}:=I_{\uf}, \quad d_i^{\vee}:=d_i^{-1}D
\]
where
\[
D:=\lcm(d_1,\ldots,d_n).
\]
The lattice, with its finite index sublattice, is 
\[
 D \cdot N =: (N^{\mch})^{\circ} \subset N^{\mch} :=  N^{\circ}
\]
and the $\bQ$-valued skew-symmetric form on $N^{\mch} = N^{\circ}$ is 
\[
\{\cdot,\cdot\}^{\mch} := D^{-1} \{\cdot,\cdot\}.
\]

For each $\s=(e_1,\ldots,e_n) \in [\s]$, we define
\[
\s^\mch :=(d_1 e_1,\dots,d_ne_n).
\]
One checks easily that $\s \mapsto \s^{\mch}$ gives a bijection between 
$[\s]$ and $[\s^{\mch}]$.

Note that for skew-symmetric cluster algebras, i.e.,
when all the multipliers $d_i =1$, Langlands duality
is the identity, $\Gamma^{\mch} = \Gamma$. 

\begin{definition}[Fock-Goncharov dual] \label{fgddef}
We write $\cA_{\Gamma}^{\mch} := \cX_{\Gamma^{\mch}}$ and
$\cX_{\Gamma}^{\mch}:=\cA_{\Gamma^{\mch}}$.
\end{definition}

Note in the skew-symmetric case, that $\cA^{\mch} = \cX$. 

One observes the elementary 

\begin{proposition} 
Given fixed data $\Gamma$, the double Langlands dual data $\Gamma^{\vee\vee}$
is canonically isomorphic to the data $\Gamma$ via the map $D\cdot N\rightarrow
N$ given by $n\mapsto D^{-1}n$.
\end{proposition}

\section{The $\cA$ and $\cX$-varieties with principal coefficients} 
\label{rdsec} 

We recall briefly the construction of principal fixed data from
\cite{P1}, Construction 2.11. For fixed data $\Gamma$, the data for the cluster
variety with principal coefficients $\Gamma_{\prin}$ is defined by:
\begin{itemize}
\item  $\tN := N \oplus M^{\circ}$ with the skew-symmetric bilinear form
\[
\{(n_1,m_1),(n_2,m_2)\}=\{n_1,n_2\}+\langle n_1,m_2\rangle-\langle n_2,m_1
\rangle.
\]
\item $\tN_{\uf} := N_{\uf} \oplus 0 \subset \tN.$
\item The  sublattice $\tN^{\circ}$ is $N^\circ \oplus M$.
\item The index set $I$ is now the disjoint union of two copies of $I$,
with the $d_i$ taken to be as in $\Gamma$. The set of unfrozen indices
$I_{\uf}$ is just the original $I_{\uf}$ thought of as a subset of the
first copy of $I$.
\item Given an initial seed $\s=(e_1,\ldots,e_n)$, we define 
\begin{equation}
\label{tildeseed}
\tilde\s=\big((e_1,0),\ldots,(e_n,0),(0,f_1),\ldots,(0,f_n)\big).
\end{equation}
We then take the mutation class $[\tilde\s]$.
\end{itemize}

Note that $[\tilde\s]$ depends on the choice of $\s$: it is not true
that if $\s'$ is obtained by mutation from $\s$ then $\tilde\s'$ is
obtained from the same set of mutations applied to $\tilde\s$.
Nevertheless, the cluster varieties 
\[
\cX_{\prin}:=\cX_{\Gamma_{\prin}}, \quad \cA_{\prin}:=\cA_{\Gamma_{\prin}}
\]
are defined independently of the seed $\s$.
This is a very important point, which we shall revisit in Remark 
\ref{initialseedindremark}.

The following summarizes all of the important relationships between
the various varieties which will be made use of in this paper.

\begin{proposition}
\label{ldpprop}
Giving fixed data $\Gamma$, we have:
\begin{enumerate}
\item
There is a commutative diagram where the dotted arrows are only present
if there are no frozen variables (i.e., $N_{\uf}=N$): 
\[
\xymatrix@C=30pt
{
\cA_t\ar[r]\ar[d]&\cA_{\prin}\ar@/^24pt/[rr]^p\ar[r]^{\tilde p}\ar[d]_{\pi}
&\cX\ar[d]_{\lambda}&\cX_{\prin}\ar[l]_{\rho}\ar[d]_w&\cA\ar[l]_{\xi}\ar[d]\\
t\ar[r]&T_M\ar@{-->}[r]&T_{K^*}&\ar@{-->}[l]T_M&\ar[l] e
}
\]
with $t\in T_M$ any point, $e\in T_M$ the identity, and with the left- and
right-hand squares cartesian and $p$ an isomorphism, canonical if there
are no frozen variables.
\item There are torus actions
\[
\hbox{$T_{N^{\circ}}$ on $\cA_{\prin}$;
$T_{K^{\circ}}$ on $\cA$;
$T_{N_{\uf}^{\perp}}$ on $\cX$;
$T_{\widetilde K^{\circ}}$ on $\cA_{\prin}$.}
\]
Here $\widetilde K^{\circ}$ is the kernel of the map 
\begin{align*}
N^{\circ}\oplus M&\rightarrow N^*_{\uf}\\
(n,m)&\mapsto p_2^*(n)-m.
\end{align*}
Furthermore $T_{N^{\circ}}$ and $T_{\widetilde K^{\circ}}$ act on $T_M$
so that the map $\pi:\cA_{\prin}\rightarrow T_M$ is $T_{N^{\circ}}$- and
$T_{\widetilde K^{\circ}}$-equivariant.
The map $\tilde p:\cA_{\prin}\rightarrow \cX=\cA_{\prin}/T_{N^{\circ}}$ 
is a $T_{N^{\circ}}$-torsor. There is a map $T_{\widetilde K^{\circ}}\rightarrow
T_{N_{\uf}^{\perp}}$ such that the map $\tilde p$ is also compatible
with the actions of these two tori on $\cA_{\prin}$
and $\cX$ respectively, so that 
\[
\tau:\cA_{\prin}\rightarrow \cX/T_{N_{\uf}^{\perp}}
\]
is a $T_{\widetilde K^{\circ}}$-torsor.
\item
$(\Gamma_{\prin})^{\vee}$ and $(\Gamma^{\vee})_{\prin}$ are
isomorphic data, so we can define 
\[
\shA^{\vee}_{\prin}:=\shX_{(\Gamma^{\vee})_{\prin}}, \quad
\shX^{\vee}_{\prin}:=\shA_{(\Gamma^{\vee})_{\prin}}
\]
\item
There is a commutative diagram
\[
\xymatrix@C=30pt
{
\cX^{\vee}\ar[r]\ar[d]&\cX_{\prin}^{\vee}\ar@/^24pt/[rr]^p\ar[r]^{\tilde p}
\ar[d]_{\pi}
&\cA^{\vee}\ar[d]_{\lambda}&\cA^{\vee}_{\prin}
\ar[l]_{\rho}\ar[d]_w&\cX^{\vee}\ar[l]_{\xi}\ar[d]\\
e\ar[r]&T_{M^{\circ}}\ar@{-->}[r]&T_{(K^{\circ})^*}&\ar@{-->}[l]T_{M^{\circ}}&\ar[l] e
}
\]
\end{enumerate}
\end{proposition}

\begin{proof}
We consider the diagram of (1). The maps with names
are given as follows on cocharacter
lattices: 
\begin{align}
\label{basicmaps}
\begin{split}
\pi:N^{\circ}\oplus M\rightarrow M,\quad & (n,m)\mapsto m\\
\tilde p:N^{\circ}\oplus M\rightarrow  M, \quad &
(n,m)\mapsto m-p^*(n)\\
\rho: M\oplus N^{\circ}\rightarrow M,\quad & (m,n)\mapsto m\\
\lambda: M\rightarrow K^*,\quad & m\mapsto m|_K\\
w: M\oplus N^{\circ}\rightarrow M,\quad & (m,n)\mapsto m-p^*(n)\\
\xi: N^{\circ} \rightarrow M\oplus N^{\circ},\quad & n\mapsto (-p^*(n),-n)\\
p: N^{\circ}\oplus M\rightarrow M\oplus N^{\circ},  \quad &
(n,m) \mapsto (m-p^*(n), n)
\end{split}
\end{align}
Note $\lambda$ is the transpose of the inclusion $K\rightarrow N$.
In the case there are no frozen variables, the two dotted horizontal
lines are just given on cocharacter lattices by $\lambda$ again.
One checks commutativity from these formulas at the level of individual
tori, and one checks the maps are compatible with mutations.
Note the left-hand diagram defines $\cA_t$, see \cite{P1}, Definition 2.12.
The statements that $\tilde p$, $\pi$ and $\lambda$ are compatible with
mutations are in \S2 of \cite{P1}, as well as the commutativity of
the second square in case of no frozen variables. It is clear that $p$
induces an isomorphism of lattices, hence an isomorphism of the relevant
tori. This isomorphism is canonical in the no frozen variable case because
$p^*$ is well-defined in this case. The fact the right-hand square is
cartesian follows from the fact that $\Im \xi=\ker w$. Note the signs
in the definition of $\xi$ are necessary to be compatible with mutations.
This gives (1).

For (2), the first action is specified on the level of cocharacter
lattices by 
\[
N^{\circ} \rightarrow N^{\circ}\oplus M, \quad n \mapsto (n, p^*(n))\\
\]
while the last three are given by the inclusions 
\[
K^{\circ}\subset N^{\circ},\quad N_{\uf}^{\perp} \subset M, \quad
\widetilde K^{\circ} \subset N^{\circ}\oplus M.
\]
One checks easily that the induced actions are compatible with mutations.
The action of $T_{N^{\circ}}$ and $T_{\widetilde K^{\circ}}$ on
$T_M$ are induced by the maps $n\mapsto p^*(n)$ and $(n,m)\mapsto m$
respectively, in order to achieve the desired equivariance.
The map $T_{\widetilde K^{\circ}}\rightarrow T_{N_{\uf}^{\perp}}$ is given by
\[
\widetilde K^{\circ}\ni (m,n)\mapsto m-p^*(n)\in N_{\uf}^{\perp}.
\]
The other statements are easily checked.

For (3), from the definitions, the lattices playing 
the role of $N^{\circ}\subseteq N$ are:
\begin{align*}
(\Gamma_{\prin})^{\vee}:& \quad
D\cdot \widetilde N=D\cdot N\oplus D\cdot M^{\circ}\subseteq 
\widetilde N^{\circ}=N^{\circ}\oplus M\\
(\Gamma^{\mch})_{\prin}:& \quad 
D\cdot N\oplus M^{\circ}\subseteq N^{\circ}\oplus D^{-1}\cdot M.
\end{align*}
These are isomorphic under the map $(n,m)\mapsto (n, D^{-1}m)$. Furthermore,
the pairings in the two cases are given by 
\[
\{(n_1,m_1),(n_2,m_2)\}
=\begin{cases}
D^{-1}(\{n_1,n_2\}+\langle n_1,m_2\rangle -\langle n_2, m_1\rangle)
&\hbox{in the $(\Gamma_{\prin})^{\vee}$ case}\\
D^{-1}\{n_1,n_2\}+\langle n_1,m_2\rangle -\langle n_2, m_1\rangle
&\hbox{in the $(\Gamma^{\vee})_{\prin}$ case}
\end{cases}
\]
respectively. The isomorphism given preserves the pairings, hence the 
isomorphism.

(4) is the same as (1), but for the 
Langlands dual data $\Gamma^{\vee}$. For reference, the maps
are given as follows: 
\begin{align}
\label{basicmapsdual}
\begin{split}
\pi:D\cdot N\oplus M^{\circ}\rightarrow M^{\circ},\quad & (n,m)\mapsto m\\
\tilde p:D \cdot N\oplus M^{\circ}\rightarrow  M^{\circ}, \quad &
(Dn,m)\mapsto m-p^*(n)\\
\rho: M^{\circ}\oplus D\cdot N\rightarrow M^{\circ},\quad & (m,Dn)\mapsto m\\
\lambda: M^{\circ}\rightarrow (K^{\circ})^*,\quad & m\mapsto m|_K\\
w: M^{\circ} \oplus D\cdot N\rightarrow M^{\circ},\quad & (m,Dn)
\mapsto m-p^*(n)\\
\xi: D\cdot N \rightarrow M^{\circ} \oplus D\cdot N,\quad & Dn\mapsto (-p^*(n),-Dn)\\
p: D\cdot N\oplus M^{\circ}\rightarrow M^{\circ}\oplus D\cdot N, \quad &
(Dn,m) \mapsto (m-p^*(n), Dn)
\end{split}
\end{align}
\end{proof}

\begin{remark}
Whenever the lattice $D\cdot N$ appears in dealing with the Langlands
dual data, we will always identify this with $N$ in the obvious way.
\end{remark}

Simple linear algebra gives:

\begin{lemma} \label{eglem}
The choice of the map $p^*$ gives an inclusion $N^{\circ}\subset \widetilde
K^{\circ}$ (see Proposition \ref{ldpprop}, (2)) given by $n\mapsto (n,p^*(n))$.
We also have $N_{\uf}^{\perp}$ (a sublattice of $M$) included in 
$\widetilde K^{\circ}$ via $m\mapsto (0,m)$. These inclusions
induce an isomorphism
$N^\circ \oplus N_{\uf}^{\perp} \to \tK^{\circ}$.
\end{lemma}

\begin{lemma} \label{taprinlem} The map $T_{\tK^\circ} \to T_M$ 
induced by the composition of the inclusion and projection $\tK^{\circ}
\subset \tN^{\circ}\rightarrow M$
is a split surjection if and only if the 
map 
\[
p_2^*|_{N^{\circ}}:N^\circ \to N_{\uf}^*, \quad n \mapsto \{n, \cdot\}|_{N_{\uf}} 
\]
is surjective. This
holds if and only if in some seed $\s=(e_i)_{i\in I}$,
the $\# I_{\uf}  \times \# I$ matrix with entries for $i\in I_{\uf}$, 
$j\in I$, 
$\epsilon_{ij} = \{e_i,d_j e_j\}$ gives a surjective map
$\bZ^{\#I} \to \bZ^{\#I_{\uf}}$. 
In this case $\pi: \cA_{\prin} \to T_M$ is isomorphic to the trivial
bundle $\cA \times T_M \to T_M$. 
\end{lemma}

\begin{proof} For the first statement, note using Lemma \ref{eglem}
that the map $\tK^{\circ}\rightarrow M$ is surjective if and only if
the map $N^{\circ}\oplus N_{\uf}^{\perp}\rightarrow M$ given by
$(n,m)\mapsto m+p^*(n)$ is surjective, and this is the case if and only
if the induced map $N^{\circ}\rightarrow M/N_{\uf}^{\perp}=N_{\uf}^*$ 
is surjective. The
given matrix is the matrix for $N^{\circ} \to N_{\uf}^*$ in the given bases, so the second
equivalence is clear. The final statement follows from the 
$T_{\tK^\circ}$-equivariance of
$\pi$ (the trivialization then comes by choosing a splitting of 
$\tK^\circ \twoheadrightarrow M$).
\end{proof}

\begin{remark}
\label{initialseedindremark}
In general, a seed is defined to be a basis of the lattice $N$ (or
$\tN$), but to define the seed mutations
\cite{P1}, (2.2) and the union of tori \eqref{aandx}, all one needs are elements
$e_i \in N$, $i \in I_{\uf}$ (the definitions as given make sense even if the
$e_i$ are dependent, or
fail to span).
If one makes the construction in this greater generality,
the characters $X_i := z^{e_i}$ on $T_{M,\s} \subset \cX$ will not be independent (if
the $e_i$ are not) and unless we take a full basis, we cannot define the
cluster variables $A_i := z^{f_i}$ on $T_{N^\circ,\s}$,
as the $f_i$ are defined as the
dual basis to the basis $(d_1e_1,\dots,d_ne_n)$ for $N^\circ$.

In the case of the principal data, given a seed $\s=(e_1,\ldots,e_n)$ for
$\Gamma$, we get a seed $\big((e_1,0),\ldots,(e_n,0)\big)$ in this modified
sense for the data $\Gamma_{\prin}$. We also write this seed as $\s$.
On the other hand, in \cite{P1}, the seed
$\tilde\s$ for $\Gamma_{\prin}$ is defined in the more traditional sense
to be the basis $\big((e_1,0),\ldots,(e_n,0),(0,f_1),\ldots,(0,f_n)\big)$.
It is not the case that if $\s'$ is obtained from $\s$ via a sequence of
mutations, then $\tilde\s'$ is obtained from $\tilde\s$ by the same
sequence of mutations. In particular, the set $[\tilde\s]$ of seeds mutation
equivalent to $\tilde\s$ depends not just on the mutation equivalence class
of $\s$, but on the original seed $\s$. However, using the seed $\s$
as a seed for $\Gamma_{\prin}$ in this modified sense, we can build
$\shA_{\prin}$, and this depends only on the mutation class of $\s$.
Thus $\cA_{\prin}$ does not depend on the initial choice of seed, but only
on its mutation equivalence class.

However, as we shall now see, the choice of initial seed does give a partial
compactification. This is a more general phenomenon when there are frozen
variables.
\end{remark}

\begin{construction}[Partial compactifications from frozen variables] 
\label{fvss} When the cluster data $\Gamma$ includes frozen
variables, $\cA$ comes with a canonical partial compactification $\cA \subset \ocA$,
given by partially compactifying each torus chart via
$T_{N^\circ,\s} \subset \TV(\Sigma^{\s})$, where for
$\s = (e_i)$,
$\Sigma^{\s} =\sum_{i\not\in I_{\uf}} \RR_{\ge 0}e_i \subset N^\circ_{\bR,\s}$. 
Thus the dual cone
$(\Sigma^{\s})^{\mch} \subset M^\circ_{\bR,\s}$ is cut out by the half-spaces 
$e_i \geq 0$, $i \not \in I_{\uf}$. 
Note that the monomials 
$A_i := z^{f_i}$, $i \not \in I_{\uf}$ are invariant under mutation. These give a canonical
map $\ocA \to \bA^{\rank N-u}$, where $u$ is the number of unfrozen variables. 
Note that the basis elements $e_i$ for $i\not\in I_{\uf}$, though they have
frozen indices, can change under mutation. What is invariant is the associated boundary divisor with valuation given by
$e_i \in N^\circ_{\bf s} = \cA^{\trop}(\bZ)$. These are the boundary 
divisors of $\cA \subset \ocA$. We remark that like $\cA$, $\ocA$ is
also separated, with the argument given in \cite{P1}, Theorem 3.14 working
equally well for $\ocA$.

Here is another way of seeing the same thing. Given any cluster variety
$V = \bigcup_{\s \in S} T_{L,\s}$ 
and a single fan $\Sigma \subset L_{\bR}$ for a toric partial compactification $T_{L,\s'} \subset \TV(\Sigma)$ for some $\s'\in S$,
there is a canonical way to build a partial compactification 
\[
V \subset \overline{V} = \bigcup_{\s \in S} \TV(\Sigma^{\s}).
\]
We let $\Sigma^{\s'} := \Sigma$ and $\Sigma^{\s} := 
(\mu_{\s,\s'}^t)^{-1}(\Sigma^{\s'})$, where 
$\mu_{\s,\s'}$ is the birational map given by the composition
\[
\mu_{\s,\s'}: T_{L,\s} \subset V \supset T_{L,\s'}
\]
and $\mu_{\s,\s'}^t$ is the geometric tropicalisation, see 
\S\ref{tropsec}.
\end{construction}

\begin{remark}
\label{Aprinremark}
We now return to the discussion of $\cA_{\prin}$. Note that the frozen
variables for $\cA_{\prin}$ are indexed by $I\setminus I_{\uf}$ in the first
copy of $I$, along with all indices in the second copy of $I$. However, we can
apply Construction \ref{fvss} taking only the second copy of $I$ as the
set of frozen indices, with the initial choice of seed $\s$
determining a partial compactification of $\cA_{\prin}$. 
In this case,
we indicate the partial compactification by 
$\cA_{\prin} \subset \ocA_{\prin}^{\s}$. It is important to keep in mind the
dependence on $\s$. Fixing $\s$ fixes $\tilde\s$, and hence cluster variables
$A_i=z^{(f_i,0)}$, $X_i=z^{(0,e_i)}$. The variables $X_i$ can then take
the value $0$ in the compactification. In particular, we obtain
an extension of $\pi: \cA_{\prin} \to T_M$ to 
$\pi: \ocA_{\prin}^{\s} \to \bA^n_{X_1,\dots,X_n}$,  $X_i := z^{e_i}$ 
pulling back to $X_i=z^{(0,e_i)}$.

Note that the seeds in $[\s]$ and $[\mts]$ are in one-to-one correspondence.
Given any seed $\s'=(e_i')_{i\in I}\in[\s]$, and seed $\tilde\s'\in[\mts]$ 
obtained via the same sequence of mutations we have $\tilde\s'=
\big((e_i',0)_{i\in I}, (g_i)_{i\in I}\big)$ for some $g_i\in \tN$. These two
seeds give rise to coordinates $A_i'$ on the chart of $\cA$ indexed by
$\s'$ and coordinates $A_i', X_i$ on the chart of $\cA_{\prin}$ indexed
by $\tilde\s'$. As $\cA$ is the fibre of $\pi$ over the point of $\AA^n$
with all coordinates $1$, the coordinate $A_i'$ on the chart of $\cA_{\prin}$
restricts to the coordinate $A_i'$ on the chart of $\cA$. 
This gives a one-to-one correspondence between cluster variables on 
$\cA$ and $A$-type cluster variables on $\cA_{\prin}$. To summarize:

\begin{proposition} \label{cvextrem} The cluster variety $\cA_{\prin} := 
\bigcup_{w\in \foT_{\s}} T_{\tN^{\circ},\s_w}$
depends only on the mutation class $[\s]$. But the choice of a seed $\s$ 
determines: 
\begin{enumerate}
\item A partial compactification $\shA_{\prin}\subset \ocA_{\prin}^{\s}$;
\item The canonical
extension of each cluster variable on any chart of $\cA$
to a cluster variable on the corresponding chart of $\cA_{\prin} \supset \cA$.
\end{enumerate}
\end{proposition} 
\end{remark}

\section{Construction of scattering diagrams}
\label{scatappendix}

This appendix is devoted to giving proofs of Theorems \ref{KSlemma2}
and \ref{scatdiagpositive}. The proof of \ref{KSlemma2} is essentially
given in \cite{GSAnnals}, but the special case here is
considerably simpler than the general case covered there, and
it is likely to be very difficult for the reader to extract the needed
results from \cite{GSAnnals}. In addition, the
details of the proof of \ref{KSlemma2} will be helpful in proving 
Theorem \ref{scatdiagpositive}.

\subsection{An algorithmic construction of scattering diagrams}\label{subsecKSlemalg}

\begin{construction} \label{KSlemalg} 
There is a simple order by order algorithm, introduced in \cite{KS06} in 
the two-dimensional case and in \cite{GSAnnals} in the higher dimensional case, 
for producing the diagram $\foD \supset \foD_{\inc}$ of Theorem \ref{KSlemma},
which we will describe shortly after a bit of preparation. This is useful
both from a computational point of view and because a more complicated version of
this will be necessary in the remainder of this Appendix.

We continue with fixed data $\Gamma$, yielding the Lie algebra $\fog$
in \S\ref{defconstsection}.

We first introduce some additional terminology.
For any scattering diagram $\foD$ for $N^+,\fog$, and
any $k > 0$ we let $\foD_k \subset \foD$ be the (by definition, finite) 
set of $(\fod,g_{\fod})$ with $g_{\fod}$ non-trivial in $G^{\le k}$.
A scattering diagram for $N^+$, $\fog$ induces a scattering diagram for
$N^+$, $\fog^{\le k}$ in the obvious way, viewing $g_{\fod} \in G^{\le k}$
for a wall $(\fod,g_{\fod})$. We say two scattering diagrams $\foD$, $\foD'$
are equivalent to order $k$ if they are equivalent as scattering diagrams
for $\fog^{\le k}$.

\begin{definition-lemma} \label{pjdef} Let $\foj$ be a joint of the scattering diagram $\foD_k$. 
Either every wall containing $\foj$ has direction tangent to $\foj$ (where the
direction of a wall contained in $n^{\perp}$ is 
$-p^*(n) =-\{n,\cdot\}$), or every wall
containing $\foj$ has direction \emph{not} tangent to $\foj$. In the first case
we call the joint {\it parallel} and in the second case {\it perpendicular}.
\end{definition-lemma} 

\begin{proof} Suppose $\foj$ spans the subspace
$n_1^{\perp}\cap n_2^{\perp}$. Then the direction of any wall containing
$\foj$ is of the form $-p^*(a_1n_1+a_2n_2)$ for some $a_1,a_2\in\QQ$. 
If this is tangent to $\foj$, then $\langle p^*(a_1n_1+a_2n_2),n_i\rangle =0$
for $i=1,2$, and hence $0=\langle p^*(n_1),n_2\rangle=\{n_1,n_2\}$. 
 From this it follows that $\langle p^*(a_1'n_1+a_2'n_2),n_i\rangle=0$
for all $a_1',a_2'$, and hence the direction of any wall containing $\foj$
is tangent to $\foj$. \end{proof}

A joint $\foj$ is a codimension two convex rational
polyhedral cone. Let $\Lambda_{\foj}\subseteq M^{\circ}$ be
the set of integral tangent vectors to $\foj$. This is a saturated sublattice
of $M^{\circ}$. Then we set
\begin{equation}
\label{fogfojdef}
\fog_{\foj} := \bigoplus_{n\in N^{+} \cap
\Lambda_{\foj}^{\perp}} \fog_n.
\end{equation}
This is closed under Lie bracket.
If $\foj$ is a parallel joint, then 
$\fog_{\foj}$ is abelian, since if $n_1,n_2\in \Lambda_{\foj}^{\perp}$
with $p^*(n_1), p^*(n_2)\in\Lambda_{\foj}$, $\{n_1,n_2\}=\langle p^*(n_1),n_2
\rangle=0$,
so $[\fog_{n_1},\fog_{n_2}]=0$.
We
denote by ${G}_\foj$ the corresponding group.

We will build a sequence of finite scattering diagrams
$\tilde \foD_1\subset \tilde \foD_2\subset\cdots$, with the property that
$\tilde \foD_k$ is equivalent to $\foD$ to order $k$. Taking $\tilde\foD=
\bigcup_{k=1}^{\infty}\tilde\foD_k$, 
we obtain $\tilde\foD$ equivalent to
$\foD$. Let $(\foD_{\inc})_k$ denote the subset of $\foD_{\inc}$ consisting
of walls which are non-trivial in $G^{\le k}$. We start with
\[
\tilde\foD_1=(\foD_{\inc})_1.
\]
If $\foj$ is a joint of a finite scattering diagram, we write
$\gamma_{\foj}$ for a simple loop around $\foj$ small enough so that 
it only intersects walls containing $\foj$. In particular, for each joint
$\foj$ of $\tilde\foD_1$, $\theta_{\gamma_{\foj},\tilde\foD_1}=
\id\in G^{\le 1}$. Indeed, $G^{\le 1}$ is abelian and by the form given
for $\foD_{\inc}$ in the statement of Theorem \ref{KSlemma}, all walls 
containing $\foj$ 
are hyperplanes. Thus the automorphism associated to crossing each wall and
its inverse occurs once in $\theta_{\gamma_{\foj},\tilde\foD_1}$, and hence
cancel.

Now suppose we have constructed $\tilde\foD_k$. For every perpendicular
joint $\foj$ of $\tilde\foD_k$, we can write uniquely in $G_{\foj}^{\le k+1}$
\begin{equation}
\label{basicKSeq2} 
\theta_{\gamma_{\foj},\tilde\foD_{k}}=\exp \left(\sum_{\alpha \in S} g_{\alpha} \right)
\end{equation}
where $S \subseteq \{\alpha \in N^{+} \cap \Lambda_{\foj}^{\perp} \,|\,
d(\alpha) = k+1\}$ and $g_{\alpha}\in \fog_{\alpha}$. Such an expression
holds because all wall-crossing automorphisms for walls containing
$\foj$ lie in $G_{\foj}$, so that $\theta_{\gamma_{\foj},\tilde\foD_k}$ can
be viewed as an element of $G_{\foj}^{\le k+1}$. Furthermore, by the
inductive hypothesis, this element is trivial in $G_{\foj}^{\le k}$.
Because $\foj$ is
perpendicular, we never have $p^*(\alpha)\in\Lambda_{\foj}$.
Now define 
\[
\foD[\foj]:=\{\big(\foj-\RR_{\ge 0}p^*(\alpha),\exp(\pm g_{\alpha})\big)
\,|\, \alpha \in S\},
\]
where the sign is chosen so that the contribution to crossing the wall
indexed by $\alpha$ in $\theta_{\gamma_{\foj},\foD[\foj]}$ is 
$\exp(-g_{\alpha})$. Note the latter element is central in
$G^{\le k+1}$. Thus $\theta_{\gamma_{\foj},\foD[\foj]}=\theta_{\gamma_{\foj},
\tilde\foD_k}^{-1}$ and 
\begin{equation}
\label{fodfodj}
\theta_{\gamma_{\foj},\tilde\foD_k\cup\foD[\foj]}
= \theta_{\gamma_{\foj},\tilde\foD_k}\circ\theta_{\gamma_{\foj},\foD[\foj]}
=\id
\end{equation}
in $G^{\le k+1}$.

We define
\[
\tilde\foD_{k+1}=\tilde \foD_k \cup ((\foD_{\inc})_{k+1}\setminus 
(\foD_{\inc})_k)
\cup \bigcup_{\foj} \foD[\foj]
\]
where the union is over all perpendicular joints of $\tilde\foD_k$.

\begin{lemma}
\label{KSalglemma}
$\tilde\foD_{k+1}$ is equivalent to
$\foD$ to order $k+1$.
\end{lemma}

\begin{proof} 
Consider a perpendicular joint $\foj$ of $\tilde\foD_{k+1}$.
If $\foj$ is contained in a joint $\foj'$ of $\tilde\foD_k$, $\foj'$ is the
unique such joint, and we constructed $\foD[\foj']$
above. If $\foj$ is not contained in a joint of $\tilde\foD_k$, we define
$\foD[\foj']$ to be the empty set.
There are three types of walls $\fod$ in
$\tilde\foD_{k+1}$ containing $\foj$:
\begin{enumerate}
\item $\fod\in\tilde\foD_k\cup\foD[\foj']$.
\item $\fod\in\tilde\foD_{k+1}\setminus(\tilde\foD_{k}\cup\foD[\foj'])$,
but $\foj\not\subseteq
\partial\fod$. This type of wall does not contribute to
$\theta_{\gamma_{\foj},\tilde\foD_{k+1}}\in G^{\le k+1}$,
as the associated automorphism
is central in $G^{\le k+1}$,
and in addition this wall contributes twice to $\theta_{\gamma_{\foj},
\tilde\foD_{k+1}}$,
with the two contributions inverse to each other.
\item $\fod\in\tilde\foD_{k+1}\setminus(\tilde\foD_{k}\cup \foD[\foj'])$ and
$\foj\subseteq\partial\fod$. Since each added wall is of the form
$\foj''-\RR_{\ge 0}m$ for some joint $\foj''$ of $\tilde\foD_k$,
where $-m$ is the direction of the wall,
the direction of the wall is parallel to $\foj$, contradicting $\foj$
being a perpendicular joint. Thus this does not occur.
\end{enumerate}
 From this, it is clear that $\theta_{\gamma_{\foj},\tilde\foD_{k+1}}=
\theta_{\gamma_{\foj},\tilde\foD_k\cup \foD[\foj']}$, which is the
identity in $G^{\le k+1}$ by \eqref{fodfodj}. This holds
for every perpendicular joint of $\tilde\foD_{k+1}$.

The result then follows from Lemma \ref{equivalencetonextorder}.
\end{proof}

\begin{lemma}
\label{equivalencetonextorder}
Let $\foD$ and $\tilde\foD$ be two scattering diagrams for $N^+, \fog$
such that 
\begin{enumerate}
\item $\foD$ and $\tilde\foD$ are equivalent to order $k$.
\item $\foD$ is consistent to order $k+1$.
\item $\theta_{\gamma_{\foj},\tilde\foD}$ is the identity for every
perpendicular joint $\foj$ of $\tilde\foD$ to order $k+1$.
\item $\foD$ and $\tilde\foD$ have the same set of incoming walls.
\end{enumerate}
Then $\foD$ and $\tilde\foD$ are equivalent to order $k+1$, and in particular
$\tilde\foD$ is consistent to order $k+1$.
\end{lemma}

\begin{proof}
We work with scattering diagrams in the group $G^{\le k+1}$.
There is a finite scattering diagram $\foD'$ with the following 
properties:
(1) $\tilde\foD\cup \foD'$ is equivalent to $\foD$;
(2) $\foD'$ consists only of walls trivial to order $k$ but non-trivial
to order $k+1$. Indeed, $\foD'$ can be chosen so that $g_x(\foD')=
g_x(\tilde\foD)^{-1} g_x(\foD)$ for any general point $x$ in any $n^{\perp}$,
$n\in N^+$. Note that $\foD'$ is finite because the same is true of $\foD$
and $\tilde\foD$.

Thus to show $\foD$ and $\tilde\foD$ are
equivalent, it is sufficient to show that $\foD'$ is 
equivalent to the empty scattering diagram.
To do so, replace $\foD'$ with
an equivalent scattering diagram with minimal support. Let $\foj$ be a
perpendicular joint of $\tilde \foD\cup \foD'$. Then in
$G^{\le k+1}$, $\id=\theta_{\foD,\gamma_{\foj}}=
\theta_{\foD',\gamma_{\foj}}$, since $\theta_{\tilde\foD,\gamma_{\foj}}
=\id$ and automorphisms in $\foD'$ are central in $G^{\le k+1}$. 
However, this implies that for each $n_0 \in N^+$ with 
$\foj\subseteq n_0^{\perp}$ and $x, x'$ two points in $n_0^{\perp}$
on either side of $\foj$, the automorphisms associated with crossing
$n_0^{\perp}$ in $\foD'$ through either $x$ or $x'$  must be the same
in order for these two automorphisms to cancel in
$\theta_{\foD',\gamma_{\foj}}$.
 From this it is easy to see that $\foD'$ is equivalent to a scattering
diagram such that for every wall $\fod\in\foD'$, each facet of $\fod$
is a parallel joint of $\foD'$, i.e., the direction $-p^*(n)$ is
tangent to every facet of $\fod$. However, such a wall must be
incoming, contradicting, if $\foD'$ is non-empty, the fact that
$\tilde\foD$ and $\foD$ have the same set of incoming walls
by assumption.
\end{proof}
\end{construction} 

\subsection{The proof of Theorem \ref{KSlemma2}}
We fix the notation of Theorem \ref{KSlemma2}, and in addition make use
of the notation $\shH_{k,\pm}$ of Definition \ref{halfspacedef} and
$\theta_{\fod_k}$ as in \eqref{thetafodkdef} the map
associated to crossing the slab $\fod_k=(e_k^{\perp},1+z^{v_k})$ from
$\shH_{k,-}$ to $\shH_{k,+}$.

We define the Lie algebra
\[
\bar\fog:= \bigoplus_{n\in N^{+,k}}  \kk z^{p^*(n)} 
\partial_{n},
\]
and set $\bar G^{\le j}:=\exp(\bar\fog/\bar\fog^{>j})$,
$\bar G=\liminv \bar G^{\le j}$ as usual, with the degree function
$\bar d:N^{+,k}\rightarrow\NN$ given by $\bar d(\sum_i a_ie_i)=
\sum_{i\not =k} a_i$. We note that $\bar G$ acts on $\widehat{\kk[\oP]}$
as usual, and if $\foD$ is a scattering diagram in the sense of
Definition \ref{escdef}, then all automorphisms associated to crossing
walls (rather than slabs) lie in $\bar G$.

Besides the Lie algebra $\bar \fog$ just defined, recall we also have
$\fog=\bigoplus_{n\in N^+} \kk z^{p^*(n)}\partial_n$ as usual.
We have the degree map $d:N^+\rightarrow\NN$ given by
$d(\sum_i a_i e_i)=\sum a_i$, but we also have $\bar d:N^+\rightarrow
\NN$ given by the restriction of $\bar d:N^{+,k}\rightarrow\NN$.
We use the notation $\fog^{d>l}$ and $\fog^{\bar d>l}$ to distinguish
between the two possibilities for $\fog^{>l}$ determined by the two choices
of degree map. Then
$G=\invlim \exp(\fog/\fog^{d>j})$ 
and we define
$\tilde G=\invlim \exp(\fog^{\bar d>0}/\fog^{\bar d>j})$. Note that $G, \tilde G$
both act faithfully on $\widehat{\kk[P]}$, where the completion is respect to
the maximal monomial ideal $P\setminus \{0\}$, and $\tilde G, \bar G$
act faithfully on $\widehat{\kk[\oP]}$. There are inclusions
$\tilde G\subset G$ and $\tilde G\subset \bar G$. Only the second inclusion
holds at finite order, i.e., $\tilde G^{\le j}\subset \bar G^{\le j}$.

For each of the above Lie algebras $\fog'$, we can now also talk about scattering
diagrams for $\fog'$ using Definitions \ref{walldef} and \ref{KSscatdiagdef},
replacing $\fog$ with $\fog'$ in those definitions.

For a joint $\foj$, we define $\bar G_{\foj}$, $\tilde G_{\foj}$ as
subgroups of $\bar G$, $\tilde G$ defined analogously to \eqref{fogfojdef}.

Finally, we will need one other group. We define, for a fixed $j$,
\[
\hat G^{\le j}:=\invlim_{j'}\exp(\fog/ (\fog^{d > j'}+\fog^{\bar d> j})).
\]
There is an inclusion $\tilde G^{\le j}=
\exp(\fog^{\bar d>0}/\fog^{\bar d> j}) \subset \hat G^{\le j}$, and surjection $G\rightarrow \hat G^{\le j}$. 

We need to understand the interaction between elements of $G$ and
the automorphism associated to crossing the slab 
(see Lemma 2.15 of \cite{GSAnnals}). Recall the notation $G_{\foj}$
from Construction \ref{KSlemalg}; this is applied also to the various
assorted groups above. 

\begin{lemma}
\label{conjugationlemma}
Let $n\in N^{+,k}$ (resp.\ $N^+$)
and let $\theta\in\bar G$ (resp.\ $\theta\in \tilde G$)
be an automorphism of the form
$\exp(f\partial_n)$ for $f=1+\sum_{\ell\ge 1} c_\ell z^{\ell p^*(n)}$. 
Let $\foj=n^{\perp}\cap e_k^{\perp}$. If
$\{n,e_k\}>0$, then
\[
\hbox{$\theta_{\fod_k}^{-1}\circ\theta\circ\theta_{\fod_k}\in \bar G_{\foj}$
(resp.\ $\tilde G_{\foj}$)}
\]
while if $\{n,e_k\}<0$, then
\[
\hbox{$\theta_{\fod_k}\circ\theta\circ\theta_{\fod_k}^{-1}\in\bar G_{\foj}$,
(resp.\ $\tilde G_{\foj}$).}
\]
Here, we view $\theta_{\fod_k}^{-1}\circ\theta\circ\theta_{\fod_k}$
or $\theta_{\fod_k}\circ\theta\circ\theta_{\fod_k}^{-1}$ as automorphisms
of $\widehat{\kk[\oP]}_{1+z^{v_k}}$, and $\bar G$ or $\tilde G$ as
subgroups of the group of automorphisms of this ring.
\end{lemma}

\begin{proof}
Let us prove the first statement, the second being similar. It is
enough to check that 
\[
\hbox{$\theta_{\fod_k}^{-1}\circ (z^{p^*(n)}\partial_n)\circ
\theta_{\fod_k}\in \bar\fog_{\foj}$ (resp.\ $\tilde\fog_{\foj}$).}
\]
But, with $h=1+z^{v_k}$,
\begin{align*}
&(\theta_{\fod_k}^{-1}\circ (z^{p^*(n)}\partial_n)\circ
\theta_{\fod_k})(z^m)\\
= {} &
(\theta_{\fod_k}^{-1}\circ (z^{p^*(n)}\partial_n))
(z^mh^{-\langle d_ke_k,m\rangle})\\
= {} & \theta^{-1}_{\fod_k}\big(\langle n,m\rangle z^{m+p^*(n)}
h^{-\langle d_ke_k,m\rangle}\big)
-\theta_{\fod_k}^{-1}
\big(\langle d_ke_k,m\rangle\langle v_k,n\rangle
z^{m+p^*(n)+v_k}h^{-\langle d_ke_k,m\rangle-1}\big)\\
={} & z^m\big(\langle n,m\rangle z^{p^*(n)}h^{\langle d_ke_k,
p^*(n)\rangle}-\langle d_ke_k,m\rangle\langle v_k,n\rangle
z^{p^*(n)+v_k}h^{\langle d_ke_k,p^*(n)+v_k\rangle-1}\big).
\end{align*}
Noting that $\langle v_k,n\rangle=\{e_k,n\}=-\{n,e_k\}=
-\langle e_k,p^*(n)\rangle=-d_k^{-1}\langle d_ke_k,p^*(n)\rangle$, and
in addition $\langle d_ke_k,v_k\rangle=0$, we see that as a derivation,
writing $\alpha=\langle d_ke_k, p^*(n)\rangle >0$,
\begin{align}
\label{conjugationeq}
\begin{split}
&\theta_{\fod_k}^{-1}\circ(z^{p^*(n)}\partial_n)\circ \theta_{\fod_k}\\
= {} & z^{p^*(n)}h^{\langle d_ke_k,p^*(n)\rangle} \partial_n
+z^{p^*(n)+v_k}\langle d_ke_k,p^*(n)\rangle h^{\langle d_ke_k,p^*(n)\rangle -1}\partial_{e_k}\\
= {} &\sum_{\beta=0}^{\alpha} z^{p^*(n)+\beta v_k} \begin{pmatrix} \alpha\\
\beta\end{pmatrix} \partial_n
+\alpha\sum_{\beta=1}^\alpha z^{p^*(n)+\beta v_k}\begin{pmatrix}
\alpha-1\\ \beta-1\end{pmatrix} \partial_{e_k}\\
={} & \sum_{\beta=0}^{\alpha}
z^{p^*(n+\beta e_k)} \begin{pmatrix} \alpha\\ \beta\end{pmatrix}
\partial_{n+\beta e_k}.
\end{split}
\end{align}
Of course $n+\beta e_k\in\Lambda_{\foj}^{\perp}$ by definition of $\foj$,
so the derivation $z^{p^*(n+\beta e_k)}\partial_{n+\beta e_k}$
lives in $\bar\fog_{\foj}$ (resp.\ $\tilde\fog_{\foj}$).
\end{proof}

We now proceed with the proof of Theorem \ref{KSlemma2}. 

\emph{Step I. Strategy of the proof.} 
We will first construct $\overline\foD_{\s}$
using essentially the same algorithm as the one given in Construction
\ref{KSlemalg}, but working with the group $\tilde G$. The algorithm
is slightly more complex because of the slab, and needs to be carried
out in two steps. To show that the diagram constructed is consistent
at each step, we compare it with the scattering diagram $\foD_{\s}$ for
the group $G$
which we know exists, using $\tilde G^{\le j}$ as an intermediary group.
Because $\tilde G\subset \bar G, G$, we obtain a consistent scattering diagram
for $\bar G$ and $G$. While
$\overline\foD_{\s}$ is equivalent to $\foD_{\s}$ as a scattering diagram
for $G$ by construction, this does not show uniqueness of $\overline\foD_{\s}$,
as there may be a different choice with wall crossing automorphisms in $\bar
G$ but not in $\tilde G$, so it cannot be compared with $\foD_{\s}$.
Thus, the final step involves showing uniqueness directly for the group
$\bar G$, again as part of the inductive proof.

We will proceed by induction on $j$,
constructing for each $j$ a finite scattering diagram $\ofoD_j$ for
$\tilde G$
containing $\foD_{\inc,\s}$ such that the following
induction hypotheses hold:
\begin{enumerate}
\item 
For every joint $\foj$ of $\ofoD_j$, there is a simple loop $\gamma_{\foj}$
around $\foj$ small enough so that it only intersects walls and slabs
containing $\foj$ and such that 
$\theta_{\gamma_{\foj},\ofoD_j}$, as an automorphism of 
$\widehat{\kk[\oP]}_{1+z^{v_k}}$, lies in $\tilde G$ and is trivial in
$\tilde G^{\le j}$, or equivalently, by the inclusion $\tilde G^{\le j}
\subset \bar G^{\le j}$, trivial in $\bar G^{\le j}$.
\item If $\ofoD'_j$ is a scattering diagram for $\bar G$ which has the same incoming walls as $\ofoD_j$ and satisfies
(1) (with $\tilde G$ replaced by $\bar G$ everywhere), 
then $\ofoD'_j$ is equivalent to $\ofoD_j$ in $\bar G^{\le j}$.
\end{enumerate}

Recall that joints of $\ofoD_j$ are either 
parallel or perpendicular, Definition-Lemma \ref{pjdef}.

\medskip

\emph{Step II. The base case.}
For $j=0$, $\ofoD_0=\foD_{\inc,\s}$ does the job. Indeed, all walls are trivial
in $\tilde G^{\le 0}=\{\id\}$, leaving just the single initial slab, and
thus there are no joints.

\medskip

\emph{Step III. From $\ofoD_j$ to $\ofoD_{j+1}$: adding walls associated to
joints not contained in $e_k^{\perp}$.}
Now assume we have found
$\ofoD_{j}$ satisfying the induction hypotheses. We need to add a finite number 
of walls to get $\ofoD_{j+1}$.
We will carry out the construction of $\ofoD_{j+1}$ in two steps, following
Construction \ref{KSlemalg}.

First, let $\foj$ be a perpendicular joint of $\ofoD_j$
with $\foj\not\subseteq e_k^{\perp}$.
Let $\Lambda_{\foj}\subseteq M^{\circ}$ 
be the set of integral tangent vectors to $\foj$. If $\gamma_{\foj}$
is a simple loop around $\foj$ small enough so that it only intersects
walls containing $\foj$, we note that every wall-crossing automorphism
$\theta_{\gamma_{\foj},\fod}$ contributing to
$\theta_{\gamma_{\foj},\ofoD_{j}}$ lies in $\tilde G_{\foj}$. Thus
as in \eqref{basicKSeq2}, in
$\tilde G^{\le j+1}$ we can write
\begin{equation}
\label{basicKSeq}
\theta_{\gamma_{\foj},\ofoD_{j}}=\exp\left(\sum_{i=1}^s c_iz^{p^*(n_i)}
\partial_{n_i}\right)
\end{equation}
with $c_i\in\kk$, 
and $n_i\in
\Lambda_{\foj}^{\perp}$ with $\bar d(n_i)=j+1$ as
$\theta_{\gamma_{\foj},\ofoD_j}$ is the identity in
$\tilde G^{\le j}$ by the induction hypothesis. Finally,
$p^*(n_i)\not\in\Lambda_{\foj}$ because the joint is perpendicular. Let
\[
\foD[\foj]:=\{(\foj-\RR_{\ge 0}p^*(n_i),(1+z^{p^*(n_i)})^{\pm c_i})
\,|\,i=1,\ldots,s \}.
\]
Here $(1+z^{p^*(n_i)})^{\pm c_i}=\exp(\pm c_i\log(1+z^{p^*(n_i)}))$ makes sense
as a power series.
The sign is chosen in each wall so that its contribution to
$\theta_{\gamma_{\foj},\foD[\foj]}$ is $\exp(-c_iz^{p^*(n_i)}\partial_{n_i})$
to $\bar d$-order $j+1$.

We now take
\[
\ofoD'_{j}:=\ofoD_{j}\cup\bigcup_{\foj}
\foD[\foj],
\]
where the union is over all perpendicular joints not contained in $e_k^{\perp}$.
We have only added a finite number of walls.

\medskip

\emph{Step IV. From $\ofoD_j$ to $\ofoD_{j+1}$: adding walls associated to
joints contained in $e_k^{\perp}$.}
If we didn't have a slab, $\ofoD'_j$ constructed above would now do the
job as in the proof of Lemma \ref{KSalglemma}. However, the
elements of $\tilde G$ trivial in $\tilde G^{\le j}$
do not commute with $\theta_{\fod_k}$ to order
$j+1$ as automorphisms of $\widehat{\kk[\oP]}_{1+z^{v_k}}$ in any reasonable
sense. As a consequence, we will need to add some additional walls
coming from joints in $e_k^{\perp}$, some of which have arisen as
the intersection of $e_k^{\perp}$ with walls added in Step III.

Consider a perpendicular joint $\foj\subseteq e_k^{\perp}$ of $\ofoD_j'$.
Necessarily the linear span of $\foj$ is $e_k^{\perp} \cap n^{\perp}$
for some $n\in N^+$.
Furthermore, we can choose $n$ so that
any wall containing $\foj$ then has linear
span $(ae_k+bn)^{\perp}$ for some $a,b$ non-negative rational
numbers. The direction of such a wall is positively proportional to
$-p^*(ae_k+bn)$. We now distinguish between two cases. Note that
$\langle e_k,p^*(n)\rangle\not=0$ as the joint is not parallel, so
we call the joint $\foj$ \emph{positive} or \emph{negative} depending
on the sign of $\langle e_k,p^*(n)\rangle=\{n,e_k\}$. Note that
if the joint is positive (negative) then $\langle e_k, p^*(ae_k+b n)\rangle$
is positive (negative) for all $b>0$.

If the joint is positive, then
choose $\gamma_{\foj}$ so that the first wall crossed is $\fod_k$, 
passing from $\shH_{k,-}$ to $\shH_{k,+}$. We can write
\begin{equation}
\label{jointexpression}
\theta_{\gamma_{\foj},\ofoD'_j}=\theta_2
\circ\theta_{\fod_k}^{-1}\circ\theta_1\circ
\theta_{\fod_k},
\end{equation}
where $\theta_i\in \tilde G_{\foj}$ are compositions of wall-crossing
automorphisms. It then follows from $\langle e_k,p^*(ae_k+bn)
\rangle>0$ for all $a\ge 0$, $b>0$ and Lemma 
\ref{conjugationlemma} that $\theta_{\fod_k}^{-1}\circ \theta_1
\circ \theta_{\fod_k}\in \tilde G_{\foj}$, 
hence $\theta_{\gamma_{\foj},\ofoD_j'}
\in \tilde G_{\foj}$. If the joint is negative, then we use a slightly
different loop:
without changing the orientation of the loop $\gamma_{\foj}$,
change the endpoints so that $\gamma_{\foj}$ now starts and ends in
$\shH_{k,+}$, crossing $\fod_k$ just before its endpoint. Then
\[
\theta_{\gamma_{\foj},\ofoD_j'}=\theta_{\fod_k}\circ \theta_2\circ
\theta_{\fod_k}^{-1}\circ\theta_1,
\]
and again by Lemma \ref{conjugationlemma}, $\theta_{\gamma_{\foj},\ofoD_j'}
\in \tilde G_{\foj}$.

Thus in both cases, $\theta_{\gamma_{\foj},\ofoD_j'}\in \tilde G_{\foj}$ 
and is
the identity in $\tilde G^{\le j}$. Thus we still have \eqref{basicKSeq} and
we can produce a scattering diagram $\foD[\foj]$ in the same way
as for the joints
$\foj$ not contained in $e_k^{\perp}$. We then set
\[
\ofoD_{j+1}=\foD_j'\cup \bigcup_{\foj}
\foD[\foj],
\]
where the union is over perpendicular joints of $\ofoD_j'$ contained in
$e_k^{\perp}$.
\medskip

\emph{Step V. (1) of the induction hypothesis is satisfied.}
Consider a perpendicular joint $\foj$ of $\ofoD_{j+1}$.
First suppose $\foj\not\subseteq e_k^{\perp}$. We proceed as in the proof
of Lemma \ref{KSalglemma}.
If $\foj$ is contained in a joint of $\ofoD_j$, there is a unique
such joint, say $\foj'$,
and we constructed $\foD[\foj']$
above. If $\foj$ is not contained in a joint of $\ofoD_j$, we define
$\foD[\foj']$ to be the empty set.
There are three types of walls $\fod$ in 
$\ofoD_{j+1}$ containing $\foj$:
\begin{enumerate}
\item $\fod\in\ofoD_j\cup\foD[\foj']$.
\item $\fod\in\ofoD_{j+1}\setminus(\ofoD_{j}\cup\foD[\foj'])$,
but $\foj\not\subseteq
\partial\fod$. This type of wall does not contribute to
$\theta_{\gamma_{\foj},\ofoD_{j+1}}$ in $\tilde G^{\le j+1}$. Indeed,
the associated automorphism is in the center of $\tilde G^{\le j+1}$
and this wall contributes twice to $\theta_{\gamma_{\foj},
\ofoD_{j+1}}$,
with the two contributions inverse to each other, so the contribution cancels.
\item $\fod\in\ofoD_{j+1}\setminus(\ofoD_{j}\cup \foD[\foj'])$ and
$\foj\subseteq\partial\fod$. Since each added wall is of the form
$\foj''-\RR_{\ge 0}m$ for some joint $\foj''$ of $\ofoD_j$,
where $-m$ is the direction of the wall,
the direction of the wall is parallel to $\foj$, contradicting $\foj$
being a perpendicular joint. Thus this does not occur.
\end{enumerate}
 From this we see by construction of $\foD[\foj']$ that 
$\theta_{\gamma_{\foj},\ofoD_{j+1}}$ is the identity in $\tilde G^{\le j+1}$.

On the other hand, suppose $\foj$ is a perpendicular
joint of $\ofoD_{j+1}$
contained in $e_k^{\perp}$. Then since no wall of $\ofoD_{j+1}\setminus
\ofoD_j'$ is contained in $e_k^{\perp}$, by definition of $N^{+,k}$,
in fact $\foj$ is a joint of $\ofoD_j'$.
Thus we see again by construction of $\foD[\foj]$ that 
to order $j+1$, $\theta_{\gamma_{\foj},\ofoD_{j+1}}$ is the identity for
$\gamma_{\foj}$ the loop around $\foj$ described in Step IV. Recall
the choice of loop depends on whether the joint is positive or negative.

Now we show that $\ofoD_{j+1}$ satisfies the induction hypothesis (1),
using the above existence of a $\gamma_{\foj}$ such that 
$\theta_{\gamma_{\foj},\overline{\foD}_j+1}=\id$ for each
perpendicular joint $\foj$.
Note that there is a map $\tilde G^{\le j+1}\rightarrow
\exp(\fog/(\fog^{d > j'} +\fog^{\bar d>j+1}))=:\hat G_{j'}$ for any $j'$.
The slab automorphism $\theta_{\fod_k}$ can be viewed as an element of
$\hat G_{j'}$ for any $j'$, and hence $\ofoD_{j+1}$ can be viewed as a
scattering diagram for $\hat G_{j'}$ in the sense of Definition 
\ref{KSscatdiagdef}.
We will first
show that $\ofoD_{j+1}$ is consistent as a diagram for $\hat G_{j'}$ 
inductively on $j'$. 

The base case is $j'=j$. All walls of $\ofoD_{j+1}\setminus \ofoD_j$
are trivial to $\bar d$-order $j$ and hence to $d$-order $j$. Now 
$\ofoD_j$ satisfies the main induction
hypothesis (1) at order $j$, which implies via the natural map 
$\tilde G^{\le j}\rightarrow
\hat G_{j}=G^{\le j}$ that $\ofoD_{j+1}$ is consistent as a diagram for
$\hat G_j$. Indeed, as $\ofoD_{j+1}$ is a finite scattering diagram, it is
enough to check that $\theta_{\gamma_{\foj},\ofoD_{j+1}}$ is the
identity in $G^{\le j}$
for any small loop $\gamma_{\foj}$ around any joint $\foj$. By
the hypothesis (1), this is the case for some loop $\gamma_{\foj}$, and
hence for all loops. Note that by uniqueness of consistent scattering diagrams
with the same incoming walls, we also record for future use:
\begin{equation}
\label{equivtoDs}
\hbox{$\ofoD_{j+1}$ is equivalent to $\foD_{\s}$ as diagrams for $G^{\le j}$.}
\end{equation}

The induction step follows from Lemma
\ref{equivalencetonextorder}, applied to $\tilde\foD=\ofoD_{j+1}$, $\foD=
\foD_{\s}$, and the group being $\hat G_{j'}$ (a quotient of
$G$, so the argument of Lemma \ref{equivalencetonextorder}
still applies). Indeed, if we assume
$\ofoD_{j+1}$ is consistent in $\hat G_{j'}$, then it is equivalent to
$\foD_{\s}$ as a scattering diagram in $\hat G_{j'}$. Furthermore,
$\foD_{\s}$ is consistent to all orders by Theorem \ref{KSlemmaGS},
and has the
same set of incoming walls as $\ofoD_{j+1}$ by construction. Finally,
$\theta_{\gamma_{\foj},\ofoD_{j+1}}$ is the identity in $\hat G_{j'+1}$
for any perpendicular joint $\foj$,
as shown above. Thus $\ofoD_{j+1}$ and $\foD_{\s}$ are equivalent
in $\hat G_{j'+1}$, and in particular $\ofoD_{j+1}$ is consistent in
$\hat G_{j'+1}$.

Thus taking the inverse limit, we see that $\ofoD_{j+1}$ is consistent
as a scattering diagram for $\hat G^{\le j+1}$. This almost completes
the proof of the induction hypothesis (1) in degree $j+1$. Indeed,
as $\tilde G^{\le j+1}$ is a subgroup of $\hat G^{\le j+1}$, certainly
$\theta_{\gamma_{\foj},\ofoD_{j+1}}$ is the identity for any joint not
contained in $e_k^{\perp}$, including the parallel joints. For
a perpendicular joint contained in $e_k^{\perp}$, if we choose $\gamma_{\foj}$
as given in Step IV, $\theta_{\gamma_{\foj},\ofoD_{j+1}}$ lies in
$\tilde G$ and is the identity in $\tilde G^{\le j+1}$ by the construction
of $\foD[\foj]$ in Step IV.
Finally, for a parallel joint $\foj$
contained in $e_k^{\perp}$,
note that all wall and slab-crossing automorphisms associated to walls 
containing
$\foj$ commute, and in particular the contribution of $\theta_{\fod_k}$
and $\theta_{\fod_k}^{-1}$ in $\theta_{\gamma_{\foj},\ofoD_{j+1}}$
as an automorphism of $\widehat{\kk[\oP]}_{1+z^{v_k}}$
cancel, so that the latter automorphism lies in $\tilde G$. 
Hence the image of this automorphism in $\tilde G^{\le j+1}\subset
\hat G^{\le j+1}$ must also be trivial. This gives the induction hypothesis
(1).

\medskip

\emph{Step VI. Uniqueness.}
Suppose we have constructed
two scattering diagrams $\ofoD_{j+1}, \ofoD_{j+1}'$ for $\bar G$
from $\ofoD_j$
which satisfy the inductive hypothesis (1) to $\bar d$-order $j+1$,
but with the group $\tilde G$ replaced with $\bar G$.
By the induction
hypothesis (2), these two scattering diagrams are equivalent to 
$\bar d$-order $j$,
and we wish to show they are equivalent to $\bar d$-order $j+1$. 
One first constructs a finite scattering diagram $\foD$ consisting only
of outgoing walls whose attached functions are of the form $1+cz^{p^*(n)}$
with $c\in\kk$ and $\bar d(n)=j+1$, with the property that
$\ofoD_{j+1}\cup \foD$ is equivalent to $\ofoD_{j+1}'$ to $\bar d$-order
$j+1$. This is done precisely as in the proof of Lemma 
\ref{equivalencetonextorder}.
We need to show $\foD$ is
equivalent to the empty scattering diagram to $\bar d$-order $j+1$.

To show this, first note that for any loop $\gamma$ which does not cross
the slab $\fod_k$,
$\theta_{\gamma,\ofoD_{j+1}}=\theta_{\gamma,\ofoD_{j+1}'}=\id$ to
$\bar d$-order $j+1$ implies that
$\theta_{\gamma,\foD}=\id$ to $\bar d$-order $j+1$. Indeed, all wall-crossing
automorphisms of $\foD$ are central in $\bar G^{\le j+1}$.
Now if $n\in N^{+,k}$ with $\bar d(n)=j+1$,
let $\foD_n\subseteq \foD$ be the set of walls in $\foD$
with attached functions
of the form $1+cz^{p^*(n)}$. Note all wall-crossing automorphisms of $\foD$,
viewed as elements of $\bar G^{\le j+1}$,
lie in $\exp(\bar\fog^{> j}/\bar\fog^{> j+1})$, which as a group
coincides with the additive group structure on 
$\bar\fog^{> j}/\bar\fog^{> j+1}$. Thus for any path $\gamma$ not
crossing $\fod_k$, we obtain a unique decomposition $\theta_{\gamma,\foD}=
\prod_n \theta_{\gamma,\foD_n}$ from the $N^{+,k}$-grading on 
$\bar\fog^{>j}/\bar\fog^{>j+1}$, 
and if $\theta_{\gamma,\foD}$ is the identity, 
so is each $\theta_{\gamma,\foD_n}$.

Fixing $n$ as above, replace $\foD_n$ with an equivalent scattering
diagram with smallest possible support, and let $C_n=\Supp(\foD_n)$.
So if $x\in n^{\perp}$ is a general point, $x\in C_n$ if and only if
$g_x(\foD_n)$ is not the identity.
Assume first that $\langle e_k,p^*(n)\rangle\ge 0$. We shall show 
$C_n\subseteq
\shH_{k,-}$. Assume not.
Taking a general point $x\in C_n\setminus\shH_{k,-}$, it is not
possible for the ray $L:=x+\RR_{\ge 0}p^*(n)$ to be contained in $C_n$. This is
because $\foD$ consists of only a finite number of walls, none of which
are incoming. Let $\lambda = \max\{ t\in \RR_{\ge 0}\,|\, x+ tp^*(n)\in C_n\}$,
and $y=x+\lambda p^*(n)$. This makes sense as $t=0$ is in the set over which
we are taking the maximum, as we are assuming $x\in C_n$. 
Then necessarily $y$ is in a joint $\foj$ of
$\foD_n$, and every wall of $\foD_n$ containing $\foj$ is contained
in $\RR\foj-\RR_{\ge 0}p^*(n)$. Furthermore, since $\langle e_k, x\rangle>0$,
$\langle e_k, p^*(n)\rangle\ge 0$, it follows that
$y\not\in e_k^{\perp}$ and $\foj$ is not contained in $e_k^{\perp}$.
Thus given a loop $\gamma_{\foj}$ around $\foj$, $\theta_{\gamma_{\foj},\foD_n}$
is the identity. This implies that in fact to $\bar d$-order $j+1$,
\[
\prod_{\fod\in \foD_n\atop \foj\subseteq\fod} \theta_{\gamma_{\foj},\fod}
=\id.
\]
In particular, a point $z=y-\epsilon p^*(n)$ for small $\epsilon$ is
contained in precisely those walls of $\foD_n$ containing $\foj$.
But then $g_z(\foD_n)=\id$, contradicting minimality of $C_n$.
Thus one finds that $C_n\subseteq \shH_{k,-}$. Similarly,
if $\langle e_k,p^*(n)\rangle\le 0$, then $C_n\subseteq \shH_{k,+}$. 
In particular,
if $\langle e_k,p^*(n)\rangle=0$, $C_n\subseteq e_k^{\perp}$, but there
are no walls contained in $e_k^{\perp}$, so in this case $\foD_n=\emptyset$.

Now consider a joint $\foj$ of $\ofoD_{j+1}'$ contained in $e_k^{\perp}$. 
There are three cases: either $\foj$ is perpendicular and positive, 
perpendicular and negative, or parallel. Consider the first case.
Take a loop $\gamma_{\foj}$ around $\foj$ as in the positive case in
Step IV. Because of positivity, if a wall $\fod$ of $\foD$ contains
$\foj$, then with $n$ chosen so that $\fod\in \foD_n$,
we must have $\langle e_k, p^*(n)\rangle >0$ and hence $\fod$
is contained in $\shH_{k,-}$. Thus we have that to $\bar d$-order $j+1$,
\[
\id=\theta_{\gamma_{\foj},\ofoD'_{j+1}}=\theta_{\gamma_{\foj},\ofoD_{j+1}\cup\foD}
=\theta_{2,\foD}\circ\theta_{2,\ofoD_{j+1}}\circ\theta_{\fod_k}^{-1}
\circ\theta_{1,\ofoD_{j+1}}\circ\theta_{\fod_k}=\theta_{2,\foD}
=\theta_{\gamma_{\foj}, \foD}
\]
as in \eqref{jointexpression}, where $\theta_{i,\foD}$ and 
$\theta_{i,\ofoD_{j+1}}$
denote the contributions coming from the scattering diagrams
$\foD$ and $\ofoD_{j+1}$ and the pieces of $\gamma_{\foj}$
not crossing $e_k^{\perp}$. The same argument works for negative
joints, while a parallel joint cannot contain any wall of $\foD$,
(as we showed above that $\foD_n=\emptyset$ if $\langle e_k,p^*(n)\rangle=0$),
so that $\theta_{\gamma_{\foj},\foD}=\id$ trivially.
We can now repeat the argument of the previous paragraph, taking for
any $n$ a general point $x\in C_n$ rather than $x\in C_n\setminus \shH_{k,-}$.
This allows us to conclude that $\foD_n=\emptyset$ for all $n$, proving
uniqueness.

\medskip

\emph{Step VII. Finishing the proof of Theorem \ref{KSlemma2}.}
Having completed the induction step, we take $\ofoD_{\s}=
\bigcup_{j=0}^\infty \ofoD_j$.  
We need to check it satisfies the stated conditions in Theorem
\ref{KSlemma2}. Certainly conditions (1) and (2) hold by construction.

For (3), first recall that because by construction $\ofoD_{\s}$ can be viewed
as a scattering diagram for $\tilde G$, it can also be viewed as a 
scattering diagram for $G$ via the inclusion $\tilde G\subset G$,
and in addition $\theta_{\fod_k}\in G$, so that $\ofoD_{\s}$ is
viewed as a scattering diagram for $G$ in the sense of Definition 
\ref{KSscatdiagdef}, i.e., with no slab. Now as a scattering diagram for $G$,
$\ofoD_{\s}$ is equivalent to $\foD_{\s}$ by \eqref{equivtoDs}. 
By consistency of $\foD_{\s}$, $\theta_{\gamma,
\ofoD_{\s}}$ is independent of the endpoints of $\gamma$ as an element
of $G$. Now suppose $g_1, g_2$ are two automorphisms of 
$\widehat{\kk[\oP]}_{1+z^{v_k}}$ which induce automorphisms of
$\widehat{\kk[P]}$, (i.e., for $p\in P\subset \oP$,
 $g_i(z^p)\in \widehat{\kk[P]}$, giving a map $g_i:\widehat{\kk[P]}\rightarrow
\widehat{\kk[P]}$ which is an automorphism)
and agree as automorphisms of the latter ring. Then $g_1, g_2$ agree
as automorphisms of $\widehat{\kk[\oP]}_{1+z^{v_k}}$. Thus in particular,
$\theta_{\gamma,\ofoD_{\s}}$ is independent of the endpoints of 
$\gamma$ as an automorphism of $\widehat{\kk[\oP]}_{1+z^{v_k}}$. This
gives condition (3).

The uniqueness of $\ofoD_{\s}$ with these properties then follows from
the induction hypothesis (2). Indeed, if $\ofoD_{\s}'$ satisfies conditions
(1)-(3) of Theorem \ref{KSlemma2}, then working by induction
on the order $j$, the
induction hypothesis (1) holds for $\ofoD_{\s}'$ (the existence of
$\gamma_{\foj}$ with $\theta_{\gamma_{\foj},\ofoD'_{\s}}\in \tilde{G}$
only being an issue for joints contained in $e_k^{\perp}$, and Step IV
explains how to choose the loop $\gamma_{\foj}$). Thus by induction
hypothesis (2), $\ofoD_{\s}$ and $\ofoD_{s}'$ are equivalent to order $j$.

This completes the proof of Theorem \ref{KSlemma2}. \qed

\subsection{The proof of Theorem \ref{scatdiagpositive}}

The key point of the proof
is just the positivity of the simplest scattering diagram
as described in Example \ref{basicscatteringexample}, which we use to
analyze general two-dimensional scattering diagrams. We will consider
a somewhat more general setup, but only in two dimensions, than considered in 
the rest of this paper. In particular, we will follow the notation of 
\cite{G11}, \S 6.3.1, taking $M=\ZZ^2$, $N=\Hom(M,\ZZ)$, and assume given
a monoid $P$ with a map $r:P\rightarrow M$, $\fom=P\setminus P^{\times}$.
We will consider scattering diagrams $\foD$ for this data as in
\cite{G11}, Def.\ 6.37, consisting of rays and lines which do not necessarily
pass through the origin. Given any scattering diagram $\foD_{\inc}$, 
the argument of Kontsevich and Soibelman from \cite{KS06} (see \cite{G11},
Theorem 6.38 for an exposition of this particular case) adds rays to
$\foD_{\inc}$ to obtain a
scattering diagram $\Scatter(\foD_{\inc})$ 
such that $\theta_{\gamma,\Scatter(\foD_{\inc})}$ is the identity for every
loop $\gamma$. This diagram is unique up to equivalence.

The fundamental observation involves a kind of universal scattering diagram:

\begin{proposition}
\label{basicpositivity}
In the above setup, suppose given $p_i\in \fom\subseteq P$, $1\le i \le s$,
with $r(p_i)\not=0$, and positive integers $d_1,\ldots,d_s$.
Consider the scattering diagram 
\[
\foD_{\inc}:=\{(\RR r(p_i), (1+z^{p_i})^{d_i})\,|\, 1\le i\le s\},
\]
$\foD:=\Scatter(\foD_{\inc})$.
We can choose $\foD$ within its equivalence class 
so that for any given ray $(\fod,f_{\fod})\in \foD\setminus
\foD_{\inc}$,
we have
\[
f_{\fod}=(1+z^{\sum_{i=1}^s n_ip_i})^c
\]
for $c$ a positive integer and the $n_i$ non-negative integers with
at least two of them non-zero.
\end{proposition}

\begin{proof}

\emph{Step I. The change of monoid trick}. 
Note that if the $r(p_i)$ generate a rank one sublattice of $M$, then
all the wall-crossing automorphisms of $\foD_{\inc}$ commute and $\foD=
\foD_{\inc}$, so we are done. So assume from now on that the $r(p_i)$
generate a rank two sublattice of $M$.

Let $P'=\NN^s$, generated by
$e_1,\ldots,e_s$, define a map $u:P'\rightarrow P$ by $u(e_i)=p_i$, and
a map $r':P'\rightarrow M$ by $r'(e_i)=r(p_i)$.
We extend $u$ to a map $u:\widehat{\kk[P']}\rightarrow\widehat{\kk[P]}$,
and define, for a scattering diagram $\foD$ for the monoid $P'$,
$u(\foD):=\{(\fod,u(f_{\fod}))\,|\,(\fod,f_{\fod})\in \foD\}$.
Clearly if $\theta_{\gamma,\foD}=\id$, then $\theta_{\gamma,u(\foD)}=\id$.
Thus if $\foD'=\Scatter(\{(\RR r(p_i), (1+z^{e_i})^{d_i})\,|\,1\le i\le s\})$,
then $u(\foD')$ is equivalent to $\Scatter(\foD_{\inc})$, 
by uniqueness of $\Scatter$ up to
equivalence. So it is sufficient to show the result with $P=\NN^s$,
$p_i=e_i$.

\medskip

\emph{Step II. Everything but the positivity of the exponents}.
We can construct $\foD$ specifically using the original method of
\cite{KS06}, already explained here in Steps III and IV of the proof of
Theorem \ref{KSlemma2}: we construct $\foD$ order by order, constructing
$\foD_d$ so that $\theta_{\gamma,\foD_d}$ is the identity modulo 
$\fom^d$ for $\gamma$ a loop around the origin. Given a description
\begin{equation}
\label{thetaeq2}
\theta_{\gamma,\foD_d}=\exp\left(\sum c_iz^{m_i}\partial_{n_i}\right)\mod
\fom^{d+1},
\end{equation}
with the $n_i$ primitive and the $m_i$ all distinct,
we add a collection of rays 
\[
\{(-\RR_{\ge 0}r(m_i), (1+z^{m_i})^{\pm c_i})\}
\]
for some $c_i\in\kk$. However, inductively, we can
show the $c_i$ can be taken to be integers. Indeed, if all rays in
$\foD_d$ have this property, then $\theta_{\gamma,\foD_d}$ is in fact
an automorphism of $\widehat{\ZZ[P]}$, 
and thus the $c_i$ appearing in \eqref{thetaeq2}
of $\theta_{\gamma,\foD_d}$ are also integers.

Next let us show that any exponent $m_i$ is of the form $\sum n_je_j\in P$
with at least two of the $n_j$ non-zero.
The pro-nilpotent group $\VV$ in which all automorphisms live is
given by the Lie algebra
\[
\fov=\bigoplus_{m\in\fom\atop r(m)\not=0} z^m\kk\otimes r(m)^{\perp}
\subseteq \Theta(\kk[P]),
\]
following the notation of \cite{G11}, pp.\ 290-291. This contains a
subalgebra $\fov'$ where the sum is taken over all $m\in\fom$ not
proportional to one of the $e_i$. Then clearly $[\fov,\fov']\subseteq
\fov'$, so the corresponding pro-nilpotent group $\VV'$ is normal in $\VV$. 
Furthermore,
$\fov/\fov'$ is abelian, hence so is $\VV/\VV'$. For any loop $\gamma$,
the image of $\theta_{\gamma,
\foD_{\inc}}$ is thus the identity in $\VV/\VV'$, as
every wall in $\foD_{\inc}$ contributes twice to $\theta_{\gamma,\foD_{\inc}}$,
but with inverse automorphisms. Assume inductively that
$\foD_d\setminus\foD_{\inc}$ only contains rays whose attached functions
$(1+z^{m_i})^{c_i}$ have $m_i$ not proportional to any $e_j$. Then the
wall-crossing automorphisms associated to these rays lie in $\VV'$, so
$\theta_{\gamma,\foD_d}$ is the identity in $\VV/\VV'$, i.e., lies in
$\VV'$. Thus the expression $\sum c_iz^{m_i}\partial_{n_i}$ of \eqref{thetaeq2}
lies in $\fov'$, hence the inductive step follows.

It remains to show that each wall added is of the form 
$(\fod,(1+z^m)^{c})$ with $c$ positive.

\medskip

\emph{Step III. The perturbation trick}. 
We will now show the result for all monoids $P=\NN^\alpha$ for all
$\alpha$, all choices of $r:P\rightarrow M$,
all choices of $p_i\in P\setminus\{0\}$ with $r(p_i)\not=0$, and all
positive choices of $d_i$. (Note by Step I this is a bit more than we need, as
we don't take the $p_i$ to necessarily be generators of $P$). 
All cases are dealt with simultaneously by induction.

We define for
$p\in P$ the \emph{order} $\ord(p)$, which is the unique $n \in \bZ_{\ge 0}$ such that
$p\in \fom^n\setminus \fom^{n+1}$. For a ray $(\fod,(1+z^p)^c)$, we write
$\ord(\fod):=\ord(p)$, and say $\fod$ is a ray of order $\ord(\fod)$.
We will go by induction on the order,
showing that a ray $(\fod,(1+z^p)^c)$ in $\foD$ of order $\le k$
for any choice of data
has $c$ positive. This is obviously the case for $k=1$, as all elements of
$\foD\setminus\foD_{\inc}$ have order at least $2$. So assume the induction
hypothesis is true for all orders $<k$, and we need to show rays added
of order $k$ have positive exponent.

We will
use the \emph{perturbation trick} repeatedly. Given a scattering diagram
$\foD_{\inc}$ for which we would like to compute $\foD=\Scatter(\foD_{\inc})$,
choose general $v_{\fod}\in M_{\RR}$ for each $\fod\in \foD_{\inc}$.
Define $\foD_{\inc}':=\{(\fod+v_{\fod},f_{\fod})\,|\, \fod\in\foD_{\inc}\}$;
this is the perturbed diagram.
We can then run the Kontsevich-Soibelman algorithm for $\foD'_{\inc}$,
for example as described in \cite{G11}, Theorem 6.38. This gives
a scattering diagram $\foD'=\Scatter(\foD'_{\inc})$ with the property that
$\theta_{\gamma,\foD'}$ is the identity for every loop $\gamma$. This is
the case in particular for $\gamma$ a very large loop
around the origin which contains all singular points of $\foD'$. We can
assume as usual that $\foD'$ has been constructed only by adding rays
of the form $(1+z^m)^c$.

Then up to equivalence,
$\foD$ can be obtained from $\foD'$ by taking the \emph{asymptotic scattering
diagram of $\foD'$}, i.e., just translate each line of $\foD'$ so it passes
through the origin and each ray of $\foD'$ so its endpoint is the origin.
See \S1.4 of \cite{GPS} for more details. If after performing this translation,
we obtain a number of rays with the same support of the form
$(\fod, (1+z^m)^{c_i})$, $i$ in some index set, we can replace all these
rays with a single ray $(\fod, (1+z^m)^{\sum c_i})$ without affecting the
equivalence class.
Thus if we want to show positivity of the exponents for $\foD$, it is
enough to show the desired positivity for $\foD'$. 

We will typically use an induction hypothesis to show positivity for
$\foD'$. Indeed, for each order, we will run the Kontsevich-Soibelman
algorithm at each singular point, and the behaviour at each singular point
is equivalent to a scattering diagram of the general type being considered.
Indeed, if $p$ is a singular point of some $\foD'_d$ constructed to order
$d$, we obtain a local version $\foD_{p}^{\loc}$ of the scattering diagram 
at $p$ by
replacing each $\fod$ with $p\in\fod$ with $\fod-p$, and replacing such
translated rays with the line spanned by the ray if the translated ray does 
not have the origin as its endpoint. As long as all attached functions
of rays and lines passing through $p$ are of the form $(1+z^m)^c$ with
$c$ a positive integer, we are back in the original situation of the
proposition. We shall write $\foD_{p,\inc}^{\loc}$ for the
set of lines in $\foD_{p}^{\loc}$.

We first observe that using the perturbation trick it is enough to show the
induction hypothesis for order $k$ 
when at most two of the $p_i$ have $\ord(p_i)=1$. Indeed,
after perturbing, the lines of $\foD'_{\inc}$ only intersect
pairwise, but as more
rays are added as the Kontsevich-Soibelman algorithm is run, one might have
more complicated behaviour at singular points. However, any ray added
has order $>1$. Thus we only have to analyze 
initial scattering diagrams $\foD^{\loc}_{p,\inc}$
with at most two lines of order $1$.
  
Next we observe the induction hypothesis allows us to show the result
only for $s=2$, with both lines having order $1$. Indeed,
write $\foD^{\loc}_{p,\inc}$ as $(\fod_i, (1+z^{p_i})^{c_i})$ and order 
the $p_i$
so that $\ord(p_1)\le\ord(p_2)\le\cdots$. Apply Step I, getting a map
$u:P'\rightarrow P$ with $u(e_i)=p_i$. 
We are trying to prove that rays with $P$-order $k$ have positive
exponent. But consider a ray 
$(\fod,(1+z^{\sum n_ip_i})^c)$ which is the image under $u$ of a ray 
$(\fod, (1+z^{\sum n_i e_i})^c)$ appearing in 
$\Scatter(\{(1+z^{e_i})^{c_i}\,|\,1\le i \le s \})$ 
with $\ord_P(\sum n_ip_i)=k$
and at least one of $n_j, j\ge 3$ non-zero. Then $\ord_{P'} \sum n_ie_i<k$,
so by the induction hypothesis, we can assume $c$ is positive. On the
other hand, rays of the form $(\fod, (1+z^{\sum n_i e_i})^c)$ with $n_j=0$
for $j\ge 3$
appearing in $\Scatter(\{(1+z^{e_i})^{c_i}\,|\,1\le i \le s \})$ already
appear in $\Scatter(\{(1+z^{e_i})^{c_i}\,|\,1\le i \le 2 \})$, as follows
easily by working modulo the ideal in $P'$ generated by the $e_j, j\ge 3$.
Thus we are only concerned about rays which arise from scattering
the two order $1$ lines.
Thus it is sufficient to show the result when $s=2$.

\medskip

\emph{Step IV. The change of lattice trick}.
To deal with the case where $\foD_{\inc}$ consists of two lines,
we use the \emph{change of lattice} trick to reduce to a simpler expression
for the scattering diagram. By Step I, we can take $P=\NN^2$, $p_i=e_i$.
Let $M^{\circ}\subseteq M$ be the sublattice
generated by $v_1=r(e_1), v_2=r(e_2)$. 
Note as in Step I we can assume that this is a rank $2$
sublattice, as otherwise the automorphisms associated to the two lines
commute.
Then $N^{\circ}:=\Hom(M^{\circ},\ZZ)$ is a superlattice
of $N$, with dual basis $v_1^*,v_2^*$. In what follows, we will talk
about scattering diagrams defined using both the lattice $M$ and $M^{\circ}$.
Bear in mind that a wall $(\fod,f_{\fod})$ could be interpreted using
either lattice, and the automorphism induced by crossing such a wall depends
on which lattice we are using, as primitive vectors in $N$ differ
from primitive vectors in $N^{\circ}$.

To see the relationship between these automorphisms,
for $w\in N^{\circ}\setminus\{0\}$, let 
\[
e(w)=\min\{e>0\,|\, ew\in N\}.
\]
Then a wall $(\fod,f_{\fod})$ for $M$ 
induces a wall-crossing automorphism of $\widehat{\kk[P]}$
which is the same as the automorphism induced by the 
wall $(\fod, f_{\fod}^{e(n_{\fod})})$ for $M^{\circ}$, where 
$n_{\fod}\in N^{\circ}$ is primitive and annihilates $\fod$.

Consider
\[
\foD_{\inc}^{\circ}:=\{(\RR v_1, (1+z^{e_1})^{d_1e(v_2^*)}),
(\RR v_2, (1+z^{e_2})^{d_2e(v_1^*)})\}
\]
as a scattering diagram for the lattice $M^{\circ}$.
Let $\foD^{\circ}=\Scatter(\foD^{\circ}_{\inc})$. Let $\foD'$ be the
scattering diagram for $M$ obtained by replacing every wall
$(\fod,(1+z^p)^c)\in\foD^{\circ}$ with $(\fod,(1+z^p)^{c/e(n_{\fod})})$.
Thus the wall-crossing automorphism for each wall in $\foD'$ as a scattering
diagram for the lattice $M$
is the same automorphism for the corresponding wall in $\foD^{\circ}$.
Then $\theta_{\gamma,\foD'}$ is the identity.
Thus by uniqueness of the scattering process up to equivalence, 
$\foD'$ is equivalent to $\Scatter(\foD_{\inc})$. (Note this implies that
$c/e(n_{\fod})\in\ZZ$ also, as $\Scatter(\foD_{\inc})$ only involves integer
exponents.)

Thus it is enough to prove the desired positivity for the scattering diagram
$\foD^{\circ}$. To do so, we use a variant of the perturbation trick,
factoring the two lines in $\foD_{\inc}^{\circ}$. We choose
general $v^1_{j_1},v^2_{j_2}\in M_{\RR}$, with $1\le j_1\le 
d_1e(v_2^*)$, $1\le j_2\le d_2e(v_1^*)$. Define
\[
\tilde\foD_{\inc}^{\circ} :=
\{(v^1_j+\RR v_1, 1+z^{e_1})\,|\,
1\le j\le d_1e(v_2^*)\}
\cup
\{(v^2_j+\RR v_2, 1+z^{e_2})\,|\,
1\le j\le d_2e(v_1^*)\}.
\] 
Again, we initially only have pair-wise intersections. The first
stage of this algorithm will then only involve points where two lines of
the form $(v^1_j+\RR v_1, 1+z^{e_1})$ and 
$(v^2_{j'}+\RR v_2, 1+z^{e_2})$ intersect. The algorithm only adds
one ray in the direction $-v_1-v_2$ with endpoint the intersection point
and attached function $1+z^{e_1+e_2}$, as follows from Example
\ref{basicscatteringexample}. This now accounts for all new rays of order
$2$. We continue to higher degree, but now we can use the induction hypothesis
at every singular point $p$ as we did in Step III, because every
line in $\foD_{p,\inc}^{\loc}$ has order $\ge 2$ except for possibly
one or two of the given lines of order $1$, and we have already accounted
for all rays produced by collisions of two lines of order $1$.
\end{proof}

\begin{corollary}
\label{positivitywithcoefs}
In the situation of Proposition \ref{basicpositivity},
suppose instead that 
\[
\foD_{\inc}:=\{(\RR r(p_i), (1+\alpha_i z^{p_i})^{d_i})\,|\, 1\le i \le s\},
\]
where now $\alpha_i\in\kk$, the ground field. Choosing 
$\foD=\Scatter(\foD_{\inc})$ up to equivalence, we can assume
that each ray $(\fod,f_{\fod})\in\foD\setminus\foD_{\inc}$ satisfies
\[
f_{\fod}= (1+\prod_i (\alpha_i z^{p_i})^{a_i})^c
\]
for some choice of non-negative integers $a_i$ and
where $c$ is a positive integer.
\end{corollary}

\begin{proof}
This follows easily from from Proposition \ref{basicpositivity}. 
First, using the change of monoid trick (Step I of the proof of
Proposition \ref{basicpositivity}), we may assume $P=\NN^s$ and $p_i=e_i$.
Consider the automorphism $\nu:\widehat{\kk[P]}\rightarrow
\widehat{\kk[P]}$ defined by $\nu(z^{e_i})=\alpha_iz^{e_i}$. Applying
$\nu$ to the function attached to each wall of $\Scatter(\{(\RR r(e_i),
(1+z^{e_i})^{d_i})\})$ gives a scattering diagram $\foD'$ whose incoming
walls are precisely those of $\foD_{\inc}$, and $\theta_{\gamma,\foD'}=\id$
for $\gamma$ a loop around the origin. Thus we can take $\foD=\foD'$ and
the result follows from Proposition \ref{basicpositivity}.
\end{proof}

\emph{Proof of Theorem \ref{scatdiagpositive}.}
In fact one can use the $\foD_{\s}$ as constructed explicitly in the
algorithm of the proof of Theorem \ref{KSlemma2}. The only issue is that
we need to 
know that the walls added at each joint have the desired positivity property. 
Note that the statement of
Theorem \ref{scatdiagpositive} involves scattering diagrams without slabs,
while the proof of Theorem \ref{KSlemma2} given involves a slab. So for the
purpose of this discussion, we can ignore all issues concerning the slab
in the proof of Theorem \ref{KSlemma2}, and the only thing we need to do is look
at the procedure for producing $\foD[\foj]$ in Step II of the proof of
Theorem \ref{KSlemma2}.

For a perpendicular joint $\foj$ of $\foD_d$, we can split 
$M=\Lambda_{\foj}\oplus M'$, where $M'$ is a rank two lattice. 
For each wall $\fod\in\foD_d$ containing $\foj$, we can inductively assume
that $f_{\fod}=(1+z^m)^c$ for some positive integer $c$, and split
$z^m=z^{m_{\foj}}z^{m'}$, with $m_{\foj}\in \Lambda_{\foj}$ and
$m'\in M'$. Because $\foj$ is perpendicular, we have $m'\not=0$.
We will apply Corollary \ref{positivitywithcoefs} to the 
case where the monoid $P$ is the one being used in Theorem \ref{KSlemma},
and $r:P\rightarrow M'$ is the projection. 
We can then view the computation at the joint as
a two-dimensional scattering situation in the lattice $M'$ over the
ground field $\kk(\Lambda_{\foj})$, the quotient field of $\kk[\Lambda_{\foj}]$.
To obtain the relevant two-dimensional scattering diagram we replace each
wall $(\fod,f_{\fod})$ with $\foj\subseteq\fod$ with $\big((\fod+\Lambda_{\foj}
\otimes\RR)/(\Lambda_{\foj}\otimes\RR),f_{\fod}\big)$ in $M'_{\RR}
=M_{\RR}/\Lambda_{\foj}\otimes\RR$. 
We are then in the situation of Corollary \ref{positivitywithcoefs}, and
the result follows. 
\qed

\end{document}

%% file: scat11.pstex_t
\begin{picture}(0,0)%
\includegraphics{scat11.pstex}%
\end{picture}%
\setlength{\unitlength}{4144sp}%
\begingroup\makeatletter\ifx\SetFigFont\undefined%
\gdef\SetFigFont#1#2#3#4#5{%
  \reset@font\fontsize{#1}{#2pt}%
  \fontfamily{#3}\fontseries{#4}\fontshape{#5}%
  \selectfont}%
\fi\endgroup%
\begin{picture}(2949,3129)(709,-2368)
\put(946,-781){\makebox(0,0)[lb]{\smash{{\SetFigFont{12}{14.4}{\rmdefault}{\mddefault}{\updefault}{\color[rgb]{0,0,0}$1+A_1^{-1}$}%
}}}}
\put(2206,254){\makebox(0,0)[lb]{\smash{{\SetFigFont{12}{14.4}{\rmdefault}{\mddefault}{\updefault}{\color[rgb]{0,0,0}$1+A_2$}%
}}}}
\put(3196,-1816){\makebox(0,0)[lb]{\smash{{\SetFigFont{12}{14.4}{\rmdefault}{\mddefault}{\updefault}{\color[rgb]{0,0,0}$1+A_1^{-1}A_2$}%
}}}}
\end{picture}%

%% file: scat13.pstex_t
\begin{picture}(0,0)%
\includegraphics{scat13.pstex}%
\end{picture}%
\setlength{\unitlength}{4144sp}%
\begingroup\makeatletter\ifx\SetFigFont\undefined%
\gdef\SetFigFont#1#2#3#4#5{%
  \reset@font\fontsize{#1}{#2pt}%
  \fontfamily{#3}\fontseries{#4}\fontshape{#5}%
  \selectfont}%
\fi\endgroup%
\begin{picture}(3042,3264)(616,-2503)
\put(2206,254){\makebox(0,0)[lb]{\smash{{\SetFigFont{12}{14.4}{\rmdefault}{\mddefault}{\updefault}{\color[rgb]{0,0,0}$1+A_2^3$}%
}}}}
\put(631,-691){\makebox(0,0)[lb]{\smash{{\SetFigFont{12}{14.4}{\rmdefault}{\mddefault}{\updefault}{\color[rgb]{0,0,0}$1+A_1^{-1}$}%
}}}}
\end{picture}%

%% file: diagram33.pstex_t
\begin{picture}(0,0)%
\includegraphics{diagram33.pstex}%
\end{picture}%
\setlength{\unitlength}{4144sp}%
\begingroup\makeatletter\ifx\SetFigFont\undefined%
\gdef\SetFigFont#1#2#3#4#5{%
  \reset@font\fontsize{#1}{#2pt}%
  \fontfamily{#3}\fontseries{#4}\fontshape{#5}%
  \selectfont}%
\fi\endgroup%
\begin{picture}(3624,3354)(664,-2548)
\end{picture}%

%% file: brokenline11.pstex_t
\begin{picture}(0,0)%
\includegraphics{brokenline11.pstex}%
\end{picture}%
\setlength{\unitlength}{4144sp}%
\begingroup\makeatletter\ifx\SetFigFont\undefined%
\gdef\SetFigFont#1#2#3#4#5{%
  \reset@font\fontsize{#1}{#2pt}%
  \fontfamily{#3}\fontseries{#4}\fontshape{#5}%
  \selectfont}%
\fi\endgroup%
\begin{picture}(3894,3669)(394,-2863)
\put(3241,-331){\makebox(0,0)[lb]{\smash{{\SetFigFont{12}{14.4}{\rmdefault}{\mddefault}{\updefault}{\color[rgb]{0,0,0}$Q$}%
}}}}
\put(2431,344){\makebox(0,0)[lb]{\smash{{\SetFigFont{12}{14.4}{\rmdefault}{\mddefault}{\updefault}{\color[rgb]{0,0,0}$A_1^{-1}A_2$}%
}}}}
\put(2116,-286){\makebox(0,0)[lb]{\smash{{\SetFigFont{12}{14.4}{\rmdefault}{\mddefault}{\updefault}{\color[rgb]{0,0,0}$1+A_2^2$}%
}}}}
\put(3421,-601){\makebox(0,0)[lb]{\smash{{\SetFigFont{12}{14.4}{\rmdefault}{\mddefault}{\updefault}{\color[rgb]{0,0,0}$A_1A_2^{-1}$}%
}}}}
\put(2296,-601){\makebox(0,0)[lb]{\smash{{\SetFigFont{12}{14.4}{\rmdefault}{\mddefault}{\updefault}{\color[rgb]{0,0,0}$A_1^{-1}A_2^{-1}$}%
}}}}
\put(1101,-781){\makebox(0,0)[lb]{\smash{{\SetFigFont{12}{14.4}{\rmdefault}{\mddefault}{\updefault}{\color[rgb]{0,0,0}$1+A_1^{-2}$}%
}}}}
\end{picture}%

%% file: brokenline22.pstex_t
\begin{picture}(0,0)%
\includegraphics{brokenline22.pstex}%
\end{picture}%
\setlength{\unitlength}{4144sp}%
\begingroup\makeatletter\ifx\SetFigFont\undefined%
\gdef\SetFigFont#1#2#3#4#5{%
  \reset@font\fontsize{#1}{#2pt}%
  \fontfamily{#3}\fontseries{#4}\fontshape{#5}%
  \selectfont}%
\fi\endgroup%
\begin{picture}(3894,3669)(394,-2863)
\put(3241,-331){\makebox(0,0)[lb]{\smash{{\SetFigFont{12}{14.4}{\rmdefault}{\mddefault}{\updefault}{\color[rgb]{0,0,0}$Q$}%
}}}}
\put(2431,344){\makebox(0,0)[lb]{\smash{{\SetFigFont{12}{14.4}{\rmdefault}{\mddefault}{\updefault}{\color[rgb]{0,0,0}$A_1^{-2}A_2^2$}%
}}}}
\put(3421,-601){\makebox(0,0)[lb]{\smash{{\SetFigFont{12}{14.4}{\rmdefault}{\mddefault}{\updefault}{\color[rgb]{0,0,0}$A_1^2A_2^{-2}$}%
}}}}
\put(2251,-331){\makebox(0,0)[lb]{\smash{{\SetFigFont{12}{14.4}{\rmdefault}{\mddefault}{\updefault}{\color[rgb]{0,0,0}$2A_1^{-2}$}%
}}}}
\put(2971,-736){\makebox(0,0)[lb]{\smash{{\SetFigFont{12}{14.4}{\rmdefault}{\mddefault}{\updefault}{\color[rgb]{0,0,0}$2A_2^{-2}$}%
}}}}
\put(2251,-691){\makebox(0,0)[lb]{\smash{{\SetFigFont{12}{14.4}{\rmdefault}{\mddefault}{\updefault}{\color[rgb]{0,0,0}$A_1^{-2}A_2^{-2}$}%
}}}}
\end{picture}%